\documentclass[10pt,reqno]{amsart}

\usepackage{graphicx}
\usepackage{tabularx}
\graphicspath{ {./images/} }
\usepackage{xargs}
\usepackage{geometry}
\usepackage{fancyhdr} 
\usepackage{lastpage} 
\usepackage{extramarks} 
\usepackage[most]{tcolorbox} 
\usepackage{xcolor} 
\usepackage[hidelinks]{hyperref} 
\usepackage[permil]{overpic}
\usepackage{pict2e} 
\usepackage{listings}
\usepackage{amssymb}
\usepackage{amsthm}
\theoremstyle{plain}
\usepackage[
backend=biber,
style=numeric,
sorting=ynt
]{biblatex}
\usepackage{xargs}
\usepackage{float}
\usepackage{lipsum}
\usepackage{geometry}
\usepackage{fancyhdr} 
\usepackage{lastpage} 
\usepackage{extramarks} 
\usepackage[most]{tcolorbox} 

\usepackage{xcolor} 
\usepackage[hidelinks]{hyperref} 
\usepackage[permil]{overpic}
\usepackage{pict2e} 
\usepackage{listings}

\newtheorem*{theorem*}{Theorem}

\theoremstyle{notation}

\numberwithin{equation}{section}
\theoremstyle{plain}
\newtheorem{definition}{Definition}[section]
\newtheorem{theorem}[equation]{Theorem}
\newtheorem{claim}[equation]{Theorem}
\newtheorem{remark}[equation]{Remark}
\newtheorem{corollary}[equation]{Corollary}
\newtheorem{lemma}[equation]{Lemma}
\newtheorem{conj}[equation]{Conjecture}
\newtheorem{proposition}[equation]{Proposition} 
\providecommand{\keywords}[1]
{
  \small	
  \textbf{\textit{Keywords---}} #1
}

\title{Essential dynamics in chaotic attractors}
\author{Eran Igra}
\address{Shanghai Institute for Mathematics and Interdisciplinary Sciences}
\email{eranigra@simis.cn}

\begin{document}

\begin{abstract}
We prove that if a smooth vector field $F$ of $S^3$ generates a sufficiently complicated heteroclinic knot, the flow also generates infinitely many periodic orbits, which persist under smooth perturbations which preserve the heteroclinic knot. Consequentially, we then associate a Template with the flow dynamics - regardless of whether $F$ satisfies any hyperbolicity condition or not. In addition, inspired by the Thurston-Nielsen Classification Theorem, we also conclude topological criteria for the existence of chaotic dynamics for three-dimensional flows - which we apply to study both the Rössler and Lorenz attractors.
\end{abstract}

\maketitle
\section{Introduction}

Consider a homeomorphism $f:D\to D$ where $D$ is an open disc, and let $\{x_1,...,x_n\}\subseteq D$, $n>1$ be some $f-$invariant set (for example, a collection of periodic orbits). Set $S=D\setminus\{x_1,...,x_n\}$ - it is well-known the dynamics of $f:S\to S$ are constrained by how $f$ permutes $\{x_1,...,x_n\}$, a result often referred to as the "Thurston-Nielsen Classification Theorem" or the "Betsvina-Handel Algorithm" (see \cite{Fat} and \cite{BeH}, respectively). In more detail, the action of $f$ on an embedded graph $\Gamma\subseteq D$ s.t. $\{x_1,...,x_n\}\subseteq \Gamma$ can teach us almost everything there is to know on the dynamics of $f$ - for example, it teaches us if $f$ generates complex dynamics which include infinitely many periodic orbits, which persist under isotopies.\\

Now, assume that instead of a two-dimensional surface homeomorphism we have a three-dimensional flow on $S^3$ generated by a smooth vector field $F$ - can the behavior of $F$ on some one dimensional set force the flow to have infinitely many periodic orbits? Even though heuristically flows are suspended homeomorphisms there is no analogue to the Betsvina-Handel algorithm for flows - and it is an open question how and whether we can actually generalize these ideas from discrete-time two-dimensional dynamics to three dimensional flows (see\cite{M} for a survey). It is precisely this question we address in this paper - inspired by the numerical results of \cite{MBKPS}, \cite{BBS} and \cite{Le} and by the results of \cite{Pi} in this paper we give a partial answer to the question above. In more detail, we prove that whenever $F$ is a smooth vector field of $S^3$ whose fixed points are all connected by a heteroclinic knot $H$, then under some topological assumptions on $H$ there exist complex dynamics for which persist under homotopies of $F$ in $S^3\setminus H$.\\

\begin{figure}[h]
\centering
\begin{overpic}[width=0.15\textwidth]{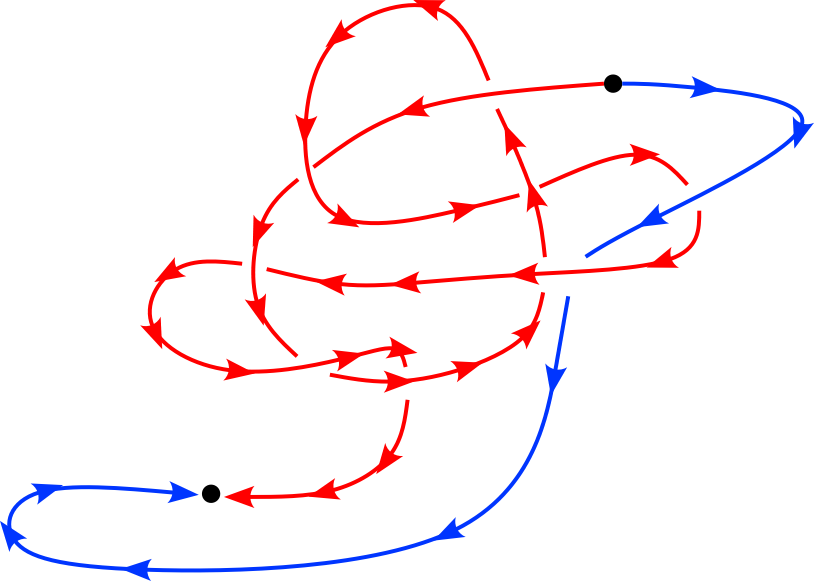}
\end{overpic}
\caption[Fig1]{\textit{A heteroclinic knot between two fixed points.}}\label{hetro1}
\end{figure} 

The importance of these results is that to our knowledge this is the first time such a persistence result had been proven for three-dimensional flows. In addition, our results imply a deep connection between the topology generated by a given three-dimensional flow $F$ and the dynamics it can generate - as such, these results can be seen as s step towards extending forcing theory from one and two-dimensional discrete-time dynamics (see \cite{Shar}, \cite{Yor} and \cite{Bo}) for three-dimensional flows. In a more applied context, our results can also be used to prove the existence of complex dynamics for flows where no hyperbolicity conditions are known to be satisfied - which we exemplify by studying the existence and persistence of complex dynamics in two famous models for chaos: the Rössler and the Lorenz attractors (see \cite{Ross76} and \cite{L}, respectively). As such, our methods can probably be applied to a large class of three-dimensional flows.\\

With these ideas in mind we are now ready to state the major results of this paper. To begin, let $F$ be a $C^k$ vector field of $S^3$ (where $k\geq3$) s.t. $F$ has a finite number of fixed points - all non-degenerate. By the \textbf{topological type} of a non-degenerate fixed point we refer to the local dynamics around it - i.e., a sink, source, a center, saddle, or a saddle focus. Moreover, assume $F$ generates a heteroclinic knot $H$ which connects all the fixed points of $F$ as in Fig.\ref{hetro1} - then, in Section \ref{perss} we prove the following result (see Th.\ref{orbipers} and Th.\ref{pers13}):

\begin{claim}
 \label{th1}   Let $F$ and $H$ be as above, and set $M=S^3\setminus H$. Moreover, assume there exists a smooth vector field $G$ on $S^3$ s.t. the following is satisfied:
\begin{itemize}
    \item $F$ is smoothly homotopic to $G$ in $M$, and for all $s\in M$ we have $G(s)\ne0$.
    \item If $x\in H$ is a fixed point for $F$ it is also a fixed point for $G$ and vice versa - and moreover, the topological type of $x$ is the same for both $F$ and $G$.
    \item There exists two cross-section $S_1\subseteq S_2$ (not necessarily connected) s.t. the following holds:
    \begin{enumerate}
        \item $\overline{S_1}\cap H$ is non-empty lies on the boundary of $S_1$ in $S_2$.
        \item The first-return map $g:S_1\to S_2$ is continuous and extends continuously to $\overline{S_1}$.
        \item Let $I$ denote the invariant set of $g$ - then, $g:I\to I$ is conjugate to either a Smale Horseshoe map or a Fake Horseshoe map.
    \end{enumerate}
\end{itemize}

Then, $F$ generates infinitely many periodic orbits. Moreover, for a generic choice of $F$ these periodic orbits persist under all sufficiently small $C^3$ perturbations of $F$ - and moreover, they do so without changing their knot type.
\end{claim}

Despite the somewhat technical phrasing, Th.\ref{th1} has the following meaning: whenever the dynamics around a heteroclinic knot $H$ can be smoothly deformed s.t. the heteroclinic orbits on $H$ suspend a horseshoe map, the dynamics of $F$ would also be essentially those of a "deformed" horseshoe map. The proof of Th.\ref{th1} is performed in two parts in Section \ref{perss} - we first prove $F$ generates infinitely many periodic orbits using the Betsvina Handel Algorithm, which we do by embedding the dynamics of $g$ inside those of a Pseudo-Anosov map (see Th.\ref{orbipers}). Following that we prove the persistence of periodicity by using the Orbit Index as introduced in \cite{PY}, \cite{PY2} and \cite{PY4}. At this point we remark the Orbit Index Theory had been applied extensively to study the existence and persistence of complex dynamics for flows (and their applications) - for a non-exhaustive list, see \cite{GG}, \cite{KW}, \cite{KB}, \cite{CS},\cite{B1},\cite{B2}, \cite{KY3}, \cite{B3}, \cite{J1}, \cite{J2} and the references therein. \\

\begin{figure}[h]
\centering
\begin{overpic}[width=0.5\textwidth]{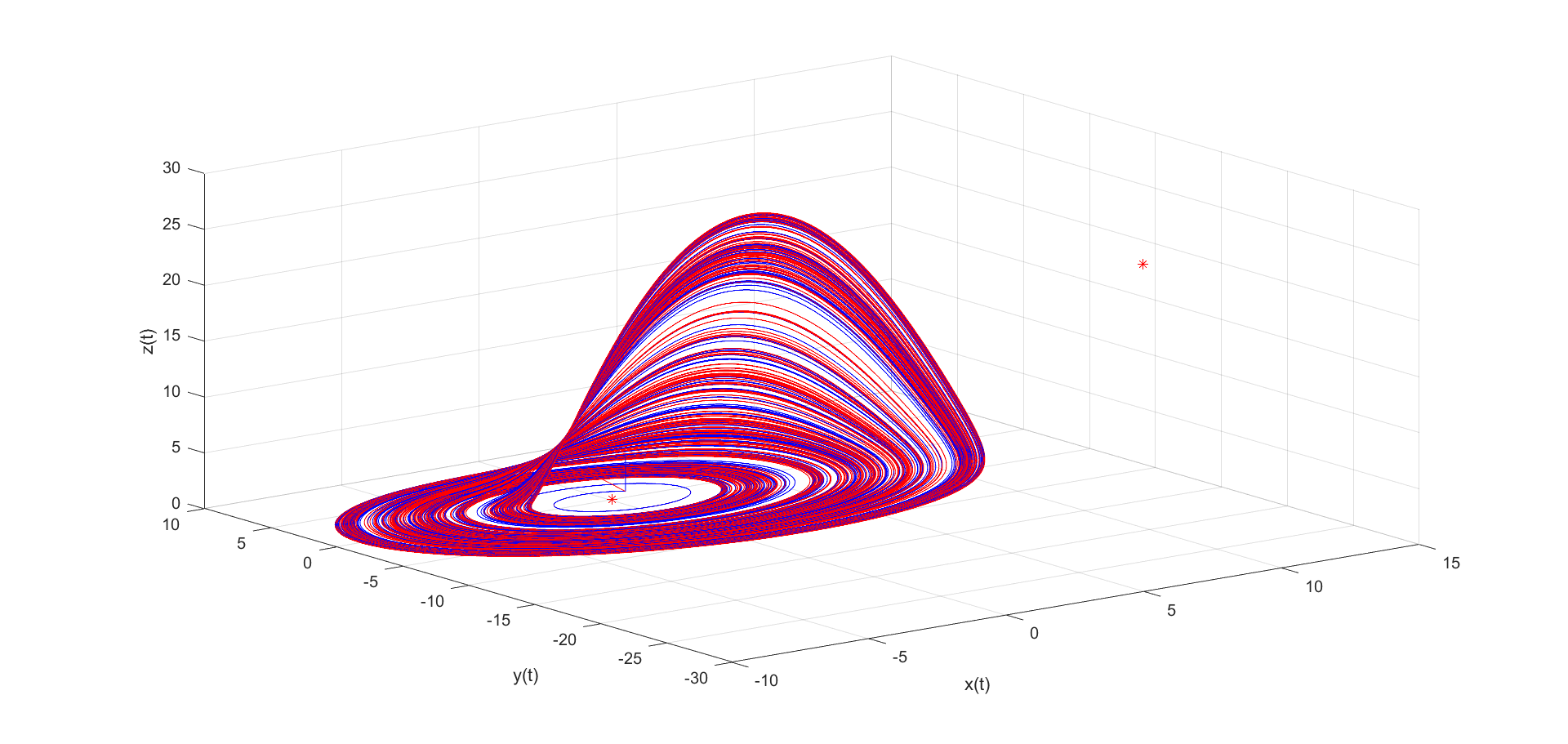}
\end{overpic}
\caption[Fig1]{\textit{The Rössler attractor in parameters $(a,b,c)=(0.2,0.2,5.7)$ (see Eq.\ref{Vect}).}}
\end{figure}
Having proven Th.\ref{th1}, we then give examples how it can be applied to study the dynamics of three dimensional flows. Inspired by \cite{MBKPS} we begin by applying Th.\ref{th1} to study the Rössler system (see \cite{Ross76}). In more detail, we apply Th.\ref{th1} to prove the following Theorem (see Th.\ref{trefoilor} and Th.\ref{trefoil1}):

\begin{claim}
    \label{th3}
    Assume $F$ is a smooth vector field of $S^3$ which generates a heteroclinic trefoil knot configured as in Fig.\ref{trefoil} - then, $F$ generates infinitely many periodic orbits, of infinitely many knot types. Moreover, whenever the Rössler system generates both a compact attractor which traps a heteroclinic orbit as in Fig.\ref{trefoil}, the same conclusion also holds for it.
\end{claim}
The main ingredient in the proof of Th.\ref{th3} is Th.\ref{orbipers} and Th.6.1 in \cite{Hol}, and more generally, Template Theory for three-dimensional flows (see \cite{KNOTBOOK} for a survey). As such, the proof of Th.\ref{th3} illustrates precisely how one can use Th.\ref{th1} to prove the existence of complex dynamics for vector fields of $\mathbf{R}^3$: first we isolate the dynamics inside the attractor away from $\infty$, after which we smoothly deform the flow inside the attractor to some idealized model whose dynamics are easier to analyze. Then, applying Th.\ref{th1} we conclude the everything proven for the idealized model also holds for the original flow.\\

\begin{figure}[h]
\centering
\begin{overpic}[width=0.4\textwidth]{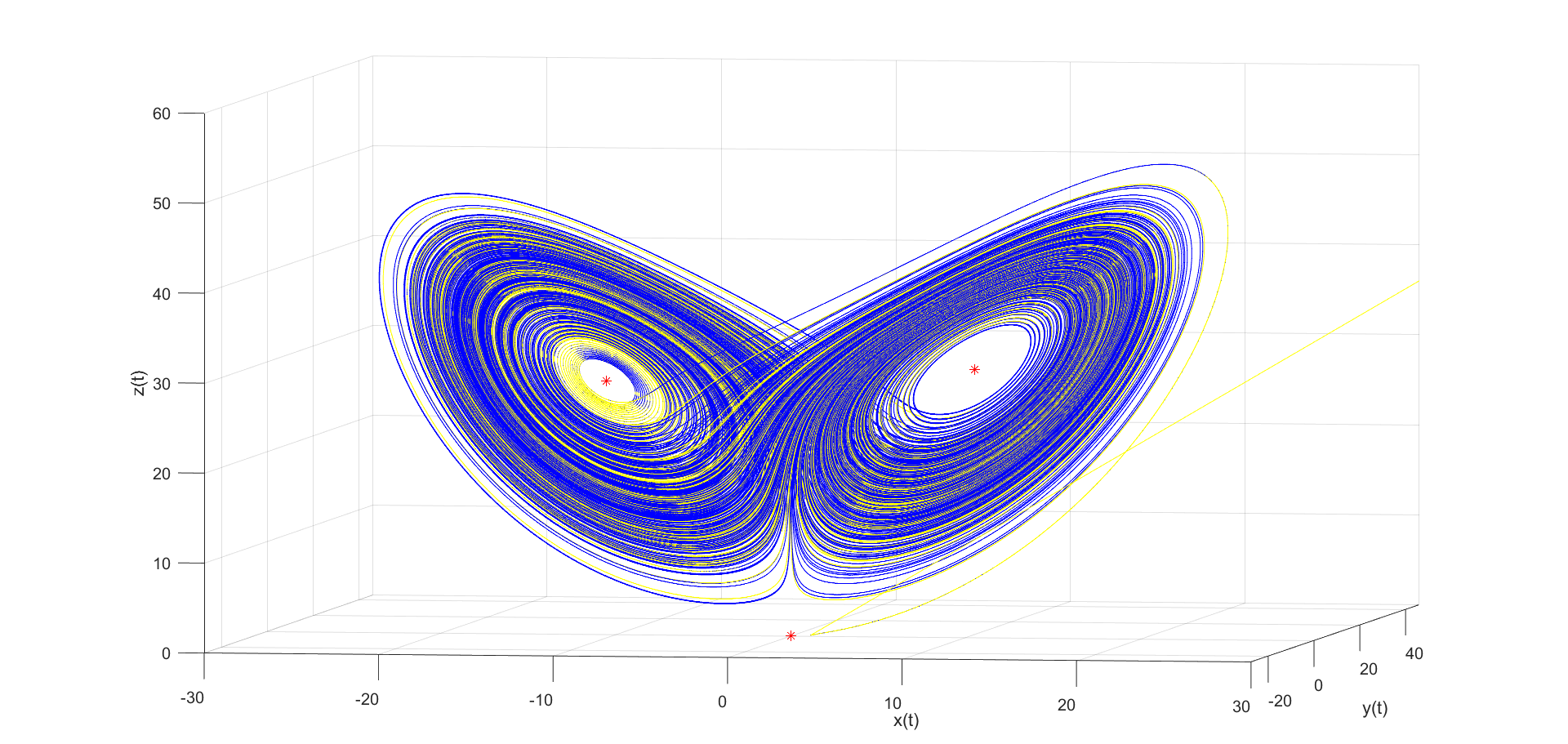}
\end{overpic}
\caption[Fig1]{\textit{The Lorenz attractor in parameters $(\beta,\rho,\sigma)=(3,30,12)$ (see Eq.\ref{Vect11}).}}
\end{figure} 

Having applied Th.\ref{orbipers} to study the Rössler system we then apply it (along with other, more general aspects of Orbit Index Theory) to study the Lorenz system (see \cite{L}). Using Th.\ref{th1}, the results of \cite{Pi} as well as Orbit Index Theory we prove the following (see Th.\ref{persistence} and Th.\ref{trefoil2}):
\begin{claim}
   \label{th4}
   Assume $F$ is a smooth vector field of $S^3$ which generates a heteroclinic trefoil knot as in Fig.\ref{heteroclo}. Then, $F$ generates infinitely many periodic orbits, of infinitely many different knot types - and furthermore, if $F$ is a Lorenz system given by Eq.\ref{Vect11}, every one of these periodic orbits persists under sufficiently small $C^k$ perturbations of $F$ (where $k\geq3$) - without changing its knot type. 
\end{claim}
 In the context of the Lorenz system Th.\ref{th4} implies the complex dynamics of the Lorenz attractor generated in the scenario of a heteroclinic trefoil knot persist under sufficiently small $C^k$ perturbations. As heteroclinic scenarios are known to exist for the Lorenz system (see Th.1.1 in \cite{Pi}) and are also known numerically to be the epicenter of complexity in the parameter space (see \cite{SR}), Th.\ref{th4} further clarifies how the importance of heteroclinic dynamics to the onset of chaos in the Lorenz model.\\

This paper is organized as follows: in Section $2$ we survey several preliminaries from the theory of Topological Surface Dynamics, Template Theory, and the basics of the basics of the Orbit Index Theory - all of which will be used throughout this paper. Following that, we devote Section \ref{perss} to the proof of Th.\ref{th1} (and its corollaries), after which we study the Rössler and Lorenz systems in Sections \ref{rossler} and \ref{lorenz} (respectively). Finally, in Section \ref{horsus} we study the implications of Th.\ref{th1} in "simpler" scenarios, i.e., where there is no infinite collection of homotopy-invariant periodic orbits. We conclude this paper by discussing how our results can be further generalized (see Conjecture \ref{genthunie}) and how they possibly relate to the well-known "Chaotic Hypothesis" (see \cite{gal}).\\

Before we begin we remark that in \cite{sixu} several results which mirror our own were proven. In more detail, it was proven that the growth rate of periodic orbits - a constant which measures the number of isolated periodic orbits for a given flow - remains constant under Lipschitz orbital equivalencies of vector fields. As such, this study (combined with \cite{Pi} and \cite{sixu}) further shows how periodic orbits can be used to study the complexity of given flows, even in the absence of any hyperbolicity conditions. In addition, we further remark that in \cite{CSS} related questions on periodic orbits were considered, albeit in the context of Anosov flows. Finally, we state that even though it does not appear so from the text our results were very much inspired by \cite{By}, \cite{GKP}, \cite{ST1} and \cite{Bely} (and more generally, by the theory of Homoclinic Bifurcations) - in fact, Th.\ref{th1} originated in an attempt to study how periodic orbits are destroyed after the breakdown of a structurally unstable heteroclinic knot connecting saddle-foci.

\subsection*{Acknowledgements}
The author would like to thank Tali Pinsky, Genadi Levin, Joshua Haim Mamou, Michael Faran, Łukasz Cholewa, and Noy Soffer-Aranov for their helpful comments and suggestions. 
\section{Preliminaries}
In this section we review the facts used to prove the results of this paper. This section is organized as follows - we begin with Section \ref{top2} where we review how topology can force isotopy-stable dynamics for surface homeomorphisms. Following that in Section \ref{templatere} we review the basic ideas of Template Theory, and recall several facts about the hyperbolic dynamics for three dimensional flows. Finally, in Section \ref{orbitin} we survey the Orbit Index Theory as given in \cite{PY}, \cite{PY2} and \cite{PY3} (among others). As will be made clear, the theme in every one of the topics surveyed below is how one can study a given dynamical system - either in discrete or continuous time - by treating its periodic dynamics as a topological invariant.

\subsection{Two-dimensional surface dynamics}
\label{top2}
As stated above in this section we give a brief survey for how topology forces dynamics for surface homeomorphism, following the ideas presented in \cite{BeH}, \cite{Han} and \cite{Bo}. We will only survey how topology can force the persistence of periodic dynamics within the isotopy class of a given surface homeomorphism - for a more extensive survey of this topic we refer the reader to \cite{Bo}. To begin, in this section $S$ will always denote a surface homeomorphic to an open disc $D$ punctured at some interior points $\gamma=\{x_1,...,x_n\},n\geq2$ - and $h:D\to D$ would always denote a homeomorphism permuting the elements of $\gamma$ (it is easy to see $h$ is also a homeomorphism of $D$ - see the illustration in Fig.\ref{relative}). We begin with the notion of a relative isotopy:

\begin{definition}
    \label{relp}
We say a homeomorphism $g:D\to D$ is \textbf{isotopic to $h$ relative to} $\gamma$ (or in short, $rel$ $\gamma$) if there exists an isotopy $f_t:D\to D$, $t\in[0,1]$ of continuous maps satisfying the following (see the illustration in Fig.\ref{relative}):
    \begin{itemize}
        \item $f_0=h$, $f_1=g$.
        \item For every $s,t\in[0,1]$ $f_t$ and $f_s$ permute $\gamma$ in the exact same way.
    \end{itemize}
\end{definition}
\begin{figure}[h]
\centering
\begin{overpic}[width=0.6\textwidth]{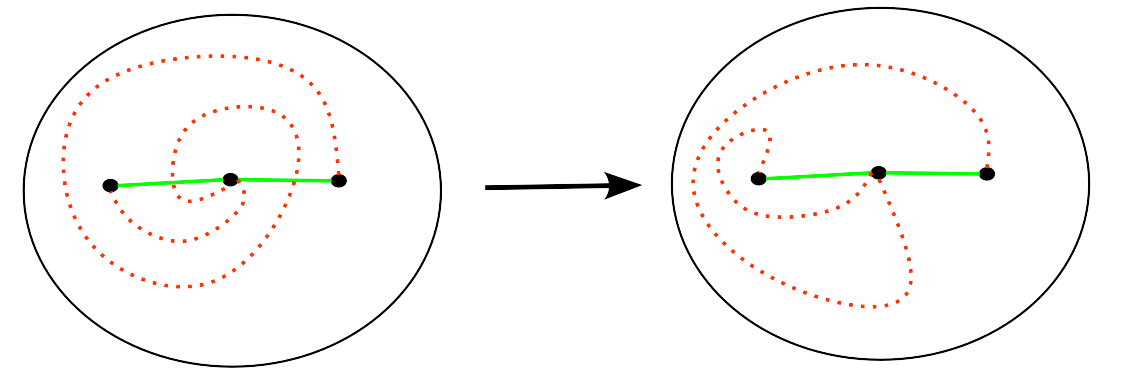}
\put(690,190){$x_1$}
\put(830,200){$x_3$}
\put(760,200){$x_2$}
\put(100,190){$x_1$}
\put(180,200){$x_{2}$}
\put(260,200){$x_3$}
\end{overpic}
\caption{\textit{The homeomorphisms $h$ (on the left) and $g$ (on the right) are isotopic on $D$ $rel$ $\gamma$, where $\gamma=\{x_1,x_2,x_3\}$ (in this scenario, $h(x_1)=g(x_1)=x_2$, $h(x_2)=g(x_2)=x_1$ and $x_3$ remains fixed). This is exemplified by how they distort the green curves connecting the elements in $\gamma$ (i.e., the spine of $S$ - see Def.\ref{spi}) -whose respective images under $h$ and $g$ are the dashed red lines.}}\label{relative}

\end{figure}
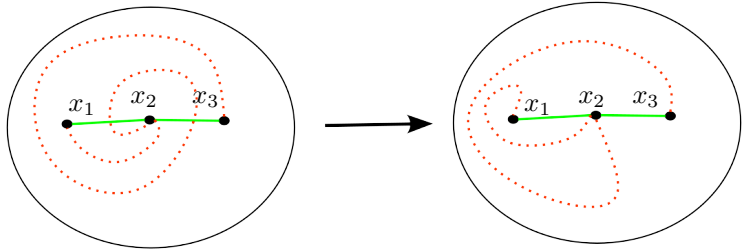

From now on, we treat relatively isotopic $h$ and $g$ as surface homeomorphisms, i.e., as homeomorphisms of $S$. We will now see that if $h$ and $g$ are isotopic $rel$ $\gamma$ as described above, their dynamics are constrained by the way they permute the elements in $\gamma$ (see the illustration in Fig.\ref{relative}). To do so, we first introduce the following definition:

\begin{definition}
    \label{pseanosov} A homeomorphism $h:S\to S$ is \textbf{Pseudu-Anosov} provided there exist two foliations of $S$, $F^u$ and $F^s$, transverse to one another throughout $S$ (but not necessarily at the punctures $\{x_1,...,x_n\}$ - see the illustration in Fig.\ref{trannn}) and some $\lambda>0$ s.t. the following is satisfied:
    \begin{itemize}
        \item Both $F^s$ and $F^u$ are measured - i.e., if we move some leaf $L_1$ of $F^i$ to another leaf $L_2$ of $F^i$ by some isotopy of $S$ (where $i\in\{u,s\}$), the Borel measure on $L_2$ is the pushforward of the Borel measure on $L_1$.
        \item $h(F^u)=\lambda F^u$ while $h(F^s)=\frac{1}{\lambda}F^s$ - i.e., $h$ stretches uniformly the unstable foliation $F^u$ and squeezes uniformly the stable foliation $F^s$. We refer to $\lambda$ as the \textbf{expansion constant}.
    \end{itemize}
\end{definition}

\begin{figure}[h]
\centering
\begin{overpic}[width=0.4\textwidth]{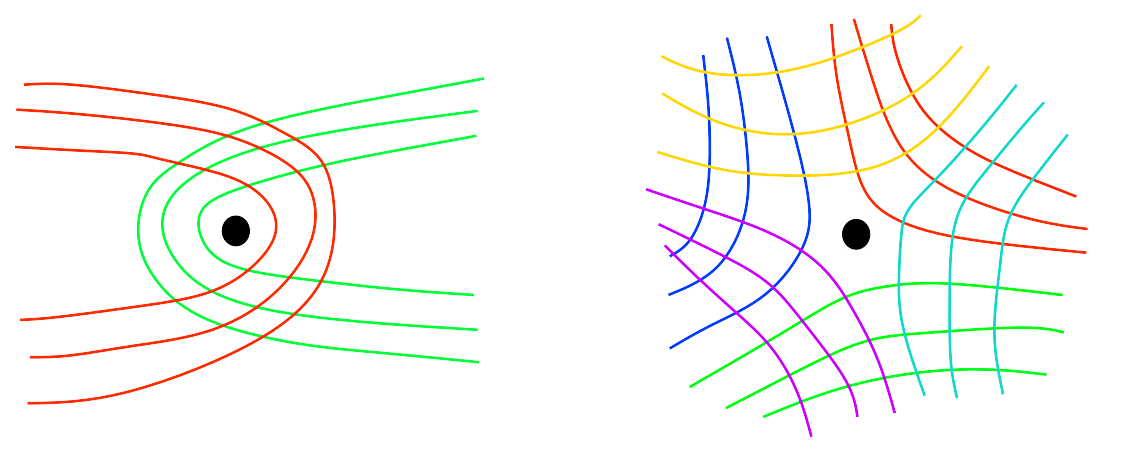}
\end{overpic}
\caption[Transverse folliations.]{\textit{Transverse foliations (with singularities) around the punctures of $S$ (i.e. the black dots). }}
\label{trannn}
\end{figure}

Pseudo-Anosov maps are essentially a generalized form of hyperbolic diffeomorphisms like Smale's Horseshoe (see \cite{S}) - that is, they contract uniformly in one direction and expand uniformly in another. The reason we are interested in Pseudo-Anosov maps is because as far as their dynamics are concerned we have the following result (see Th.$7.2$ in \cite{Bo}):

\begin{theorem}
    Assume $h:S\to S$ is Pseudo-Anosov - then, $h$ has infinitely many periodic orbits in $S$.
\end{theorem}
In addition, Pseudo-Anosov maps are dynamically minimal in the following sense - if $h:S\to S$ is a homeomorphism isotopic to a Pseudo-Anosov map $g:S\to S$, then the dynamics of $g$ are complex at least like those of $h$. More precisely, we have the following result proven in Th.2 and Remark 2 in \cite{Han}:

\begin{theorem}
\label{stability}    Let $g:S\to S$ be a Pseudo-Anosov map and let $h:S\to S$ be a homeomorphism isotopic to $g$. Then, there exists a closed set $Y\subseteq S$ and a continuous, surjective $\pi:Y\to S$ s.t. $\pi\circ h=g\circ \pi$. In particular, we have the following:

\begin{itemize}
    \item If $x$ is periodic of minimal period $k$ for $g$, then as $g$ is isotoped to $h$ the point $x$ is continuously deformed to $y$ - a periodic point for $h$ of minimal period $k$. That is, the periodic dynamics of $g$ are \textbf{unremovable} - i.e., they are not destroyed by an isotopy.
    \item  If $x\in S$ is periodic of minimal period $n$ for $g$, then $\pi^{-1}(x)$ includes at least one periodic orbit of minimal period $n$ for $h$ - i.e., when we isotope $g$ to $h$, no two periodic orbits of $h$ collapse into one another by a bifurcation. That is, the periodic dynamics of $g$ are \textbf{uncollapsible}.
\end{itemize}
\end{theorem}
In other words, per Th.\ref{stability} we know the dynamical complexity of a Pseudo-Anosov map in $g$ serves as a "lower bound" of sorts for the possible dynamical complexity of any $h$ isotopic to it - or, in other words, the dynamics of Pseudo-Anosov maps are \textbf{dynamically minimal}. This can be stated in a more precise manner - to do so we first introduce the following notion:

\begin{definition}
\label{essential2}    Let $h:S\to S$ be a surface homeomorphism - then the \textbf{Essential Class} of $h$ would be the collection of periodic orbits which persist under isotopies of $h$ without changing their minimal period or collapsing into one another.
\end{definition}
It is easy to see Th.\ref{stability} can be restated as follows: whenever $h:S\to S$ is isotopic the Pseudo-Anosov $g:S\to S$ then the essential class of $h$ includes infinitely many periodic orbits. With these ideas in mind we are now led to the following question - given a homeomorphism $h:S\to S$, can we tell if it is isotopic to a Pseudo-Anosov map? To state the answer to that question we first introduce several concepts. We begin with the notion of a spine:

\begin{definition}
  \label{spi}  Recall $S$ is a surface homeomorphic to a disc punctured at $n-$interior points, $n>1$. The \textbf{spine} of $S$ is a graph $\Gamma$ embedded in $\overline{S}$ with $n$ vertices and $n-1$ edges s.t. $\Gamma$ is a retract of $S$ (see the illustration in Fig.\ref{spine}).
\end{definition}

\begin{figure}[h]
\centering
\begin{overpic}[width=0.4\textwidth]{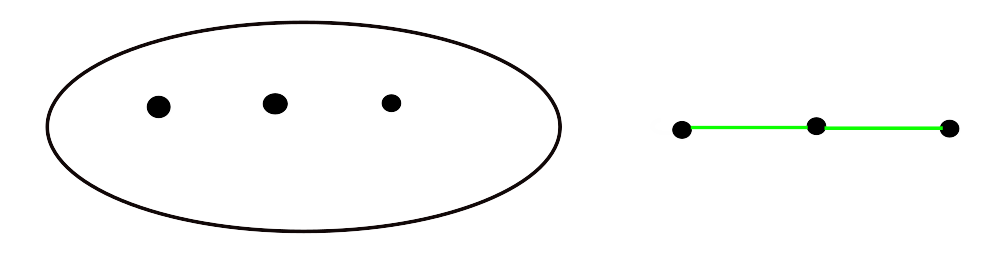}
\end{overpic}
\caption[A spine for a surface.]{\textit{On the left we have a surface $S$ homeomorphic to $D$ punctured at $3$ points, and on the right we have its spine, the graph $\Gamma$. As can be seen, $\Gamma$ has two edges and three vertices.} }
\label{spine}
\end{figure}

To continue, given a spine $\Gamma$ let us denote by $S_1,...,S_{n-1}$ the edges on it (see the illustration in Fig.\ref{spine}), and let $g':\Gamma\to \Gamma$ be some continuous map. We now define a matrix $A=\{a_{i,j}\}_{1\leq i,j\leq n-1}$ s.t. $a_{i,j}$ is the number of times $g(S_j)$ covers $S_i$ - and moreover, we define the \textbf{spectral radius} of $A$ as the maximal eigenvalue for $A$. Now, note that given a homeomorphism $h:S\to S$ by retracting $S$ to its spine $\Gamma$ we can collapse $h$ to a graph map $g':\Gamma\to\Gamma$ (see the illustration in Fig.\ref{cover2}). As a consequence, there exists a matrix $A$ as defined above associated with $h$ via $g'$, the induced graph map, with a spectral radius $\gamma$. With these ideas in mind we now state the following fact, proven in Sections $3.4$ and $4.4$ of \cite{BeH}:

\begin{theorem}
    \label{betshan} Let $h:S\to S$ be a homeomorphism. Whenever $\gamma>1$ $h$ is isotopic to a Pseudo-Anosov map $g:S\to S$. Moreover, if there exists an edge $S_i$, $1\leq i\leq n-1$ s.t. $g'(S_i)$ covers itself at least twice (i.e., $a_{i,i}>1$), the dynamics of $h$ in the regions of $S$ collapsed to $S_i$ include periodic orbits of all minimal periods.
\end{theorem}
\begin{figure}[h]
\centering
\begin{overpic}[width=0.6\textwidth]{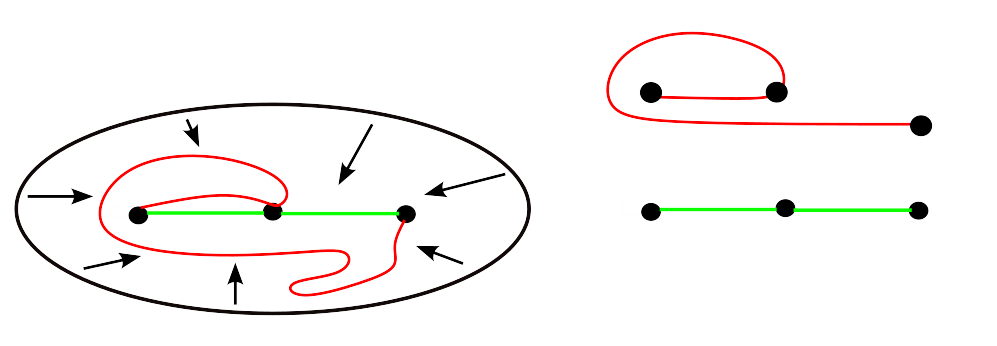}
\put(705,160){$S_1$}
\put(810,160){$S_2$}
\put(770,90){$x_0$}
\put(900,90){$x_1$}
\put(620,90){$x_{-1}$}
\put(960,220){$g'_(x_{-1})$}
\put(810,270){$g'(x_0)$}
\put(510,250){$g'(x_1)$}
\end{overpic}
\caption{\textit{ $h$ distorts the spine $\Gamma$ inside $S$ into the red curve, s.t. $h(x_{-1})=x_1$, $h(x_1)=x_{-1}$ and $h(x_0)=x_0$ (while the outer circle remains fixed). Consequentially, as we collapse $h$ to $g'$, $g'(S_1)$ covers itself twice while $g'(S_2)=S_1$. One can verify the spectral radius in this scenario is $1+\sqrt{2}$.}}\label{cover2}

\end{figure}
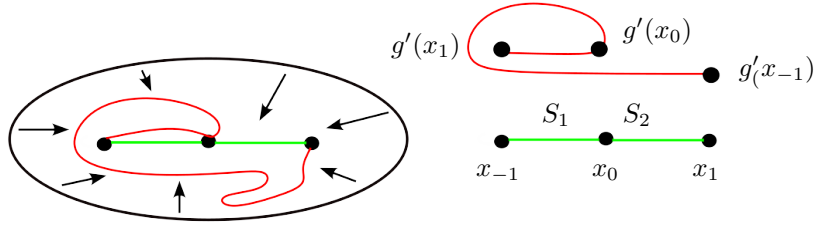

Before concluding this section we remark that Th.\ref{betshan} above is a private case of a more general result, often referred to as the \textbf{Thurston-Nielsen Classification Theorem} (see \cite{BeH} or \cite{Fat}). For completeness, we conclude this section with a statement of the said result in its full generality:

\begin{theorem}
\label{thurston-nielsen}    Let $S$ be a surface with a negative Euler characteristic and let $h:S\to S$ be a homeomorphism. Then, $h$ can be isotoped in $S$ to precisely one of the following possibilities:
    
\begin{enumerate}
    \item A periodic homeomorphism $g:S\to S$ - i.e., there exists some $k>0$ s.t. $g^k=Id$. 
    \item A Pseudo-Anosov homeomorphism $g:S\to S$ with infinitely many periodic orbits which persist under isotopies (see Th.\ref{stability}).
    \item A reducible homeomorphism $g:S\to S$ - i.e., $S$ can be decomposed to a finite number of surfaces $S_1,...,S_n$, glued to one another along some invariant curve(s), s.t. for every $i=1,...,n$ $g|_{S_i}$ is either periodic or Pseudo-Anosov.
\end{enumerate}

\end{theorem}

\subsection{Hyperbolic flows and Template Theory}
\label{templatere}
In this section we review several facts from both Template Theory and the topological dynamics of three-dimensional hyperbolic flows. We begin by recalling several notions from Knot Theory:

\begin{definition}
    \label{knot} A \textbf{knot} is an isotopy class of an embedding of the circle $S^1$ into $\mathbf{R}^3$. We say two knots in $S^3$, $T_1$ and $T_2$ \textbf{have the same knot type}, provided there exists an isotopy of $h_t:S^1\to\mathbf{R}^3$, $t\in[0,1]$ s.t. $h_0(S^1)=T_1$ and $h_1(S^1)=T_2$ (see the illustration in Fig.\ref{typek}). We often refer to such an isotopy as an \textbf{ambient isotopy}. 
\end{definition}

\begin{figure}[h]
\centering
\begin{overpic}[width=0.3\textwidth]{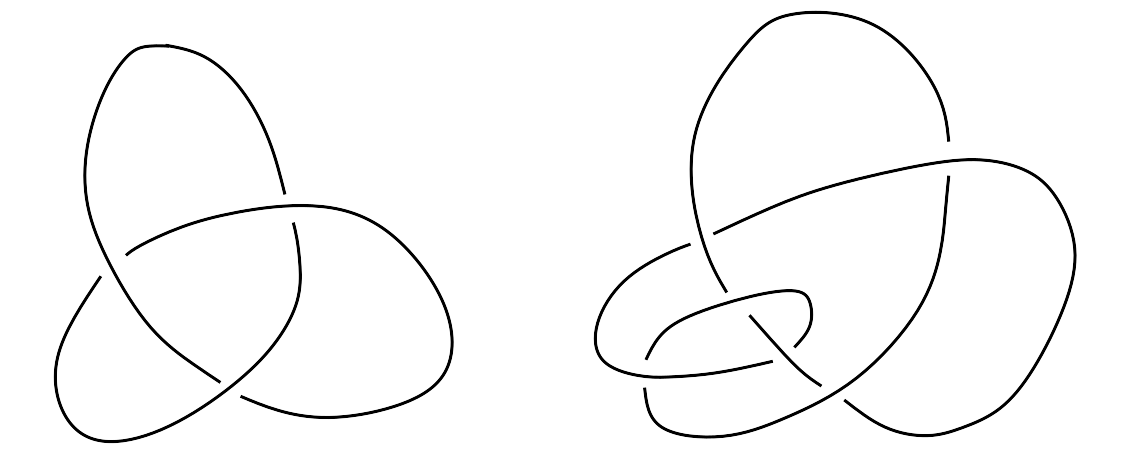}
\end{overpic}
\caption[Two knot types.]{\textit{Two knots, belonging to different knot types. The knot on the left has the same type as the trefoil knot, while that on the right is ambient isotopic to $S^1$, the unknot.}}
\label{typek}
\end{figure}

In addition, we will also need the definition of a Torus knot:

\begin{definition}
\label{torusknot}    A knot is said to be a \textbf{Torus Knot} provided it can be looped on a two-dimensional, unknotted Torus $\mathbf{T}=C_1\times C_2$ - where both $C_1,C_2$ are homeomorphic to $S^1$. We say a Torus knot is a $(p,q)$ \textbf{Torus knot} (where $p,q$ are co-prime integers) provided it winds around the $z-$axis $p$ times and around $C_1$ $q$ times (as illustrated in Fig.\ref{toru}). It is easy to see a $(p,q)$ Torus knot is ambient isotopic to a $(p',q')$ Torus knot if and only if $(p,q)=(p',q')$.
\end{definition}

\begin{figure}[h]
\centering
\begin{overpic}[width=0.5\textwidth]{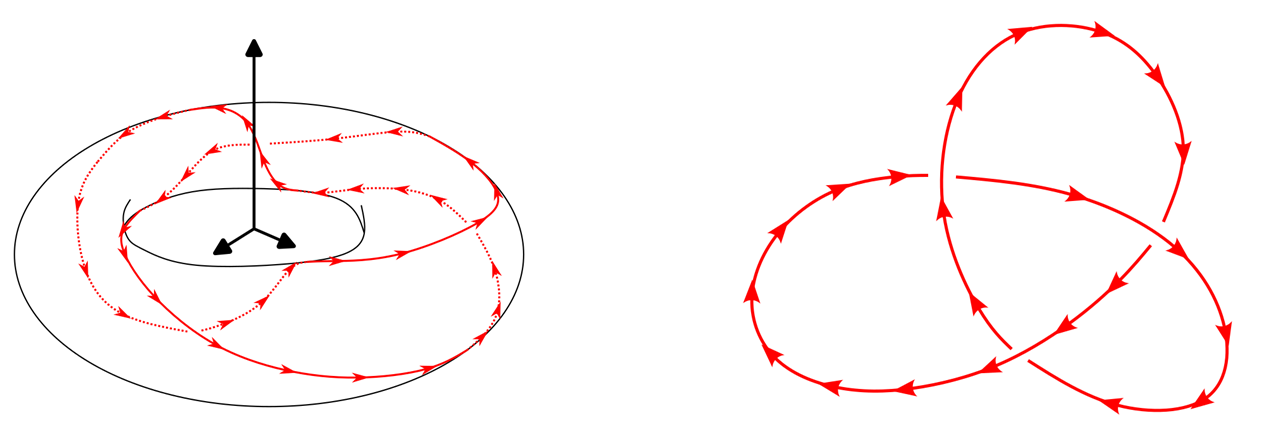}
\put(185,320){$z$}
\put(240,135){$y$}
\put(135,130){$x$}
\end{overpic}
\caption[A Template.]{\textit{A trefoil knot on the right and its representation as a loop on a Torus on the left. As the knot winds $3$ times around one circle and $2$ times around the $z-$axis, it follows the trefoil is a $(2,3)$ Torus knot.}}
\label{toru}
\end{figure}

It is immediate that given any smooth vector field $F$ on $\mathbf{R}^3$ or $S^3$ every periodic solution for the dynamical system $\dot{s}=F(s)$ is a knot. It is well-known in the theory of topological dynamics of surface homeomorphisms that one can study the dynamics by considering the braid types of periodic orbits (for a survey of these results, see \cite{Bo} and \cite{CH}). This leads us to ask the following question - can we do something similar for three-dimensional flows? It turns out that when we impose a some restrictions on the dynamics of a given vector field, this question can be answered - an answer which we review below. To do so we first introduce the notion of hyperbolicity:

\begin{definition}
    \label{hyperbolicvector} Let $M$ be a smooth $3-$manifold, let $F$ be a $C^k, k\geq1$ vector field on $M$, and denote by $\phi_t$ the resulting flow on $M$ (where $t\in\mathbf{R}$). An invariant set $\Lambda$ w.r.t. the flow is said to \textbf{have a} \textbf{hyperbolic structure w.r.t. $F$} or in short, \textbf{hyperbolic w.r.t. $F$}, provided at every $x\in\Lambda$ one can split the tangent space, $T_xM=E^s_x\oplus E^u_x\oplus E^c(x)$ s.t. the following is satisfied:
    \begin{itemize}
        \item $E^c(x)$ is spanned by $F(x)$.
        \item For every $t\in\mathbf{R}$, $E^s_x,E^u_x$ and $E^c_x$ vary continuously to $E^s_{\phi_t(x)},E^u_{\phi_t(x)}$ and $E^c_{\phi_t(x)}$ (respectively) as $x$ flows to $\phi_t(x)$, $t\in\mathbf{R}$. In particular, $D_{\phi_t}(x)E^j_x=E^j_{\phi_t(x)}$ - where $j\in\{s,u,c\}$ and $D_{\phi_t}(x)$ denotes the differential of the time $t$ map for the flow.
        \item There exist constants $C>0$ and $\lambda>1$ s.t. for every $t>0$ and every $x\in\Lambda$ we have: 
        
        \begin{enumerate}
            \item For $v\in E^s_x$, $||D_{\phi_t}(x)v||<Ce^{-\lambda t}||v||$.
            \item For $v\in E^u_x$, $||D_{\phi_t}(x)v||>Ce^{\lambda t}||v||$. 
        \end{enumerate}
    \end{itemize}

In other words, the flow expands (or stretches) uniformly in one invariant direction along the orbit on $\Lambda$, and contracts uniformly in the other invariant direction. 
\end{definition}
We will also need the following definition:

\begin{definition}
    \label{nonwanderchain}
    Let $M$ be a smooth $3-$manifold (possibly with boundary), let $F$ be a $C^k$, $k\geq1$ vector field on $M$, and denote by $\phi_t$ the resulting flow on $M$ (where $t\in\mathbf{R}$). We define the \textbf{chain-recurrent set in $M$}, $\Lambda$, as the collection of all points $x\in M$ s.t. for any $\epsilon>0$ there exists a sequence of points $\{x_1,x_2,...,x_n\}$, $x=x_1=x_n$ and times $\{t_1,...,t_{n-1}\}$ satisfying:
    \begin{itemize}
        \item $t_i>1$, $1\leq i\leq n-1$.
        \item For all $1\leq i\leq n-1$, $||\phi_{t_i}(x_i)-x_{{i+1}}||<\epsilon$.
    \end{itemize}

\end{definition}
When $M$ is a manifold with a boundary s.t. an orbit can escape $M$ in finite time under the flow, the chain-recurrent set includes the maximal set of initial conditions which never escape $M$ (in particular, it includes every periodic orbit in $M$). It is easy to see the chain recurrent set in $M$ is invariant under the flow.\\

\begin{figure}[h]
\centering
\begin{overpic}[width=0.3\textwidth]{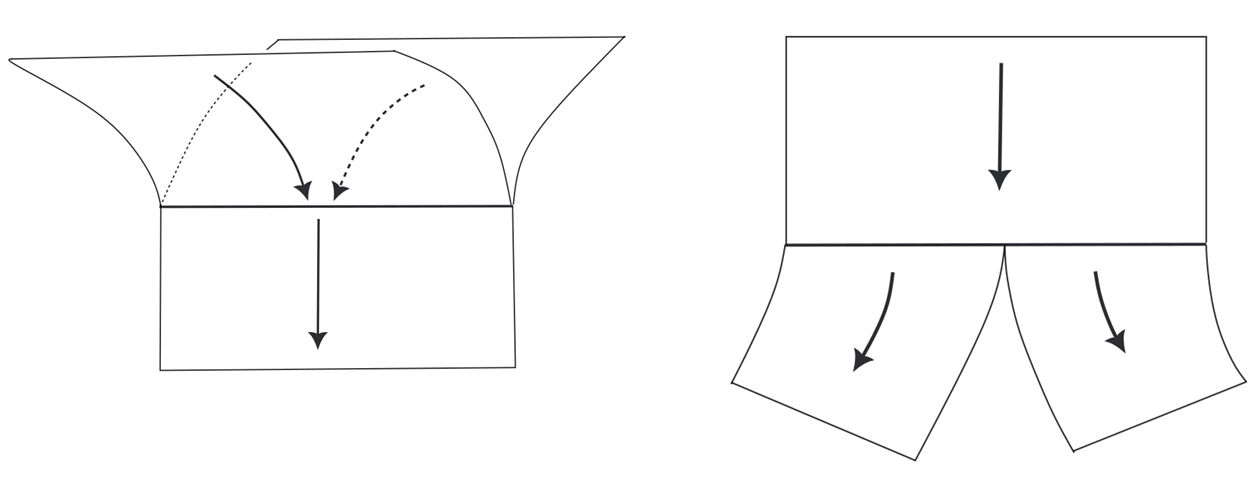}
\end{overpic}
\caption[Joining and splitting charts.]{\textit{A joining chart on the left, and a splitting chart on the right.}}
\label{TEMP1}
\end{figure}
Getting back on topic, as it turns out, given a $3-$manifold $M$ endowed with a flow $\phi_t$ hyperbolic on its chain-recurrent set in $M$ one can describe the dynamical complexity of $\phi_t$ in $M$ - i.e., we can answer the question which knot types $\phi_t$ can and cannot generate as periodic orbits. This is done by reducing the flow to an object called a Template (see Th.\ref{BIRW} below). We first define:

\begin{definition}\label{deftemp}
    A \textbf{Template} is a compact branched surface with a boundary endowed with a smooth expansive semiflow, built locally from two types of charts - \textbf{joining} and \textbf{splitting} (see the illustration in Fig.\ref{TEMP1}). Additionally, the gluing maps between charts all respect the semiflow, and act linearily on the edges (see Fig.\ref{TEMP2} for an illustration).
\end{definition}

\begin{figure}[h]
\centering
\begin{overpic}[width=0.2\textwidth]{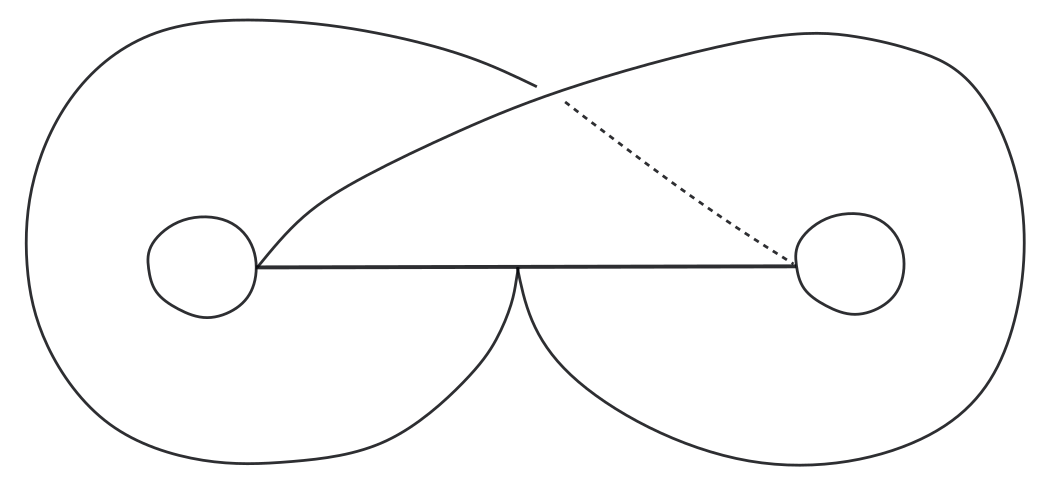}
\end{overpic}
\caption{\textit{A Template, often referred to as the "Lorenz Template" (see \cite{Hol}).}}
\label{TEMP2}
\end{figure}

That is, Templates encode an infinite collections of Knot types - which correspond to the periodic orbits of the semiflow defined on them as illustrated in Fig.\ref{tempknot} (in particular, every Template encodes infinitely many different knot types - see Cor.3.1.14 in \cite{KNOTBOOK}). Now, let $M$ be a three-manifold and let $\phi_t$ be flow with a hyperbolic chain-recurrent set $\Lambda\subseteq M$ (for example, consider a flow generated by some suspension of a Smale Horseshoe inside some solid torus). Moreover, let $P(\Lambda)$ denote the periodic orbits for $\phi_t$ in $\Lambda$ - then, under these assumptions we have the following result often called the \textbf{Birman-Williams Theorem}:
\begin{claim}{\textbf{The Birman-Williams Theorem} -}\label{BIRW}
    Given $M,\phi_t,\Lambda$ as above, there exists a unique Template $T$ embedded in $M$ and endowed with a smooth semiflow $\psi_t$. Moreover, letting $P(T)$ denote the periodic orbits of $\psi_t$ in $T$ there exists an injective projection $f:P(\Lambda)\to P(T)$ s.t. $P(T)\setminus f(P(\Lambda))$ contains at most two periodic orbits for $\psi_t$. 
\end{claim}
For a proof, see either \cite{BW} or Th.2.2.4 in \cite{KNOTBOOK}. In other words, the Birman-Williams Theorem states that whenever we have a flow hyperbolic on its chain-recurrent set $\Lambda$ the Knot types present as periodic orbits on $\Lambda$ are encoded by some Template.\\

\begin{figure}[h]
\centering
\begin{overpic}[width=0.4\textwidth]{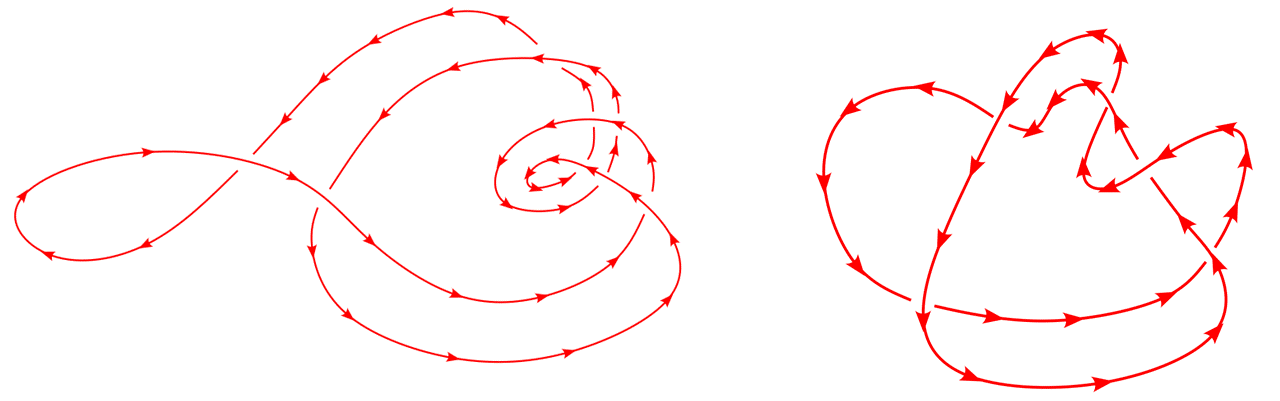}
\end{overpic}
\caption{\textit{A $(2,5)$ Torus knot on the right and its projection to the Lorenz Template from Fig.\ref{TEMP2} on the left (it exists there by Th.6.1.2.b in \cite{Hol}).}}
\label{tempknot}
\end{figure}

Having surveyed the basics of Template Theory we conclude this section with a uniqueness result for hyperbolic flows. To do so, we first introduce the following definitions:
\begin{definition}
\label{model} Let $\phi_t$, $t\in\mathbf{R}$ be a flow on a $3$-manifold $M$. A hyperbolic invariant set for $\phi_t$, $K\subseteq M$, is said to be a \textbf{basic set} if there are no fixed points in $K$, the periodic orbits are dense in $K$, and $K$ includes a dense orbit (by Def.\ref{nonwanderchain}, every basic set is at the very least a subset of the chain recurrent set in $M$). Given a basic set $K\subseteq M$ a \textit{\textbf{plug}} for $\phi_t$ is a connected, orientable $3$-manifold $N$ with boundary and a flow $\varphi_t$ defined on $N$ s.t. the following is satisfied:
\begin{itemize}
    \item $\varphi_t$ is transverse to the boundary of $N$.
    \item Let $K'$ denote the maximal invariant set of $\varphi_t$ in $N$. Then:
    \begin{enumerate}
        \item $K'$ is interior to $N$.
        \item $\varphi_t$ is hyperbolic on $K'$.
        \item The dynamics of $\varphi_t$ on a neighborhood of $K'$ are orbitally equivalent to those of $\phi_t$ on a neighborhood of $K$.
    \end{enumerate}
\end{itemize}    
  
Finally, recall the \textbf{total genus of $N$} is defined as the sum of the genus for all the components of $\partial N$. We refer to a plug with a minimal total genus as a \textbf{model}.
\end{definition}
As far as models are concerned, we have the following theorem proven in \cite{BB} (see Th.0.3):

\begin{claim}\label{begbon}
Given $\phi,M,K$ as above, there exists a model - and moreover, the model is unique up to orbital equivalence of the flows.
\end{claim}

\subsection{The Orbit Index}
\label{orbitin}
As stated in the introduction the proofs of all the results in this paper are heavily based on the notion of the Orbit Index for flows (or in short, the Orbit Index). Therefore, for the sake of completeness, in this section we survey the basics of the Orbit Index Theory as introduced in \cite{PY}, \cite{PY2}, \cite{PY3}, \cite{PY4}, \cite{KE} and \cite{KY2} (among others) - in particular we review how it can be used to study the persistence of periodic dynamics for three-dimensional flows. Even though we will only consider the Orbit Index in a three-dimensional context we remark that most of the results stated below also hold for higher dimensions (for more details, see the papers cited above and the references therein).\\

This section is organized as follows: for motivation, we begin by briefly recalling the Fixed Point Index from Algebraic Topology and its homotopy invariance properties (see Th.\ref{lefschetztheorem}) - after which we discuss the persistence of periodic dynamics for three-dimensional flows which we describe using a notion called "Global Continuability" (see Def.\ref{globalconti}). Following that, we use the Fixed Point Index to define the Orbit Index for generic families of vector fields (see Def.\ref{index2} and \ref{index1}) - and finally we generalize it to a more general context and connect it with the notion of global continuability mentioned above (see Def.\ref{index22} and Th.\ref{contith}). As will be made clear, whenever the Orbit Index of a given periodic orbit is non-zero that periodic orbit must persist under smooth perturbations - at least generically - and moreover, this persistence is analogous to the homotopy invariance of the Fixed Point Index (see Th.\ref{invar} later on in this section). Per the discussion above we begin with the following definition: 

\begin{definition}
    \label{lefschetz}
     Let $D$ be an open planar domain, and let $f:D\to \mathbf{R}^2$ be continuous. Assume $\{x\in D|f(x)=x\}$ is compact in $D$. Then, the \textbf{Fixed Point Index} or the \textbf{Lefschetz number} of $f$ at $D$ will be defined as the degree of $f(x)-x$ in $D$. 
\end{definition}
It is well-known that when the Lefschetz number of $f$ in $D$ is non-zero $f$ has a fixed point (see Prop.VII.5 in \cite{Dold}). Moreover, it is also well-known that the Fixed Point Index is invariant under "well-behaved" homotopies. More precisely, we have the following result:

\begin{theorem}
    \label{lefschetztheorem} With the notations above, assume $g_t:D\to\mathbf{R}^2, t\in[0,1]$ is a homotopy of continuous maps s.t. $\{(x,t)\in D\times[0,1]|g_t(x)=x\}$ is compact in $D\times[0,1]$. Then, if the Fixed Point Index of $g_0$ is $k$ the Fixed Point Index of $g_1$ is also $k$.
\end{theorem}
For a proof, see Th.VII.5.8 in \cite{Dold}. Now, assume $f:S\to S$ is a local first-return map for a three-dimensional vector field $F$ - where $S$ is a cross-section transverse to some periodic orbit $T$. It is easy to see that if the Fixed-Point Index of $f$ on $S\cap T$ is non-zero by Th.\ref{lefschetztheorem} we would expect $T$ to persist under sufficiently small smooth perturbations of $F$. This leads us to ask the following question - just how much can we perturb $F$ s.t. $T$ persists as a periodic orbit? Obviously, by Th.\ref{lefschetztheorem} we know $T$ persists for sufficiently small $C^k$ perturbations of $F$ - but can we give a more global answer? \\

\begin{figure}[h]
\centering
\begin{overpic}[width=0.3\textwidth]{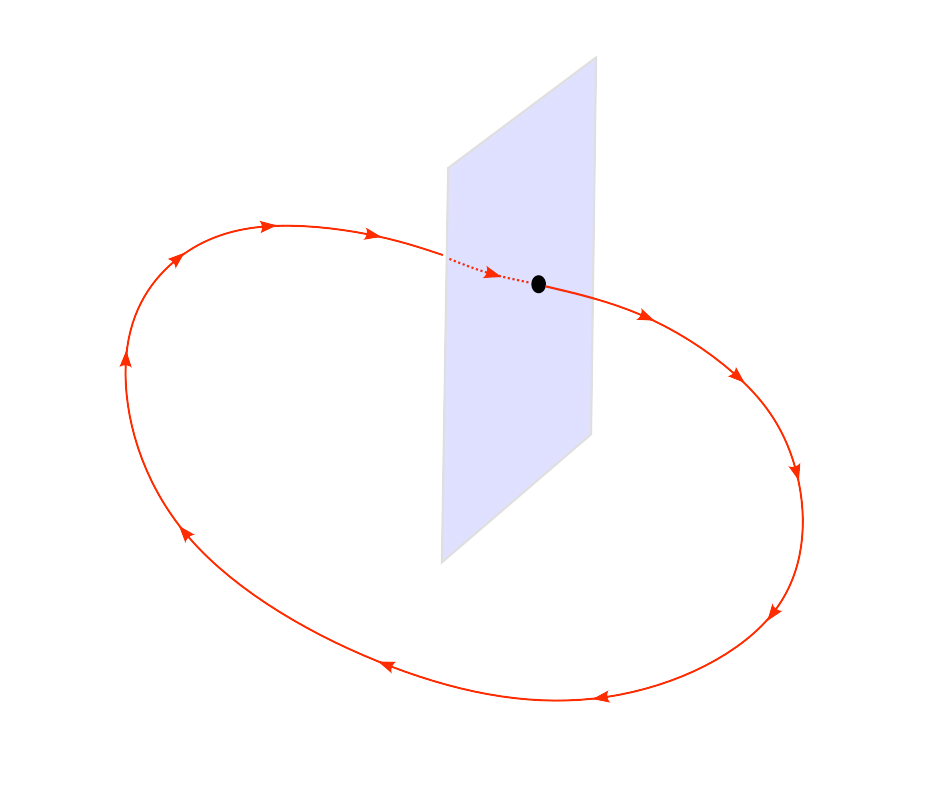}
\put(350,530){$T$}
\put(550,470){$x$}
\put(650,750){$S$}
\end{overpic}
\caption[A Template.]{\textit{A first return map for a periodic orbit $T$. $x$ is the sole point of intersection for $T$ with the cross-section $S$, hence $S$ is an isolating cross-section for $T$.}}
\label{iso}
\end{figure}

It is precisely this question that the Orbit Index Theory addresses, and to introduce it we first introduce several notations and conventions. To begin, consider a smooth three-manifold $M$. Moreover, from now on unless explicitly stated otherwise $F$ would always denote a $C^k$ vector field, $k\geq3$ - we will often refer to such a vector field $F$ as \textbf{smooth}, regardless of its differentiability class.\\

In addition, from now on, by a \textbf{periodic orbit} $T$ for $F$ we will always mean a periodic solution for $F$ that is not a fixed point, and the \textbf{period of $T$} would always denote its minimal period w.r.t. the flow - that is, if $\phi_t,t\in\mathbf{R}$ is the flow generated by $F$ and $x\in T$, $\tau>0$ is the minimal real number s.t. $\phi_\tau(x)=x$. Finally, from now on we say $T$ is \textbf{isolated} whenever the following is satisfied (see the illustration in Fig.\ref{iso}):
\begin{enumerate}
    \item There exists a cross-section $S$ transverse to $T$ s.t. $S\cap T=\{x\}$.
    \item If $f:S\to S$ is some local first-return map, then $\{y\in S|f(y)=y\}=\{x\}=T\cap S$.
\end{enumerate}

We will often refer to such a cross-section $S$ as an \textbf{isolating cross-section}. With these ideas in mind, we now define the differential of $T$ - 

\begin{definition}
    \label{dfT} Let $T$ be an isolated periodic orbit as above, and let $S$ be some isolating local cross-section s.t. $\{x\}=T\cap S$. The differential of the first-return map $f:S\to S$ will be referred to as the \textbf{differential of $T$}, and will be denoted by $D(T)$. We will always treat $D(T)$ as a $2\times2$ matrix - it is easy to see the eigenvalues of $D(T)$ are independent of the choice of $x,S$ and $f$.
\end{definition}

With these notations and conventions in mind, we now study the persistence of periodic orbits under smooth deformations. To do so let $M$ be a (smooth) three-manifold, let $t\in\mathbf{R}$ be a parameter and let $\dot{s}=F_t(s)$, $s\in M$ be a smooth curve of vector fields, varying smoothly in $(s,t)\in M\times\mathbf{R}$. Now, denote by $Per\subseteq M\times\mathbf{R}$ the collection of periodic orbits for the curve $\{F_t\}_t$, $t\in\mathbf{R}$ - that is, if $T\times\{t\}$ is a component of $Per\cap(M\times\{t\})$, then $T$ is periodic orbit for the vector field $F_{t}$ (see the illustration in Fig.\ref{period}). Now, endow $Per$ with the topology induced by $M\times\mathbf{R}$ and given a periodic orbit $T_0$ for the vector field $F_0$ we denote by $Per_0$ the component of $Per$ s.t. $T_0\times\{0\}\subseteq Per_0$. It is easy to see that whenever $Per_0\setminus T_0\times\{0\}$ contains more than two components the vector fields $\{F_t\}_t$ go through a bifurcation at $F_0$ - i.e., $T_0$ splits into two (or more) periodic orbits under perturbation of $t$ (see the illustration in Fig.\ref{period}). Similarly, if $(Per_0\setminus T_0\times\{0\})\cap M\times(0,\infty)=\emptyset$, it implies $T_0$ does not persist as we smoothly deform $F_0$ to $F_t,t>0$.\\

\begin{figure}[h]
\centering
\begin{overpic}[width=0.3\textwidth]{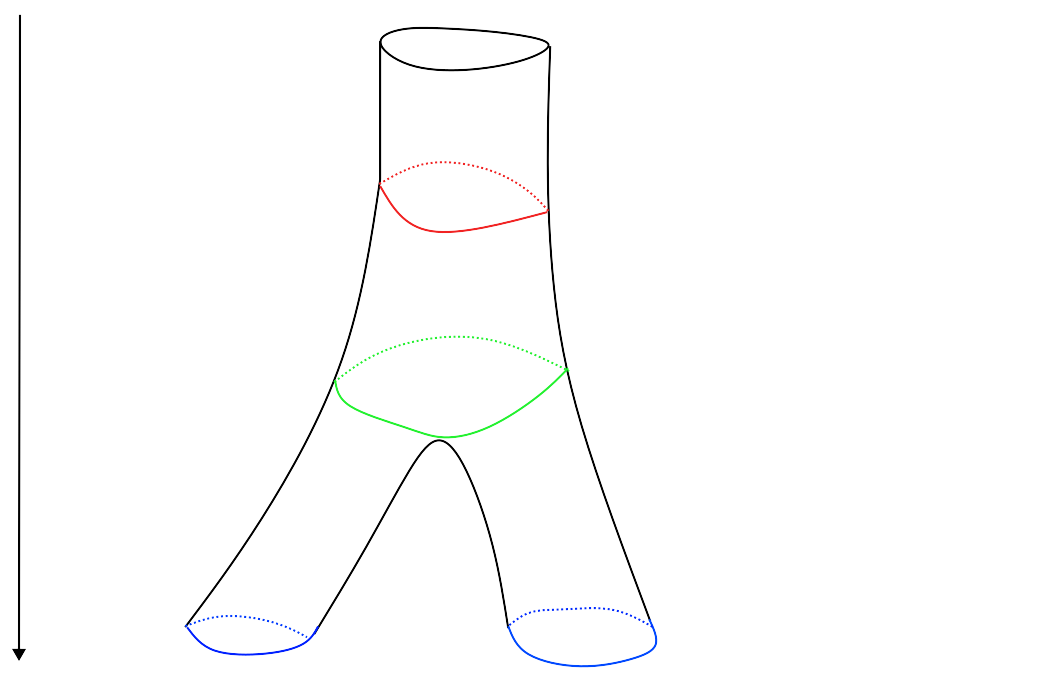}
\put(10,-40){$t$}
\put(560,450){$T_{-1}$}
\put(570,250){$T_0$}
\end{overpic}
\caption{\textit{A component of periodic orbits $Per_0$ for a curve of vector fields $\{F_t\}_{t\in\mathbf{R}}$, which, for simplicity, we sketch as a "cylinder", splitting at $T_0$. In this case, the periodic orbit $T_{-1}$ is smoothly deformed to $T_0$, where a bifurcation occurs causing $Per_0$ to split in two.}}
\label{period}
\end{figure}

As we will soon see, by studying periodic orbits for the curve $\dot{s}=F_t(s),t\in\mathbf{R}$ as components in the set $Per$ one can infer a large amount of information on their persistence and bifurcations. With this intuition in mind, we first define the notion of a globally continuable periodic orbit for a smooth curve of vector fields (see Def.2.1 in \cite{PY2}):

\begin{definition}
    \label{globalconti}
    With the notations above, we say $T_0$ is \textbf{globally continuable} provided $Per_0\setminus T_0\times\{0\}$ includes either one or two components, and in addition, one of the following holds:

    \begin{itemize}
        \item $Per_0\setminus T_0\times\{0\}$ is connected (see the illustration in Fig.\ref{con}). 
        \item  $Per_0\setminus T_0$ has precisely two components, $B_1$ and $B_2$, each satisfying precisely one of the following (where $i=1,2$):
        \begin{enumerate}
            \item $\overline{B_i}$ includes a center fixed point - i.e., a fixed point whose linearization has a pair of imaginary eigenvalues.
            \item The set $B_i$ is unbounded in $M\times\mathbf{R}$, i.e., it is not trapped inside any compact set.
            \item The periods of periodic orbits $T\times\{t\}$ on $B_i$ are unbounded.
        \end{enumerate}
\end{itemize}
\end{definition}
\begin{remark}
    We remark that in \cite{PY2} the notion of global continuability is slightly different, and involves the notion of virtual periods (see Def.3.5 in \cite{PY2} and Def.\ref{virtper} below). However, by Th.4.2 in \cite{PY2} it follows that for three-dimensional flows Def.\ref{globalconti} above and the Definition of global continuability in \cite{PY2} are one and the same.
\end{remark}
\begin{figure}[h]
\centering
\begin{overpic}[width=0.3\textwidth]{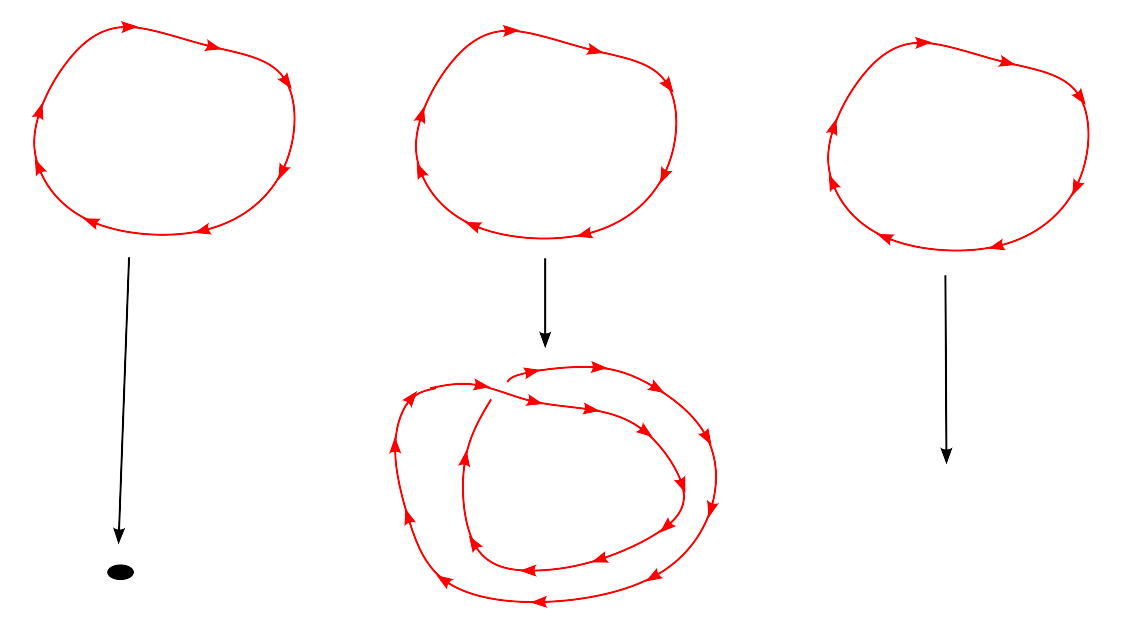}
\put(800,50){$\infty$}
\end{overpic}
\caption[A horseshoe template.]{\textit{The three ways a globally continuable periodic orbit is expected to be destroyed in $\mathbf{R}^3$ - by closing to a fixed point by a Hopf Bifurcation (left), collapsing to an a-periodic motion via a Period-Doubling Cascade (middle), or vanishing to $\infty$ in a Blue Sky Catastrophe (right).}}
\label{period2}
\end{figure}

In order to further motivate the definition above, note the definition of global continuability practically describes how the periodic orbit $T_0$ can be destroyed as we smoothly vary the vector fields $F_t,t\in\mathbf{R}$ along the curve. Informally, the definition above implies that as we vary $t\in\mathbf{R}$, $T_0$ can be destroyed in one of three ways, as illustrated in Fig.\ref{period2}: 
\begin{enumerate}
    \item By a Hopf-bifurcation (i.e., the case where $\overline{B_i}$ includes a center fixed point) - that is, as we vary $t$ away from $0$ in some direction $T_0$ closes to a fixed point by a Hopf bifurcation.
    \item  When $B_i$ is not compact in $M\times\mathbf{R}$, $T_0$ is possibly destroyed by collapsing into $\partial M$. To illustrate, when $M=\mathbf{R}^3$ this could imply a Blue Sky Catastrophe (see \cite{Ab} and \cite{ST}): that is, as we vary $t$ away from $0$ both the diameter and period of $T_0$ diverge to $\infty$ (i.e., the orbit "vanishes into the blue sky").
    \item By some period-doubling cascade, or by collapsing into a homoclinic (or a heteroclinic) orbit for some fixed point - which causes the period to explode.
\end{enumerate}

It is easy to see that globally continuable orbit persist under sufficiently small smooth perturbations - and more importantly, that heuristically they cannot be destroyed by arbitrarily small perturbations. This begs the following question - given an isolated periodic orbit $T_0$ for a smooth vector field $F_0$, can we tell when it is globally continuable? As proven at Th.3.1 in \cite{KY2} this question has the following answer:

\begin{claim}
    \label{continugeneral} Assume that for every $n>0$ the matrix $D(T_0)^n-Id$ is invertible and that $D(T_0)$ has no negative eigenvalues. Then, $T_0$ is globally continuable.
\end{claim}
\begin{remark}
    When $T_0$ is a saddle periodic orbit s.t. $D(T_0)$ has negative eigenvalues, even though $T$ would persist under sufficiently small perturbations of the flow it need not necessarily be globally continuable - see \cite{ky4} for a four-dimensional example. 
\end{remark}
\begin{figure}[h]
\centering
\begin{overpic}[width=0.25\textwidth]{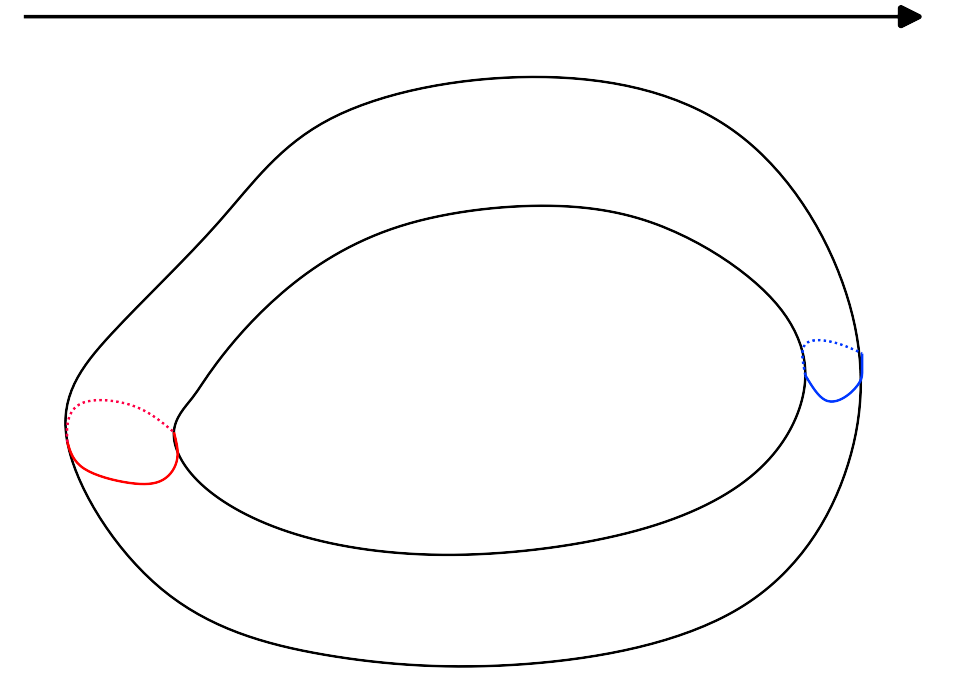}
\put(960,670){$t$}
\put(920,300){$T_{1}$}
\put(200,250){$T_0$}
\end{overpic}
\caption{\textit{An example of a globally continuable periodic orbit $T_0$ s.t. $Per_0\setminus T_0\times\{0\}$ is connected. In this scenario, $T_0$ is a saddle node bifurcation orbit (see Def.\ref{type}) which persists until terminating in $T_1$ (which is also a saddle-node bifurcation orbit).}}
\label{con}
\end{figure}

Even though Th.\ref{continugeneral} does give us some criterion to determine the global continuability of a given periodic orbit, alone it does not give us any concrete tool to describe how $T_0$ possibly bifurcates as the vector field $F_0$ is varied. To do that we will need the Orbit Index - which we now define. As stated at the beginning of this section, we first define the Orbit Index only for generic families of vector fields - after which we extend the definition to a general class of three-dimensional flows.\\

Therefore, we begin by defining the generic family of vector fields we will be working with. To do so, consider an isolated periodic orbit $T_0$ for a smooth vector field $F_0$ on $M$, and its differential $D(T_0)$ - see Def.\ref{dfT} and the discussion immediately before it. It is easy to see the situation where $D(T_0)$ has no eigenvalues on the unit circle $S^1$ is generic (and consequentially, so is the isolation of $T_0$). Now, let us consider again $\dot{s}=F_t(s)$, a smooth curve in $(s,t)\in M\times\mathbf{R}$, and recall we denote by $Per$ the collection of periodic orbits for the curve. Following Sect.2 in \cite{PY2}, we begin by defining three types of periodic orbits:

\begin{definition}
\label{type}    Let $T_0$ denote an isolated periodic orbit for $\dot{s}=F_0(s)$ as above, and let $Per_0$ denote the component of $Per$ s.t. $T_0\times\{0\}\subseteq Per_0$. The type of $T_0$ is defined as follows (see the illustration in Fig.\ref{period3}):
    \begin{itemize}
        \item $T_0$ is \textbf{Type} $0$ if the matrix $D(T_0)$ has no eigenvalues which are roots of unity.
        \item  $T_0$ is \textbf{Type} $I$ or a \textbf{Saddle-Node bifurcation orbit} if the following is satisfied: 
        \begin{enumerate}
            \item $1$ is a simple eigenvalue of $D(T_0)$, and $D(T_0)$ has no other eigenvalues which are roots of unity (see the illustration in Fig.\ref{period3}).
            \item $Per_0\setminus T_0$ has two components.
            \item The period of $T_0$ varies continuously as we vary periodic orbits on a neighborhood of $T_0$ on $Per_0$.

        \end{enumerate}
        \item We say $T_0$ is \textbf{Type} $II$ or \textbf{Period-Doubling bifurcation orbit} provided the following is satisfied (see the illustration in Fig.\ref{period3}):
        \begin{enumerate}
            \item $-1$ is a simple eigenvalue of $D(T_0)$, and the only eigenvalue of $D(T_0)$ which is a root of unity. 
            \item $Per_0\setminus T_0$ includes three components.
            \item There exists a component of $Per_1\subseteq Per_0\setminus T_0$ s.t. for periodic orbits on $Per_1$ sufficiently close to $T_0$, the period is approximately double that of $T_0$.
        \end{enumerate}
    \end{itemize}
\end{definition}
\begin{figure}[h]
\centering
\begin{overpic}[width=0.3\textwidth]{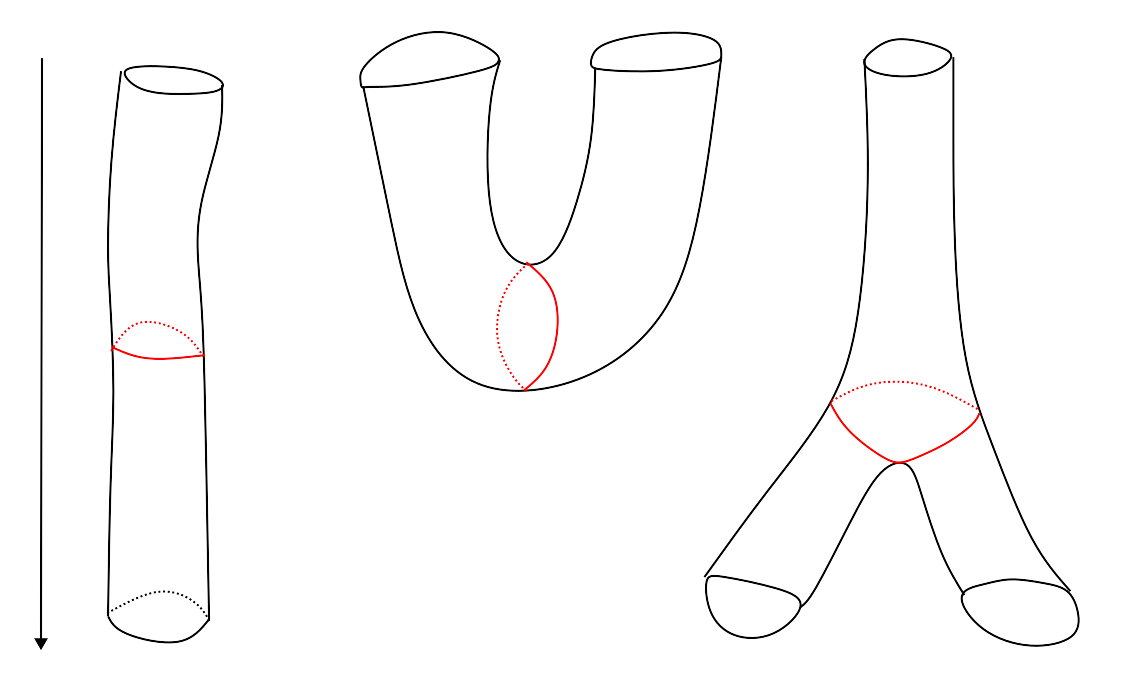}
\put(20,-30){$t$}
\end{overpic}
\caption[A horseshoe template.]{\textit{The three types of periodic orbits, represented by the red loop - Type $0$ which persist under perturbation (left), Type $I$ (middle), and Type $II$ (right), where one arm of periodic orbits has its period doubled.}}
\label{period3}
\end{figure}

At this point we remark given a Type $0$ periodic orbit $T_0$ it is common in literature that it is further classified as follows:\label{type2}

\begin{enumerate}
    \item  $T_0$ is said to \textbf{Hyperbolic} whenever $D(T_0)$ has one eigenvalue in $(0,1)$ and another at $(1,\infty)$ - by Th.\ref{continugeneral} it follows $T$ is globally continuable.
    \item When $D(T_0)$ has one eigenvalue at $(-\infty,-1)$ and another at $(-1,0)$ we say $T_0$ is \textbf{Möbius} - with the justification for the name being that for Möbius orbit the unstable manifold is non-orientable (see the illustration in Fig.\ref{MOB}).
    \item Otherwise, we say $T_0$ is \textbf{Elliptic} - for example, any attracting or repelling Type $0$ orbit is elliptic.
\end{enumerate}

However, for the sake of clarity we will avoid using this this terminology in this paper, and reserve the notion of hyperbolicity strictly for the dynamics of vector fields.\\

\begin{figure}[h]
\centering
\begin{overpic}[width=0.25\textwidth]{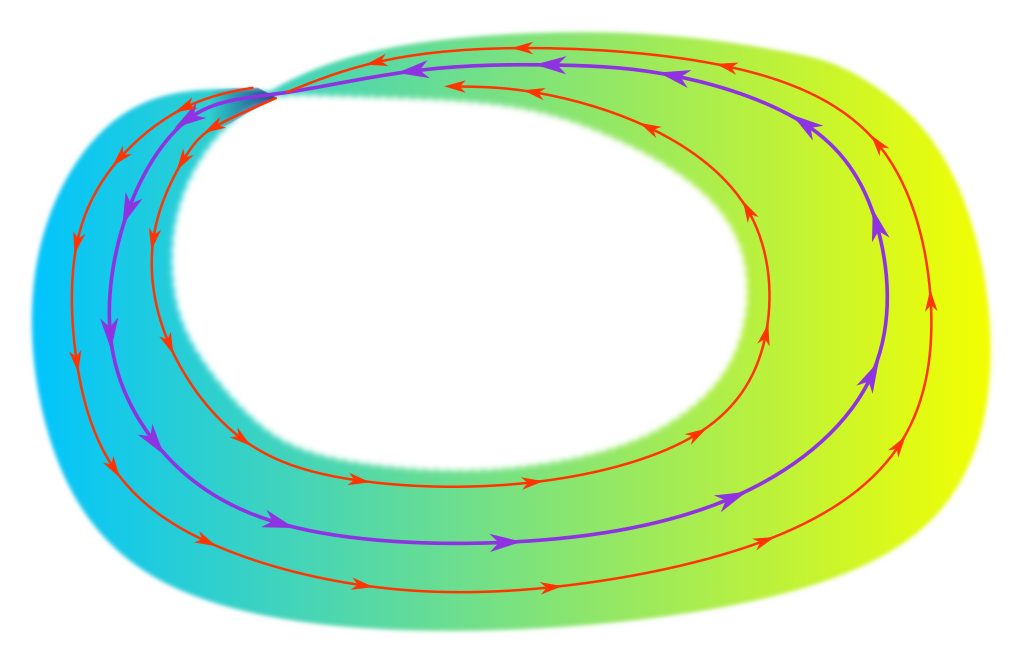}

\end{overpic}
\caption{\textit{A Möbius periodic orbit $T$ (the purple curve), whose knot type is the unknot. It is easy to see that due to the unstable manifold inorientability, locally around $T$ its shape is that of a Möbius strip.}}
\label{MOB}
\end{figure}

We are now ready to introduce a family of vector fields in which we define the Orbit Index. To do so, recall that as proven at the appendix of \cite{PY2} we have the following genericity result:
\begin{proposition}
  \label{propk}  Given a three-dimensional manifold $M$, let $K$ denote the collection of $C^k,k\geq3$ curves of vector fields $\dot{s}=F_t(s)$, $(s,t)\in M\times\mathbf{R}$, s.t. every periodic orbit for $F_t$, $t\in\mathbf{R}$ is either Type $0$, Type $I$ or Type $II$. Then, $K$ is generic in the space of $C^3-$maps $F:M\times\mathbf{R}\to M$.
\end{proposition}

Before moving on we remark that in \cite{PY2} the assertion above is only proven for curves $F:\mathbf{R}^3\times\mathbf{R}\to\mathbf{R}^3$ - however, as the proof does not depend on the specific properties of $\mathbf{R}^3$ these arguments hold for every three-dimensional manifold $M$. Informally, Prop.\ref{propk} teaches us that vector fields whose periodic orbits can undergo only saddle node, period-doubling and Hopf bifurcations are generic - that is, we can approximate any bifurcation scenario with them. With these ideas in mind, we immediately conclude:

\begin{corollary}
    \label{hyper}
Assume $F$ is a vector field hyperbolic on its collection of periodic orbits (see Def.\ref{hyperbolicvector}) - then, $F$ lies on some curve in $ K$.
\end{corollary}
\begin{remark}
It is not always easy to see how vector fields in $K$ approximate bifurcation scenarios which are not permitted in $K$. For completeness, in Appendix $A$ we show how we can approximate a derived from Anosov bifurcation using vector fields from $K$. A more complex example of such approximations can be found at Section 2 in \cite{PY4}.
\end{remark}
Having defined the generic family of vector fields $K$ we are now ready to introduce our first definition of the Orbit Index. To do so, we first note that if $T$ is a Type $0$ periodic orbit for some vector field $F\in K$ it is isolated - that is, there exists an isolating cross-section $S$ for $T$ with a local first-return map $f:S\to S$ s.t. $S\cap T=\{x\in S|f(x)=x\}$. Moreover, since $T$ is Type $0$ by definition the eigenvalues of $D(T)$ are not roots of unity - hence the Fixed Point Index (see Def.\ref{lefschetz}) of $f^n$, the local $n-$first hit map at $x$, is well defined and non-zero as well. Consequentially, for every $n>0$ there exists an isolating cross-section $S_n$ s.t. $f^n:S_n\to S_n$ is a well-defined $n-$hit map, and. $\{x\in S_n|f^n(x)=x\}=T\cap S_n$ (in particular, $S_n\subseteq S_{n-1}$). With these ideas in mind we introduce our first Definition for the Orbit Index, as formulated in \cite{PY3}:

\begin{definition}
    \label{index2}
    Let $F,T,f$ and $S_n, n>0$ be as above and set $k_n$ as the Fixed Point Index for $f^n:S_n\to S_n$. We define the \textbf{$\psi$-index} of $T$, denoted by $\psi(T)$, as the average $\lim_{N\to\infty}\frac{1}{N}\sum_{i=1}^N k_i$.
\end{definition}
Per Prop.3.1 in \cite{PY3}, the $\psi-$index is well-defined and always forms an integer - it is easy to see that when a vector field is hyperbolic on a periodic orbit $T$, we have $k_n=-1$ for every $n$ - hence in that case $\psi(T)=-1$. By the discussion preceding Def.\ref{index2}, it is also easy to see that $\psi(T)$ depends more on the eigenvalues of $D(T)$ than it does on the iterates of $f$. And indeed, following \cite{PY} and \cite{PY2}, one can also define the Orbit Index alternatively using only the eigenvalues of $D(T)$:
\begin{definition}
\label{index1}     Consider a vector field $F\in K$ and some Type $0$ periodic orbit $T$ for $F$. Let $k^-$ and $k^+$ denote the (respective) number of eigenvalues for the two-dimensional matrix $D(T)$ in $(-\infty,-1)$ and $(1,\infty)$ (both counted with multiplicity). We define the $\phi-$\textbf{index} of $T$, as follows:
     $$
\phi(T)=\begin{cases}
			(-1)^{k^+}, & \text{if $k^-$ is even}\\
            0, & \text{otherwise}
		 \end{cases}
$$
 \end{definition}
Again, when the flow is hyperbolic on $T$ we have by definition $\phi(T)=-1$. In other words, for hyperbolic periodic orbits we have $\psi(T)=\phi(T)$. Motivated by this computation, for the sake of completeness we now quickly prove that Def.\ref{index2} and Def.\ref{index1} are actually one and the same:
\begin{lemma}
    \label{index3}
    Let $F\in K$ be a smooth vector field, and let $T$ be a Type $0$ periodic orbit for $F$ (see Def.\ref{type}). Then $\phi(T)=\psi(T)$ - and consequentially, we can define $\phi(T)=\psi(T)=\varphi(T)$, the \textbf{generic Orbit Index}.
\end{lemma}
\begin{proof}
We prove that whenever $T$ is Type $0$ and satisfies $\phi(T)=s$ for some $s\in\{1,-1,0\}$ we also have $\psi(T)=s$ as well. To begin, let $S$ be some isolating cross-section for the local first-return map $f:S\to S$ at $T$. Since a $C^k$ flow on $\mathbf{R}^3$, $k\geq3$, is both orientation preserving and transverse to $S$, it follows that either $D(T)$ has two real eigenvalues of the same sign, or two complex conjugate eigenvalues - therefore, provided we prove the Lemma for the real and complex cases separately, we're done.\\

We first prove Lemma \ref{index3} for the case when there are only real eigenvalues. To do so, recall that since $T$ is Type $0$, whenever $\phi(T)=\pm 1$ the matrix $D(T)$ has no eigenvalues which are $\pm 1$. It is easy to see that whenever $\phi(T)=1$ we have two non-zero eigenvalues with the same sign, one of which lies at $(-1,1)$ and the other at $(-\infty,-1)\cup(1,\infty)$ - and conversely, when $\phi(T)=-1$ one eigenvalue for $D(T)$ lies in $(1,\infty)$ and another at $(0,1)$. Similarly, when $\phi(T)=0$ the matrix $D(T)$ has one eigenvalue in $(-\infty,-1)$ and another in $(-1,0)$. Now, let us denote by $deg(g)$ the degree of a given map $g:D\to\mathbf{R}^2$ (where $D$ is an open planar domain), and recall the Fixed Point Index of $g$ in $D$ is given by $deg(g(x)-x)$. By the discussion above it now follows that if $\phi(T)=1$ then for every $n>0$, $deg(f^n(x)-x)=1$ - and when $\phi(T)=-1$ then for every $n>0$, $deg(f^n(x)-x)=-1$. As a consequence, it follows that whenever $\phi(T)=\pm1$ we have $\psi(T)=\phi(T)$. Similarly, when $\phi(T)=0$ and $T$ is Type $0$, for every odd $n$ we have $deg(f^n(x)-x)=-1$ while for all even $n$ we have $deg(f^n(x)-x)=1$, which implies $\frac{1}{2N}\sum_{j=1}^{2N}k_j=0$. As the $\psi-$index is well defined as the limit of such sums, it follows again $\psi(T)=0$, i.e., again we have $\phi(T)=\psi(T)$.\\
    
All in all, we have proven that when $D(T)$ has two real eigenvalues then $\phi(T)=\psi(T)$ - therefore, it remains to prove the same is true when $D(T)$ has two complex conjugate eigenvalues. To do so, note that when $D(T)$ has complex conjugate eigenvalues $\lambda\pm i\omega$ which are not roots of unity, it is easy to see $\phi(T)=1$ - and that the matrix $D_{f^n}(T)-Id$ is invertible for all $n$. It is also easy to see the eigenvalues of $D_{f^n}(T)$ are two complex conjugate numbers, $\lambda_n\pm i\omega_n\ne 1$ - consequentially, $\lambda_n-1\pm i\omega_n$ are also complex-conjugate and non-zero, which implies $det(D_{f{^n}}(T)-Id)=||\lambda_n-1+i\omega_n||^{2}>0$. Therefore, since the degree of $f^n(x)-x$ is given by the sign of $det(D_{f{^n}}(T)-Id)$, it follows that with the notation of Def.\ref{index1} we have $k_n=1$ for all $n$ - which again implies $\psi(T)=1$, and again $\psi(T)=\phi(T)$. All in all, the proof of Lemma \ref{index3} is now complete.
\end{proof}
Having established the generic Orbit Index $\varphi(T)$ is well defined in generic scenarios, our current goal is to show one can generalize the $\varphi$-index from vector fields in $K$ to a general class of $C^k$ vector fields, $k\geq3$. However, before doing so we recall several of the properties of the $\varphi$-index which will be used throughout the proof - we begin with its bifurcation invariance property (see Th.5.5 in \cite{PY3}), which can be thought of as a continuous time analogue of the Lefschetz Fixed Point Theorem:

\begin{claim}
\label{invar} Assume $\dot{s}=F_t(s), t\in\mathbf{R}$ is a $C^3$-curve of vector fields in $K$. Let $T$ be a Type $I$ or Type $II$ periodic orbit for the vector field $F_0$, and let $Per_0$ denote the component of periodic orbits for the curve $\{F_t\}_t$ s.t. $T_0\times\{0\}\subseteq Per_0$. Now, let $Per_1,...,Per_n, n\in\{2,3\}$ denote the components of $Per_0\setminus T_0\times\{0\}$. Then, assuming every periodic orbit on $Per_0\setminus T_0\times\{0\}$ is Type $0$ we have the following:
    \begin{enumerate}
        \item   The $\varphi-$index is constant on every $Per_i$, $i=1,...,n$.
        \item Assume $Per_1,...,Per_k\subseteq\mathbf{R}^3\times(-\infty,0)$ and that $Per_{k+1},...,Per_n\subseteq\mathbf{R}^3\times(0,\infty)$. Then, denoting by $\varphi_i$ the Orbit Index on $per_i$, we have the following equality (see the illustration in Fig.\ref{coll223} and Fig.\ref{colli}): 
        \begin{equation*}
            \sum_{i=1}^k\varphi_i=\sum_{i=k+1}^r\varphi_i
        \end{equation*}
        
    \end{enumerate}
In other words, the $\varphi-$ index is a bifurcation invariant of $Per_0$ (see the illustration in Fig.\ref{coll223}). Moreover, by using the formula above we extend the Orbit Index to Type $I$ and $II$ orbits.
\end{claim}
\begin{remark}
    In fact, Th.\ref{invar} holds in a more general setting. For more details, see Th.5.5 in \cite{PY3}.
\end{remark}
\begin{figure}[h]
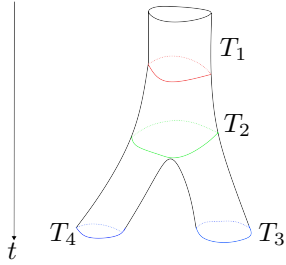

\centering
\begin{overpic}[width=0.3\textwidth]{images/coll223.png}
\put(0,-30){$t$}
\put(650,15){$T_3$}
\put(110,15){$T_4$}
\put(550,500){$T_1$}
\put(560,300){$T_2$}
\end{overpic}
\caption[A horseshoe template.]{\textit{The invariance of the $\varphi-$index under regular bifurcations, an illustration: assume $\varphi(T_1)=-1$ - consequentially, by Th.\ref{invar} we have $\varphi(T_3)=-1$ and $\varphi(T_4)=0$. It is easy to see that $T_2$ is a Type $II$ orbit (see Def.\ref{type}). }}
\label{coll223}
\end{figure}

It is easy to see Th.\ref{invar} strongly constrains the way periodic orbits for vector fields in $K$ can bifurcate. As noted in Prop.3.2 at \cite{Perd}, Th.\ref{invar} implies the following fact:

\begin{corollary}
    \label{slice} Let $Per$ be the collection periodic orbit for $\{F_t\}_{t\in[0,1}$ as in Th.\ref{invar}, and for any $t\in[0,1]$ and $Per_0$ a component of $Per$, set $S_t=Per_0\cap M\times\{t\}$. Then, provided the periods of periods in $Per_0$ are bounded $S_t$ includes a finite number of orbits and $\lambda=\sum_{T\in S_t}\varphi(T)$ is independent of $t$ and depends only on $Per_0$.
\end{corollary}

One immediate implication of Cor.\ref{slice} is the following - if $T_0$ is a Type $II$ Orbit for a smooth vector field $F_0$ which lies on some curve of vector fields $\dot{s}=F_t(s)$, $(s,t)\in M\times\mathbf{R}$, then there are precisely two ways the bifurcation can unfold. More precisely, if $Per_0$ is the component of periodic orbits s.t. $T_0\times\{0\}\subseteq Per_0$ then $Per_0\setminus T_0\times\{0\}$ must intersect \textbf{both} $M\cap(0,\infty)$ and $M\cap(-\infty,0)$ (see the illustration in Fig.\ref{posib}). Similarly, it also follows that if $T_0$ is a saddle node bifurcation then Cor.\ref{slice} implies that all periodic orbit on $Per_0\setminus\{T_0\}$ sufficiently close to $T_0\times\{0\}$ are trapped in the same component of $M\times\mathbf{R}\setminus M\times\{0\}$ (see the illustration in Fig.\ref{posib}).\\

\begin{figure}[h]
\centering
\begin{overpic}[width=0.3\textwidth]{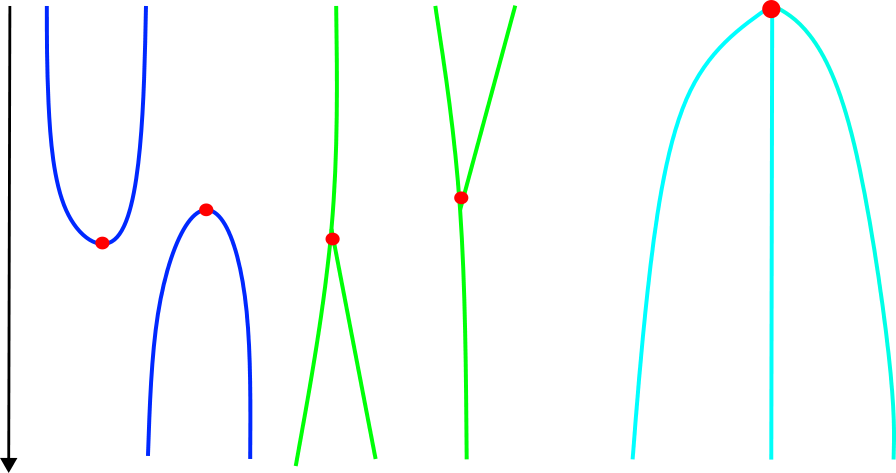}
\put(-10,-50){$t$}
\end{overpic}
\caption{\textit{How possible bifurcation diagrams of periodic orbits for vector fields in $K$ can (and cannot) look like, where the red dots denote the bifurcation points. The blue diagrams represent two periodic orbits colliding in a saddle-node bifurcations, while the green represent period-doubling bifurcation. By Cor.\ref{slice}, the cyan diagram does not describe any admissible period-doubling bifurcation in $K$. }}
\label{posib}
\end{figure}

Another tool we will use to rule out certain bifurcations is the following technical Lemma, which is a slightly generalized form of a phenomenon originally observed in \cite{Perd} and \cite{KE} (for more details, see the beginning of Sect.4 in \cite{KE} or Lemma 3.4 in \cite{Perd}):

\begin{lemma}
    \label{notyp2}
    Let $T$ be a type-$2$ orbit for some $F\in K$. Then, we cannot approximate $F$ in the $C^k$ metric (for $k\geq3$) with vector fields $\{F_n\}_n\subseteq K$ which generate periodic orbits $\{T_n\}_n$ (respectively) satisfying the following:
    \begin{enumerate}
        \item $\varphi(T_n)=-1$ for all $n$.
        \item $T_n\to T$ (as a limit set).
    \end{enumerate}
    
Consequentially, if $T$ is a Type $II$ orbit then given a sequence of $F_n$ and $T_n$ as above s.t. $F_n\to F$ and $T_n\to T$,  for every $n$ we have either $\varphi(T_n)=0$ or $\varphi(T_n)=1$.
\end{lemma}
\begin{proof}
Consider a sequence of vector fields $F_n$ which generate periodic orbits $T_n$ s.t. for every $n$ we have $\varphi(T_n)=-1$ by Def.\ref{index1} it follows the matrix $D(T_n)$ has two positive eigenvalues - one in $(1,\infty)$ and another in $(0,1)$. Since $T$ is a Type $II$ orbit, $D(T)$ admits $-1$ as an eigenvalue (see Def.\ref{type}) - therefore, by $\varphi(T_n)=-1$ for all $n$ we cannot have $T_n\to T$ and the assertion follows.
\end{proof}
Lemma \ref{notyp2} leads us to a very important heuristic: if a periodic orbit $T$ satisfies $\varphi(T)=-1$, then generically, in order for $T$ to undergo a period doubling bifurcation it must first change its stability and become either an attractor or a repeller. Or, in other words, $T$ must first bifurcate in a way which changes its $\varphi-$index to $1$. We conclude our discussion about bifurcations of periodic orbits for vector fields in $K$ by recalling the following Lemma (see Prop.4.1.1 in \cite{KNOTBOOK} and the illustration in Fig.\ref{colli}):

\begin{lemma}
\label{perserve} Let $\dot{x}=F_t(x)$ be a smooth curve of vector fields, $t\in\mathbf{R}$, and assume $T_0$ is a saddle-node bifurcation orbit for $F_0$ at which two periodic orbits collide and vanish (see the illustration in Fig.\ref{colli}). Then, the index of one such orbit is $-1$, while the other is $1$. Moreover, the knot type of $T_0$ does not change when it undergoes a saddle node bifurcation.
\end{lemma}
\begin{remark}
    For more details about the connection between the Orbit Index and the persistence of Knot types under bifurcations, see \cite{PY4} and Ch.4.1 in \cite{KNOTBOOK}.
\end{remark}
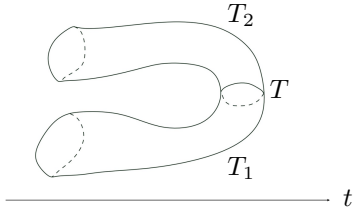
\begin{figure}[h]
\centering
\begin{overpic}[width=0.3\textwidth]{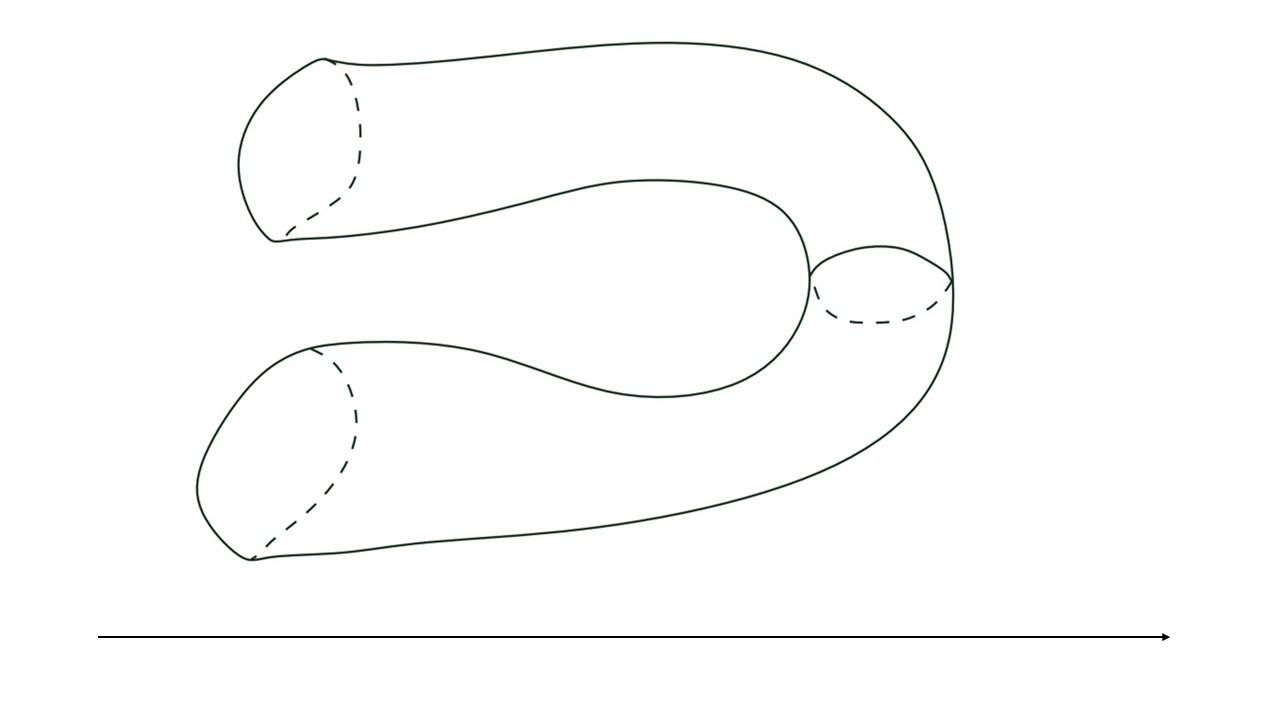}
\put(950,50){$t$}
\put(650,130){$T_1$}
\put(760,330){$T$}
\put(650,530){$T_2$}
\end{overpic}
\caption{\textit{A saddle node bifurcation at a the Type $I$ orbit $T$. By Prop.\ref{perserve}, the Orbit Index of $T_1$ is $1$ while that of $T_2$ is $-1$.}}
\label{colli}
\end{figure}

Having studied the basic properties of the Orbit Index in generic situations, we now proceed to define it for periodic orbits in a general setting - that is, we now extend the Orbit Index for vector fields not in $K$. The main difficulty we will have to contend with is that in general, given a smooth vector field $F$ of $M$ (where $M$ is some smooth three-dimensional manifold) there is no general reason to assume its periodic orbits are isolated - which prevents us from directly generalizing Def.\ref{index1} or \ref{index2} to such scenarios. To bypass this problem, following Section 3 in \cite{PY2} we first introduce the notion of a virtual period:\
\begin{definition}
    \label{virtper} Let $F$ be a smooth vector field on $M$, let $T_0$ be a periodic orbit for $F$ of period $\tau$ (not necessarily isolated), and recall $D(T_0)$ denotes the differential of some local first-return map for $T_0$. Assume there exists some vector $v\in\mathbf{R}^2$ and some minimal $m\geq1$ s.t. $D^m_f(T_0)v=v$. In that case, we say $\overline\tau=m\tau$ is a \textbf{virtual period} of $T_0$, whose order is $m$.
\end{definition}
It is easy to see that $v=(0,0)$ satisfies $D(T_0)v=v$, hence the period of $T_0$ is always also a virtual period. It is also easy to see that if there exists some $m>1$ s.t. $m\tau$ is a virtual period for $T$, then $D(T_0)$ admits an eigenvalue which is an $m-$th root of unity (and vice versa) - consequentially $T_0$ has at most a finite number of virtual periods. Using the virtual period, following Section 3 in \cite{PY2} we now define the notion of an isolated collection of periodic orbits:
\begin{definition}
     \label{relis} Let $M$ be a smooth three-dimensional manifold. We say a collection $C$ of periodic orbits for some $C^3$ vector field $F$ of $M$ is \textbf{relatively isolated} provided the following is satisfied:
     \begin{enumerate}
         \item $C$ is compact.
         \item The set of virtual periods for periodic orbits in $C$ is uniformly bounded.
         \item If $\{\beta_i\}_i$ is a sequence of periodic orbits for $F$ which lies outside $C$ and accumulates on $C$, the periods of orbit in $\{\beta_i\}_i$ are unbounded.
     \end{enumerate}
 \end{definition}

In other words, relatively isolated sets of periodic orbits are compact sets of periodic orbits which cannot be approximated from outside by low-period orbits (it is easy to see every isolated periodic orbit is also relatively isolated). Now, given a relatively isolated set $C$ of periodic orbits for a vector field $F$ as above we have the following result (see Prop.3.2 in \cite{PY2}):

\begin{lemma}
\label{isol} With previous notations, let $p_0>0$ denote the maximum over all the virtual periods in $C$, and let $p_1$ and $z$ be some real numbers s.t. $p_0<p_1<2p_0$ and $z>4p_0$. Then, there exists a neighborhood $N$ of $C$ in $M$ s.t. if $T$ is a periodic orbit for $F$ satisfying $T\subseteq N$, the period of $T$ lies outside $[p_1,z]$.  
\end{lemma}

With these ideas in mind, following Def.3.5 in \cite{PY2} we are finally ready to give a general definition of the Orbit Index. To do so, let $F$ be a smooth vector field of $M$, and let $C$ denote a relatively isolated set of periodic orbits for $F$. Now, let $G_t$, $t\in(0,1]$ be a smooth curve of vector fields s.t. for every $t>0$, $G_t\in K$, and $G_0=F$, and let $N\subseteq M$ be a neighborhood of $C$ as in Lemma \ref{isol}. For every $t\in(0,1]$ set $J_t=\sum_{T\in Per_t}\varphi(T)$ - where $Per_t$ denotes the collection of all periodic orbits for the vector field $G_t$ inside $N$ whose period less than $p_1$, and $\varphi$ is the generic Orbit Index (see Lemma \ref{index3}). Set $i(C)=\lim_{t\to0} J_t$ - as proven at Prop.3.3 in \cite{PY2}, $i(C)$ is independent of the choice of $\{G_t\}_{t\in(0,1]}$ (for more details, see Prop.3.3 in \cite{PY2}). We now define the general Orbit Index as follows: 
\begin{definition}
     \label{index22}
     Given a relatively isolated collection of periodic orbits $C$ for the vector field $F_0$, its \textbf{Orbit Index} will be defined as the limit $i(C)$ above.  \end{definition}
As mentioned at the beginning of this Section, the Orbit Index is strongly related to the notion of global continuability (see Def.\ref{globalconti}) - to be precise, by Th.4.2 in \cite{PY2} we have the following result:

\begin{claim}
\label{contith}   Let $M$ be a smooth three-manifold and let $\dot{x}=F_t(x)$ denote a curve of vector fields varying smoothly in $(x,t)\in M\times[0,\infty)$. Assume $C$ is a relatively isolated collection of periodic orbits for $F_0$ s.t. $i(C)\ne0$ - then, $C$ includes a globally continuable periodic orbit for $F_0$.
\end{claim}
\begin{figure}[h]
\centering
\begin{overpic}[width=0.5\textwidth]{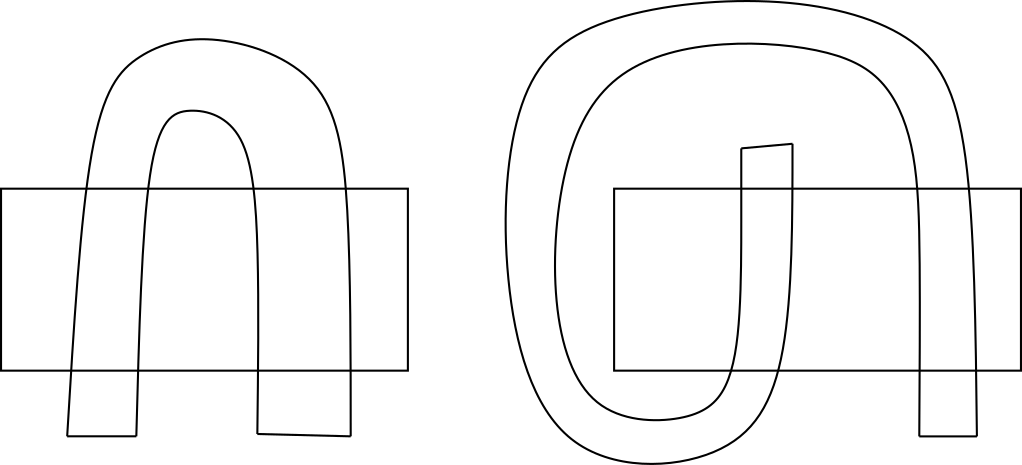}
\put(980,60){$B$}
\put(700,330){$h(CD)$}
\put(980,280){$D$}
\put(900,-30){$h(AB)$}
\put(590,60){$A$}
\put(590,280){$C$}
\put(380,60){$B$}
\put(40,-20){$h(CD)$}
\put(380,280){$D$}
\put(240,-20){$h(AB)$}
\put(-10,60){$A$}
\put(-10,280){$C$}
\end{overpic}
\caption{\textit{A Smale Horseshoe map on the left and a Fake Horseshoe on the right.}}
\label{horss}
\end{figure}
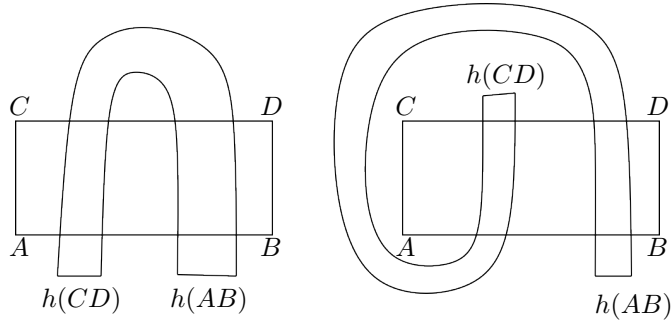

At this point we remark one can prove global continuability results even for some cases in which $i(C)=0$ - this can be achieved by studying the Linking Numbers of periodic orbits of Orbit Index $0$ (for more details, see Section $5$ in \cite{PY4}). Moreover, we conclude this section with the following fact which we will also use throughout this paper:

\begin{proposition}
    \label{dens1}
    Let $ABCD$ be a topological rectangle and assume $h:ABCD\to\mathbf{R}^2$ is either a Smale Horseshoe map or a Fake Horseshoe map (see the illustration in Fig.\ref{horss}). Then, the following is satisfied:
    \begin{itemize}
        \item If $h$ is a Smale Horseshoe map then any suspension of $h$ creates a countable collection of periodic orbits whose Orbit Index is $-1$. The Orbit Index of all other periodic orbits is $0$.
        \item  If $h$ is a Fake Horseshoe every periodic orbit generated by suspending $h$ has Orbit Index $-1$.
    \end{itemize}
\end{proposition}
In order not to interrupt the flow of ideas we defer the proof of Prop.\ref{dens1} to Appendix $B$ (see Prop.\ref{dens}), where we will discuss in greater detail what possible types of periodic orbits can be created under the suspension of either stretch-and-fold or tearing mechanisms. We will state however that similar results can also be obtained for many other maps $f:ABCD\to\mathbf{R}^2$ whose dynamics on their invariant sets are conjugate to the double-sided shift $\sigma:\{1,2,...,n\}^\mathbf{Z}\to\{1,2,...,n\}^\mathbf{Z}$ - for more details, see Appendix $B$.

\section{On the general persistence of periodic orbits under three-dimensional perturbation}
\label{perss}
In this section we prove the main result of this paper: Th.\ref{th1} and its corollaries, where we study the persistence of periodic dynamics under smooth deformations that preserve a certain heteroclinic condition fixed (we do so in Th.\ref{orbipers} and Th.\ref{pers13}). This section is organized as follows: we begin by generalizing the definition of relative isotopy from two-dimensional homeomorphisms to flows (see Def.\ref{relt}) - after which we briefly discuss how it manifests in a heteroclinic context. We then state and prove Th.\ref{orbipers} where we give sufficient conditions for the persistence of periodic dynamics under such "relative isotopies" (or more precisely, homotopies) of vector fields. Following that we use the Orbit Index Theory as sketched above to prove Th.\ref{pers13}, which together with Th.\ref{orbipers} completes the proof of Th.\ref{th1}. Finally, using Th.\ref{orbipers} we conclude this section by proving Cor.\ref{templateth} which allows us in some scenarios to associate a Template with a flow generating complex dynamics - even when no hyperbolicity conditions are satisfied by the flow.\\

Per the paragraph above we begin by discussing the generalization of relative isotopies from surface homeomorphisms to three-dimensional flows. To motivate the said generalization let us first consider some heuristics: assume $F$ is a smooth vector field of $S^3$, and let $H$ denote some invariant one-dimensional set for the flow - say, a periodic orbit, or a collection of fixed points and their one-dimensional invariant manifolds. Naively, one would generalize Def.\ref{relp} by fixing the behavior of $F$ on $H$ alone, and consider how this behavior constrains the possible smooth deformations of $F$ on $S^3\setminus H$. It is easy to see that under these (very) mild restrictions one can change the type of the fixed points on $H$, or allow periodic orbits to bifurcate from $H$ (or the fixed points on it - see the illustration in Fig.\ref{naive}).\\

\begin{figure}[h]
\centering
\begin{overpic}[width=0.35\textwidth]{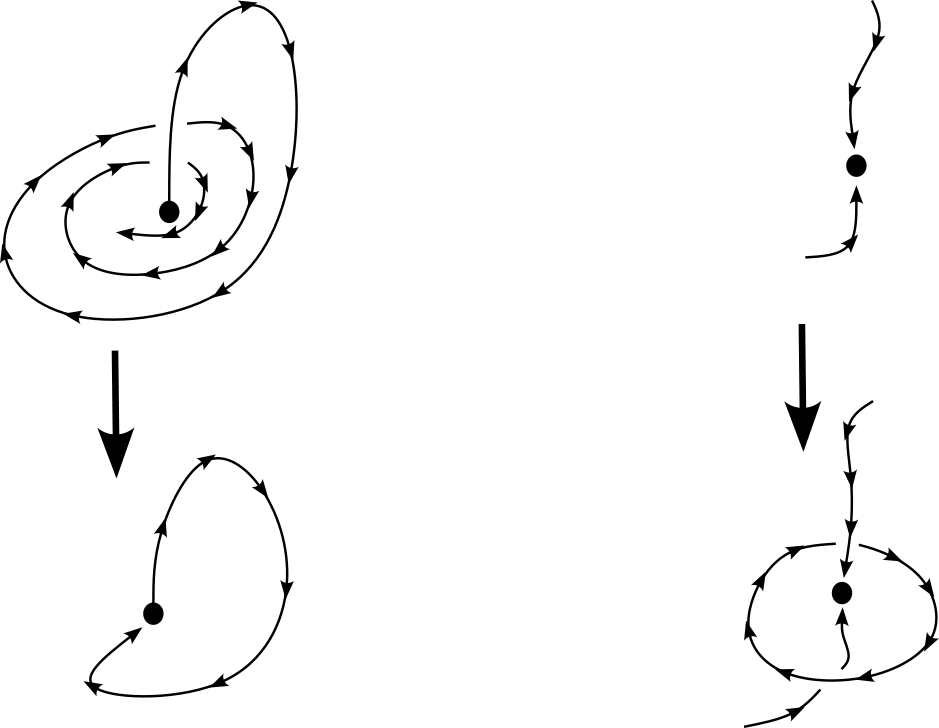}

\end{overpic}
\caption{\textit{Two deformations of $F$ in $S^3\setminus H$ which do not change the behavior on $H$ (or its topology) while changing the local dynamics around it - a saddle-focus becoming a real saddle on the left, and a Hopf bifurcation on the right.}}
\label{naive}
\end{figure}

Unfortunately, this approach is too naive to work with - mostly because it gives us too much freedom to radically alter the dynamics around and away from $H$. For example, we can change a complex saddle on $H$ to a real one, add fixed points away from $H$, make the invariant manifolds of $H$ intersect transversely (when $H$ is a periodic orbit), induce Hopf bifurcations, and remove fixed points of Poincare Index $0$ from $H$ - to name a few (see the illustration in Fig.\ref{remover}). These actions, even though they do not change too much the "essential" behavior of $F$ on $H$ can greatly change the behavior of $F$ in $S^3\setminus H$ - therefore, under such mild restrictions it is probably pointless to talk about any homotopy invariant properties of $F$ in $S^3\setminus H$.\\

\begin{figure}[h]
\centering
\begin{overpic}[width=0.4\textwidth]{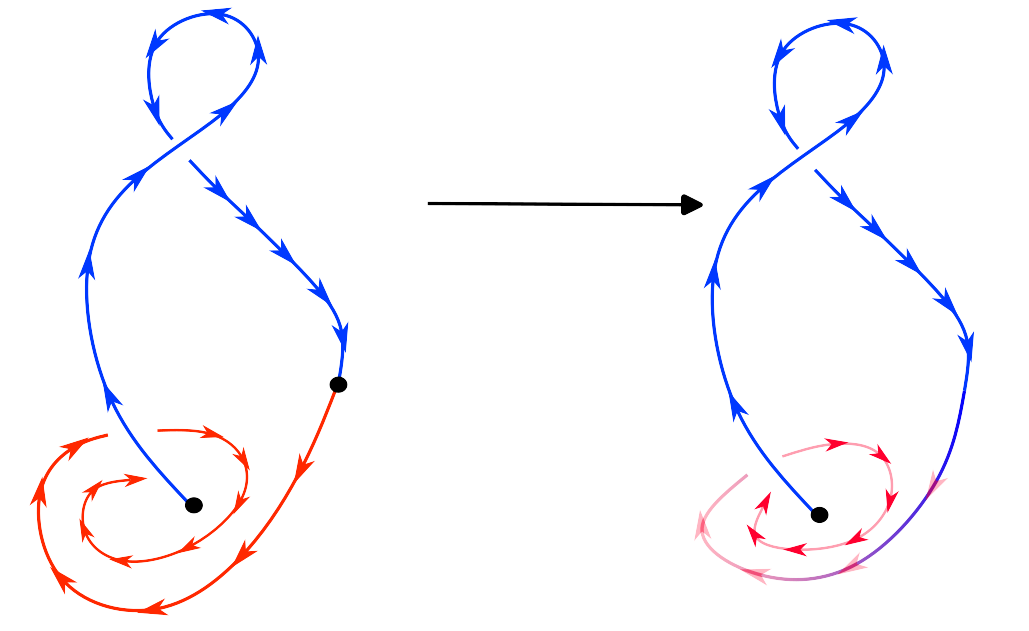}

\put(115,90){$O_1$}
\put(795,130){$O_1$}
\put(355,200){$O_2$}

\end{overpic}
\caption{\textit{In this illustration $H$ is composed from two fixed points (a saddle-focus $O_1$ and $O_2$ of Poincare Index $0$), and the heteroclinic orbit connecting them. A deformation of the vector field on the left removes $O_2$ (as it is of Poincare Index $0$) and turns $H$ into a homoclinic orbit to $O_1$ (on the right).}}
\label{remover}
\end{figure}

These examples show that if we wish to meaningfully generalize the notion of relative isotopy from two-dimensional dynamics to three-dimensional flows we must take into account two things:

\begin{enumerate}
    \item First, we must take the local dynamics of $F$ around $H$ into account - that is, the type of the fixed points for the flow which lies on $H$ matters. That is, we cannot change a complex saddle into a real one, or allow Hopf Bifurcations.
    \item Second, we cannot consider $F$ s.t. it has fixed points away from $H$, nor can we deform the dynamics s.t. fixed points are added in $S^3\setminus H$, as this could change the constrains placed by $H$ on the dynamics in $S^3\setminus H$.
\end{enumerate}

With these ideas in mind, we now generalize the notion of relative isotopies from surface homeomorphisms (see Def.\ref{relp}) to flows as follows:

\begin{definition}
\label{relt}    Let $H\subseteq S^3$ be an invariant, one-dimensional set for a $C^k$ vector field $F$ of $S^3$, where $k\geq3$. Let $G$ be another smooth vector field - we say $F$ can be deformed to $G$ relative to $H$ (or in short, $rel$ $H$) if there exists a homotopy of vector fields $\{G_t\}_{t\in[0,1]}$ varying smoothly in $S^3\times[0,1]$ s.t. the following is satisfied:
    \begin{itemize}
        \item $G_0=F$ and $G_1=G$.
        \item For every $t\in[0,1]$ the set $H$ is invariant under $G_t$.
        \item No fixed points on $H$ are added (or removed) from $H$ during the deformation - that is, for every $t_1,t_2\in[0,1]$, the vector field $G_{t_1}$ has precisely as many fixed points on $H$ as $G_{t_2}$. 
        \item For every $t_1,t_2\in [0,1]$, if $x\in H$ is a fixed point of some topological type for $G_{t_1}$ it remains a fixed point of the same topological type for $G_{t_2}$. Moreover, for all $t\in[0,1]$ the vector field $G_t$ has no fixed points in $S^3\setminus H$.
    \end{itemize}
\end{definition}
\begin{figure}[h]
\centering
\begin{overpic}[width=0.4\textwidth]{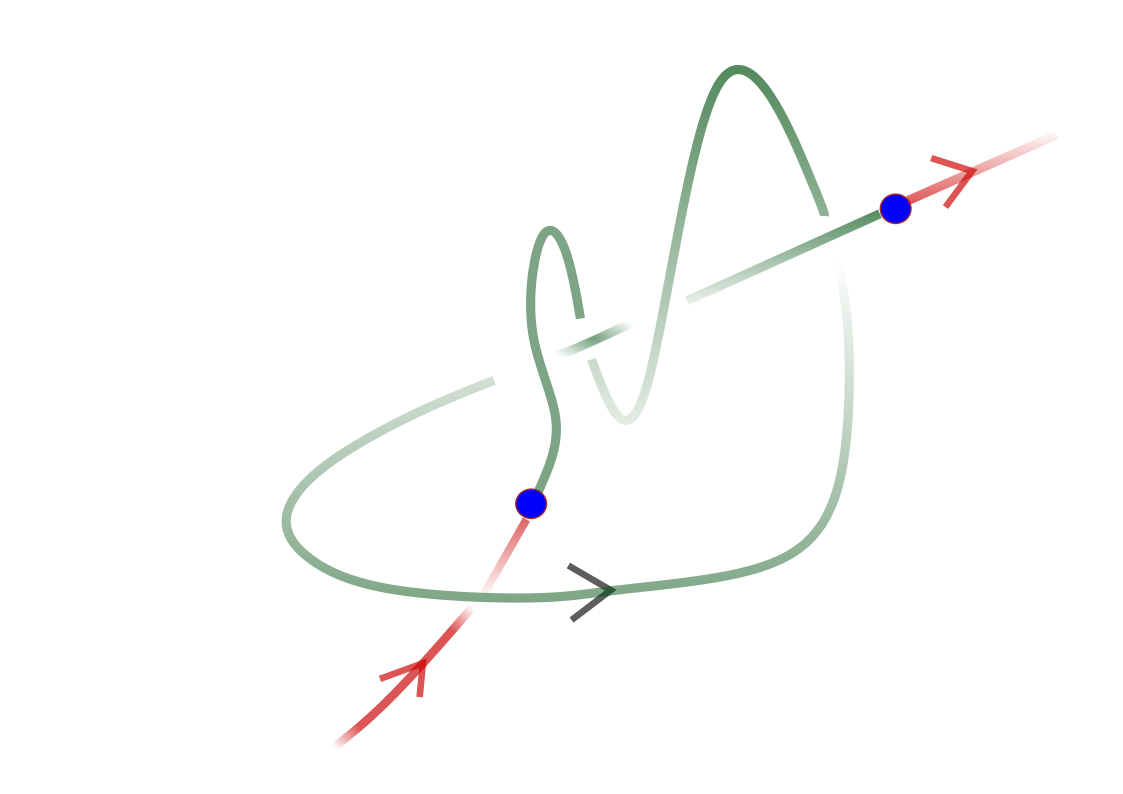}
\put(350,245){$P_{1}$}
\put(780,440){$P_{2}$}
\end{overpic}
\caption[Fig39]{\textit{A heteroclinic knot $H$ generated by joining two finite fixed points, $P_1$ and $P_2$, a fixed point at $\infty$, and the heteroclinic orbit connecting them. }}
\label{heterock}
\end{figure}

Put simply, when we homotopically deform flows $rel$ $H$ in addition to fixing the qualitative behavior of $F$ on $H$ we also require that no bifurcations on the fixed-points in $H$ take place - in particular, no fixed points can be removed or added to the flow, and no Hopf bifurcations can take place either. Moreover, we also fix the behavior of $F$ away from $H$, in the sense that we cannot add fixed points to it - or in other words, we can at most "stretch or bend" flow lines away from $H$. With these ideas in mind, inspired by the notion of Essential Dynamics for surface homeomorphisms (see Def.\ref{essential2}) we now introduce the following definition:

\begin{definition}
    \label{essentialclass} Let $F$ be a smooth vector field in $S^3$, and let $H$ denote some invariant, one-dimensional set for $F$ s.t. whenever $F$ has fixed points they all lie on $H$. Then, the \textbf{Essential Class} \textbf{of} $F$ \textbf{w.r.t.} $H$ will be defined as the collection of periodic orbits for $F$ which persist under smooth deformations $rel$ $H$ without changing their knot type (see Def.\ref{knot}) or collapsing to one another by some bifurcation. We will denote the said set by $Ess(F)$.
    \end{definition}
\begin{figure}[h]
\centering
\begin{overpic}[width=0.35\textwidth]{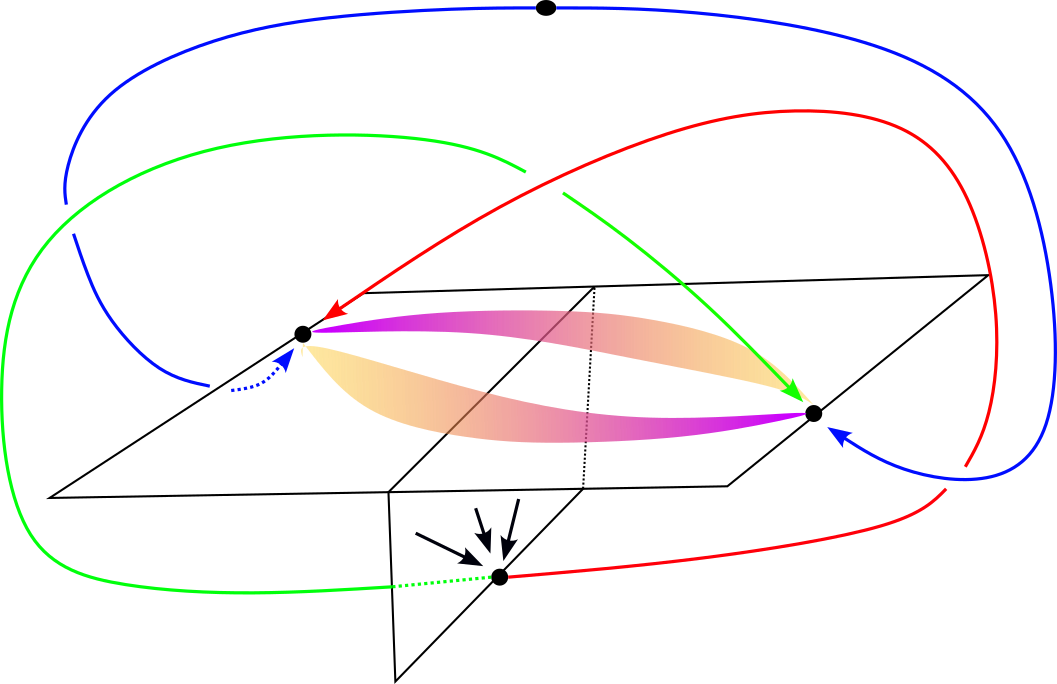}
\put(810,250){$p_-$}
\put(470,50){$o$}
\put(210,340){$p_+$}
\put(490,590){$\infty$}

\end{overpic}
\caption{\textit{A variant on the Geometric Lorenz Model (see \cite{SAB}), where we have a heteroclinic knot connecting two saddle foci $p_\pm$, a real saddle $o$, and a repeller at $\infty$. In the next section we will prove using Th.\ref{orbipers} the essential class of this vector field is infinite.}}
\label{suspendda}
\end{figure}
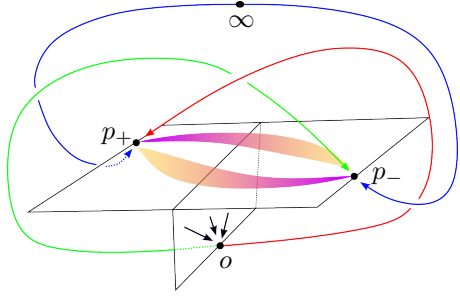

Needless to say, given an arbitrary $F-$ invariant one dimensional set $H$ as above, in general there is no good reason to assume $Ess(F)$ is non-empty, and this would very much depend on the choice of $H$. However, regardless of whether $Ess(F)$ is empty or not it is easy to see that if $G$ can be deformed $rel$ $H$ to $F$, we have $Ess(F)=Ess(G)$ - i.e., at any case the Essential Class is a topological invariant depending only on $H$. In this spirit we state that for the remainder of this paper we will only study the Essential Class of a vector field $F$ when $H$ is a curve composed of heteroclinic orbits, connecting all the fixed points of $F$. To make these notions clear, we now define the following:

\begin{definition}
    \label{heteroknot}
   Let $F$ be a $C^k$, $k\geq1$ vector field of $S^3$ s.t. $F$ has finitely many fixed points in $S^3$, $x_1,...,x_n$, $n\geq2$, where every $x_i$ is connected to $x_{i+1}$ by a heteroclinic orbit $H_i$ (where $x_{n+1}=x_1$). Then, the set $H=\{x_1,...,x_n\}\cup H_1\cup...\cup H_n$ is said to be a \textbf{heteroclinic knot} (see the illustrations in Fig.\ref{suspendda} and \ref{suspenddb}).
\end{definition}

To justify our choice of words observe that given a heteroclinic knot $H$ by definition it is homeomorphic to $S^1$ - hence it is a knot per Def.\ref{knot}. We further remark that similarly to Def.\ref{relt} one can think of a heteroclinic knot $H$ and the constrains it places on the $rel$ $H$ deformations of $F$ as an analogue to how the permutation a surface homeomorphism $f:S\to S$ induces on $\partial S$ constrains the dynamics of $f$ in $S$. With these ideas in mind, we are finally ready to state and prove Th.\ref{orbipers}:
\begin{theorem}
    \label{orbipers}
    Assume $F$ is a $C^k$ vector field of $S^3$, $k\geq3$, which generates a heteroclinic knot $H$ connecting all the fixed points of $F$- all of which are non-degenerate. Moreover, assume $F$ can be smoothly deformed $rel$ $H$ to a $C^3$ vector field $G$ s.t. there exists a surface $R$ of negative Euler characteristic endowed with a Pseudo-Anosov map $P:R\to R$ and two cross-sections $S_1\subseteq S_2\subseteq S^3$ for which the following is satisfied (see Fig.\ref{suspendda} and Fig.\ref{suspenddb})
            \begin{enumerate}
            \item $S_1$ is open in $S_2$ and transverse to $G$ - moreover, $S_1\cap H=\emptyset$ while $\overline{S_1}\cap H\ne\emptyset$. Similarly, $S_2\setminus H$ is transverse to the $G$ at its interior (note we do not require $S_1$ to be connected).
            \item The first-return map w.r.t. $G$, $g:{S_1}\to S_2$, is a homeomorphism. Moreover, $g$ extends continuously to $\overline{S_1}\cap H$ (not necessarily injectively - see Fig.\ref{suspenddb} or Fig.\ref{suspendda}).
            \item There exists some open set $R_1\subseteq R$ s.t. the first return map $g:S_1\to S_2$ is conjugate to $P:R_1\to R$ away from $\partial S_1$ (where the boundary is taken in $S_2$). 
            \item Let $\Gamma$ denote the spine of the surface $R$ which is embedded in $R$ (see Def.\ref{spi}), let $\gamma\subseteq\Gamma$ denote the sub-graph s.t. $R_1$ is retracted to $\gamma$, and let $p':\Gamma\to\Gamma$ denote the induced graph map (see the discussion immediately before Th.\ref{betshan}). Then, there exists $E$, a collection of edges in $\gamma$ s.t. $p':\gamma\to\gamma$ covers $E$ at least twice. 
            \item  Finally, if $h:S_1\to R_1$ is the homeomorphism s.t. $h^{-1}\circ P\circ h=g$, then given a vertex $p\in \gamma$ there exists some $x\in \partial S_1\cap H$ s.t. if $\{x_n\}_n\subseteq S_1$ satisfies $x_n\to x$ then $h(x_n)\to p$ (that is, the points in $\partial S_1\cap H$ correspond to the vertices of $\gamma$). 
        \end{enumerate} 
 
Then, under these assumptions $Ess(F)$ includes infinitely many periodic orbits. Moreover, as $F$ is deformed $rel$ $H$ all these periodic orbits retain their linking type with one another and with $H$.
\end{theorem}

Despite its technical formulation Th.\ref{orbipers} has a clear geometric meaning. To clarify it assume $F$ is a smooth vector field of $S^3$ which generates a bounded chaotic attractor $A$ inside inside $S^3\setminus H$. Then, Th.\ref{orbipers} tells us that if we can "straighten" the dynamics of $F$ $rel$ $H$ around $A$ s.t. the dynamics on $A$ become either those of a suspended Smale Horseshoe, a Fake Horseshoe, a Baker's Map (or any other such mechanism), the original dynamics of $F$ on $A$ must also be very complex. Or, put simply, Th.\ref{orbipers} implies two things:

\begin{itemize}
    \item First - that the dynamics of $F$ on $A$ are at most those of a deformed hyperbolic attractor - and in particular, the dynamics of the hyperbolic attractor (i.e., the vector field $G$) serve as a "topological lower bound" for the complexity of $F$.
    \item Second - that in order to verify the complex nature of the dynamics on $A$, it is enough to verify it for some idealized model $G$ s.t. $F$ can be deformed $rel$ $H$ to $G$. 
\end{itemize}

With these ideas in mind, we are now ready to prove Th.\ref{orbipers}.
\begin{proof}

The proof of Th.\ref{orbipers} will be mostly geometric - therefore, before giving a sketch of proof we give the basic intuition behind it. To begin, recall that given a surface homeomorphism $f:S\to S$ s.t. $S$ has negative Euler characteristic (for example, when $S$ is a disc punctured at three - or more - interior points) the behavior of $f$ on $\partial S$ determines the isotopy class of $f$ (see the illustration in Fig.\ref{relative}). In other words, the way $f$ permutes the punctures determines completely its essential class of periodic dynamics in $S$. The idea here is similar, in the sense that in our case instead of $f$ we have a flow generated by a vector field $F$. Since we can deform $F$ $rel$ $H$ to $G$ it is easy to see that heuristically $F$ and $G$ "permute" the fixed points in the same way - as such, since $F$ is at most a "deformed" version of $G$ away from $H$ we would expect it to also generate whatever interesting dynamics $G$ generates.\\

With this intuition in mind we now give a sketch of proof. To do so, recall the first return map $g:S_1\to S_2$ for $G$ as defined above (in particular, recall $\partial{S_1}\cap H\ne\emptyset$ -where that boundary is taken in $S$). We first prove that as we deform $G$ $rel$ $H$ back to $F$ the cross-sections $S_1$ and $S$ also persist, thus inducing an isotopy of the first return map $g:S_1\to S_2$ to $f:S_1\to S_2$, a first-return map for the original vector field $F$. Second, we show that since by assumption the dynamics of $g$ on $S_1$ can be embedded as a subset of a Pseudo-Anosov map it will follow the dynamics of $f$ on $S_1$ can be embedded inside those of a map that can be isotoped to a Pseudo-Anosov map. Finally, by Th.\ref{stability} we will conclude that whatever periodic dynamics exist at the invariant set of $g$ in $S_1$, these dynamics persist as $g$ is isotoped to $f$ - thus implying Th.\ref{orbipers}.\\

\begin{figure}[h]
\centering
\begin{overpic}[width=0.4\textwidth]{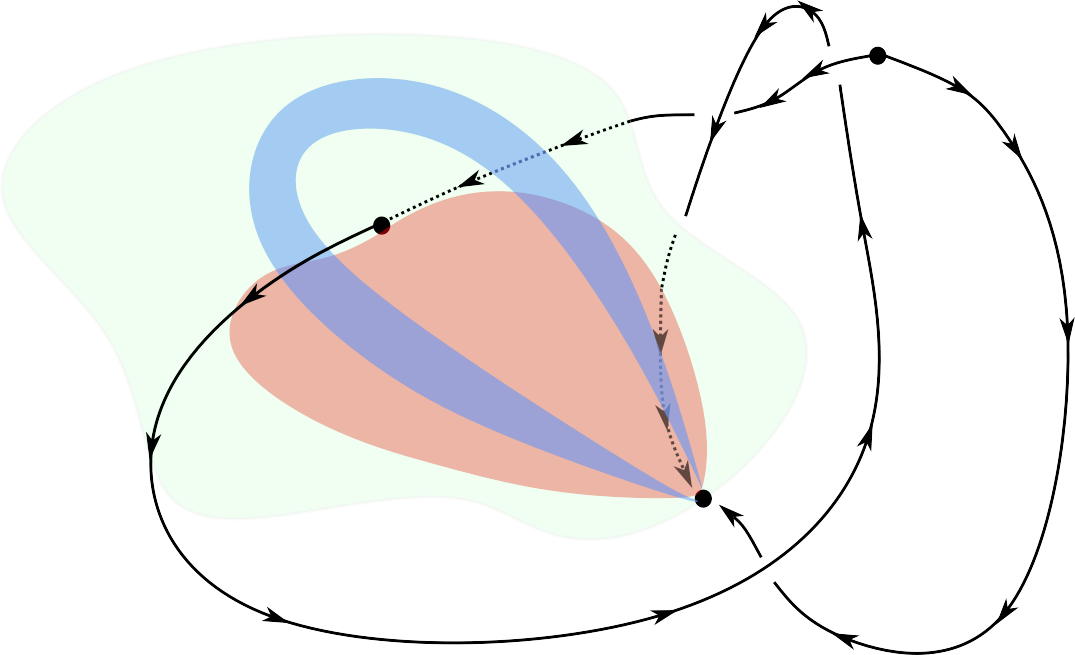}
\put(700,170){$p_2$}
\put(980,50){$H$}
\put(820,590){$p_3$}
\put(230,80){$\Theta$}
\put(330,430){$p_1$}
\put(150,480){$S$}
\put(450,350){$S_1$}
\end{overpic}
\caption{\textit{A cross-section $S$ transverse to a heteroclinic trajectory $\Theta\subseteq H$ connecting the fixed points $p_2$ and $p_3$, which stretches and folds $S_1\subseteq S$ on itself as in Th.\ref{orbipers} (the blue region denotes the image of the first return map). Using Th.\ref{orbipers}, in the next section we will prove the essential class of this flow is infinite. Note that in this scenario the first-return map extends continuously to $p_1$ and maps it to the fixed point $p_2$ (it is not injective at $p_2$).}}
\label{suspenddb}
\end{figure}
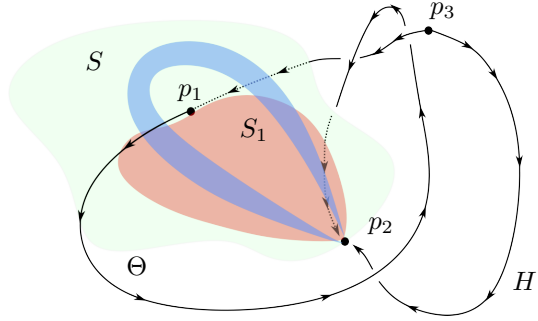

To begin, note that as we deform $G$ $rel$ $H$ back to $F$ we can assume the cross-section $S_2\setminus H$ remains transverse to the flow - to see why, recall that as remarked earlier deformations $rel$ $H$ are very strict: that is, we only deform the flow away from some solid, possibly knotted torus $T$ which includes $H$ in its interior. It is easy to see that as we deform the flow inside $T$ back to $F$ we can deform it s.t. around points in $S_2\cap H$ the cross-section remains transverse to the flow (see the illustration in Fig.\ref{deform221}). In more detail, by noting the vector field $G$ is transverse to $S_2\setminus H$ (per assumption) and by moving $S_2$ along flow lines and adding small, two-dimensional sets to it which are transverse to the flow (if necessary), we can ensure that as we deform $G$ back to $F$ the vector fields all remain transverse to $S_2$ (see the illustration in Fig.\ref{deform221}). That is, we deform homotopically the surface $S_2\setminus H$ as we smoothly deform $G$ back to $F$(if necessary) s.t. $S_2\setminus H$ remains transverse to $H$.\\

Therefore, because no fixed points are add to the flow outside of $T$ as we deform $G$ $rel$ $H$ back to $F$ it is easy to see the only ways we are allowed to deform the flow away from $T$ are to either move, stretch, and fold flow lines intersecting $S_2\setminus H$ - or alternatively, change the velocities along solution curves (without adding any new fixed points). This implies we can deform $S_2\setminus T$ s.t. it remains transverse to the flow as we deform $G$ $rel$ $H$ back to $F$ - and in particular, we can homotopically deform $S_2\setminus H$ s.t. $S_1$ also remains transverse to the flow throughout the deformation (see the illustration in Fig.\ref{deform221}).\\ 

\begin{figure}[h]
\centering
\begin{overpic}[width=0.4\textwidth]{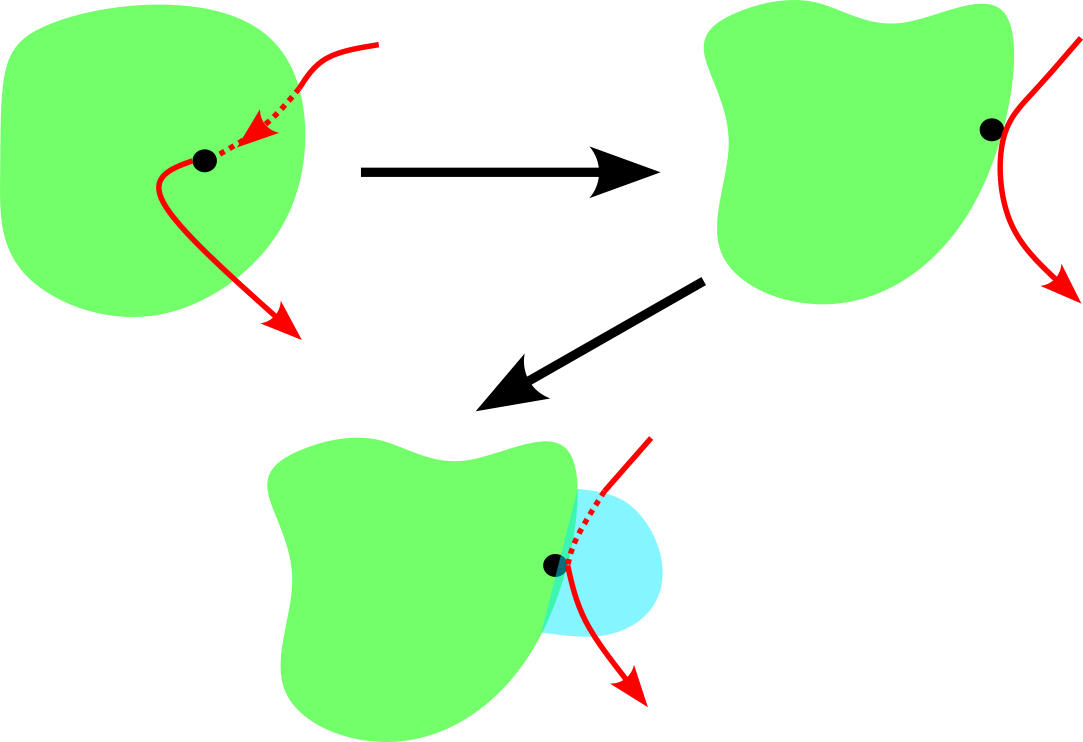}

\end{overpic}
\caption{\textit{Deforming the cross-section with the flow - on the upper left side we have a flow line transverse to $S_2$ (w.r.t. $G$) that as we deform $G$ $rel$ $H$ back to $F$ moves to the boundary of $S_2$. We extend $S_2$ as appears below (the cyan region which we add to the cross-section) s.t. the said flow line remains transverse to $S_2$. By $S_1\subseteq S_2$ this implies $g:S_1\to S_2$ varies homotopically as well.}}
\label{deform221}
\end{figure}

We now apply this argument to the first-return map $g:S_1\to S_2$. To do so, note that as the first return map is constrained by the heteroclinic orbits on $H$ which permutes the points of $\overline{S_1}\cap H$, it is easy to see that as we deform $G$ $rel$ $H$ back to $F$ we can homotopically deform the cross-sections $S_1\setminus H$ and $S_2\setminus H$ s.t. the first-return map from $S_1$ to $S_2$ remains continuous (see the illustration in Fig.\ref{deform22}). Or in other words, the deformation of $G$ to $F$ $rel$ $H$ induces an isotopy of continuous maps $g_t:S_1\to S_2$, $t\in[0,1]$, s.t. $g_0=g$ and $g_1=f$, where $f:S_1\to S_2$ is the first-return map corresponding to $F$ (see the illustration in Fig.\ref{deform22}). Moreover, it is easy to see that since per assumption $g$ extends continuously to $\partial{S_1}\cap H$ (where the closure is taken in $S_1$) the same is true for all $g_t,t\in[0,1]$ (although that extension needs not be injective on $\overline{S_1}$). As such given $p\in\partial{S_1}\cap H$ and $\{x_n\}_n\subseteq S_1$ s.t. $x_n\to p$ the limit $\lim_{n\to\infty}g_t(x_n)=g_t(p)$ is independent of $t\in[0,1]$, and depends only on how $G$ (and consequentially $F$) permutes the points of $H\cap\overline{S_1}$.\\

\begin{figure}[h]
\centering
\begin{overpic}[width=0.5\textwidth]{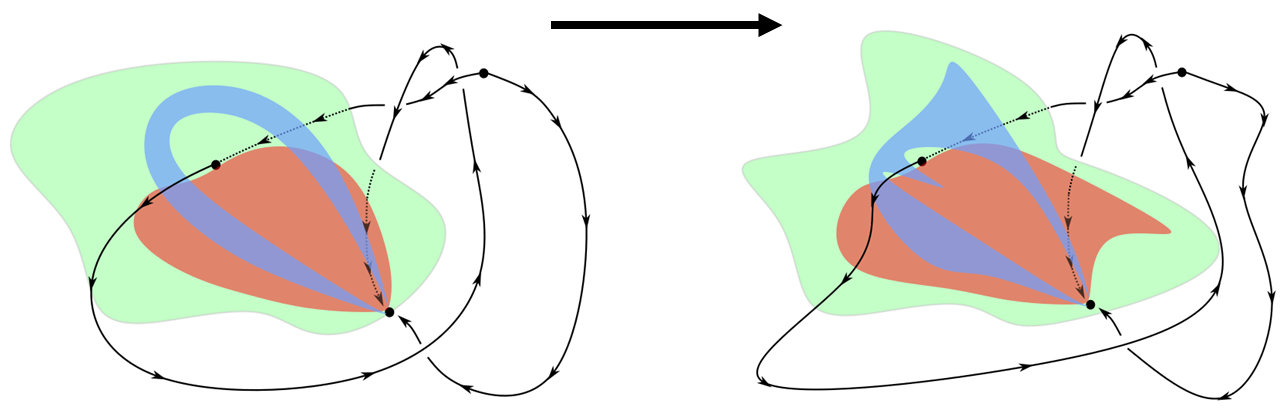}
\end{overpic}
\caption{\textit{Deforming the cross-section and the first-return map $rel$ $H$ as $G$ is deformed back to $F$. Note this deformation induces an isotopy on the first-return maps.}}
\label{deform22}
\end{figure}

To continue set $V=\overline{S_1}\cap H$ and recall the map $g:S_1\to S$ can be embedded inside some Pseudo-Anosov map $P:R\to R$ (where $R$ has negative Euler characteristic) s.t. the map $g:S_1\to S_2$ is conjugate to $P:R_1\to R$ (where $R_1$ is some open subset of $R$). In addition, recall that per assumption there exists some sub-graph $\gamma\subseteq\Gamma$ (where $\Gamma$ is the spine of $S$) s.t. $\gamma$ is the retract of $R_1$ and the graph map induced by $P$, $p':\gamma\to\Gamma$, includes an edge in $\gamma$ which covers itself twice. Further recalling that per assumption every vertex in $\gamma$ corresponds to a point in $\partial S_1\cap H$, this implies the isotopy $g_t:S_1\to S$ induces an isotopy $P_t:R\to R$, $t\in[0,1]$ s.t. the following is satisfied (see the illustration in Fig.\ref{deform23}):
\begin{enumerate}
    \item $P_0=P$.
    \item For every $t$, $g_t|_{S_1}$ is conjugate to $P_t|_{R_1}$.
    \item For every $P_t$, the induced graph map $p'_t:\Gamma\to\Gamma$, $p'_0=p'$ satisfies that $p_t(\gamma)$ covers some collection of edges in $\gamma$ at least twice.
\end{enumerate}

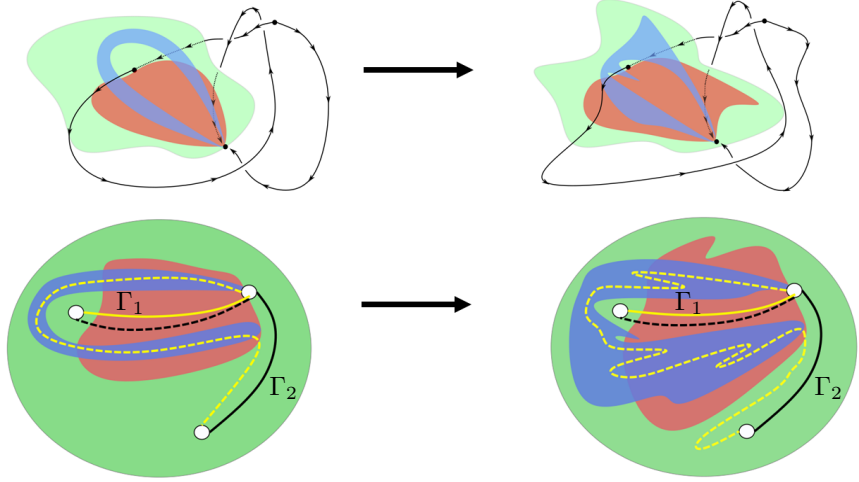
\begin{figure}[h]
\centering
\begin{overpic}[width=0.7\textwidth]{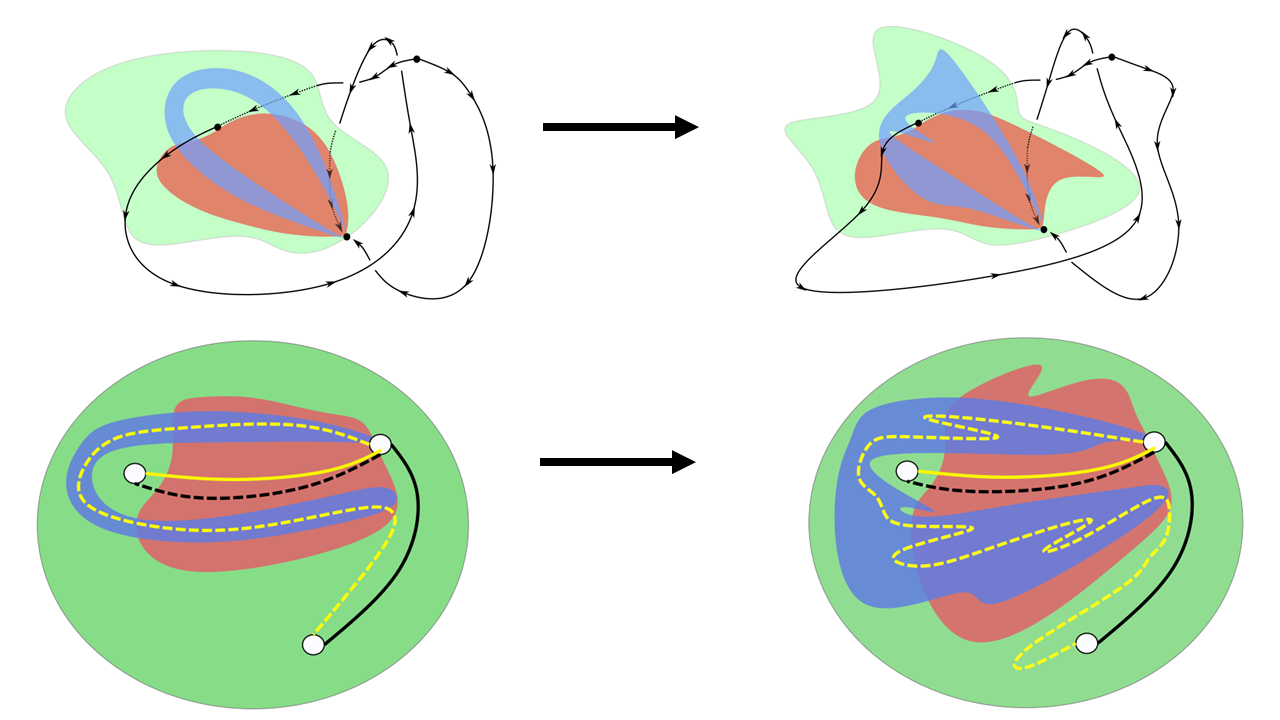}
\put(770,195){$\Gamma_1$}
\put(150,195){$\Gamma_1$}
\put(320,100){$\Gamma_2$}
\put(925,100){$\Gamma_2$}
\end{overpic}
\caption{\textit{Embedding of the first return maps inside $R$ (in this case, a thrice punctured disc) and the relative isotopies their deformation induces by permuting the punctures. The spine of $R$ is $\Gamma_1\cup\Gamma_2$, and $R_1$ is represented by the red set (n this scenario $\Gamma_1=\gamma$). The action on the spine is denoted by the dashed lines (note $\Gamma_1$ is always covered twice). }}
\label{deform23}
\end{figure}

Since $P:R\to R$ is a Pseudo-Anosov map by Th.\ref{stability} and by $P=P_0$ we know all its periodic orbits persist under the isotopy $P_t:R\to R$, $t\in[0,1]$ - and moreover, they do so without collapsing into one another or changing their (discrete-time) minimal periods w.r.t. $P_t$. In addition, since the induced graph maps $p'_t:\gamma\to\Gamma$  all cover some collection of edges in $\gamma$ twice throughout the isotopy, we know there exists an infinite class of periodic orbits in $R_1$ which persist throughout the isotopy. We remark that as we vary $P_t,t\in[0,1]$ none of these periodic orbits can lie on $\partial R_1$ - for if that was the case we could always isotopically deform $P_t$ by some other isotopy on $R$ s.t. this periodic orbit collapses to some vertex of $\gamma$ (i.e., into some boundary component of $R_1$ - see the illustration in Fig.\ref{dest}). Since by Th.\ref{stability} all the periodic orbits of $P$ are essential hence unremovable (as $P$ is Pseudo-Anosov), we conclude no periodic orbit for $P$ in $R_1$ escapes to $\partial R_1$ along the isotopy.\\  

\begin{figure}[h]
\centering
\begin{overpic}[width=0.5\textwidth]{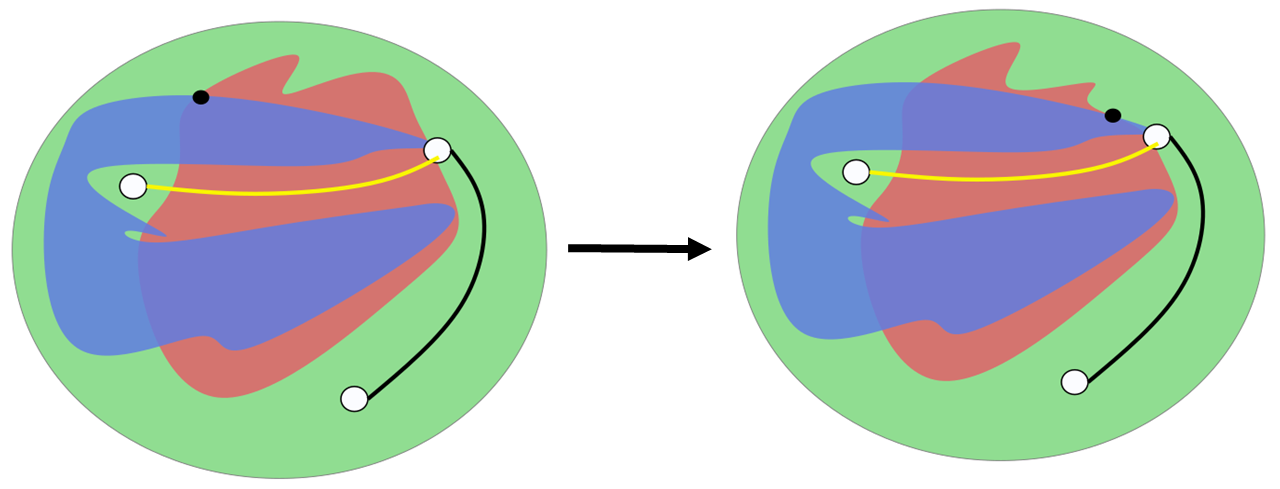}
\end{overpic}
\caption{\textit{Removing a periodic point on $\partial R_1$ by pushing it to a component of $\partial R$ by an isotopy. As this proccess destroys a periodic point, throughout the isotopy periodic points in the essential class of $P$remain away from $\partial R_1$.}}
\label{dest}
\end{figure}

By the arguments above it follows all the periodic points for $P$ in $R_1$ remain strictly interior to $R_1$ throughout the isotopy $P_t:R_1\to R$, $t\in[0,1]$ - and since for all $t\in[0,1]$ the map $P_t:R_1\to R$ is conjugate to the first-return map $g_t:S_1\to S_2$ we conclude that for all $t\in[0,1]$ the map $g_t$ has infinitely many periodic points in $S_1$, which vary smoothly with $t$. Moreover, as the period of these periodic orbits is independent of $t$ and since they do not collapse to one another we further conclude the following:

\begin{enumerate}
    \item Every periodic orbit for $G$ which intersects transversely $S_1$ persists as we deform $G$ $rel$ $H$ back to $F$, without changing its knot type or its linking with $H$.
    \item No two periodic orbits for $G$ which intersect transversely with $S_1$ collapse into one another as we return $rel$ $H$ to $F$ - and again, their linking with one another persists as well. 
\end{enumerate}

Finally, since $R_1$ has infinitely many periodic points for $P$ we conclude the same is true for $G$, i.e., there exist infinitely many periodic orbits for $G$ which intersect $S_1$ and satisfy the above - i.e., $F$ also generates infinitely many periodic orbits which intersect $S_1$ transversely. Moreover, since we chose $F$ s.t. $F$ can be deformed $rel$ $H$ to $G$ by Def.\ref{essentialclass} it follows all these periodic orbits all lie in $Ess(F)$. The proof of Th.\ref{orbipers} is now complete.
\end{proof}
\begin{remark}
    In fact, we can say a little more - by Th.\ref{stability} we know all the dynamics of $P:R_1\to R$ persist, i.e., for every $t\in(0,1]$ there exists some $R'_1\subseteq R_1$ and a continuous, surjective $\pi:R'_1\to R_1$ s.t. $\pi\circ P_t=P\circ\pi$ (in particular, the pre-image of a periodic point under $\pi$ would include a fixed point). As such, something similar remains true for the isotopy $g_t:S_1\to S_2$ (and consequentially, for the resulting deformation of the flows). 
\end{remark}
\begin{remark}
The assumption that the cross-section $S$ intersects the heteroclinic knot cannot be removed, as exemplified in Fig.\ref{removv}. In that image we show how complex Horseshoe dynamics generated by a first-return map that is not constrained by a heteroclinic knot can be removed by a very simple deformation of the flow (for more details on the destruction of Horseshoes in such non-constrained a scenario, see \cite{Perd}). 
\end{remark}
\begin{figure}[h]
\centering
\begin{overpic}[width=0.5\textwidth]{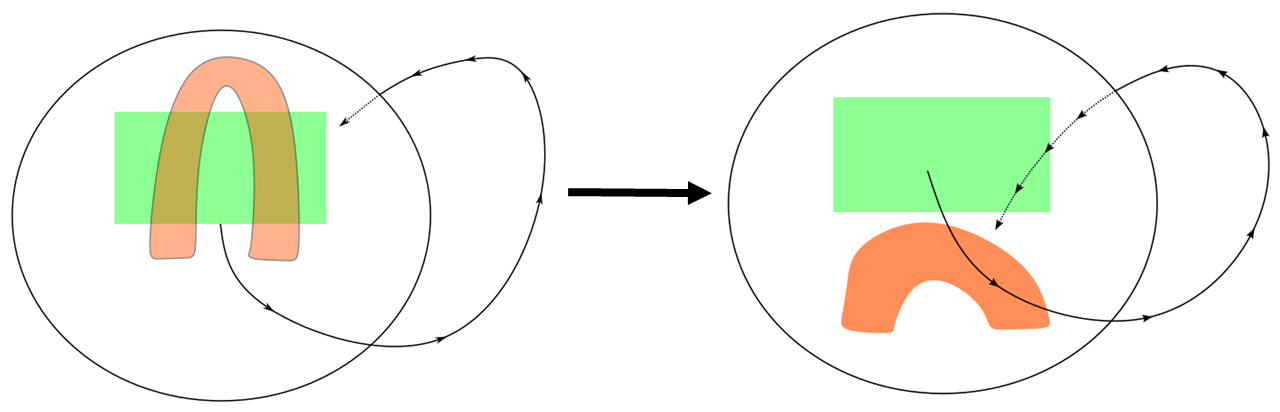}
\end{overpic}
\caption{\textit{On the left we have a suspension of the disc which creates a first-return map that includes a horseshoe. Since there is no constrain placed upon that suspension, we can deform it as described above and destroy these horseshoe dynamics.}}
\label{removv}
\end{figure}

Having proven Th.\ref{orbipers}, for the remainder of this section we will be concerned with how it can be applied to study both the bifurcations and the topology of periodic dynamics for three-dimensional flows - which we do by combining the conclusions of Th.\ref{orbipers} both with the Orbit Index Theory and Template Theory. To begin, recall $K$, the family of vector fields considered in Prop.\ref{propk} - i.e., $K$ is the family of vector fields where periodic orbits can undergo precisely three types of bifurcations:

\begin{enumerate}
    \item Saddle node bifurcations.
    \item Period-doubling bifurcations.
    \item Hopf bifurcations.
\end{enumerate}

In particular, for every vector field $F\in K$ and any periodic orbit $P$ for $F$ either the Orbit Index of $P$ is well-defined (see Def.\ref{index1}), or $P$ is a bifurcation orbit (w.r.t. perturbations of $F$ in $K$). In addition, recall $K$ is generic (see Prop.\ref{propk}). We begin with the following corollary of Th.\ref{orbipers} about the persistence of periodic dynamics in the breakdown of a heteroclinic knot:

\begin{proposition}
\label{pers11}    Let $F, H$ and $G$ be as in Th.\ref{orbipers} and assume that in addition if $T$ is a periodic orbit for $G$ then $i(T)\in\{0,-1\}$ (where $i$ is the Orbit Index). Now, let $P\in Ess(F)$ be a periodic orbit for $F$ s.t. when $F$ is deformed to $G$ $rel$ $H$, $P$ is deformed to a periodic orbit $T$ s.t. $i(T)=-1$. Then, for a generic choice of such $F$ the periodic orbit $P$ persists under all sufficiently small $C^k$ perturbations of $F$ (where $k\geq3$) - and moreover, it does so without changing its knot type. 
\end{proposition}
\begin{proof}
To begin, let $\{G_t\}_{t\in[0,1]}$ denote a smooth curve of vector fields in $S^3\times[0,1]$ s.t. $G_0=G$ and $G_1=F$ which smoothly deforms $G$ $rel$ $H$ to $F$ - we will prove Prop.\ref{pers11} under the assumption $\{G_t\}_{t\in[0,1]}\subseteq K$. By the $C^3$-density of $K$ in the space of $C^3$ vector fields on $S^3\setminus H$ (or alternatively, just $S^3$ - see Prop.\ref{propk}) this would suffice to imply Prop.\ref{pers11}. To continue, let $Per$ denote the collection of periodic orbits for the curve $\{G_t\}_{t\in[0,1]}$ - i.e., every component of $Per\cap S^3\times\{t\}$ is a periodic orbit for the vector field $G_t$. Now, given any periodic orbit $T$ for $G_0=G$ s.t. $i(T)=-1$ set $Per_T$ as the component of $Per$ s.t. $T\times\{0\}\subseteq Per_T$. Since we know by Th.\ref{orbipers} that $T$ persists as we vary $G=G_0$ to $F=G_1$ along the curve, it follows there are no center fixed points (i.e., Hopf bifurcations) on $\overline{Per_T}$. Therefore, by $\{G_t\}_{t\in[0,1]}\subseteq K$ we know there are precisely two types of bifurcation orbits which can exist on $Per_T$: period-doubling and saddle node (see Def.\ref{type} and the illustration in Fig.\ref{nosn}).\\

\begin{figure}[h]
\centering
\begin{overpic}[width=0.3\textwidth]{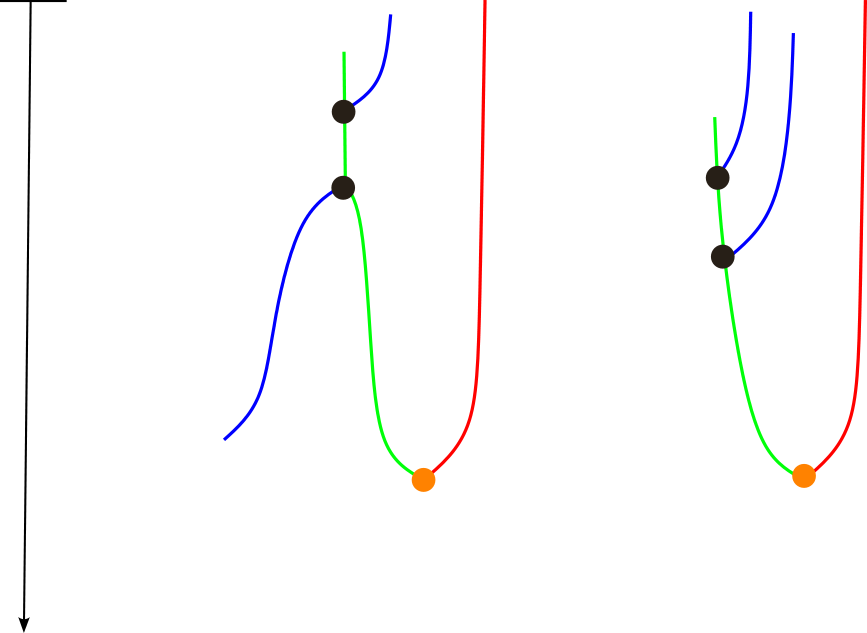}
\put(10,735){$0$}
\put(10,-50){$t$}
\end{overpic}
\caption{\textit{Two (partial) bifurcation diagrams for two distinct periodic orbits with Index $-1$ which begin at $G_0$ - the red curves denote when the Orbit Index is $-1$, the green when it is $1$, and the blue when it is $0$. The orange dots denote saddle node bifurcation orbits while the black denote period-doubling bifurcation orbits. We prove both these scenarios cannot occur on $Per_T$. }}
\label{nosn}
\end{figure}

We now prove that if $T'\times\{t\}$ is a periodic orbit for $G_t$ on $Per_T$ (for some $t\in[0,1]$) then $i(T')=-1$ - i.e., we prove the Orbit Index is constantly $-1$ on $Per_T$. To do so observe that since $i(T)=-1$ by Th.\ref{invar} we know the Orbit Index of any periodic orbit on $Per_T$ sufficiently close to $T\times\{0\}$ would also be $-1$ - and by Lemma \ref{notyp2} we know that in order for the Orbit Index on $Per_T$ to change from $-1$ to either $0$ or $1$ a saddle-node bifurcation must occur first, as illustrated in Fig.\ref{nosn}. This discussion shows that if we rule out the existence of saddle node bifurcation orbits on $Per_T$ it will immediately imply the Orbit Index has be $-1$ throughout $Per_T\setminus S^3\times\{1\}$ - therefore, we now proceed to rule out the existence of saddle node bifurcation orbits on $Per_T$.\\

To do so recall that if $T_0$ is a saddle node bifurcation for some $G_{t_0}$, $t_0\in(0,1)$, then two periodic orbits collide at $T_0$, of Orbit Indices $0$ and $-1$ - and moreover, that $T_0$ does not persist as we vary $G_{t_0}$ to $G_t$, $t>t_0$ (see the illustration in Fig.\ref{nosn}). In addition, as observed in the proof of Th.\ref{orbipers} there exists a connected subset of periodic orbits $Per'\subseteq Per_T$ which connects $T\times\{0\}$ with some $P\times\{1\}$ s.t. where $P$ is a periodic orbit for $F$). In more detail, the smooth deformation of $T\times\{0\}$ to $P\times\{1\}$ implies the existence of a smooth map $\gamma:[0,1]\to Per'$ s.t. $\gamma(t)=(x(t),y(t))$ is injective in both $x(t)\in S^3$ and $y(t)\in[0,1]$. By the discussion above this implies there can be no saddle-node bifurcation orbits on $Per'$ - and since $Per'$ includes every periodic orbit on $Per_T$ from $T\times\{0\}$ up to at least the first saddle-node bifurcation point (as $Per_T$ cannot split in a period-doubling bifurcation before such a bifurcation occurs - see Lemma \ref{notyp2}), we must have $Per'=Per_T$. Or in other words, there are no saddle node bifurcations on $Per_T$. By the discussion above we conclude the Orbit Index of every periodic orbit on $Per_T\setminus S^3\times\{1\}$ is $-1$.\\

We now prove $i(P)=-1$ as well, where $P$ denotes the periodic orbit for $G_1=F$ s.t. $T\times\{0\}$ connects with $P\times\{1\}$ through $Per_T$. To do so, note that since by assumption $F=G_1$ is a vector field in $K$ we know $P$ is isolated - hence by $P$ also relatively isolated (see Def.\ref{relative}). Now smoothly extend the curve $\{G_t\}_{t\in[0,1]}$ to $S^3\times[1,\frac{3}{2}]$ s.t. the extended curve $\{G_t\}_{t\in[0,\frac{3}{2}]}$ is also a deformation $rel$ $H$ through $K$ - using similar arguments to those used above it is easy to see that under this extension $Per_T$ connects $S^3\times\{0\}$ to $S^3\times\{\frac{3}{2}\}$ through $P\times\{1\}$. Therefore, since $P\times\{1\}$ is a periodic orbit on $Per_T$ w.r.t. this extension we conclude $i(P)=-1$ - and by Def.\ref{index22} we know this property is independent of our choice of curve $\{G_t\}_{t\in[0,1]}$. By Th.\ref{contith} and by $F=G_1$ we conclude $P$ is also globally continuable, i.e., it persists under sufficiently small $C^k$ perturbations of $F$, where $k\geq3$.\\

To conclude the proof it now remains to prove $P$ persists without changing its knot type under sufficiently small generic perturbations. To do so, note that by $F\in K$ we know there exists some cross-section $S$ transverse to $P$ s.t. the first-return map $f:S\to S$ is well-defined and satisfies $P\cap S=\{x\in S|f(x)=x\}$. By $i(P)=-1$, by Def.\ref{lefschetz} and Def.\ref{index1} we conclude the Fixed Point Index of $f$ in $x$ is also $-1$. Now, note that if the knot type of $P$ would change for all sufficiently small perturbations of $F$ it would follow $P$ is a period-multiplying bifurcation (see the illustration in Fig.\ref{cable}) - however, at such bifurcation orbits the Fixed Point Index of $f$ at $x$ would have to be $0$. Since by $i(P)=-1$ we know the Fixed Point Index of $f$ at $x$ is also $-1$ - i.e., one eigenvalue of the differential $D_f(x)$ is in $(1,\infty)$ and another is in $(0,1)$ - we conclude no such period-multiplying bifurcation takes place, i.e., $P$ persists without changing its knot type. The proof of Cor.\ref{pers11} is now complete.
\end{proof}
Having proven Prop.\ref{pers11} in the case when $i(T)=-1$ we now extend it to the case where it is not - i.e., we now prove the Conclusion of Prop.\ref{pers11} holds even when the periodic orbit $P\in Ess(F)$ is deformed to a periodic orbit $T$ for $G$ s.t. $i(T)=0$. It turns out that even though we can no longer apply Th.\ref{contith} in this setting, $P$ still persists in generic scenarios - without changing its knot type. In more detail we prove the following Corollary of Th.\ref{orbipers} and Prop.\ref{pers11}: 
\begin{corollary}
    \label{pers12} Let $F,G$ and $H$ be as in Prop.\ref{pers11}, and assume $P\in Ess(F)$ is a periodic orbit s.t. a deformation of $F$ to $G$ $rel$ $F$ deforms $P$ to $T$, a periodic orbit for $G$ s.t. $i(T)=0$. Then for a generic choice of $F$ the orbit $P$ and its knot type persist under sufficiently small, generic perturbations of $F$.
\end{corollary}
 \begin{proof}
As in the case of Cor.\ref{pers11}, it would suffice to prove the assertion for a curve of vector fields $\{G_t\}_{t\in[0,1]}$ s.t. the following holds:

\begin{enumerate}
    \item The curve is a deformation $rel$ $H$ and satisfies  $\{G_t\}_{t\in[0,1]}\subseteq K$.
    \item $F=G_1$, $G=G_0$ (per our assumption we already know $G\in K$ - see Prop.\ref{pers11}).
\end{enumerate}

As the proof is very similar to the previous one, so we only give a sketch of proof in broad strokes and avoid the technical details. To begin, let $T$ be a periodic orbit for $G$ s.t. $i(T)=0$. By the proof of Th.\ref{orbipers} we know that as we deform $G$ back to $F$ $rel$ $H$ there exists a set of periodic orbits $Per_T$ connecting $T\times\{0\}$ and $P\times\{1\}$, where $P$ is a periodic orbit for $F$ of the same knot type.\\

Similarly to the proof of the previous corollary we know there exists some subset $Per'\subseteq Per_T$ and a smooth map $\gamma(t):[0,1]\to Per'$ s.t. $\gamma(t)=(x(t),y(t))$ is injective in both $x(t)$ and $y(t)$. Furthermore, since $i(T)=0$ and since by definition $Per'\setminus S^3\times\{t\}$ includes precisely two components for all $t\in(0,1)$ we know the only possible bifurcations on $Per'$ are period-doubling bifurcations - as illustrated in Fig.\ref{nosn3} . Note that in this scenario, at the period-doubling bifurcations on $Per'$ there may be a change of Orbit Index from $0$ to $1$ as we vary $T\times\{0\}$ to $P\times\{1\}$ along $Per'$ (and that in general, we need not expect $Per'=Per_T$ - see the illustration in Fig.\ref{nosn3}). However, regardless of whether $Per'=Per_T$ or not, a similar argument to the one used in the proof of Cor.\ref{pers11} shows $P$ is not a saddle-node bifurcation orbit for $F$ when we extend the curve $\{G_t\}_{t\in[0,1]}$ to some deformation $rel$ $H$, $\{G_t\}_{t\in[0,\frac{3}{2}]}\subseteq K$. This implies either $P$ is Type $0$ (in which case its Orbit Index, $i(P)$, is well-defined), or it is a Type $II$ Orbit, i.e., a period-doubling bifurcation point.\\

\begin{figure}[h]
\centering
\begin{overpic}[width=0.2\textwidth]{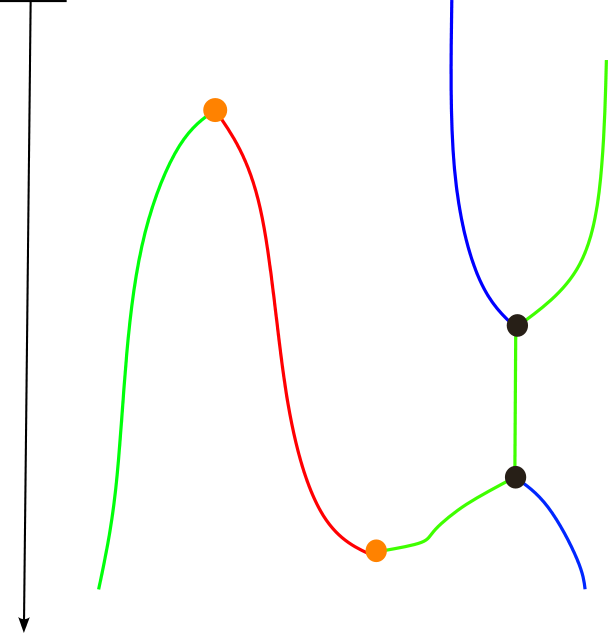}
\put(10,1025){$0$}
\put(10,-70){$t$}
\put(730,780){$\Gamma_1$}
\put(840,350){$\Gamma_2$}
\put(950,50){$\Gamma_3$}
\end{overpic}
\caption{\textit{A (partial) bifurcation diagrams for a periodic orbit with Index $0$ which begin at $G_0$, where $\Gamma_1\cup\Gamma_2\cup\Gamma_3$ corresponds to $Per'$ - again the red curves denote when the Orbit Index is $-1$, the green when it is $1$, and the blue when it is $0$ (note the Orbit Index is not constant along $Per'$). The orange dots denote saddle node bifurcation orbits while the black denote period-doubling bifurcation orbits.} }
\label{nosn3}
\end{figure}
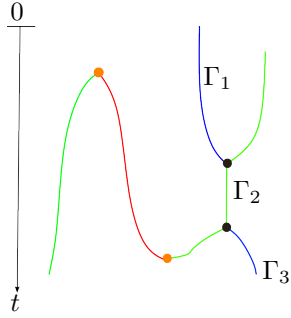

We therefore differentiate between the cases when $i(P)$ is well defined and when $P$ is a period-doubling bifurcation orbit. To do so, choose any smooth extension of $\{G_t\}_{t\in[0,1]}$ to $K$, i.e., some $\{G_t\}_{t\in[0,\frac{3}{2}]}\subseteq K$ where the deformation is $rel$ $H$ from $G_0=G$ up to $G_1=F$. When $i(P)$ is well-defined, from $F\in K$ we know $i(P)=\phi(P)$ (see Def.\ref{index1}), hence the Fixed Point Index for any local first-return map is also well-defined (see Def.\ref{lefschetz}): that is, there exists an isolating cross-section $S$ transverse to $P$ and a first-return map $f:S\to S$ s.t. $P\cap S=\{x\in S|f(x)=x\}$ and the Fixed-Point Index of $f$ in $S$ is non-zero - i.e. the differential $D_f(x)$ has one eigenvalue in $(-\infty,-1)$ and another in $(-1,0)$. When this is the case it is easy to see $P$ persists under all sufficiently small $C^k$ perturbations of $F$, $k\geq3$ - moreover, as $D_f(x)$ has no eigenvalues on the circle $S^1$ $P$ persists without bifurcating, i.e., it does not change its knot type as $F$ is perturbed in sufficiently small $C^k$ perturbations of $F$. In particular, it persists without changing its knot type as we vary $F=G_1$ to $G_t,t>1$.\\

It remains to prove the generic persistence of the knot type of $P$ when $P$ is a bifurcation orbit for $F$, i.e., when it is a period-doubling bifurcation point. In this case, by Def.\ref{type} this implies that as we vary $F$ in $K$ (say, by breaking the heteroclinic knot in some way sufficiently away from $P$) the periodic orbit $P$ persists as it splits into two orbits - one of which retains the original knot type of $P$. All in all the proof of Cor.\ref{pers12} is now complete.
\end{proof}
All in all, Prop.\ref{pers11} and Cor.\ref{pers12} can be summarized in the following theorem:

\begin{theorem}
  \label{pers13} Let $F,G$ and $H$ be as in the setting of Th.\ref{orbipers}, and assume in addition every periodic orbit for $G$ has Orbit Index which is either $-1$ or $0$. Then, for a generic choice of $F$ if $P$ is a periodic orbit in $Ess(F)$, the following holds:
  \begin{enumerate}
      \item If $P$ can be deformed into a periodic orbit $T$ for $G$ s.t. $i(T)=-1$ then $i(P)=-1$ as well.
      \item If the Orbit Index $i(P)$ is well-defined then $i(P)\in\{0,-1\}$. In that case, $P$ persists under all sufficiently small $C^k$ perturbations of $F$ (where $k\geq3$).
      \item If $i(P)$ is undefined it persists under all sufficiently small generic $C^k$ perturbations of $F$, where $k\geq3$.
  \end{enumerate}
  
\end{theorem}

Th.\ref{pers13} has the following heuristic meaning - assume $G$ is a smooth vector field of $S^3$ which generates a heteroclinic knot $H$ s.t. every periodic orbit of $S^3$ has Orbit Index $-1$ or $0$ (for example, this is the case when there exists some singular Horseshoe suspended around $H$ - see Prop.\ref{dens1}). Further assume $F$ is a smooth vector field of $S^3$ s.t. $F$ can be deformed $rel$ $H$ to $G$ - then, per Th.\ref{pers13} we know that for a generic $F$ the $C^k-$closer a given vector $F'$ is to $F$ the more periodic orbits $F'$ is expected to generate. In other words, Th.\ref{pers13} implies that despite the extreme generic instability of heteroclinic knots, their periodic dynamics are maximal in the sense that their complexity "wears off" on nearby vector fields.\\

We further remark Th.\ref{pers13} probably cannot be generalized to the case where the vector field is non-generic - as in that case there is no reason to assume the Orbit Index for $P$ is even well-defined, and moreover, the possibilities of degenerate bifurcations for $P$ can further complicate matters. That being said, later in this paper we will prove that at some specific cases one can use similar methods to derive general persistence results which hold even in a non-generic context - for more details, see Th.\ref{persistence} in the next section, where we prove such a result for the Lorenz attractor.\\

Having studied the persistence properties of periodic dynamics in the essential class, we conclude this section by studying their topology. To do so first recall that by the Birman-Williams Theorem (see Th.\ref{BIRW}) provided the vector field $G$ in Th.\ref{orbipers} is hyperbolic on its chain recurrent set there exists a Template for the flow - i.e., a branched surface which encodes all the knot types realized as periodic orbits for $G$ (save possibly for two extra knots). This begs the following questions - assuming we can associate a Template with $G$, can we somehow associate the same Template with all vector fields which can be smoothly deformed $rel$ $H$ to $G$? If yes, this would imply that in such scenarios $G$ can be thought of as an idealized model for the dynamics of $F$.\\

Under the general assumptions of Th.\ref{orbipers} it is not at all clear the answer should be positive - mostly because it is not at all clear whether the periodic orbits for $G$ even lie in the same connected invariant subset, or that $G$ is even singular hyperbolic (let alone hyperbolic) on that set. However, as we will now prove, provided we impose some extra assumptions on $G$ we can answer this question in the affirmative. To this end let $G$ be as in Th.\ref{orbipers}, let $V$ denote the collection of all essential periodic orbits for $G$ (i.e., $V=Ess(G)$ w.r.t. the heteroclinic knot $H$) and set $\Phi=\overline V$. We now prove the following Corollary of Th.\ref{orbipers}, with which we conclude this section:

\begin{corollary}
    \label{templateth}
Let $F$, $G$ and $H$ be as in Th.\ref{orbipers}, and let $V$ and $\Phi$ be as above. Assume that in addition to the assumptions of Th.\ref{orbipers} the following additional assumptions are also satisfied: 

    \begin{enumerate}
        \item The Orbit Index of every periodic orbit in $\Phi$ is well defined, and it is either $0$ or $-1$.
        \item Every periodic orbit for $G$ is in $V$.
        \item $\Phi$ includes a dense orbit.
        \item Every saddle fixed point on $H$ is a saddle focus, and lies in $\Phi$.
        \item $S^3\setminus\Phi$ is open and connected.

    \end{enumerate}
    
Then, there exists a unique Template $\Theta$ s.t. every knot type in $\Theta$ (save possibly for a finite number) corresponds to some periodic orbit in $Ess(F)$ - in particular, $F$ generates periodic orbits of infinitely many different knot types.
\end{corollary}
In other words, Cor.\ref{templateth} says the following thing: assume we can deform $F$ $rel$ $H$ to some dynamically minimal vector field as above, then we can associate a Template with the essential dynamics of $F$.
\begin{proof}

We will prove Cor.\ref{templateth} as follows: we first prove we can push the set $\Phi$ away from the fixed points by adding (at most) a finite number of periodic orbits - which we do by performing Hopf bifurcations on the fixed points in $H$ (if necessary). This will deform $\Phi$ to a new set $\Phi'$ which differs from $\Phi$ by at most a finite number of periodic orbits - and consequentially, this will allow us to endow $\Phi'$ with a hyperbolic structure. Following that we apply both Th.\ref{BIRW} and Th.\ref{begbon} (along with Th.\ref{orbipers}) from which the assertion would follow.\\

As indicated above we begin by smoothly deforming $G$ to the vector field $G'$ - which we do by expanding all the fixed-points of saddle focus type on $H$ by Hopf bifurcations, thus adding a finite number of periodic orbits to the flow, all of which are saddle periodic orbits (recall that by assumption, there are no fixed points whose type os a real saddle on $H$). In addition, we do so s.t. the dynamics in $\Phi$ all persist as $G$ is deformed to $G'$ - that is, no periodic orbit in $\Phi$ is destroyed or changes its knot type under this deformation. It is east to see this deformation replaces every saddle-focus type fixed point with either a sink or a source, while adding at most a finite number of periodic orbits to the flow - one per each saddle focus (this deformation also expands the heteroclinic knot $H$ to a knotted tube - see the illustration in Fig.\ref{blow}). Now, let $x_1,...,x_n$ denote the fixed points of $G'$, $n>1$ - it is easy to see all of them are sources and sinks . Consequentially, there exist sufficiently small (closed) balls $S_1,...,S_k$ centered at $x_1,...,x_k$ s.t. $G'$ is transverse to the boundary of $M=S^3\setminus(\cup_{i=1}^kS_k)$ (in particular, every periodic orbit for $G$ is deformed to a periodic orbit inside $M$).\\

\begin{figure}[h]
\centering
\begin{overpic}[width=0.6\textwidth]{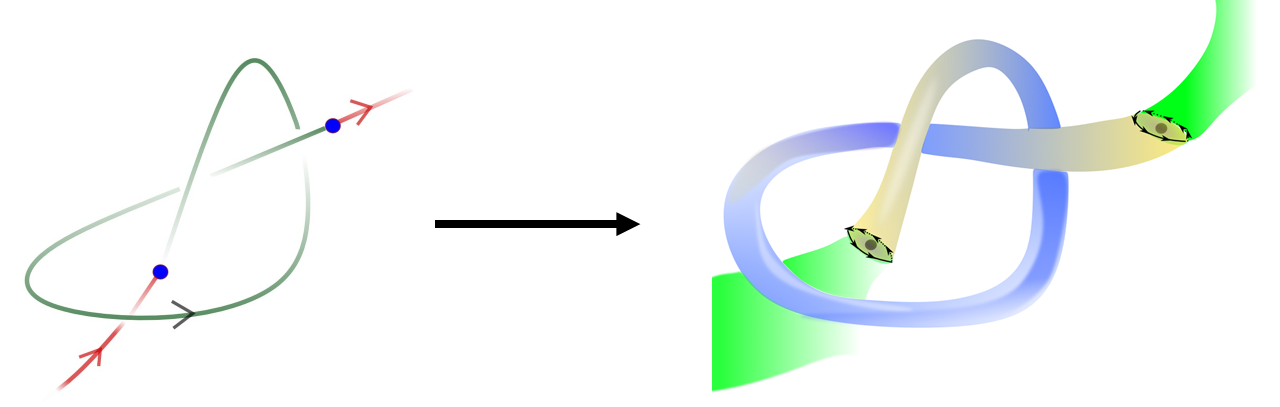}

\end{overpic}
\caption[Fig39]{\textit{Deforming a heteroclinic knot to a knotted tube via Hopf bifurcations. }}
\label{blow}
\end{figure}

To continue, let $V'$ denote the collection of periodic orbits for $G'$, and set $\Phi'=\overline{V'}$. We note that since $V'$ is created from $V$ by adding a finite number of periodic orbits (all of which are saddles) it follows every periodic orbit in $V'$ has Orbit Index which is either $-1$ or $0$ - it is also easy to see that since $V$ is dense in $\Phi$ the set $V'$ is also dense in $\Phi'$, and that $\Phi'$ includes a dense orbit. Moreover, by construction we have $\Phi'\subseteq M$ - that is, $\Phi'$ lies away from the fixed points of $G'$.\\

Now, recall that by construction on every component of $\partial M$ the vector field $G'$ points either outside or inside $M$. In addition, note that since the Euler characteristic of $S^3$ is $0$, by the Poincare-Hopf Theorem there must exist at least one source and one sink among the fixed points of $G'$ in $S^3$ (see Th.6.1 in \cite{Mil}). With these ideas in mind we now deform $G'$ to a new vector field $G''$ s.t. for all $x\in M\setminus \Phi'$ either its backward orbit accumulates on some source or its forward orbit eventually accumulates on some sink (or both) - i.e., we deform the flow s.t. $\Phi'$ becomes the maximal invariant set for the flow in $M$ (we can do so since $S^3\setminus\Phi$ - and hence also $M\setminus\Phi'$ - is open and connected). As $\Phi'$ includes a dense orbit, it follows that after this deformation $\Phi'$ becomes becomes the chain-recurrent set for $G''$ in $M$.\\

Having done that, our new goal is to smoothly deform $G''$ to a vector field $\Gamma$ which is hyperbolic on its maximal invariant set in $M$, i.e., on $\Phi'$. We do so as follows: we first note that by assumption, since every periodic orbit for $G$ has Orbit Index which is either $-1$ or $0$ by definition the vector field $G$ is in $K$ (see Prop.\ref{propk}) - and that in addition, the deformation of $G$ to $G''$ via $G'$ was done without adding or bifurcating any periodic orbits, save for the saddle periodic orbits created by the Hopf bifurcations. As such, it is easy to see that locally, around any periodic orbit in $\Phi'$, the deformation of $G$ to $G''$ is a deformation through $K$ - Consequentially, by Th.\ref{invar} we know the Orbit Index of every periodic orbit in $\Phi$ persists, unchanged, as $\Phi$ is transformed to $\Phi'$. Moreover, since $V$ and $V'$ differ at most by a finite number of saddle periodic orbits (whose Orbit Index is either $-1$ or $0$ - see Def.\ref{index1}), it follows the Orbit Index of every periodic orbit in $\Phi'$ is either $-1$ or $0$.\\

To continue, recall that since there are no fixed points in $\Phi'$ by Lemma $7$ in \cite{BWW} we know there exists a finite collection of cross-sections $S$ transverse to $\Phi'$ s.t. the first-return map $f:S\to S$ is well-defined and continuous at $S\cap\Phi'$. As the Orbit Index of every periodic orbit in $\Phi'$ is already either $-1$ or $0$, this allows us to smoothly deform $G''$ to some vector field $\Gamma$ by an isotopy of $f:S\to S$, s.t. the following is satisfied:

    \begin{itemize}
        \item For every $x\in\Phi'$, the tangent space of $x$ can be split into $E^c_x\oplus E^u_x\oplus E^s_x$, where $E^c_x$ is spanned by $\Gamma(x)$.
        \item $E^s_x,E^u_x$ and $E^c_x$ vary continuously to $E^s_{\phi_t(x)},E^u_{\phi_t(x)}$ and $E^c_{\phi_t(x)}$ (respectively) as $x$ flows to $\phi_t(x)$, $t\in\mathbf{R}$ (where $\phi_t$ denotes the flow corresponding to $\Gamma$). In particular, $D_{\phi_t}(x)E^j_x=E^j_{\phi_t(x)}$ where $j\in\{s,u,c\}$ and $D_{\phi_t}(x)$ denotes the differential of the time $t$ map w.r.t. the flow.
        \item There exist constants $C>0$ and $\lambda>1$ s.t. for every $t>0$ and every $x\in\Lambda$ we have: 
        
        \begin{enumerate}
            \item For $v\in E^s_x$, $||D_{\phi_t}(x)v||<Ce^{-\lambda t}||v||$.
            \item For $v\in E^u_x$, $||D_{\phi_t}(x)v||>Ce^{\lambda t}||v||$. 
        \end{enumerate}
    \end{itemize}

In other words, we deform the flow by making the dynamics hyperbolic on $\Phi'$. As previously remarked, $\Phi'$ forms the chain-recurrent set for $\Gamma$ in $M$ due to the existence of a dense orbit in $\Phi'$ - consequentially, by Th.\ref{BIRW} we can associate a Template $\Theta$ with the dynamics of $V$ on $\Phi'$ s.t. every knot type encoded by $\Theta$ (save possibly for two periodic orbits) is realized as a periodic orbit for $\Gamma$. And since every periodic orbit for $\Gamma$ can be deformed to a periodic orbit for $G$ without changing its knot type (save possibly for a finite collection of periodic orbits generated by Hopf bifurcations), it follows every knot type on $\Theta$ is realized as a periodic orbit for $G$ (again, save possibly for some finite collection of knot types on $\Theta$). Finally, since by Th.\ref{orbipers} every periodic orbit for $G$ is deformed to a periodic orbit for $F$ without changing its knot type the same is true for $F$. All in all, it follows every periodic orbit in $Ess(F)$ corresponds to some knot on $\Theta$ - and since every Template encodes infinitely many knot types (see Cor.3.1.14 in \cite{KNOTBOOK}) we conclude $Ess(F)$ includes periodic orbits of infinitely many distinct knot types.\\

It remains to prove the Template $\Theta$ is unique (up to isotopy) - or in other words, that the Template $\Theta$ does not depend on how we construct $\Gamma$. To do so note that no matter how we construct the vector field $\Gamma$ it is transverse to $\partial M$ hence it is a plug on $M$ (see Def.\ref{model}). In fact, since every component in $\partial M$ is a sphere it follows the total genus of $M$ is $0$ - which, per Def.\ref{model}, makes $\Gamma$ a model flow on $M$. It is easy to see by our construction that if $\Gamma$ and $\Gamma'$ are two vector fields obtained as described above, they both suspend the same basic set (namely, $\Phi'$) in the exact same way around $H$ - therefore by Th.\ref{begbon} it follows $\Gamma$ and $\Gamma'$ are orbitally equivalent on $\Phi'$. This argument proves $\Gamma$ and $\Gamma'$ define the same Template $\Theta$ and the proof of Cor.\ref{templateth} is now complete.
\end{proof}

\section{The applications:}

In this section we apply Th.\ref{orbipers} and Th.\ref{pers13} (along with the results in Sect.\ref{orbitin}) to study the dynamics of three-dimensional flows via three concrete models - which we do to showcase how these tools can be applied to study three-dimensional real-life chaotic dynamical systems. This section is organized as follows - 

\begin{itemize}
    \item In subsection \ref{rossler} using Th.\ref{orbipers} we state and prove a sufficient condition for the existence of a chaotic attractor in the Rössler system (see Th.\ref{trefoil1}). In the proccess we will also prove the existence of smooth vector fields $F$ of $S^3$ s.t. $Ess(F)$ includes infinitely many periodic orbits (see Th.\ref{trefoilor}).
    \item In subsection \ref{lorenz} we use similar tools and ideas to study the Lorenz attractor. Following that, using the unique properties of the Lorenz system as analyzed in \cite{Pi} we will sharpen the proof of Th.\ref{pers13} to derive a general persistence theorem for the periodic orbits on the Lorenz attractor (see Th.\ref{persistence}).
    \item Finally, inspired by the Moore-Spiegel Oscillator (see \cite{SM}) in subsection \ref{horsus} we study a model flow $F$ for which the essential class is at most a finite set (see Prop.\ref{nohyp}). As we will prove, the fact $Ess(F)$ is finite constrains the type of complex dynamics $F$ can generate (see Th.\ref{nohyp2})- which will lead us to conjecture about the general state of affairs when the Essential Class is finite.
\end{itemize}

\subsection{Heteroclinic chaos in the Rössler attractor}
\label{rossler}
As described above, in this section (and the next) we apply Th.\ref{orbipers} to study the periodic dynamics on chaotic attractors. In this specific (sub)section we will apply these results to study the Rössler attractor. To begin, recall we define the Rössler system as the flow generated by the following system of ordinary differential equations (where $a,b,c>0$):

\begin{equation} \label{Vect}
\begin{cases}
\dot{x} = -y-z \\
 \dot{y} = x+ay\\
 \dot{z}=b+z(x-c)
\end{cases}
\end{equation}

\begin{figure}[h]
\centering
\begin{overpic}[width=0.7\textwidth]{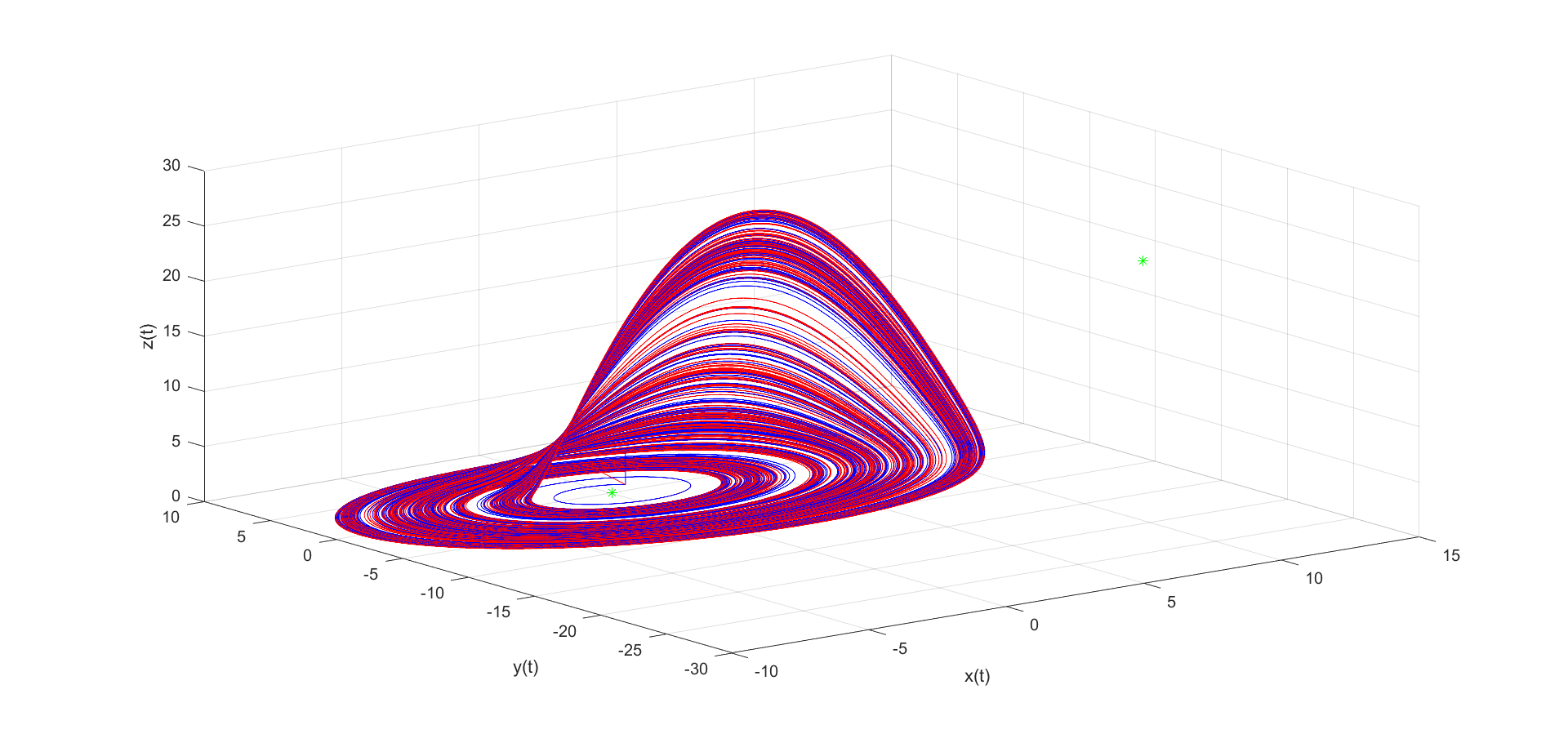}
\end{overpic}
\caption{\textit{The Rössler attractor at $(a,b,c)=(0.2,0.2,5.7)$}}
\end{figure}

As shown numerically at \cite{Ross76} (and later proven in \cite{Zgli97} and \cite{XSYS03}) at $(a,b,c)=(0.2,0.2,5.7)$ the Rössler system generates chaotic dynamics - i.e., there exists a bounded invariant set on which first-return map can be factored to a shift on two symbols (in particular, that set includes infinitely many periodic orbits).\\

To begin our analysis of the Rössler system we first remark that whenever $c^2-4ab>0$ the flow generates precisely two fixed points, which we always denote as follows - $P_{In}=(\frac{c-\sqrt{c^2-4ab}}{2},-\frac{c-\sqrt{c^2-4ab}}{2a},\frac{c-\sqrt{c^2-4ab}}{2a})$ and $P_{Out}=(\frac{c+\sqrt{c^2-4ab}}{2},-\frac{c+\sqrt{c^2-4ab}}{2a},\frac{c+\sqrt{c^2-4ab}}{2a})$. In addition, to conform with the parameter space considered in \cite{MBKPS}, \cite{Le}, \cite{BBS} (and the references therein), from now we will only consider the dynamics of the Rössler system at the three-dimensional parameter space $P$ - i.e., the collection of $(a,b,c)$ parameters s.t. the following is satisfied\label{parspace}:

\begin{enumerate}
    \item For every $p=(a,b,c)$ we have $a,b\in(0,1)$, $c>1$ and $c^2-4ab>0$.
    \item For every $p\in P$ the fixed points $P_{In}$ and $P_{Out}$ are both saddle foci of opposing indices (see the illustration in Fig.\ref{local}).
    \item For all $p\in P$, $P_{In}$ generates a two dimensional unstable manifold $W^u_{In}$, and a one-dimensional stable manifold $W^s_{In}$. Conversely, $P_{Out}$ generates a two-dimensional stable manifold $W^s_{Out}$ and a one-dimensional unstable manifold $W^u_{Out}$ (see the illustration in Fig.\ref{local}).
\end{enumerate}

\begin{figure}[h]
\centering
\begin{overpic}[width=0.4\textwidth]{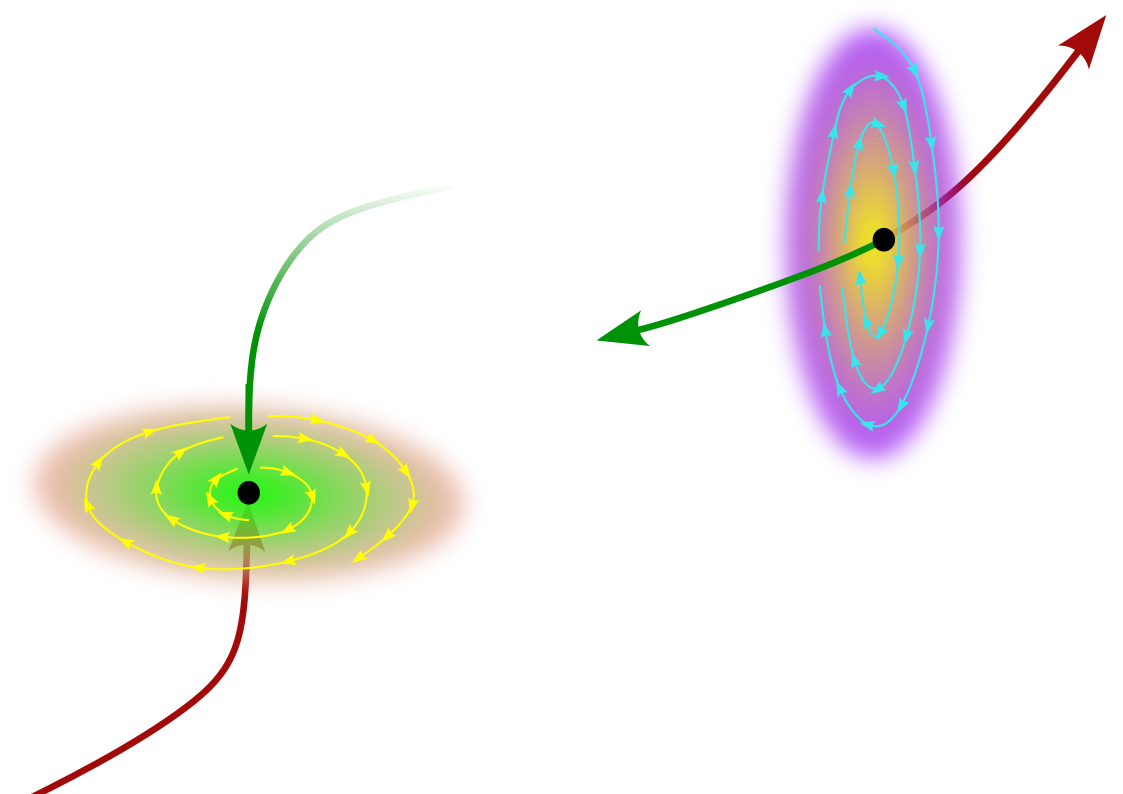}%
\put(600,360){$W^s_{Out}$}
\put(850,350){$W^u_{Out}$}
\put(760,440){$P_{Out}$}
\put(940,560){$W^s_{Out}$}
\put(350,170){$W^u_{In}$}
\put(230,250){$P_{In}$}
\put(130,400){$W^s_{In}$}
\put(230,110){$W^s_{In}$}
\end{overpic}
\caption{\textit{The local dynamics for the flow around the fixed points $P_{In}$ and $P_{Out}$.}}\label{local}
\end{figure}
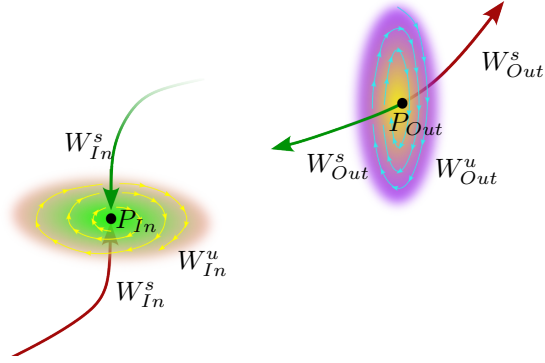

It is easy to see $P$ forms a three-dimensional open set of parameters for the Rössler system. Moreover, it is also easy to verify it is non-empty - for example, the parameter $(a,b,c)=(0.2,0.2,5.7)$ lies in $P$. Since the introduction of the Rössler system, the parameter space $P$ was the object of many numerical studies - it is well known to exhibit period-doubling cascades and homoclinic bifurcations (among other nonlinear phenomena). As far as the general dynamics of the Rössler system in $P$ goes, we have the following collection of results proven both in \cite{LiLl} and by the author in \cite{I}:

\begin{theorem}
    \label{generaldynamics}
    For every $p\in P$, the corresponding Rössler system satisfies the following:
    \begin{itemize}
        \item There exists a cross-section $U_p$ transverse to the flow and homeomorphic to a half-plane, s.t. if $s\in \mathbf{R}^3$ is an initial condition whose orbit does not limit to a fixed point (or to $\infty$), its trajectory under the flow intersects $U_p$ transversely. In particular, every periodic orbit intersects with $U_p$ transversely at least once (see the illustration in Fig.\ref{cross11}). 
        \item The two-dimensional invariant manifolds $W^u_{In}$ and $W^s_{Out}$ are transverse to $U_p$ at the fixed-points (see the illustration in Fig.\ref{cross11}).
        \item The Rössler system extends continuously to $S^3$, where $\infty$ is added as a fixed point for the flow. In particular, $\infty$ is a fixed-point for the flow of Poincare-index $0$. 
        \item The one-dimensional invariant manifolds $W^u_{Out}$ and $W^s_{In}$ both include respective heteroclinic orbits $\Gamma_{Out}$ and $\Gamma_{In}$ which connect $\infty$ to both $P_{Out}$ and $P_{In}$ (see the illustration in Fig.\ref{cross11}). Moreover, $\Gamma_{In}$ and $\Gamma_{Out}$ are not linked with one another.
        \item Given any sufficiently large $r>0$, there exists a smooth vector field $F$ of $S^3$ which coincides with the Rössler system corresponding to $p$ in $B_r(0)$. Moreover, $F$ satisfies the following properties:
        \begin{enumerate}
            \item $F$   has precisely two fixed points in $S^3$, $P_{In}$ and $P_{Out}$, connected by a heteroclinic orbit $\Gamma$ which passes through $\infty$.
            \item $U_p$ is deformed into $U$, a cross-section homeomorphic to a disc. Moreover, $P_{In}$ and $P_{Out}$ lie on $\partial U$ and both two dimensional invariant manifolds $W^u_{In}$ and $W^s_{Out}$ are transverse to $U$ at $P_{In}$ and $P_{Out}$ (respectively). In addition, $\Gamma\cap U=\emptyset$.
            \item Set $H'=W^s_{In}\cup W^u_{Out}$ (i.e., the union of the invariant manifolds - then the orbit of any initial condition $s\in S^3\setminus H'$ eventually hits $U$ transversely. Consequentially, the first-return map $f:U\to U$ is well defined.
        \end{enumerate}
            \end{itemize}
\end{theorem}
For a proof, see Lemma 2.1, Lemma 2.2, Lemma 2.6 and Th.2.8 in \cite{I} and Th.1 in \cite{LiLl}.\\

\begin{figure}[h]
\centering
    \begin{overpic}[width=0.5\textwidth]{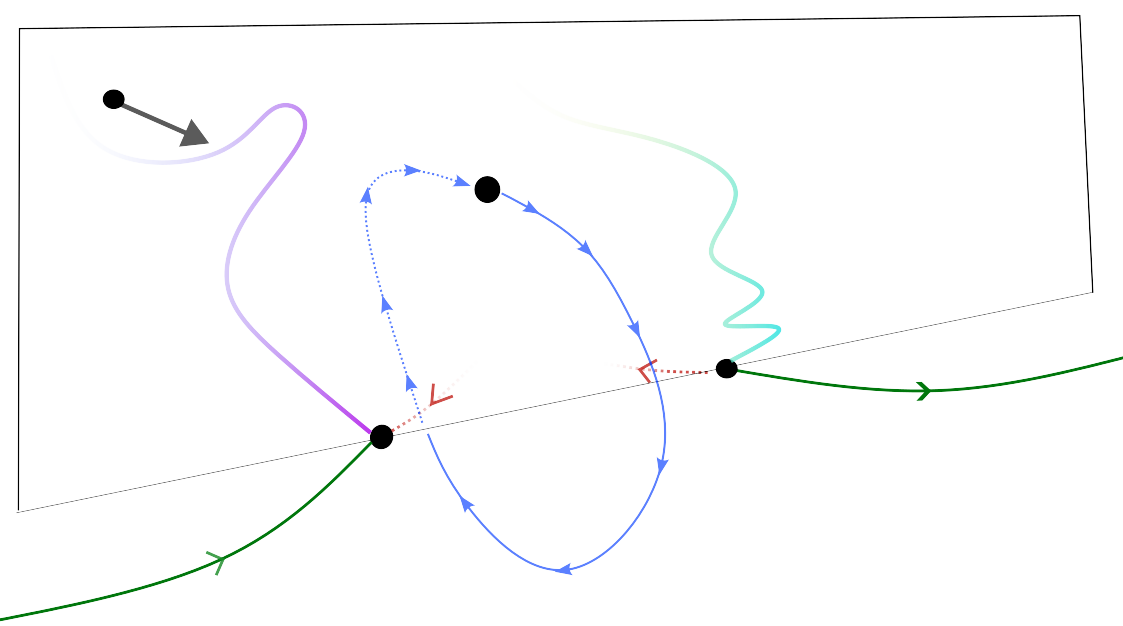}
\put(830,160){$\Gamma_{Out}$}
\put(640,180){$P_{Out}$}
\put(350,140){$P_{In}$}
\put(740,450){$U_p$}
\put(480,420){$W^s_{Out}$}
\put(290,430){$W^u_{In}$}
\put(250,45){$\Gamma_{In}$}
\put(150,150){$l_p$}
\end{overpic}
\caption{\textit{The cross-section $U_p$ and the heteroclinic orbit $\Gamma_{Out}$ and $\Gamma_{In}$. The purple and blue curves denote the intersection of the two-dimensional invariant manifolds $W^u_{In}$ and $W^s_{Out}$ with $U_p$.}}\label{cross11}
\end{figure}
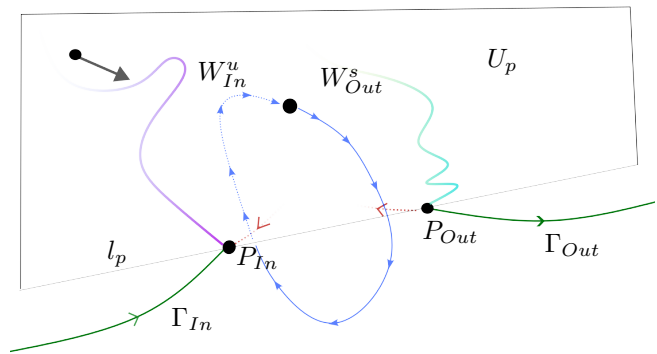

To continue, recall it was observed numerically there exist parameters $p\in P$ at which the Rössler system generates a bounded heteroclinic orbit $\Theta$, which connects $P_{In}$ and $P_{Out}$ as in Fig.\ref{trefoil} (see Fig.5.B1 in \cite{MBKPS}). It is easy to see that whenever $\Theta$ is not linked with either $\Gamma_{Out}$ or $\Gamma_{In}$, the set $H=\Theta\cup\Gamma_{In}\cup\Gamma_{Out}\cup\{P_{In},P_{Out},\infty\}$ forms a trefoil heteroclinic knot in $S^3$ - see Def.\ref{heteroknot} the illustrations in Fig.\ref{trefoil} and Fig.\ref{type}.\\

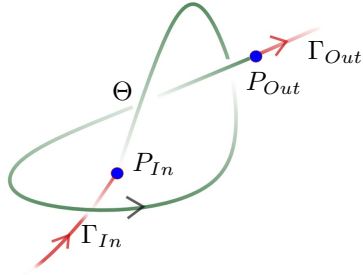
\begin{figure}[h]
\centering
\begin{overpic}[width=0.4\textwidth]{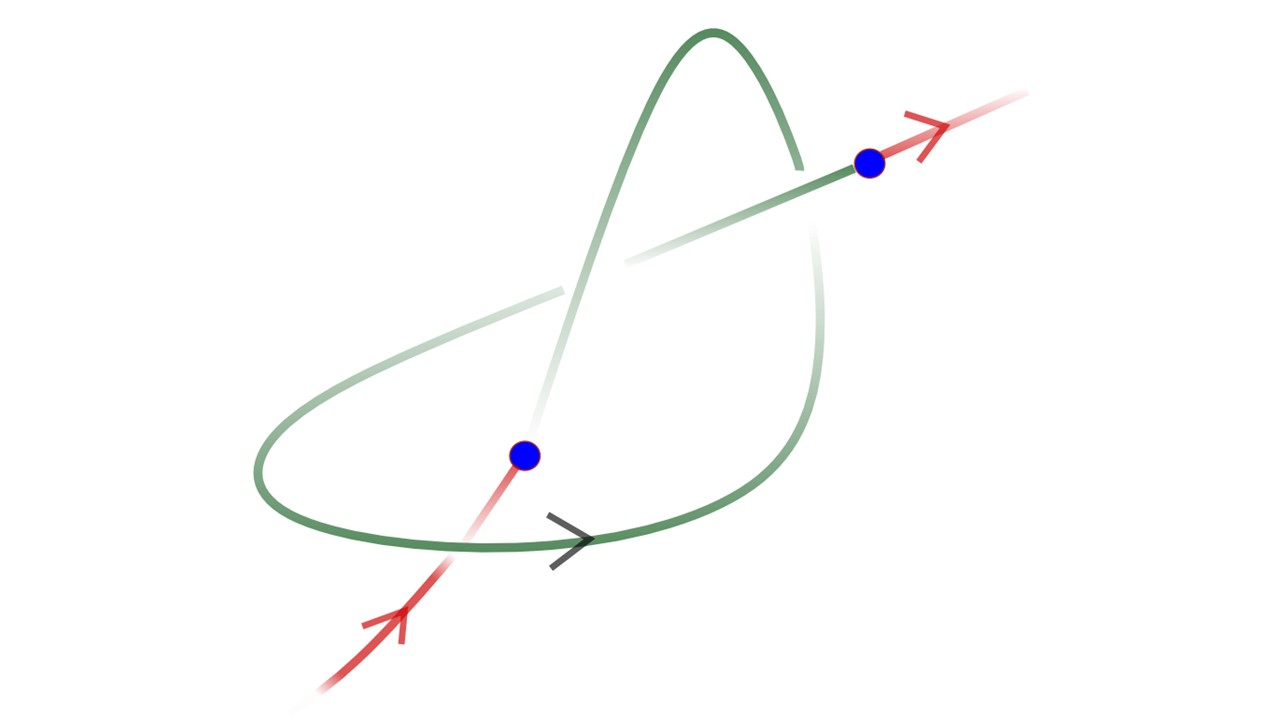}
\put(440,210){$P_{In}$}
\put(340,70){$\Gamma_{In}$}
\put(400,350){$\Theta$}
\put(660,370){$P_{Out}$}
\put(780,430){$\Gamma_{Out}$}
\end{overpic}
\caption{\textit{A heteroclinic trefoil knot. $\Theta$ denotes the bounded heteroclinic orbit, while $\Gamma_{In},\Gamma_{Out}$ denote the unbounded heteroclinic orbit given by Th.\ref{generaldynamics}. This heteroclinic knot includes three fixed points - two saddle foci ($P_{In}$ and $P_{Out}$), and one degenerate fixed point at $\infty$ of Poincare index $0$.}}
\label{trefoil}
\end{figure}
Unfortunately, the existence of a heteroclinic trefoil knot in itself is probably not sufficient to force the existence of an observable chaotic attractor for the Rössler system: as observed numerically in \cite{MBKPS} there exist parameter values where the flow generates intricate homoclinic orbits and complex dynamics (per Shilnikov's Theorem - see \cite{LeS}) - yet there is no attracting invariant set and the orbits of most initial conditions diverge to $\infty$. However, as we will now prove, given any smooth vector field $F$ of $S^3$ which generates a heteroclinic trefoil knot configured like the one in Fig.\ref{trefoil} the dynamics of $F$ include periodic orbits of infinitely many different knot types (later on in this section we will show how these results can be "pulled back" to prove a similar result about existence of a chaotic attractor for the Rössler system).\\

To do so, we first consider an idealized version of the Rössler system. In more detail, we consider $F$, a smooth vector field of $S^3$ which generates precisely two fixed points in $S^3$, both saddle foci of opposing indices - which, for simplicity, we will also denote as $P_{In}$ and $P_{Out}$. We further assume these fixed points form a heteroclinic trefoil knot configured as in Fig.\ref{treff} - that is, $F$ is an idealized form of the Rössler system in the sense that the main difference between the two vector fields is that $F$ is smooth at $\infty$ (in particular, one can think of $F$ as orbitally equivalent to the Rössler system away from $\infty$). With these ideas in mind, we now prove:

\begin{theorem}
    \label{trefoilor}
    Let $F$ be a smooth vector field of $S^3$ which generates a heteroclinic knot $H$ connecting two saddle foci as in Fig.\ref{treff}. Then, the essential class of $F$ w.r.t. $H$ includes infinitely many Torus knots, of infinitely many different knot types (see Def.\ref{torusknot}).
\end{theorem}

\begin{proof}

By Th.\ref{orbipers} all we need to do is find a specific example of a smooth vector field $F$ defined in $S^3\setminus H$ which satisfies the assumptions of Th.\ref{orbipers}. To do so, let $F$ be a smooth vector field of $S^3$ which generates both a heteroclinic orbit $H$ as in Fig.\ref{treff} and a cross-section $U$, homeomorphic to a half-plane and transverse to $F$, s.t. the following is satisfied (see the illustration in Fig.\ref{fig16}):

\begin{itemize}
    \item $\Gamma\cap U=\emptyset$.
    \item $\Theta\cap U=\{p_0\}$.
    \item $P_{In}$ and $P_{Out}$ both lie on the boundary of $U$, and their two dimensional invariant manifolds, $W^u_{In}$ and $W^s_{Out}$, are transverse to $U$ at the fixed points.
\end{itemize}

        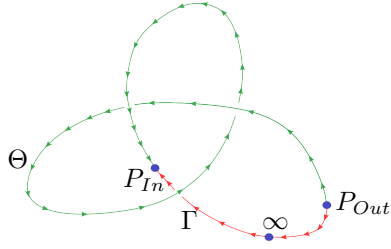
\begin{figure}[h]
   \centering
    \begin{overpic}[width=0.3\textwidth]{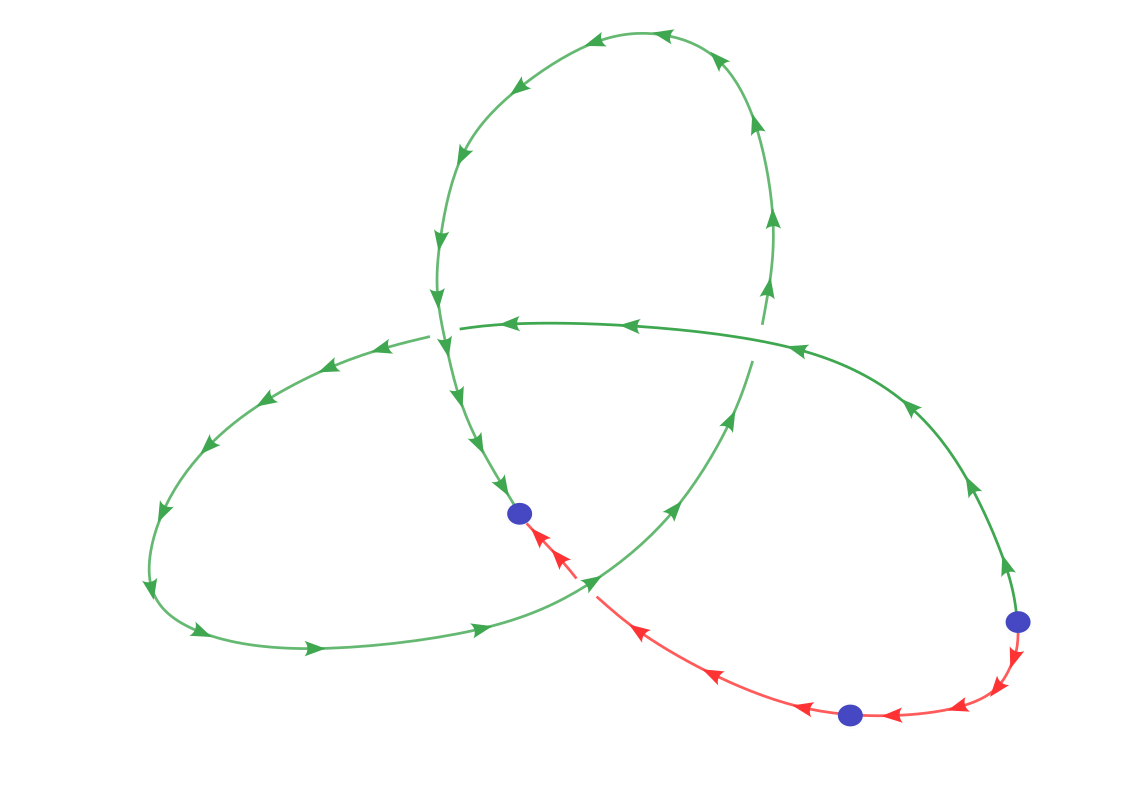}%
\put(380,200){$P_{In}$}
\put(80,250){$\Theta$}
\put(530,90){$\Gamma$}
\put(930,140){$P_{Out}$}
\put(740,90){$\infty$}
\end{overpic}
\caption{\textit{A heteroclinic trefoil knot inspired by the Rössler system (and Fig.\ref{trefoil}) connecting two saddle-foci, $P_{In}$ and $P_{Out}$ with two heteroclinic orbit: $\Theta$ and $\Gamma$. In this scenario $\infty$ is a regular point on the heteroclinic orbit connecting $P_{In}$ and $P_{Out}$.}}
\label{treff}
    \end{figure}

In addition, we further choose $F$ s.t. $U$ is a universal cross-section for the flow, homeomorphic to a disc (or a half plane) - that is, given any initial condition $s\in S^3\setminus H$ the orbit of $s$ under the flow eventually hits $U$ transversely. We remark that by Th.\ref{generaldynamics}, such a vector field $F$ exists - and moreover, by the same theorem we also know the first-return map $f:U\to U$ w.r.t. $F$ is well defined. To continue let $\gamma\subseteq U$ be curve connecting $P_{In}$ and $p_0$ and let us suspend $\gamma$ with the flow generated by $F$ - by possibly deforming $F$ $rel$ $H$ (if necessary) we can ensure $\overline{f(\gamma)}$ is a closed curve which winds around $p_0$ precisely once, as illustrated in Fig.\ref{fig16}.\\

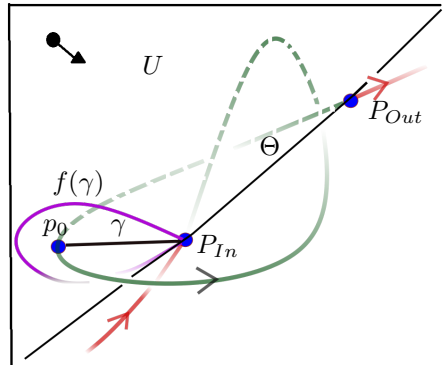
\begin{figure}[h]
\centering
\begin{overpic}[width=0.35\textwidth]{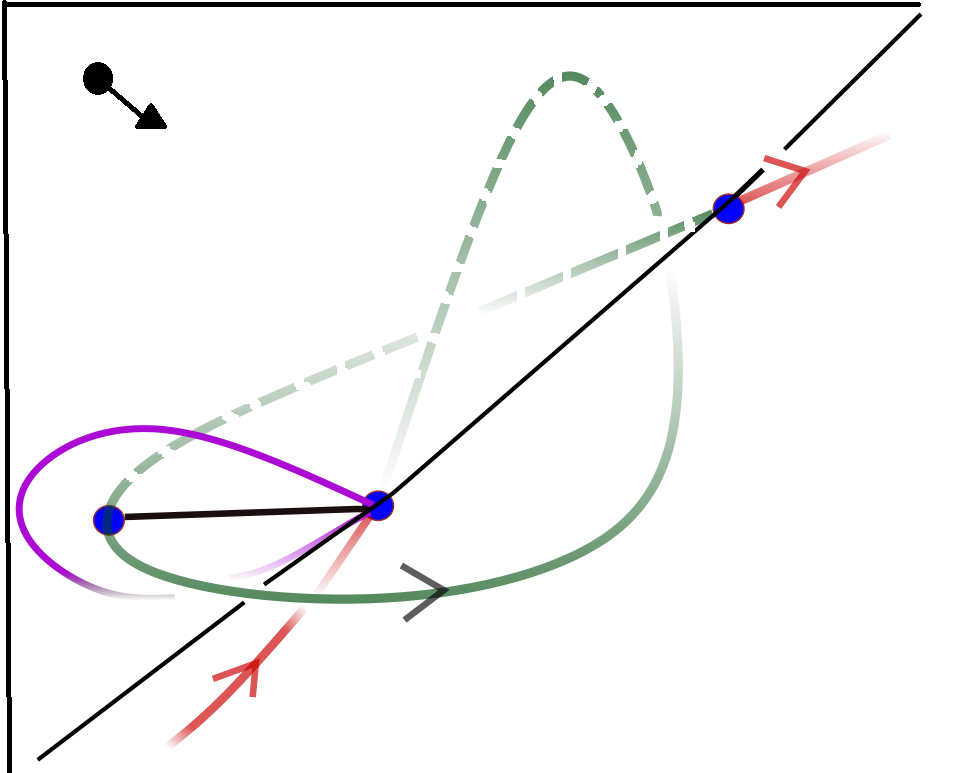}
\put(420,250){$P_{In}$}
\put(230,300){$\gamma$}
\put(80,300){$p_0$}
\put(100,390){$f(\gamma)$}
\put(300,650){$U$}
\put(800,550){$P_{Out}$}
\put(560,470){$\Theta$}
\end{overpic}
\caption{\textit{Flowing $\gamma$ along the trefoil. $\Theta$ is the green heteroclinic orbit while $\Gamma$ corresponds to the red curve(s). The vector field is transverse to the half-plane $U$.}}
\label{fig16}
\end{figure}

We now choose some domain $S_1\subseteq U$, a topological triangle with vertices $P_{In},p_1$ and $p_2$ as in Fig.\ref{D}, s.t. $P_{In},p_0\in\partial S_1$ and $P_{Out}\not\in\overline{S_1}$. By further deforming $F$ $rel$ $H$ (if necessary), we can ensure the flow folds $S_1$ on itself. That is, we smoothly deform the flow such that the first-return map $f:\overline{S_1}\to \overline U$ is well-defined and continuous with $f(S_1)$ as appears as in Fig.\ref{D} - in particular, $p'_i=f(p_i)$, $i=1,2$ are as indicated in Fig.\ref{D}.\\

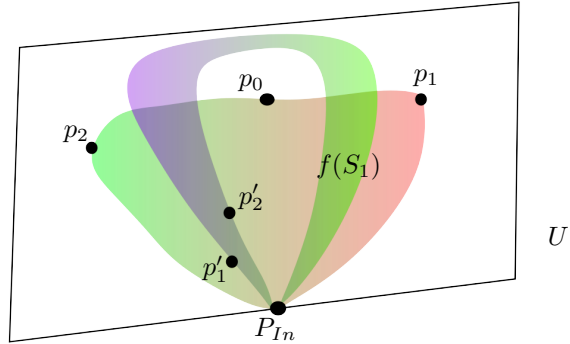
\begin{figure}[h]
\centering
\begin{overpic}[width=0.4\textwidth]{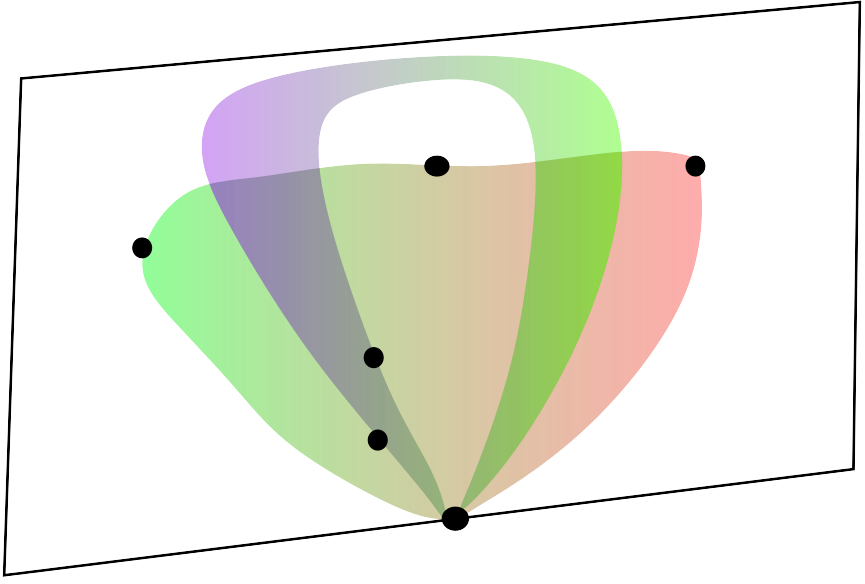}
\put(600,340){$f(S_1)$}
\put(1050,200){$U$}
\put(450,280){$p'_2$}
\put(790,520){$p_1$}
\put(450,510){$p_0$}
\put(480,20){$P_{In}$}
\put(110,410){$p_2$}
\put(380,140){$p'_1$}
\end{overpic}
\caption[The disc $D_\alpha$.]{\textit{The domain $S_1$ and its image under $f$. $p'_i$ and $v'_i$, $i=1,2$ are the images of $p_i$ and $v_i$ under $f$.}}
\label{D}

\end{figure}

We now embed the first-return map $f:S_1\to U$ inside a disc map as appears in Fig.\ref{embedd} - that is, we consider a homeomorphism of $\mathbf{R}^2$ punctured at $3$ points: $P_1$, $P_2$ and $P_3$, s.t. $P_2$ and $P_3$ lie on some sphere. In more detail we choose some rectangle $R_1\subseteq\mathbf{R}^2\setminus\{P_1,P_2,P_3\}$ s.t. the following holds (see the illustration in Fig.\ref{embedd}):

\begin{itemize}
    \item $P_1\in\partial R_1$.
    \item There exists  homeomorphism $P:\mathbf{R}^2\setminus\{P_1,P_2,P_3\}\to\mathbf{R}^2\setminus\{P_1,P_2,P_3\}$ s.t. $P(P_1)=P_3$, $P(P_3)=P_1$ and $P(P_2)=P_2$.
    \item $P|_{R_1}$ is conjugate to $f:S_1\to U$ away from $P_{In}$ and $p_0$.
\end{itemize}

\begin{figure}[h]
\centering
\begin{overpic}[width=0.4\textwidth]{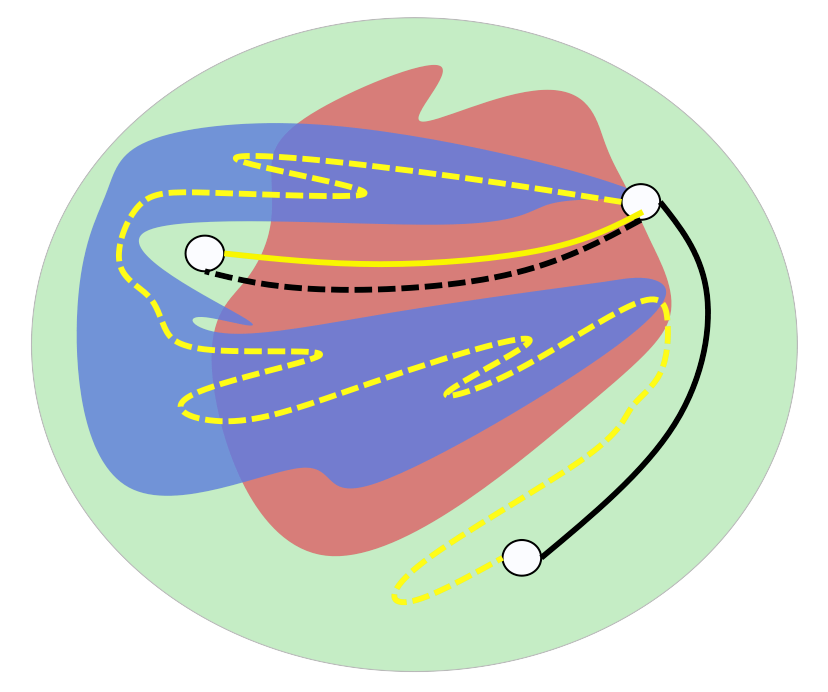}
\put(800,590){$P_2$}
\put(280,510){$P_1$}
\put(550,90){$P_{3}$}
\end{overpic}
\caption[The disc $D_\alpha$.]{\textit{A homeomorphism of $P:\mathbf{R}^2\to\mathbf{R}^2$ punctured at $P_1,P_2$ and $P_3$ (where for simplicity we sketch it as a disc). $R_1$ corresponds to the red region while $P(R_1)$ is the red region. Note $P(P_1)=P_3$, $P(P_3)=P_1$ and $P(P_2)=P_2$. The yellow and black lines denote $\Gamma_1$ and $\Gamma_2$ (which together compose the spine of $R$), while the dashed lines denote their respective images under $P$.}}
\label{embedd}

\end{figure}

Now, let us apply the Betsvina Handel Algorithm to $P$ (see Th.\ref{betshan}). That is, consider the spine of $\mathbf{R}^2\setminus\{P_1,P_2,P_3\}$, a graph $\Gamma$ with two edges - $\Gamma_1$, connecting $P_1$ and $P_3$ and $\Gamma_2$ connecting $P_2$ and $P_3$ (see the illustration in Fig.\ref{embedd2}). It is easy to see that as we collapse $P$ to a graph map $g':\Gamma\to \Gamma$, $g'(\Gamma_1)$ covers $\Gamma_1$ twice (see Fig.\ref{embedd2}), and hence the matrix corresponding to it is the following:
\begin{equation*}
\begin{pmatrix}
    2&1\\
    1&0
\end{pmatrix}
\end{equation*}

By computation it follows the spectral radius (i.e., maximal eigenvalue) of this matrix is greater than $1$ - consequentially, by Th.\ref{betshan} we conclude the invariant set in the region collapsed to $\Gamma_1$, i.e., the rectangle $R_1$, includes infinitely many periodic orbits, all of which persist as $P:R\to R$ is isotoped to a Pseudo-Anosov map $P'':R\to R$. It is easy to see that as we isotope $P$ to a Pseudo-Anosov map $P''$ we also induce an isotopy on $f:S_1\to U$, which in turn induces a deformation $rel$ $H$ on the dynamics of $F'$ in $S_1$ - thus deforming $F'$ $rel$ $H$ to some vector field $F''$ s.t. $F''$ satisfies the assumptions of Th.\ref{orbipers}. In particular, $F''$ generates infinitely many periodic orbits which intersect $S_1$ transversely. Finally, since all these periodic orbits are in the essential class of $F''$ (w.r.t. $H$) and because we can deform $F''$ $rel$ $H$ back to the original vector field $F$ it now follows $Ess(F)$ w.r.t. $H$ includes infinitely many periodic orbits.\\

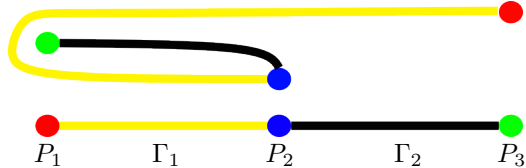
\begin{figure}[h]
\centering
\begin{overpic}[width=0.4\textwidth]{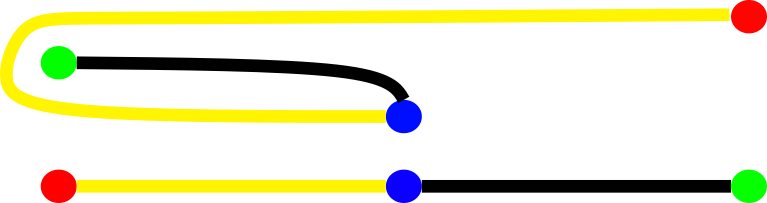}
\put(950,-50){$P_3$}
\put(750,-50){$\Gamma_2$}
\put(50,-50){$P_1$}
\put(280,-50){$\Gamma_1$}
\put(500,-50){$P_2$}
\end{overpic}
\caption[The disc $D_\alpha$.]{\textit{The induced graph map. It is easy to see $\Gamma_1$ covers itself twice under this map.}}
\label{embedd2}

\end{figure}

Having proven $Ess(F)$ w.r.t. $H$ includes infinitely many periodic orbits our next goal is to prove $Ess(F)$ includes periodic orbits of infinitely different many knot types. In fact, we will do more - that is, we will prove we can associate a Template with the dynamics of $F$, using a similar argument to the one used to prove Cor.\ref{templateth} (however, we will not use Cor.\ref{templateth} directly). To begin, let $V$ denote the collection of periodic orbits for $F''$ intersecting $R$ and set $\Phi=\overline{V}$. We first deform $F''$ $rel$ $H$ s.t. every periodic orbit $T$ for $F''$ which lies away from $V$ is generic - i.e., we deform the flow s.t. every eigenvalue of the differential $D(T)$ does not lie on $S^1$ (see Def.\ref{dfT} - note that by construction every periodic orbit in $V$ is already generic). Following that, given any periodic orbit $T$ s.t. $i(T)=1$ we enclose it inside some solid, knotted Torus $\mathbf{T}$ (there exists such a Torus, as by $i(T)=1$ it follows $T$ is either repelling or attracting) - after which we destroy $T$ by deforming the dynamics around it in $\mathbf{T}$ to those of the Kuperberg plug, as given in \cite{Kup}. This proccess removes all the periodic orbits of Orbit Index $1$ - and consequentially the new vector field $F'''$ only has periodic orbits of Orbit Indices $0$ and $-1$.\\

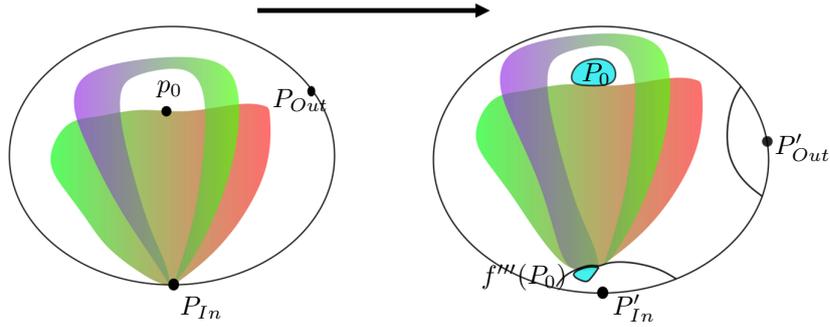
\begin{figure}[h]
\centering
\begin{overpic}[width=0.6\textwidth]{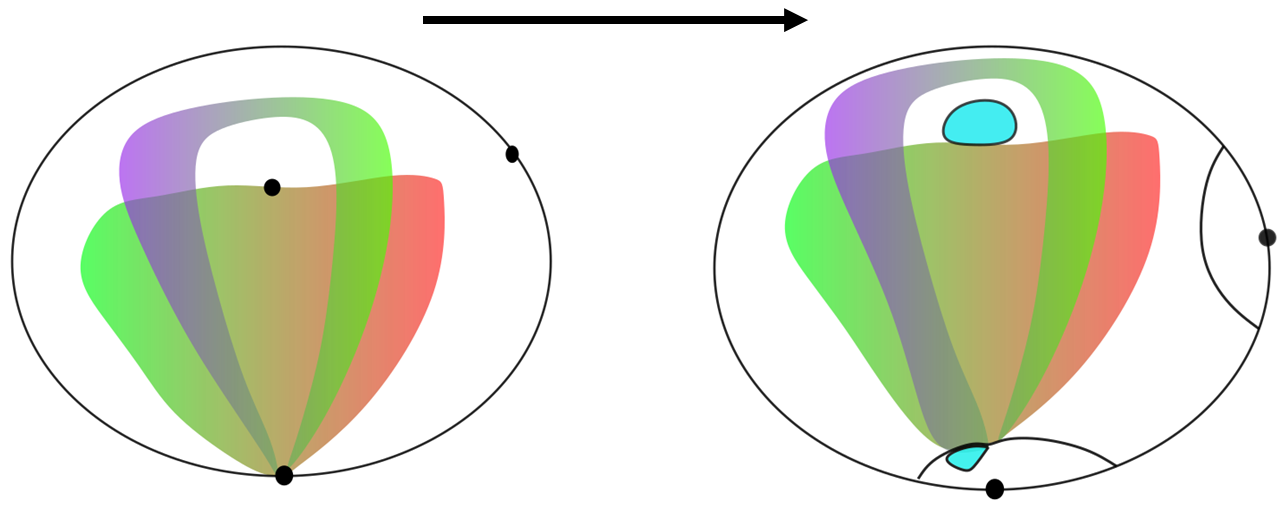}

\put(350,260){$P_{Out}$}
\put(200,280){$p_0$}
\put(230,-10){$P_{In}$}

\put(1000,200){$P'_{Out}$}
\put(790,-10){$P'_{In}$}
\put(620,30){$f'''(P_0)$}
\put(750,295){$P_0$}
\end{overpic}
\caption{\textit{Blowing up the fixed points by a Hopf bifurcation - the first-return map $f''':R\to U$ is deformed to $f'''':R\to U$, while $P_{In}$ and $P_{Out}$ are deformed to the fixed point $P'_{In}$ and $P'_{Out}$ (in this illustration, $U$ is sketched as a disc instead of a half plane). Consequentially, $p_0$ is opened to the disc $P_0$ which flows towards the sink $P'_{In}$.}}
\label{hopf}

\end{figure}

To continue, with previous notations let $f''':S_1\to U$ denote the first-return map for $F'''$.  We first note the invariant set of $f'''$ in $S_1$ lies away from the fixed point $P_{Out}$, i.e., $\Phi$ is disjoint from some ball $B_r(P_{Out})$ (for some sufficiently small $r$). We further note that by definition $P_{Out}$ is a fixed point with a one-dimensional unstable manifold and a two-dimensional stable manifold - hence its Poincare Index is $1$. By the Poincare-Hopf Theorem this implies we can deform the dynamics of $F$ on $\partial B_r(P_{out})$ s.t. $F$ points out of $B_r(P_{Out})$. Moreover, as the invariant set of $f'''$ in $S_1$ lies away from $P_{Out}$ it follows we can do so s.t. no periodic orbit in $\Phi$ is destroyed.\\

Having deformed the local dynamics around $P_{Out}$ to those of a repeller, we continue by opening $P_{In}$ to a periodic orbit by a Hopf bifurcation as illustrated in Fig.\ref{hopf}, thus replacing $P_{In}$ with a sink $P'_{In}$. This proccess adds precisely one periodic orbit to $\Phi$, and similarly to the proof of Cor.\ref{templateth} we do so s.t. no periodic orbit in $\Phi$ is destroyed. In particular, this deformation of $F'''$ to a new vector field $F''''$ induces an isotopy of the first return map $f''':S_1\to U$ to some $f'''':S_1\to U$, s.t. $f'''':S_1\to U$ acts as a Smale Horseshoe map on $S_1$. At this point we remark that as we deform $F'''$ to $F''''$ the heteroclinic knot $H$ persists - but since we change the type of the fixed points on it the deformation is no longer $rel$ $H$. \\

\begin{figure}[h]
\centering
\begin{overpic}[width=0.3\textwidth]{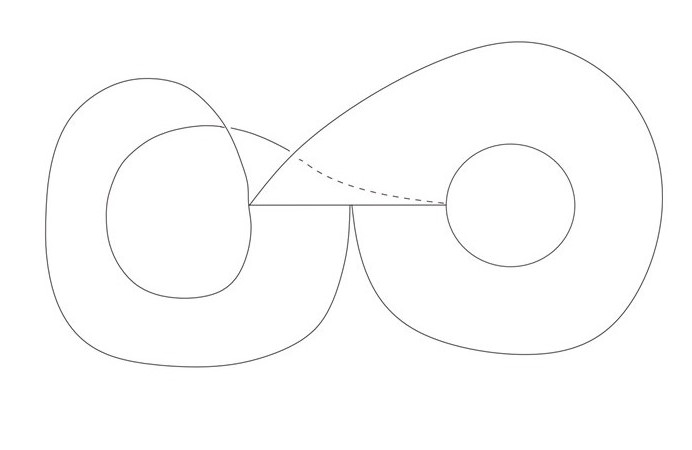}
\end{overpic}
\caption{\textit{The $L(0,1)$ Template.}}
\label{tenmplate}

\end{figure}

By construction, it is easy to see the periodic orbits in $\Phi$ do not bifurcate and keep their knot type as we smoothly vary $F'''$ to $F''''$. We now deform $F''''$ to $F'''''$ by moving the orbit of any initial condition in $S^3\setminus(\Phi\cup B_r(P_{Out}))$ into the basin of attraction of $P'_{In}$ - which leaves $\Phi$ as the maximal invariant set of $F'''''$ in $S^3\setminus (B_r(P_{Out})\cup B_r(P'_{In}))$ (for some sufficiently small $r>0$). As such, similarly to the proof of Cor.\ref{templateth} by the Birman-Williams Theorem we conclude we can associate a Template $\tau$ with the periodic dynamics of $\Phi$. Moreover, since those periodic dynamics are generated by suspending a Smale Horseshoe around $H$ by Def.\ref{model} we know $\Phi$ is a basic set - hence similarly to the proof of Cor.\ref{templateth} we further infer $F'''''$ is a model flow on $S^3\setminus (B_r(P_{Out})\cup B_r(P'_{In}))$.\\

We now note one can embed the $L(0,1)$ Template (see Fig.\ref{tenmplate}) inside $S^3\setminus (B_r(P_{Out})\cup B_r(P'_{In})\cup H)$ as illustrated in Fig.\ref{embed}. Or, in other words, we can suspend a Smale Horseshoe map inside $S^3\setminus (B_r(P_{Out})\cup B_r(P'_{In})\cup H)$ s.t. the resulting flow is both a model flow and generates the $L(0,1)$ Template. It is easy to see $F'''$ and this embedding of the $L(0,1)$ Template are orbitally equivalent around their maximal invariant sets in $S^3\setminus (B_r(P_{Out})\cup B_r(P'_{In})\cup H)$ - therefore by Th.\ref{begbon} we conclude $\tau=L(0,1)$. This yields every knot type on the $L(0,1)$ Template (save possibly for some finite number) is realized as a periodic orbit for $F'''''$ - and since by construction every periodic orbit for $F'''''$ (save possibly for a finite number) can be deformed to a periodic orbit for $F$ without changing its knot type we conclude every knot on the $L(0,1)$ Template (save possibly for a finite set) is realized as a periodic orbit for $F$, the original vector field. Therefore, since $L(0,1)$ includes infinitely many Torus knots of different knot types (see Th.6.1.2.a in \cite{Hol}) we conclude the same is true for $Ess(F)$. The proof of Th.\ref{trefoilor} is now complete.
\end{proof}
\begin{remark}
    As observed numerically in \cite{Le}, there are regions in the parameter space of the Rössler system at which the Template corresponding to the attractor is the Horseshoe Template from the proof above. Moreover, using similar ideas and arguments to those above one can also prove directly the $L(0,1)$ Template describes the dynamics of an idealized model for the Rössler system (see \cite{I1} for details).
\end{remark}

\begin{figure}[h]
\centering
\begin{overpic}[width=0.3\textwidth]{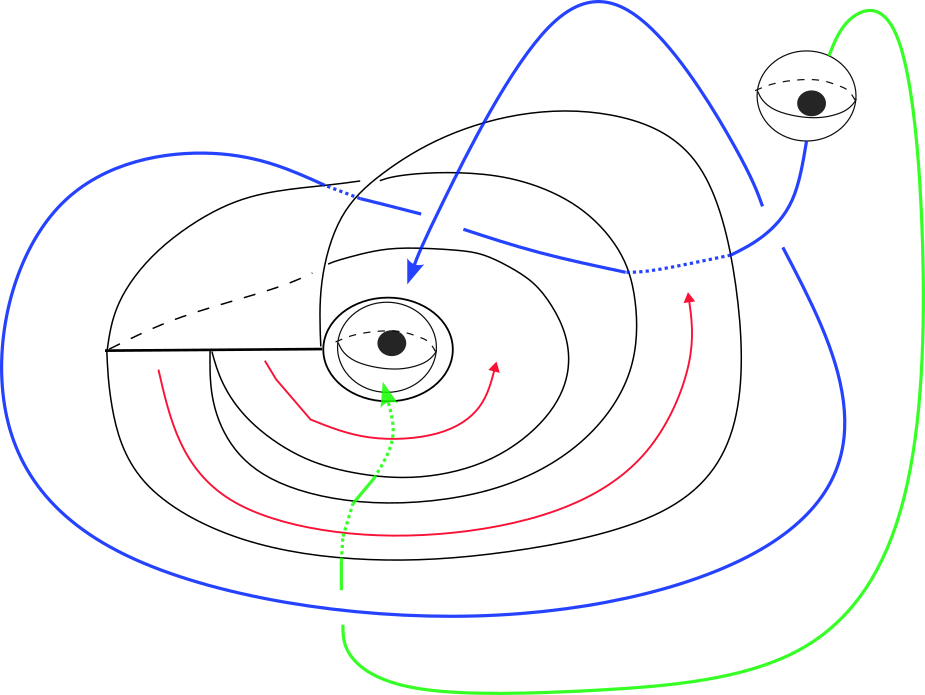}
\put(300,270){$P'_{In}$}
\put(900,570){$P'_{Out}$}
\end{overpic}
\caption{\textit{A horseshoe template, embedded on the complement of a trefoil knot connecting the sink $P'_{In}$ and the source $P'_{Out}$. The vector field points out of the sphere encasing $P'_{Out}$ and inside the sphere enclosing $P'_{In}$. It is easy to see the Templates in Fig.\ref{tenmplate} and the one above are ambient isotopic, hence the same.}}
\label{embed}
\end{figure}

Having proven Th.\ref{trefoilor} we now use it to study the dynamics of the Rössler attractor. To do so, recall we denote by $P$ the parameter space of the Rössler system (see page \ref{parspace}), and that $W^s_{Out}$ denotes the two-dimensional stable manifold of the saddle focus $P_{Out}$ (see the illustration in Fig.\ref{local}). Inspired by Th.\ref{stability} and Th.\ref{thurston-nielsen} we now prove:

\begin{theorem}
\label{trefoil1}   Assume $p\in P$ is a parameter at which the Rössler system satisfies the following:

\begin{enumerate}
    \item There exists a a heteroclinic knot $H$ as in Fig.\ref{trefoil}.
    \item The backwards orbit of any initial condition $s\in W^s_{Out}$ either intersects transversely with the plane $O=\{(x,y,0)|x,y\in\mathbf{R}\}$ or tends to $\infty$ through $\{(x,y,z)|z>0\}$ in backwards time. 
\end{enumerate}

Then, the corresponding Rössler system generates a chaotic attractor which includes periodic orbits of infinitely different Torus knots.
\end{theorem}
\begin{proof}
We prove Th.\ref{trefoil1} as follows - first we prove that under the assumptions above there exists a bounded trapping set $K_1$ for the flow (this is based on ideas originally introduced in \cite{I2} - which, for completeness, we include here as well). After that, we deform the flow inside $K_1$, after which we use both Th.\ref{orbipers} and Th.\ref{trefoilor} to prove the dynamics in $K_1$ include infinitely many periodic orbits - thus concluding the proof of Th.\ref{trefoil1}. To begin, recall that by Th.\ref{generaldynamics} there exist two one-dimensional invariant manifolds, $\Gamma_{In}\subseteq W^s_{In}$ and $\Gamma_{Out}\subseteq W^u_{Out}$, which connect the fixed points $P_{In}$ and $P_{Out}$ to $\infty$ - and in particular, $H=\{P_{In},P_{Out},\infty\}\cup\Theta\cup\Gamma_{In}\cup\Gamma_{Out}$ (where $\Theta$ is the bounded heteroclinic orbit - see the illustration in Fig.\ref{trefoil}).\\

In addition, from now on throughout the proof we denote the vector field generating the Rössler system corresponding to the parameter $p=(a,b,c)$ by $F_p$ (see Eq.\ref{Vect}). It is easy to see the normal vector to the set $O\setminus\{\infty\}$ is $(0,0,1)$ - hence by computation for all $s\in O\setminus\{\infty\}$ we have $F_p(s)\bullet(0,0,1)=b>0$ (see the illustration in Fig.\ref{P}) - consequentially, once the orbit of an initial condition $s\in\mathbf{R}^3$ enters the region $\{(x,y,z)|z>0\}$ it cannot escape it under the flow.\\ 

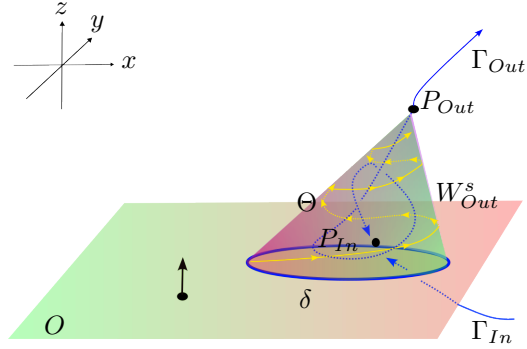
\begin{figure}[h]
\centering
\begin{overpic}[width=0.4\textwidth]{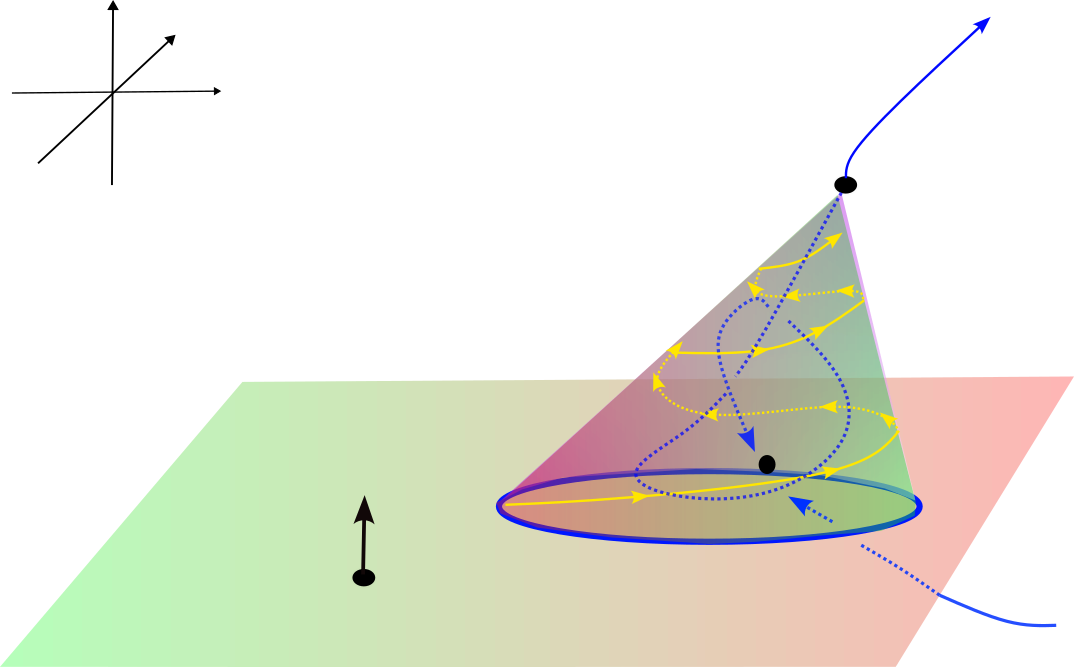}
\put(70,10){$O$}
\put(900,0){$\Gamma_{In}$}
\put(900,530){$\Gamma_{Out}$}
\put(800,450){$P_{Out}$}
\put(830,270){$W^s_{Out}$}
\put(600,180){$P_{In}$}
\put(560,250){$\Theta$}
\put(565,60){$\delta$}
\put(220,530){$x$}
\put(160,610){$y$}
\put(90,635){$z$}
\end{overpic}
\caption[Fig31]{The plane $O$, the intersection $\delta=W^s_{Out}\cap O$, and the region $C_1$ trapped inside the cone (note $\Theta,P_{In}\subseteq K_1$). On $O$ the vector field $F_p$ points into $\{(x,y,z)|z>0\}$.}
\label{P}
\end{figure}

We now recall that since for $p=(a,b,c)$ we have both $P_{In}=(\frac{c-\sqrt{c^2-4ab}}{2},-\frac{c-\sqrt{c^2-4ab}}{2a},\frac{c-\sqrt{c^2-4ab}}{2a})$ and $P_{Out}=(\frac{c+\sqrt{c^2-4ab}}{2},-\frac{c+\sqrt{c^2-4ab}}{2a},\frac{c+\sqrt{c^2-4ab}}{2a})$. Since $a,b\in(0,1),c>1$ (see the discussion in page \ref{parspace}) we have $c^2-4ab>0$ - therefore we conclude $P_{In},P_{Out}\in\{(x,y,z)|z>0\}$ (see the illustration in Fig.\ref{P}). Now, let us consider the set $C=\{(x,y,z)|z>0\}\setminus \overline{W^s_{Out}}$ - by our assumption it is easy to see $\overline{W^s_{Out}}\cap \overline{O}$ is a curve on $\overline{O}$ homotopic to $S_1$ (where the closure of $O$ is taken in $S^3$). This implies $C$ is composed of two components - $C_1$ and $C_2$, as illustrated in Fig.\ref{sb1} and \ref{P}. As such, since the one-dimensional unstable manifold $W^u_{Out}$ is transverse to $W^s_{Out}$ at the saddle-focus $P_{Out}$ by $W^s_{Out}=\Gamma_{Out}\cap\Theta$ it follows $W^u_{Out}$ separates $\Gamma_{Out}$ and $\Delta_{Out}$ in $\{(x,y,z)|z>0\}$ (see the illustration in Fig.\ref{sb1} and Fig.\ref{P}). Therefore, let us denote by $C_1$ the component of $C$ s.t. $\Theta\subseteq C_1$ (see the illustration in Fig.\ref{sb1} and Fig.\ref{P}).\\

\begin{figure}[h]
\centering
\begin{overpic}[width=0.3\textwidth]{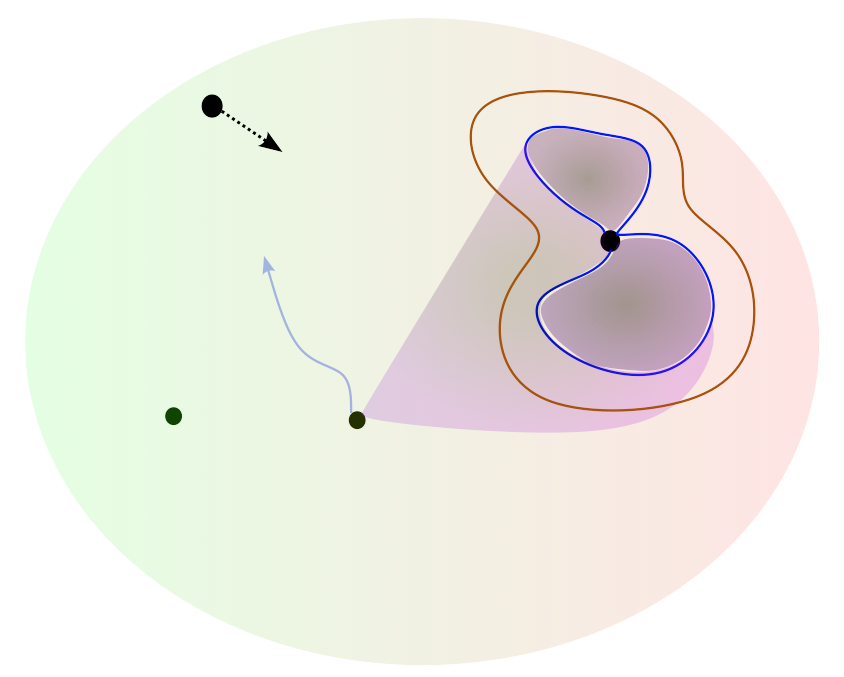}
\put(370,240){$P_{Out}$}
\put(150,270){$P_{In}$}
\put(720,560){$\delta$}
\put(750,490){$\infty$}
\put(830,270){$\gamma$}
\put(235,430){$\Theta$}
\put(555,230){$W^s_{Out}$}
\end{overpic}
\caption{\textit{The plane $\overline{O}$, sketched as a sphere in $S^3$ (with $\infty$ on it). The region $Q=\{(x,y,z)|z>0\}$ is the region trapped inside the sphere, while $W^s_{Out}$ is the purple surface. The set $C_1$ is the component of ${Q}\setminus \overline{W^s_{Out}}$ which includes $\Theta$. In this scenario, $\overline{W^s_{Out}}\cap \overline{O}$ is the blue curve $\delta$ which is homotopic (and not homeomorphic) to $S^1$, with singularities at $\infty$. The dark orange curve is $\gamma$.}}
\label{sb1}
\end{figure}
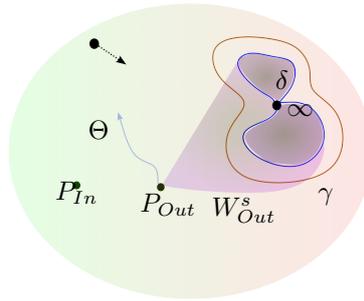

To continue, let $S_1$ denote a local cross-section transverse to $\Theta$ - moreover, we choose $S_1$ to be sufficiently small s.t. it is transverse to the flow. Since $\overline{W^s_{Out}}\cap \overline{S_b}$ is a curve homotopic to $S^1$ it follows there exists a curve $\gamma\subseteq {{O}}\cap\partial C_1$ (where both boundaries are taken in $\mathbf{R}^3$) s.t. the forward orbit of any $s\in\gamma$ hits $S_1$ transversely (see the illustration in Fig.\ref{sb1} and Fig.\ref{sb2}). Now, consider the region $K_1$, trapped between $O$ and the flow lines connecting $\gamma$ and $S_1$, as illustrated in Fig.\ref{sb2}. Since on $\partial K_1\cap O$ the vector field $F_p$ points inside $\{(x,y,z)|z>0\}$, summarizing our discussion above it is easy to see the following is satisfied:

    \begin{enumerate}
        \item    At every $s\in\partial K_1$ the vector $F_p(s)$ is either tangent to $\partial K_1$ or points inside $K_1$ (depending on whether $s$ is on a flow line connecting $\gamma$ to $S_1$ or $s\in O$). In particular, no orbit can escape $K_1$ under the flow.
        \item  $\Theta$ enters $K_1$ - and once it does it never escapes its interior. Moreover, we can choose $K_1$ s.t. it includes an arbitrarily large subset of $\Theta$ (i.e., $S_1$ is arbitrarily close to the fixed point $P_{Out}$).
        \item $K_1$ is bounded in $\mathbf{R}^3$, and homeomorphic to a ball. Moreover, $\partial K_1$ can be chosen to be arbitrarily and uniformly close to $\partial C_1\setminus O$.
        \item $F_p$ has no fixed points on $\partial K_1$ - i.e., for all $s\in\partial K_1$ we have $F_p(s)\ne0$.
        \item $P_{In}\in K_1$, and we can choose $K_1$ s.t. $P_{Out}$ is arbitrarily close to $\partial K_1$.
    \end{enumerate}

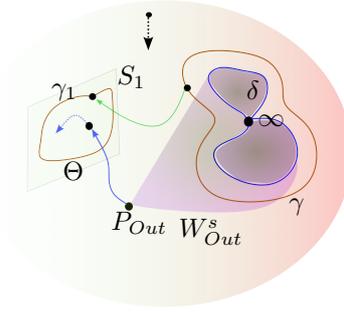
\begin{figure}[h]
\centering
\begin{overpic}[width=0.3\textwidth]{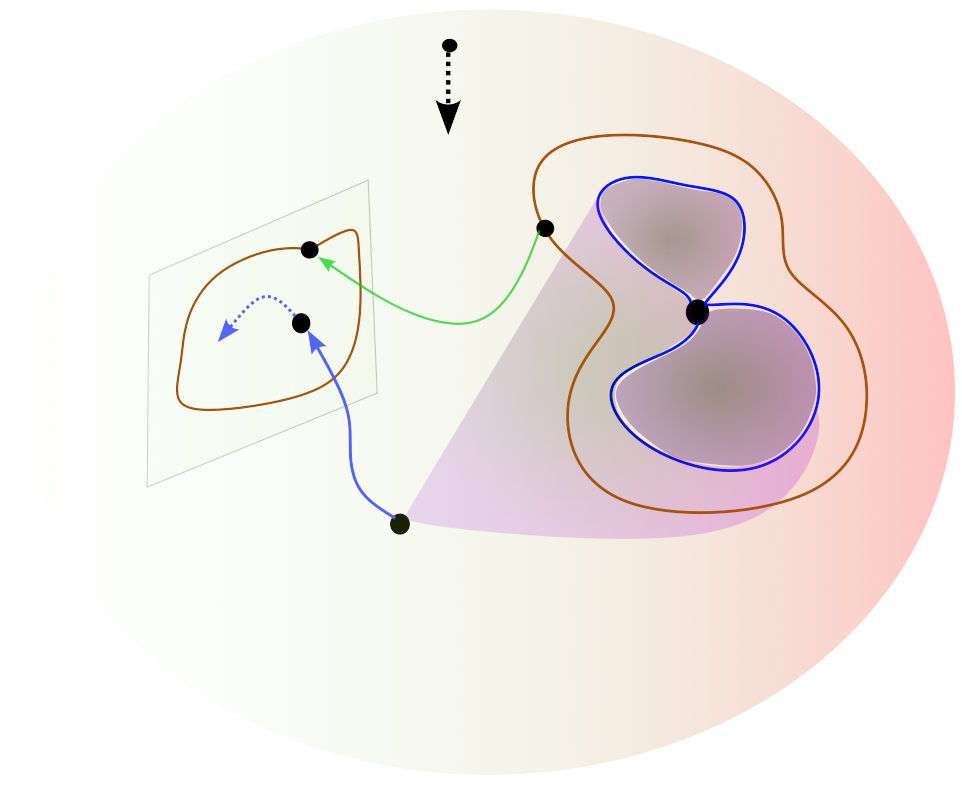}
\put(370,220){$P_{Out}$}
\put(720,560){$\delta$}
\put(750,490){$\infty$}
\put(830,270){$\gamma$}
\put(245,350){$\Theta$}
\put(215,570){$\gamma_1$}
\put(385,600){$S_1$}
\put(545,200){$W^s_{Out}$}
\end{overpic}
\caption{\textit{The cross section $S_1$ - by the continuity of the flow, the orbit of every initial condition on $\gamma$ flows to $\gamma_1$, some curve on $S_1$ - where $S_1$ is a cross-section transverse to the flow and to $\Theta$. Consequentially, there exists a connected region $K_1\subseteq C_1$ from which no orbit can escape.}}
\label{sb2}
\end{figure}

Having proven the existence of a trapping set $K_1$ for $F_p$ we are now ready to prove Theorem \ref{trefoil1}. We begin by deforming $F_p$ around the point at $\infty$ as follows: we smoothen $F_p$ around some arbitrarily small neighborhood of $\infty$ in $S^3$, $N_\infty$, by removing the fixed point at $\infty$ (we can do so by Th.\ref{generaldynamics}) - which we do by connecting the unbounded heteroclinic orbits $\Gamma_{Out}$ and $\Gamma_{In}$ to form $\Gamma$, a heteroclinic orbit connecting $P_{In}$ and $P_{Out}$ through $\infty$ (see the illustration in Fig.\ref{inftyy}). As a consequence, $F_p$ is deformed to $F'$, a smooth vector field of $S^3$ which creates a heteroclinic knot $\{P_{In}, P_{Out}\}\cup\Theta\cup\Gamma=H$ satisfying the assumptions of Th.\ref{orbipers} and Th.\ref{trefoilor}. Consequentially, $F'$ generates periodic orbits of infinitely many different Torus knot types. In addition, by Th.\ref{generaldynamics} we now we can choose the deformation s.t. the vector fields $F_p$ and $F'$ coincide on $\overline{K_1}$.\\

\begin{figure}[h]
\centering
\begin{overpic}[width=0.45\textwidth]{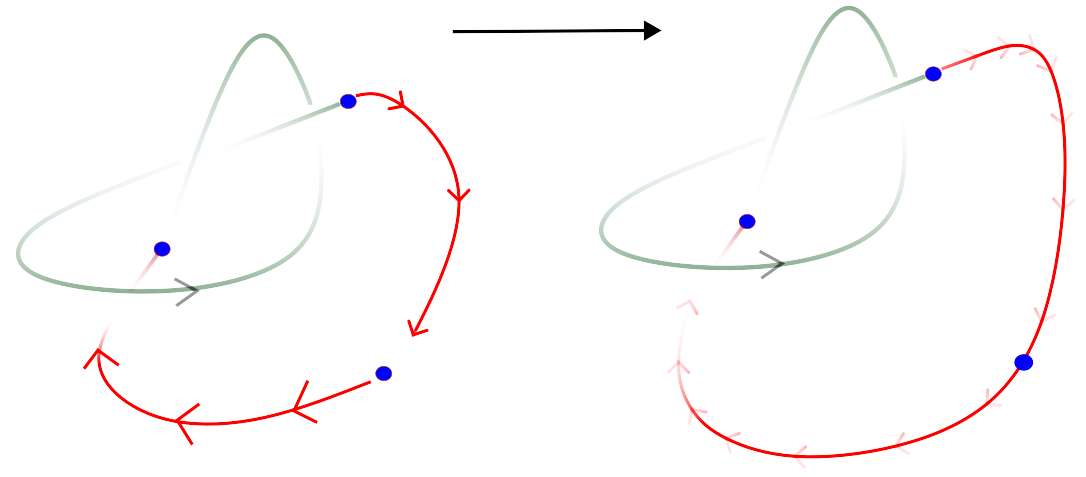}%
\put(930,280){$\Gamma$}
\put(840,410){$P_{Out}$}
\put(320,380){$P_{Out}$}
\put(390,80){$\infty$}
\put(60,200){$P_{In}$}
\put(600,230){$P_{In}$}
\put(970,120){$\infty$}
\put(420,180){$\Gamma_{Out}$}
\put(250,110){$\Gamma_{In}$}
\end{overpic}
\caption[Deforming $F_p$ to $R_p$ for trefoil parameters.]{\textit{The deformation of $F_p$ to $R_p$ - $\Gamma_{Out},\Gamma_{In}$ are connected by the deformation to a heteroclinic orbit $\Gamma$ (thus destroying the fixed point at $\infty$).}}\label{inftyy}
\end{figure}
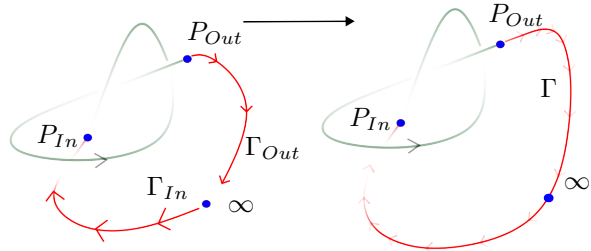

Having defined $F'$, our next goal is to prove its periodic dynamics given by Th.\ref{orbipers} and Th.\ref{trefoilor} lie in $K_1$. By the previous paragraph this will immediately imply all these periodic orbits are also generated by the vector field $F_p$ - i.e., by the Rössler system corresponding to the parameter $p$. To do so, we deform the dynamics of $F'$ $rel$ $H$ on the interior of $K_1$ s.t. there exists a cross-section $R$, a topological triangle with a tip at $P_{In}$ which is suspended with the flow as appears in Fig.\ref{ross1} (we can do so since $K_1$ is homeomorphic to a ball). Moreover, we do so s.t. the deformation does not change the behavior of $F'$ on $\partial K_1$.\\

\begin{figure}[h]
\centering
\begin{overpic}[width=0.4\textwidth]{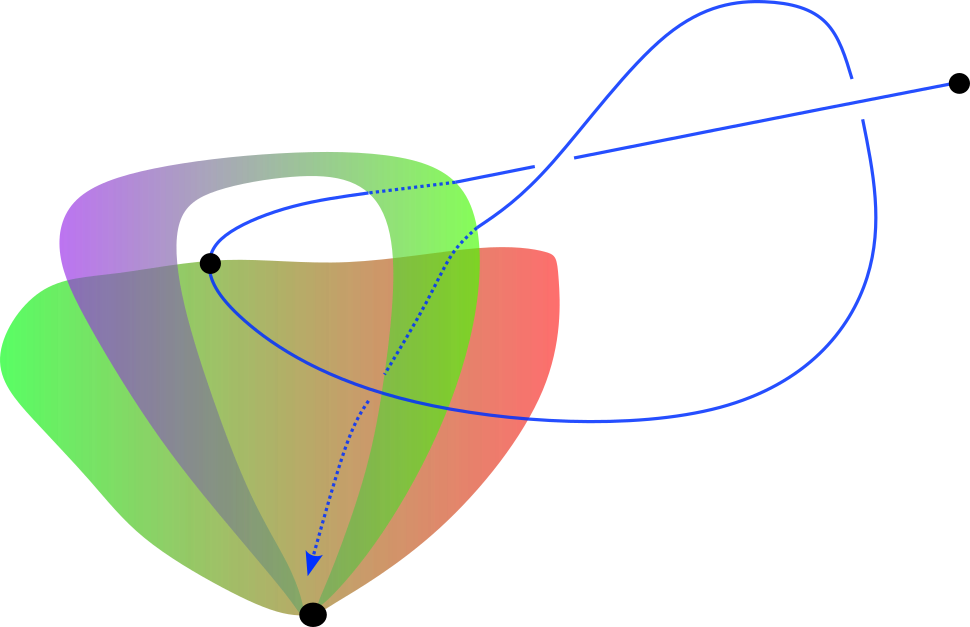}%
\put(940,500){$P_{Out}$}
\put(250,380){$p_0$}
\put(70,200){$R$}
\put(600,230){$\Theta$}
\put(340,0){$P_{In}$}
\end{overpic}
\caption{\textit{The suspension of the cross-section $R$ along $\Theta$ w.r.t. $F$ inside $K_1$.}}\label{ross1}
\end{figure}
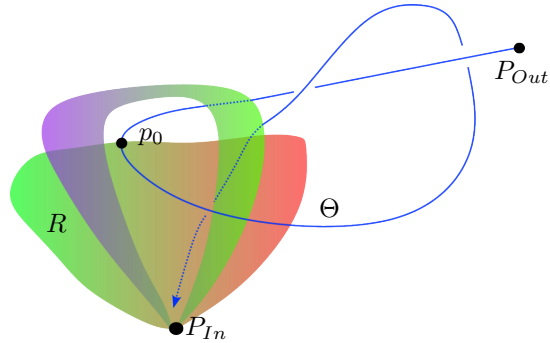

Denote the corresponding vector field generated after the deformation described above by $F$ - it is easy to see $F$ satisfies the assumptions of Th.\ref{trefoilor}, and that per the proof of Th.\ref{trefoilor} $R$ intersects periodic orbits of infinitely many knot type - all of which must be interior to $K_1$. Moreover, since the deformation from $F'$ to $F$ takes place away from $\partial K_1$ (i.e., the two vector fields coincide on $\partial K_1$), it is easy to see that as we return from $F$ to $F'$ no periodic orbit $T$ for $F$ in $K_1$ can escape $K_1$ by colliding with $\partial K_1$ - since that would necessitate the creation of a surface on $\partial K_1$ on which the vector field points outside of $K_1$, and by our construction of $K_1$ there is no such surface. Consequentially, if $T\subseteq K_1$ is a periodic orbit in $Ess(F)$ (w.r.t. $H$), $T$ persists as a periodic orbit in $K_1$ without changing its knot type as we return from $F$ to $F'$.\\

To summarize, we have just proven $K_1$ includes infinitely many periodic orbits for $F'$, of infinitely many different knot types - and since by construction $F'$ and $F_p$ coincide on $K_1$ the same is true for the vector field $F_p$. Or, in other words, we have just proven the dynamics of the Rössler system corresponding to $p$ include periodic orbits of infinitely many knot types in $K_1$  - moreover, by the proof of Th.\ref{trefoil}, we know these periodic orbits are all Torus knots on the $L(0,1)$ Template. Finally, since $K_1$ is a trapping set for $F_p$ we conclude these periodic orbits all lie inside and attracting invariant set, and the proof of Th.\ref{trefoil1} is now complete.
    \end{proof}

\begin{figure}[h]
\centering
\begin{overpic}[width=0.5\textwidth]{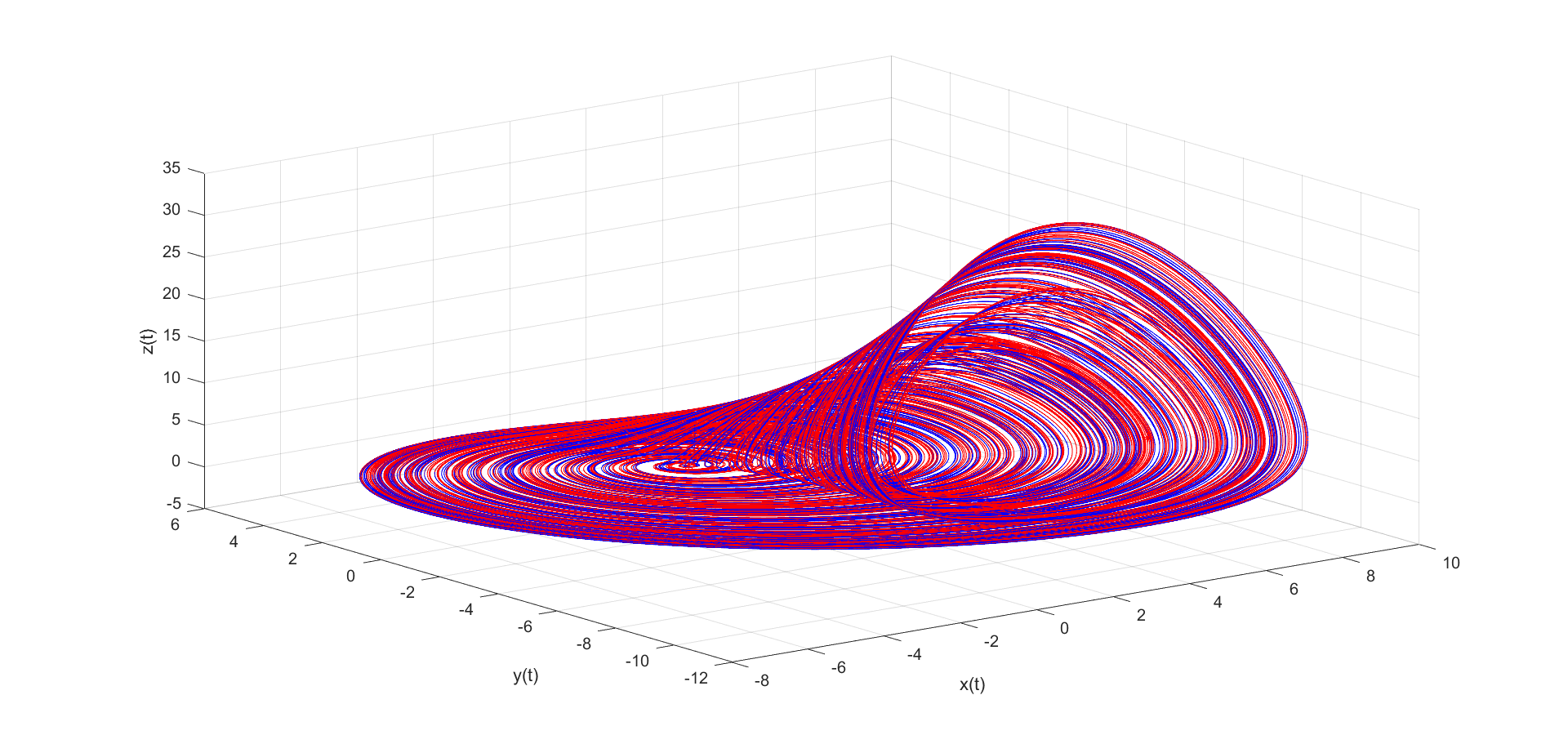}
\end{overpic}
\caption[A heteroclinic attractor.]{\textit{The Rössler attractor at parameter values at which the vector field was observed numerically to generate a heteroclinic trefoil knot as in Fig.\ref{trefoil} (see Fig.B.1 \cite{MBKPS}).}}
\label{attractortref}
\end{figure}

Before we conclude this section we remark there is some evidence that heteroclinic dynamics play a role in the bifurcations of the Rössler system. As observed numerically in \cite{MBKPS}, \cite{BBS} and \cite{G} (among others), $P$ includes spiral bifurcation structures centered around parameters $p$ which are often called"Periodicity Hubs" or "Spiral Centers". It would be an interesting question for a future study to check whether Th.\ref{trefoilor} and Th.\ref{trefoil1} are somehow connected to such bifurcation phenomena - especially since some periodicity hubs were observed to possibly correspond to heteroclinic dynamics (see \cite{SR}).

\subsection{The persistence of periodicity in the Lorenz attractor}
\label{lorenz}
Having applied Th.\ref{orbipers} to study the Rössler system, in this section we perform a similar analysis on the Lorenz system - however, this time our result would be much more dependent on the Orbit Index Theory (see Sect.\ref{orbitin}). To begin, recall we define the Lorenz system by the following system of differential equations:

\begin{equation} \label{Vect11}
\begin{cases}
\dot{x} = \sigma(y-x) \\
 \dot{y} = x(\rho-z)-y\\
 \dot{z}=xy-\beta z
\end{cases}
\end{equation}

Where $\sigma,\rho,\beta>0$. From now on, given $(\sigma,\rho,\beta)\in\mathbf{R}^3$ we will always denote the corresponding vector field by $L_{\sigma,\rho,\beta}$. Note that by direct computation it follows that when both $\rho>1,\beta>0$ the vector field $L_{\sigma,\rho,\beta}$ has precisely three fixed points::
\begin{itemize}
    \item The origin, $0$, a real saddle with a two-dimensional stable manifold $W^s(0)$ and a one-dimensional unstable manifold (in particular, the Poincare Index of $0$ is $1$)
    \item  $p^+,p^-$, two symmetric fixed points (w.r.t. the $z$-axis), given by $(\pm\sqrt{\beta(\rho-1)},\pm\sqrt{\beta(\rho-1)},\rho-1)$ of Poincare index $-1$ - moreover, $p_{\pm}$ are (generically) either saddles, saddle-foci, or attracting fixed points (as shown in \cite{L}, when $\sigma<\beta+1$ both $p^\pm$ are saddle foci).
    \item $\infty$, which is a repelling fixed point (see, for example, Th.2 in \cite{XYS}).
\end{itemize}

As proven in \cite{L}, at $(\sigma,\rho,\beta)=(10,28,\frac{8}{3})$ the Lorenz system possesses an attracting invariant set - and as shown numerically in \cite{L}, its appearance is as in Fig.\ref{attractorlo}. \\
\begin{figure}[h]
\centering
\begin{overpic}[width=0.6\textwidth]{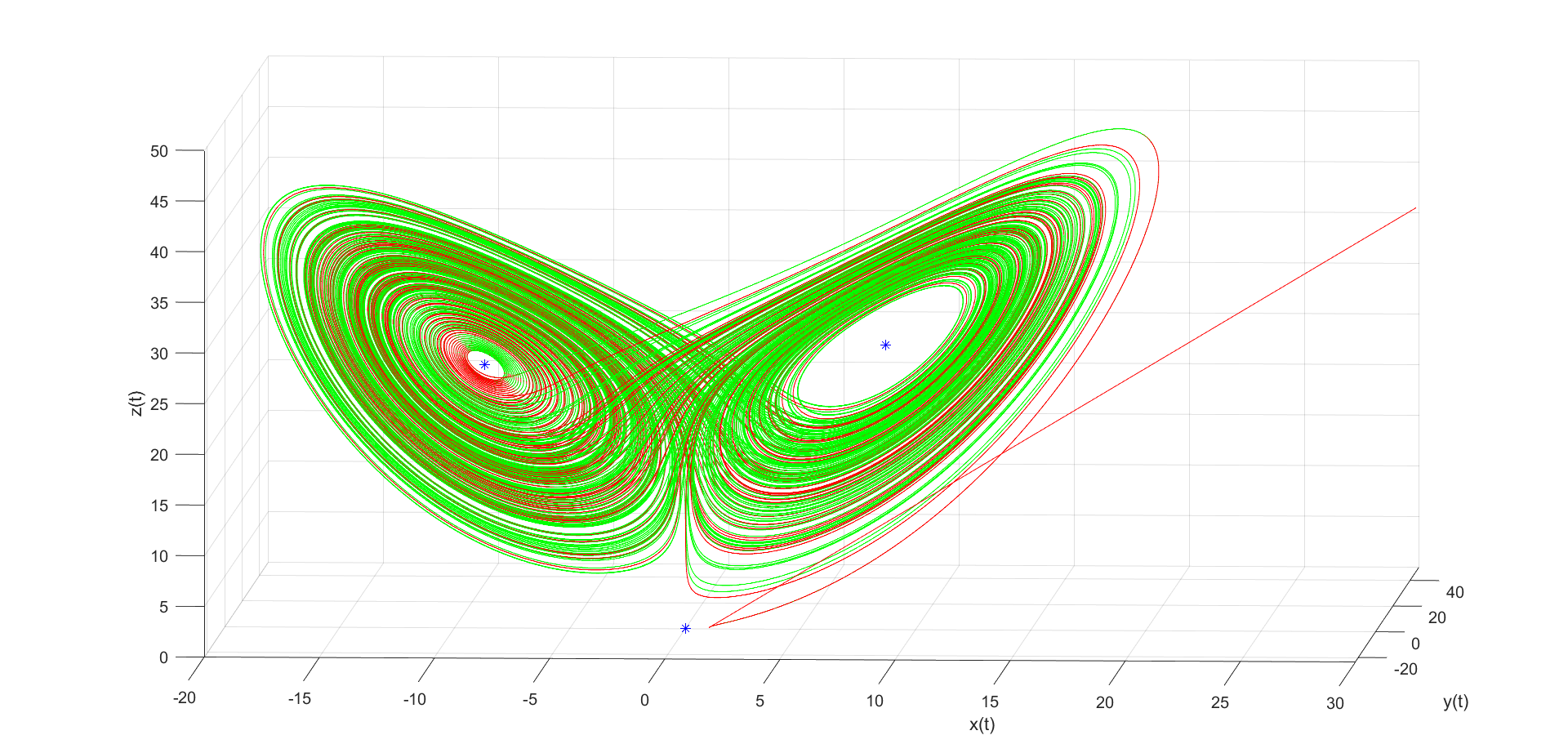}
\end{overpic}
\caption{\textit{The Lorenz attractor at $(\sigma,\rho,\beta)=(10,28,\frac{8}{3})$. The blue dots denote the fixed points for the flow.}}
\label{attractorlo}
\end{figure}

This section is organized as follows - we begin with recalling several facts from \cite{Pi} on the dynamics of the Lorenz system. Following that we use Th.\ref{orbipers} to study the dynamics of the flow, where we perform a similar analysis to the one carried in Th.\ref{trefoil1} and Th.\ref{trefoilor} (see Th.\ref{trefoil2}). Finally we conclude this section with Th.\ref{persistence} in which we use the Orbit Index Theory to study the persistence of periodic orbits in the Lorenz attractor in the breakdown of the heteroclinic knot.\\

To begin, recall that given $v=(\sigma,\rho,\beta)$ we denote the vector field generating the Lorenz system by $L_v$ and that $\sigma:\{1,2\}^\mathbf{Z}\to\{1,2\}^\mathbf{Z}$ always denotes the double-sided shift. In addition, from now on we denote the stable, invariant manifold of the origin by $W^s(0)$. We begin by recalling the following fact (see Prop.2.1 and Lemma 2.2 in \cite{Pi}):
\begin{proposition}
   \label{talilemma} There exists an open set of parameters $P=\{(\sigma,\rho,\beta)|\sigma,\beta>0,\rho>1,\rho>\frac{(\sigma+1)^2}{4\sigma}\}$ s.t. for every $v\in P$ the following is satisfied (see the illustration in Fig.\ref{RECT1}):
   
   \begin{itemize}
       \item The corresponding Lorenz system generates a universal cross-section $R_v$. In particular, $R_v$ intersects transversely with every periodic orbit for the vector field $L_v$.
       \item $R_v\cap W^s(0)$ includes a curve which divides $R_v$ in two sub-domains, $R_{1,v}$ and $R_{2,v}$ as depicted in Fig.\ref{RECT1}. Moreover, $p^-\in\partial R_{1,v}$ and $p^{+}\in\partial R_{2,v}$.
       \item The first-return map $\psi_v:R_{1,v}\cup R_{2,v}\to R_v$ is well-defined and continuous. Consequentially, there exists a $\psi_v-$invariant set $I_v\subseteq\psi_v:R_{1,v}\cup R_{2,v}$ and a continuous $\pi:I_v\to\{1,2\}$ s.t. $\pi\circ\psi_v=\sigma\circ\pi$.
   \end{itemize}

\end{proposition}
\begin{figure}[h]
\centering
\begin{overpic}[width=0.5\textwidth]{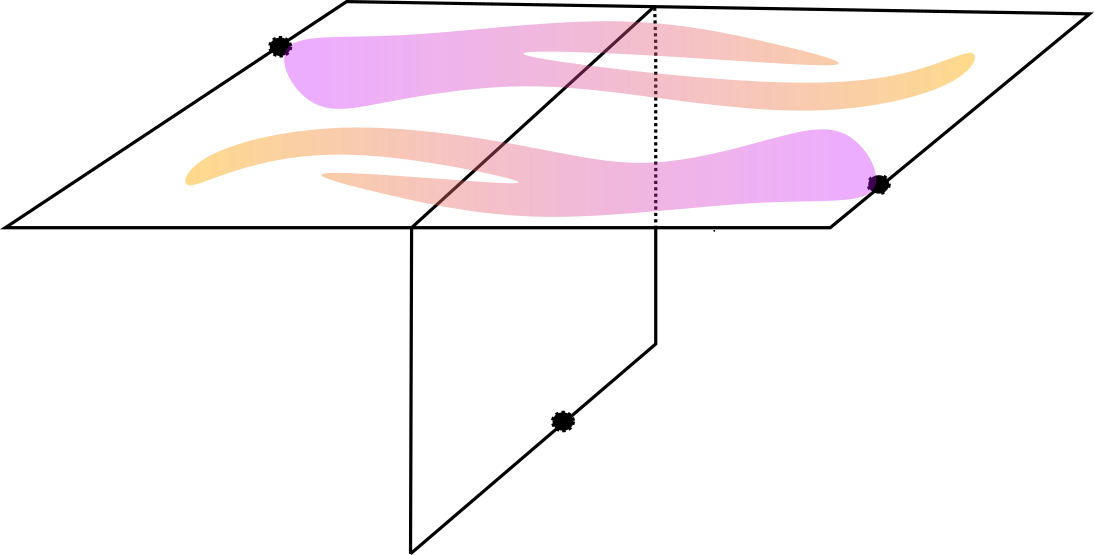}
\put(250,70){$W^s(0)$}
\put(860,470){$R_{2,v}$}
\put(100,320){ $R_{1,v}$}
\put(200,460){$p^-$}
\put(500,60){$0$}
\put(830,330){$p^+$}
\put(300,430){$\psi_v(R_{1,v})$}
\put(570,330){$\psi_v(R_{2,v})$}
\end{overpic}
\caption{\textit{The cross-section $R$ and the first-return map associated with it per Prop.\ref{talilemma}.}}
\label{RECT1}
\end{figure}
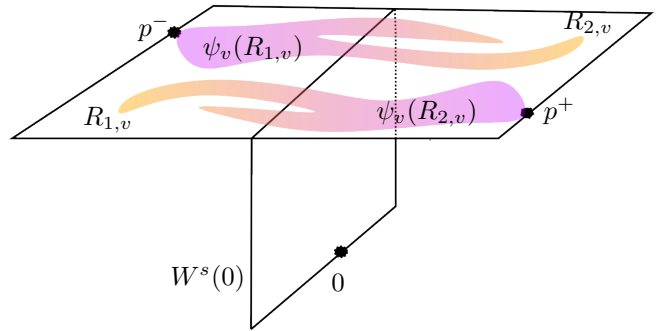

To continue, we further recall that as proven in Th.1.1 in \cite{Pi} there exist parameters $p\in P$ at which the Lorenz system generates a heteroclinic trefoil knot as in Fig.\ref{heteroclo}.  In addition, as proven in \cite{Pi} the first-return map at such parameters is as appears in Fig.\ref{RECT} - and in particular its dynamics can be described by a coding on two symbols. In more detail, recall the Lorenz Template (see Fig.\ref{TEMP2}) - then, the following fact was proven (see Cor.1.3 and Th.1.1 in \cite{Pi}):
\begin{theorem}
    \label{basis} Let $p\in P$ be a parameter s.t. the Lorenz generates a heteroclinic knot $H$ as in Fig.\ref{heteroclo}. Then, every knot type on the Lorenz Template (save possibly for two) is realized as a periodic orbit for the flow inside the attractor. Moreover, such parameters $p$ exist in $P$.
\end{theorem}
\begin{figure}[h]
        \centering
\begin{overpic}[width=0.4\textwidth]{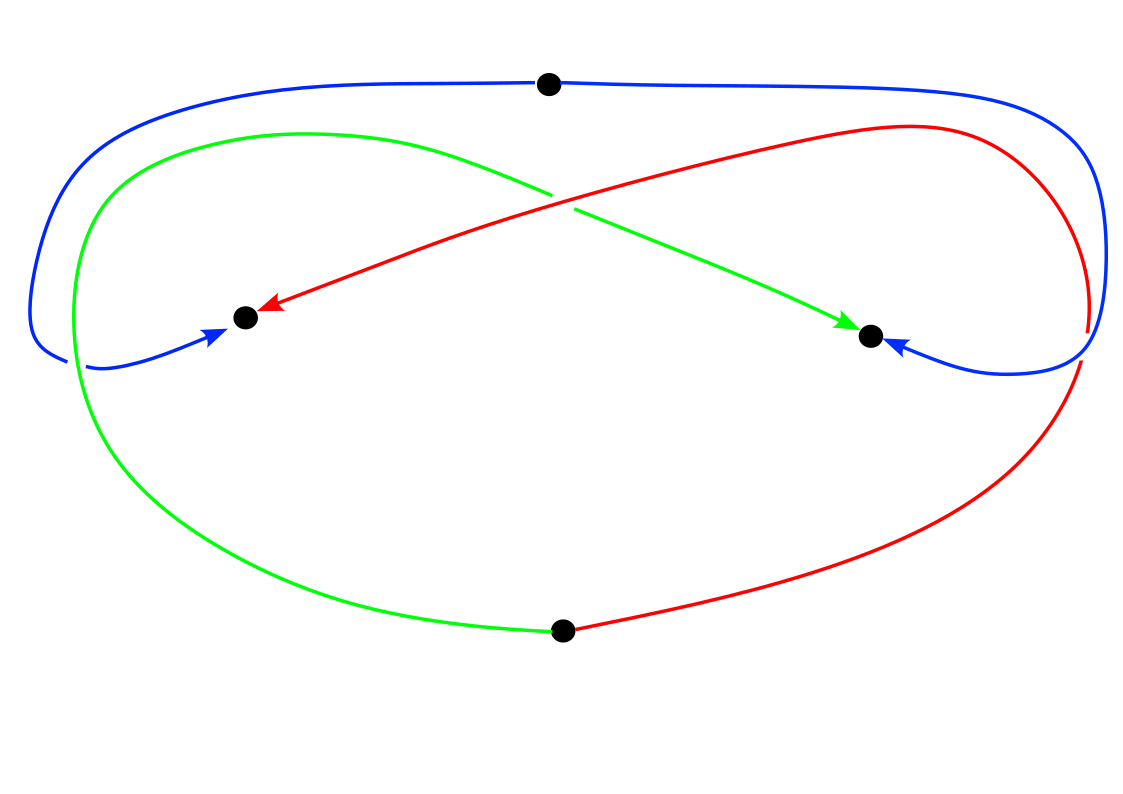}
\put(490,100){$0$}
\put(470,650){$\infty$}
\put(750,370){$p^+$}
\put(200,370){$p^-$}
\end{overpic}
\caption{\textit{A heteroclinic trefoil knot in the Lorenz system. In our setting, $p^\pm$ would always be either sinks or saddle foci, while $\infty$ and $0$ would always be a repeller and a saddle (respectively).}}
  \label{heteroclo}
    \end{figure}
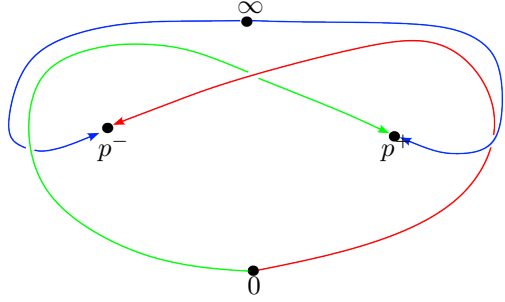
    
We are now finally ready to prove the results of this section. We first prove an analogue to Th.\ref{trefoilor} and Th.\ref{trefoil1}, using Th.\ref{basis} and Th.\ref{orbipers}. We begin with the following fact: 
\begin{theorem}
    \label{trefoil2}
   Assume $F$ is a $C^k,k\geq3$ vector field of $S^3$ which generates a heteroclinic trefoil knot $H$ as in Fig.\ref{heteroclo}, and moreover, assume the fixed points $p^\pm$ are either sinks, real saddles, or saddle foci. Then, every knot type on the Lorenz Template (save possibly for two) is realized as a periodic orbit for $F$ - in particular, $F$ generates infinitely many periodic orbits of infinitely many knot types.
   \end{theorem}

\begin{proof}
We begin by considering a Lorenz system corresponding to some parameter $p\in P$ at which the flow satisfies the assumptions (and conclusions) of Th.\ref{basis} (by Th.\ref{basis}, such a parameter exists). Let us denote the vector field generating the said system by $L$ (see Eq.\ref{Vect11}) - it is easy to see $L$ is a smooth vector field of $\mathbf{R}^3$. Moreover, by Th.\ref{basis} we know every knot type on the Lorenz Template (save possibly for two) is realized as a periodic orbit for $L$. Similarly to the proof of Th.\ref{trefoil1} we would like to apply Th.\ref{orbipers} to prove there exists a smooth vector field $F$ of $S^3$ as indicated above, whose dynamics are complex - which, by Th.\ref{orbipers}, would make them persistent $rel$ $H$. Intuitively, by Th.\ref{basis} we would expect the said vector field to be $L$ - however, since $L$ is not necessarily a smooth vector field of $S^3$ it is not immediate this argument actually holds. In more detail, there one major obstacle we must deal with in order for this reasoning to work: namely, even though $\infty$ is a repelling fixed point for the flow, the vector field $L$ does not necessarily extend smoothly to $\infty$.\\
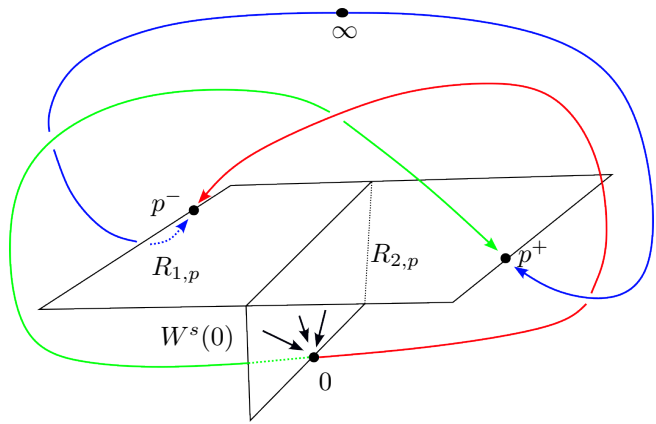
\begin{figure}[h]
\centering
\begin{overpic}[width=0.5\textwidth]{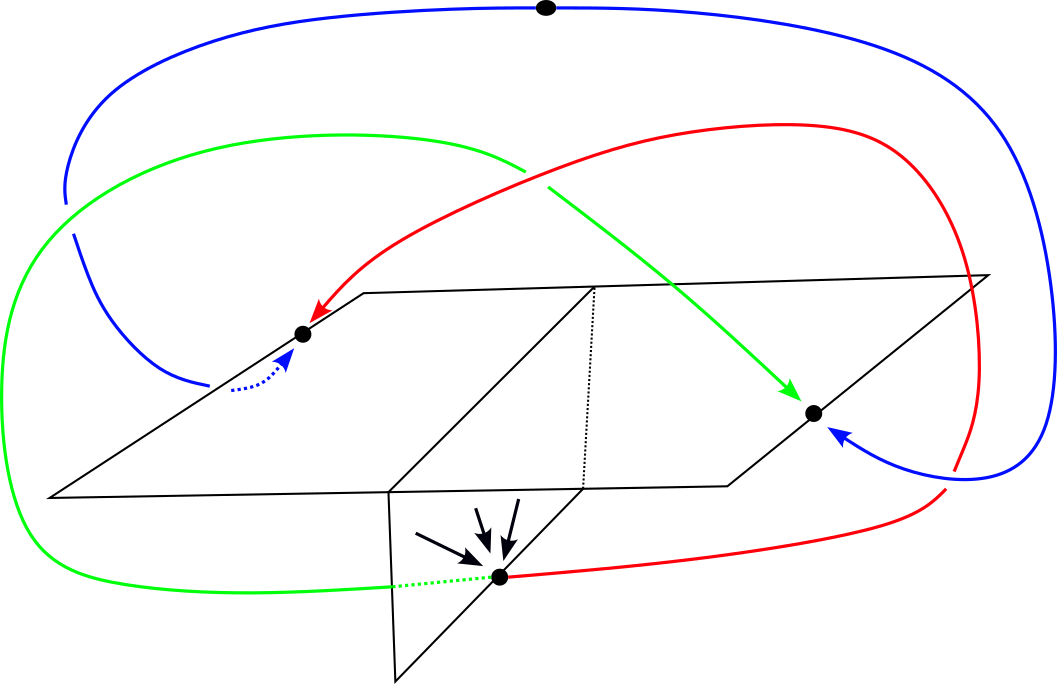}
\put(220,230){$R_{1,p}$}
\put(500,600){$\infty$}
\put(220,120){ $W^s(0)$}
\put(220,330){$p^-$}
\put(790,250){$p^+$}
\put(480,50){$0$}
\put(560,250){$R_{2,p}$}
\end{overpic}
\caption{\textit{The heteroclinic knot $H$ superimposed on the cross-section $R$.}}
\label{lorenz2}
\end{figure}

To overcome this obstacle, recall $L$ extends continuously to $\infty$ and that after this extension $\infty$ is a repelling fixed point for $L$. Consequentially, the dynamics given by Th.\ref{basis} and Prop.\ref{talilemma} all lie in some compact set $K$ which lies away from $\infty$ - which allows us to replace $L$ with $L'$, a smooth vector field of $S^3$ s.t. $L$ and $L'$ coincide on $K$. In particular, we choose $L'$ s.t. it generates the same heteroclinic knot $H$ as $L$ - see the illustration in Fig.\ref{lorenz2}. As such, by Prop.\ref{talilemma} and Th.\ref{basis} we conclude the dynamics of $L'$ inside $K$ satisfy the following (see the illustration in Fig.\ref{lorenz2}):

\begin{itemize}
    \item $L'$ has three fixed points - a real saddle $0$ with a stable manifold $W^s(0)$, and two fixed points of the same type $p^\pm$.
    \item There exists a cross-section, a topological rectangle $R$ s.t. $p^\pm$ lie on its boundary and $R\setminus W^s(0)$ is composed of two sub-domains $R_{1,p}$ and $R_{2,p}$. Moreover, $R$ is universal hence intersects transversely with all the periodic orbits of $L'$.
    \item $0$ generates two heteroclinic orbit, connecting it with $p^{\pm}$ as in Fig.\ref{lorenz2}. In addition, the first-return map $\psi_p:R_{1,p}\cup R_{2,p}\to R$ is well-defined and continuous - and consequentially, appears as in Fig.\ref{RECT}.
\end{itemize}

\begin{figure}[h]
\centering
\begin{overpic}[width=0.35\textwidth]{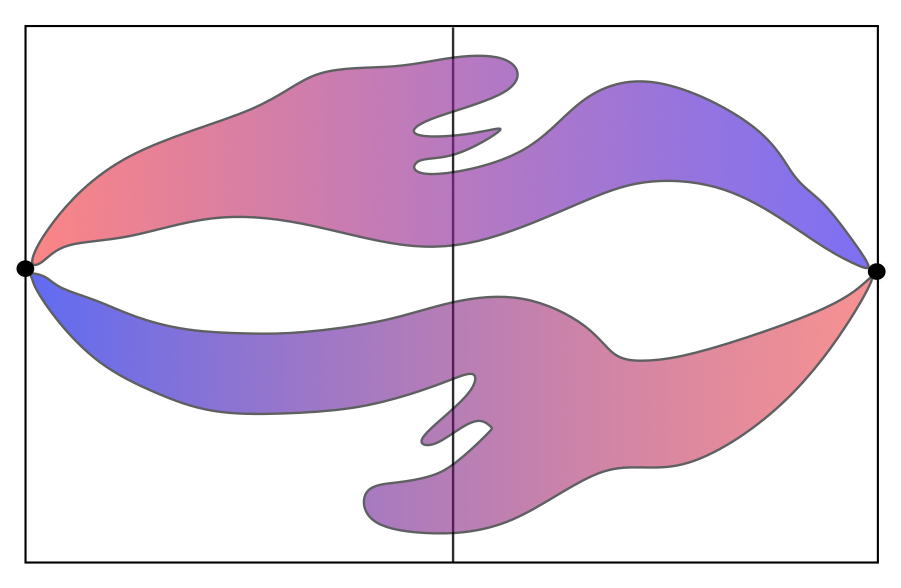}
\put(260,70){$R_{1,p}$}
\put(800,570){$R_{2,p}$}
\put(365,340){ $W^s(0)$}
\put(-65,350){$p^-$}
\put(1000,350){$p^+$}
\put(200,460){$\psi_p(R_{1,p})$}
\put(620,200){$\psi_p(R_{2,p})$}
\end{overpic}
\caption{\textit{The cross-section $R$ with the first-return map for $L'$ and $L''$ defined on it.}}
\label{RECT}
\end{figure}
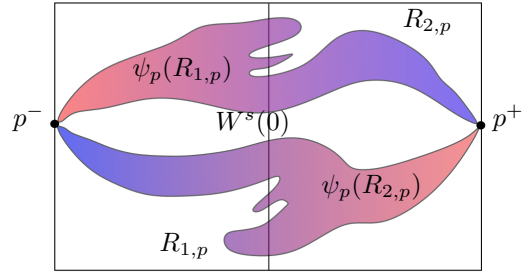

Having constructed $L''$, our second goal is to deform $L''$ $rel$ $H$ to some vector field s.t. we can apply Th.\ref{orbipers} - provided we manage that, Th.\ref{trefoil2} would follow. We first show we can deform $L''$ $rel$ $H$ to a vector field s.t. the Orbit Index of each one of its periodic orbits is $-1$. To do that, we smoothly deform $L''$ to $L'''$ by moving flow lines inside $K$, as described in Fig.\ref{deform3} - that is, we "straighten" the first return map $\psi_p:R_{1,p}\cup R_{2,p}\to R$ by an isotopy to $\varphi:R_{1,p}\cup R_{2,p}\to R$ s.t. if $I$ its maximal invariant set of $\varphi$ in $R\setminus W^s(0)$ there exists a homeomorphism $\pi:I^\mathbf{Z}\to\{1,2\}^\mathbf{Z}$ s.t. $\pi\circ\varphi\circ\pi^{-1}=\sigma$ (where $\sigma:\{1,2\}^\mathbf{Z}\to\{1,2\}^\mathbf{Z}$ is the double-sided shift). It is easy to see that after this straightening deformation all the periodic orbits of $L'''$ intersect $I$. In addition, note that by opening up the fixed points $p^\pm$ and $\cup_{n\geq0}\varphi^{-n}(W^s(0))$ to intervals and rectangles (respectively) we can conjugate $\varphi$ on its maximal invariant set in $R\setminus W^s(0)$ to that of the fake horseshoe $f:abcd\to\mathbf{R}^2$ (see the illustration in Fig.\ref{fake}).\\

\begin{figure}[h]
        \centering
\begin{overpic}[width=0.7\textwidth]{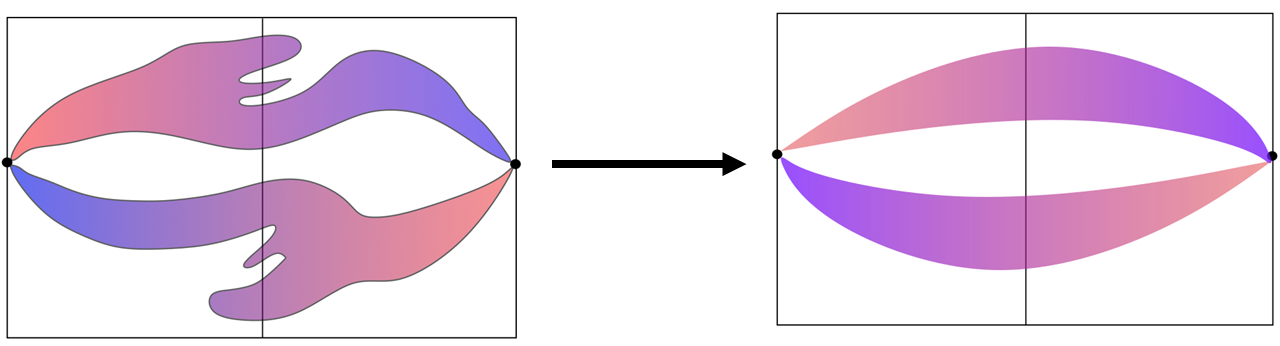}
\end{overpic}
\caption{\textit{The first-return map for $L''$ is on the left, while that for $L'''$ is on the right.}}
  \label{deform3}
    \end{figure}

To continue, note that at every point in the invariant set of $f$ in $abcd$ the differential has two positive eigenvalues and that all the periodic orbits are saddles - consequentially, the Orbit Index of every differential orbit for $L'''$ is $-1$. Moreover, as shown in the proof of Th.3.1 in \cite{Pi} the deformation of $L$ to $L'''$ induces an isotopy on the first-return map, s.t. $\varphi:R_{1,p}\cup R_{2,p}\to R$ can be embedded as a part of a fake horseshoe map inside some Pseudo-Anosov map $P:\mathbf{T^*}\to\mathbf{T^*}$  - where $\mathbf{T^*}$ is a punctured Torus (in particular, in the notations of Th.\ref{orbipers} we have $S_1=R_{1,p}\cup R_{2,p}$ and $S_2=R$). As such, by Th.\ref{orbipers} we conclude that for $F-L'$ (where $L'$ is as above) $Ess(F)$ w.r.t. $H$ is infinite.\\

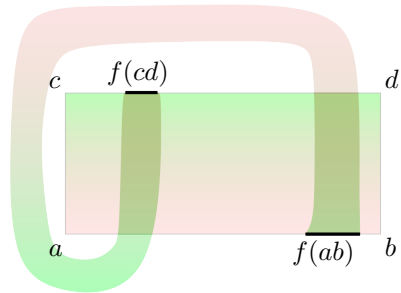
\begin{figure}[h]
\centering
\begin{overpic}[width=0.3\textwidth]{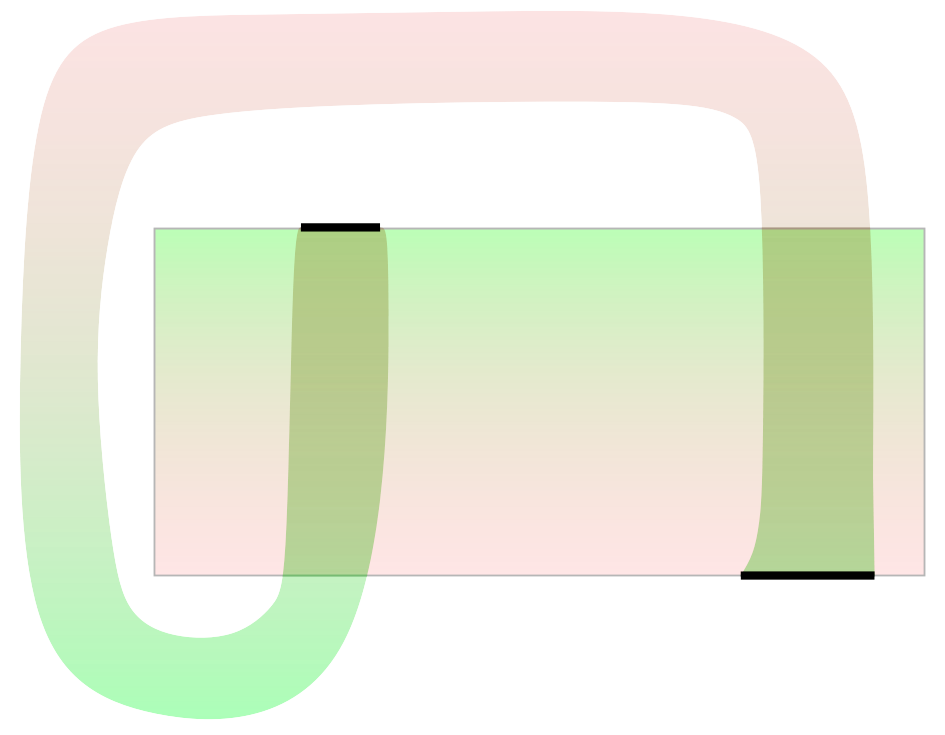}
\put(120,550){$c$}
\put(120,110){$a$}
\put(990,550){$d$}
\put(990,110){$b$}
\put(270,570){$f(cd)$}
\put(750,100){$f(ab)$}
\end{overpic}
\caption{\textit{The fake horseshoe map. $p^+$ corresponds to an arc on the $ab$ side, while $p^-$ corresponds to an arc on the $cd$ side.}}
\label{fake}
\end{figure}

To conclude the proof it remains to show the periodic orbits in $Ess(F)$ given by the proccess above are all knot types on the Lorenz Template. To do so, recall Th.3.1, Th.1.2 and Cor.1.3 in \cite{Pi}, where the following was proven:
\begin{claim}
    \label{talitheorem} Assume $F'$ is a smooth vector field of $S^3$ s.t. the following is satisfied:
    \begin{enumerate}
        \item $F'$ generates a heteroclinic knot configured like in Fig.\ref{heteroclo}.
        \item $F'$ generates a cross-section $R$ as in Fig.\ref{RECT1} and a continuous first-return map as in Fig.\ref{RECT}.
    \end{enumerate}

    Then, every knot type on the Lorenz Template (see Fig.\ref{TEMP2}) is realized as a periodic orbit for $F'$. 
\end{claim}

It is easy to see $L'''$ satisfies the assumptions of Th.\ref{talitheorem}, hence all its periodic orbits are knots on the Lorenz Template. And since the Lorenz Template includes infinitely many knot types (see Cor.3.1.14 in \cite{KNOTBOOK}) it follows $L'''$ also generates periodic orbits of infinitely many knot types. Finally, by Th.\ref{orbipers} the same is true for any smooth vector field of $S^3$ which can be deformed to $L'''$ $rel$ $H$- for example, $L$ - and setting $F=L$ the proof of Th.\ref{trefoil2} is now complete.\end{proof}

Having proven an analogue of Th.\ref{trefoil1} for the Lorenz system we conclude this section by showing another application of the Orbit Index Theory - and this time we do so to study the persistence of periodic dynamics when the heteroclinic knot $H$ breaks down. The reason we do is because the assumption a given Lorenz system generates a heteroclinic knot $H$ as in Fig.\ref{RECT1} is a very strong assumption - in the sense that it is highly non-generic. This leads us to ask the following - just how much of the complexity associated with heteroclinic dynamics persists under small $C^k$ perturbations of the the flow? In Th.\ref{pers13} in the previous section we gave a generic answer to this question - and now, using the unique properties of Eq.\ref{Vect11} we give a precise answer for the Lorenz system. In more detail, we prove:

\begin{theorem}
    \label{persistence} Assume $p\in P$ is a parameter at which the following is satisfied:
    \begin{enumerate}
        \item The corresponding Lorenz system generates a heteroclinic knot $H$ as in Fig.\ref{heteroclo}.
        \item There exists a cross-section $R$ configured with $H$ as in Fig.\ref{RECT1}.
    \end{enumerate}
    
Now, let $T_s$ be a periodic orbit given by Th.\ref{basis} - then, for any sufficiently small $C^k$-perturbation of $p$, $k\geq3$, the periodic orbit $T_s$ persists as a periodic orbit as the Lorenz system is perturbed from $p$ to $v$. Moreover, $T_s$ persists without changing its knot type.
\end{theorem}

\begin{proof}

The proof of Th.\ref{persistence} will be based on the following heuristic: since by Th.\ref{basis} the dynamics of the Lorenz system at $p$ are essentially those of a the Geometric Lorenz attractor, they are also essentially hyperbolic - and since hyperbolic dynamics persist under sufficiently small smooth perturbations we would expect any periodic orbit $T$ to persist under small perturbations in the parameter space $P$. In practice, this heuristic cannot be implemented as is, if only because there is no good reason to apriori assume the dynamics of the Lorenz system are actually hyperbolic - however, as we will soon show, by carefully applying the Orbit Index Theory as developed in Sect.\ref{orbitin} we can bypass such obstacles, thus guaranteeing the persistence of periodic dynamics. Before beginning we remark that unlike the previous section, this time our major tools would be Def.\ref{index22}, Th.\ref{contith} and Def.\ref{globalconti}, as we will work in a more general setting.\\

To begin, recall we denote the vector field corresponding to the Lorenz system at $p$ by $L_p$. We first recall that by $p\in P$ there exists a universal cross-section $R$ for $L_p$ which intersects transversely with $W^s(0)$ - the two-dimensional stable manifold of the origin (see Prop. \ref{talilemma}). In particular, there exists a curve $\gamma\subseteq R\cap W^s(0)$ s.t. $R\setminus \gamma=R_{1,p}\cup R_{2,p}$ where $R_{1,p}$ and $R_{2,p}$ are two (topological) rectangles, and the first-return map $\psi_p:R_{1,p}\cup R_{2,p}\to R$ is continuous and well-defined (see Prop.\ref{talilemma} and the illustrations in Fig.\ref{lorenz2} and Fig.\ref{rect2}). Moreover, recall that as shown in the proof of Th.\ref{talitheorem} in \cite{Pi}, we can deform $L_p$ to $F$ s.t. $F$ is a vector field satisfying the following:

\begin{itemize}
    \item The periodic orbits of $F$ are in one-to-one correspondence with those on the Lorenz Template.  
    \item Save for two periodic orbits, when we deform $F$ back to $L_p$, the periodic orbits of $F$ are deformed to those of $L_p$ by an ambient isotopy (see Def.\ref{knot} and the illustration in Fig.\ref{rect2}) - as such, they have the same knot type.
    \item The two periodic orbits for $F$ which do not persist as we return to $L_p$ are closed to the fixed points $p^\pm$ by Hopf bifurcations (see the illustration in Fig.\ref{rect2}). 
\end{itemize}

\begin{figure}[h]
        \centering
\begin{overpic}[width=0.7\textwidth]{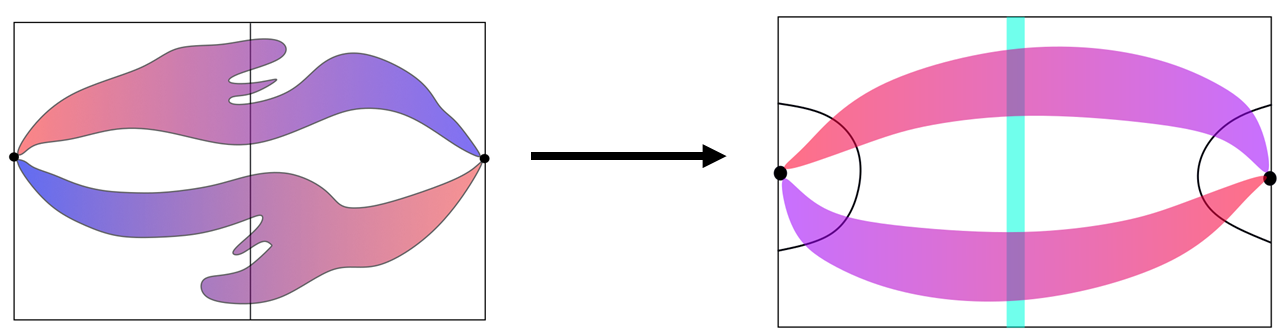}
\end{overpic}
\caption{\textit{The deformation per \cite{Pi} of the first-return map for $L_p$ (on the left) to that for $F$ (on the right) - note we expand $p^\pm$ by Hopf bifurcations to two sinks. $W^s(0)$ is deformed to the cyan rectangle denotes a region which either flows towards the basin of attraction for the sinks, or never returns to the cross-section.}}
  \label{rect2}
    \end{figure}

In more detail, let $R$ be the cross-section given by Prop.\ref{talilemma} - as shown in the proof of Th.1.2 in \cite{Pi}, we deform $L_p$ to $F$ by "straightening" the first-return map for $L_p$, $\psi_p:R_{1,p}\cup R_{2,p}\to R$ into a map $f:R_{1,p}\cup R_{2,p}\to R$, s.t. $f$ is hyperbolic on its invariant set in $R_{1,p}\cup R_{2,p}=R\setminus W^s(0)$ - which we do by expanding the fixed points $p^\pm$ by Hopf bifurcations (for more details, see the proofs of Th.1.2 and Th.3.2 in \cite{Pi}). Consequentially, setting $I\subseteq R$ as this invariant set there exists a homeomorphism $\pi:I\to\{1,2\}^\mathbf{Z}$ s.t. $\pi\circ f=\sigma\circ\pi$ where $\sigma:\{1,2\}^\mathbf{Z}\to\{1,2\}^\mathbf{Z}$ is the double-sided shift - in particular, $f$ is conjugate to the fake horseshoe map. As such, given $s\in\{1,2\}^\mathbf{Z}$, a periodic symbol of minimal period $k>1$, it follows $\pi^{-1}(s)$ forms a singleton, $\{x_s\}$ - and since there is only a finite number of periodic symbols in $\{1,2\}^\mathbf{Z}$ of minimal period $k$ it follows there exists a neighborhood $N$ of $x_s$ in $R$ s.t. the following holds:

\begin{enumerate}
    \item $N$ is a topological disc.
    \item For every $1\leq i<k$, $f^i(N)\cap N=\emptyset$.
    \item The boundary of $N$ in $R$ is composed of arcs on $W^s(0)\cap R$ and the forward iterates of $\partial R_{i,p}$, $i=1,2$ under $f$.
\end{enumerate}

Since neither $W^s(0)$ nor $\partial R_{i,p}$, $i=1,2$ include periodic orbits for the flow it is easy to see there are no fixed points for $f^k$ on $\partial N$ - consequentially, by $\{y\in N|f^k(y)=y\}=\{x_s\}$ it follows the Fixed Point Index of $f^k$ on $N$ is $-1$. Now, note that as we vary $F$ back to $L_p$ both the boundary of $N$ on $R$ and the set $W^s(0)\cap R$ change continuously as well. It is easy to see that as $W^s(0)$ is a stable manifold there are no periodic orbits for the flow on it at every stage of the deformation - similarly, as the flow always maps $\partial R\setminus W^s(0)$ inside the cross-section $R$ (save for the fixed points), no periodic orbits can emerge on the forward iterates of $\partial R_{i,p}$, $i=1,2$ under the first-return map either. As such, it follows that as we smoothly deform $F$ to $L_p$ no periodic orbits are added in $\partial N$ - and therefore, by Th.\ref{lefschetztheorem} and the discussion above we conclude the following is satisfied:

\begin{enumerate}
    \item The boundary of $N$ in $R$ is composed from a finite number of arcs on $W^s(0)\cap R$ and the forward iterates of $\partial R_{i,p}$, $i=1,2$ under $\psi_p$.
    \item For every $1\leq i<k$, $\psi_p^i(N)\cap N=\emptyset$.
    \item The Fixed Point Index of $\psi_p^k$ on $N$ is $-1$.
\end{enumerate}

In particular, it follows the tube of flow lines connecting $N$ to $f^k(N)$ is deformed by a smooth isotopy of $\mathbf{R}^3$ to the tube of flow lines connecting $N$ to $\psi^k_p(N)$ (see the illustration in Fig.\ref{tube}) - and moreover, it is easy to see that these two tubes of flow lines are knotted with themselves in the exact same way. Consequentially, if $y_s\in N$ is s.t. $\psi^k_p(y_s)=y_s$ the knot type generated by suspending $x_s$ w.r.t. $F$ and the knot type generated by suspending $y_s$ w.r.t. $L_p$ are the same. Moreover, since the Fixed Point Index of $\psi^k_p$ in $N$ is $-1$, such a $y_s$ exists - and by the discussion above we may assume $x_s$ is continuously deformed to some $y_s$.\\
 
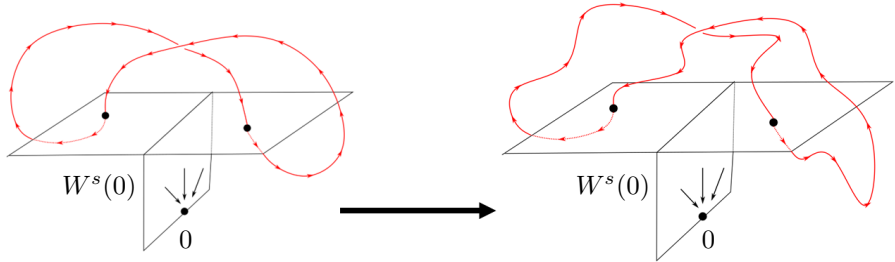
\begin{figure}[h]
\centering
\begin{overpic}[width=0.7\textwidth]{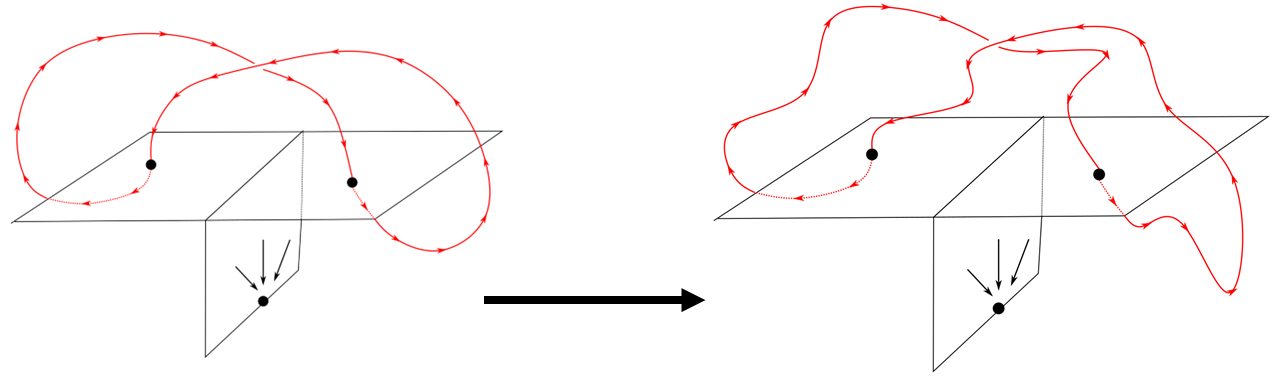}
\put(200,20){$0$}
\put(70,80){$W^s(0)$}
\put(640,80){$W^s(0)$}
\put(780,20){$0$}
\end{overpic}
\caption{\textit{The persistence of periodic orbits - due to the transverse intersection with the cross-section, they persist for sufficiently small perturbations in the parameter space.}}
\label{pert2}
\end{figure}

We are now ready to conclude the proof of Th.\ref{persistence}. To do so, let $T$ denote the periodic orbit for $F$ s.t. $x_s\in T$, and let $T_s$ denote the periodic orbit s.t. $y_s\in T_s$ - as $x_s$ is deformed to $y_s$ we know $T$ is smoothly deformed to $T_s$ as $F$ is deformed to $L_p$. Moreover, note that since $f$ is conjugate to the fake horseshoe map, using a similar argument to the one used in the proof of Th.\ref{trefoil2} we know $i(T)=-1$ (where $i$ is the Orbit Index - see Def.\ref{index22}). In addition, from now on we denote by $\{L_t\}_{t\in[0,\frac{1}{2}]}$ the curve of vector fields deforming smoothly $F=L_0$ to $L_p=L_{\frac{1}{2}}$ in $\mathbf{R}^3\times[0,\frac{1}{2}]$ - and moreover, $Per$ will always denote the collection of periodic orbits for $\{L_t\}_{t\in[0,\frac{1}{2}]}$ (i.e., for every $t\in[0,\frac{1}{2}]$ all components of $Per_s\cap \mathbf{R}^3\times\{t\}$ are periodic orbits for $L_t$).\\

We first prove Th.\ref{persistence} under the generic assumption that $T_s$ is an isolated periodic orbit for $L_p$ i.e., that there exists some local cross-section $S$ transverse to $T_s$ s.t. if $f':S\to S$ is some local first-return map, $\{y\in S|f(y)=y\}=T_s\cap S$ is a singleton (note we may well assume $S\subseteq N$ - see the illustration in Fig.\ref{lorenz1}). It is easy to see that for such a generic choice of $L_p$ we can choose the curve $\{L_t\}_{t\in[0,\frac{1}{2}]}$ s.t. the continuum of periodic orbits connecting $T$ and $T_s$ cannot be approximated by low-period orbit. That is, if $Per_s\subseteq Per$ is the component of periodic orbits in $R^3\times[0,1]$ s.t. $T\times\{0\},T_s\times\{\frac{1}{2}\}\subseteq Per_s$, then every periodic orbit on $Per_s$ would also be isolated as in the scenario described in Def.\ref{dfT}. Consequentially, by $i(T)=-1$ and by the invariance of the Orbit Index (see Th.\ref{invar}) it follows the Orbit Index is constant on $Per_s$, i.e., on every periodic orbit in $Per_s$ the Orbit Index is $-1$. As such, since $T_s$ is isolated by Def.\ref{index22} it follows $i(T_s)=-1$ - which implies the Fixed Point Index for $\psi^k_p$ on some neighborhood of $\{y_s\}=T_s\cap N$ is $-1$.\\

\begin{figure}[h]
\centering
\begin{overpic}[width=0.35\textwidth]{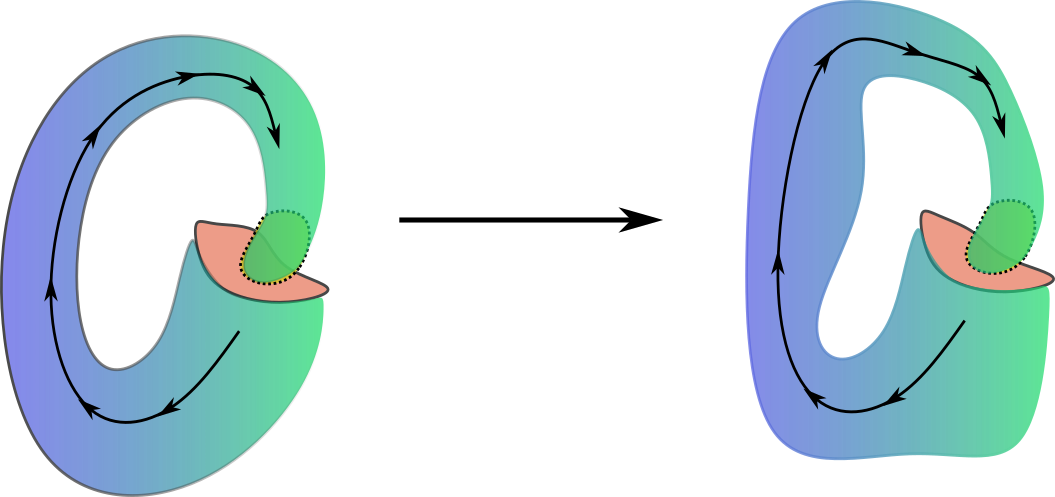}
\put(-50,360){$\Theta_p$}
\put(1000,360){$\Theta$}
\end{overpic}
\caption{\textit{The tubes of flow lines $\Theta_p$ and $\Theta$ (for the vector fields $L_p$ and $L$ respectively) connecting $N$ (the red cross-section) to itself. It is easy to see that provided $L$ is sufficiently $C^k-$close to $L_p$ the tubes $\Theta_p$ and $\Theta$ are both isotopic, bounded, and knotted in the same way.}}
\label{tube}
\end{figure}

To conclude the proof of the generic case, assume $L$ is some $C^3$ vector field of $\mathbf{R}^3$ and extend the curve to $[0,1]$, s.t. $\{L_t\}_{t\in[0,1]}$ is a smooth curve of vector fields in $\mathbf{R}^3\times[0,1]$ and $\{L_t\}_{t\in[\frac{1}{2},1]}$ is a curve in the parameter space $P$ connecting $L_p=L_\frac{1}{2}$ to $L=L_1$. Since $i(T_s)=-1$ by Th.\ref{contith} we conclude $T_s$ is globally continuable (see Def.\ref{globalconti}), hence persists when we perturb $L_p$ to any $L$ which is sufficiently $C^k$- close to $L_p$, $k\geq3$ - or in other words, we have proven that provided $L$ is sufficiently $C^k-$close to $L_p$, $T_s$ persists as we perturb $L_p$ to $L$.\\

Therefore, to conclude the proof of the generic case it remains to show that provided $L$ s sufficiently $C^k$-close $T_s$ persists without changing its knot type. To see why this is so, note that if that were not the case $L_p$ is a bifurcation set for $T_s$ (w.r.t. $\{L_t\}_{t\in[0,1]}$). Since $T_s$ is isolated we know this implies the matrix $D(T_s)$ (see Def.\ref{dfT}) must have a multiplier on $S^1$ - and since $T_s$ can be approximated by periodic orbits on $Per_s\cap\mathbf{R}^3\times[0,\frac{1}{2})$ whose Orbit Index is $-1$ by Def.\ref{index1} we conclude that multiplier can only be $1$. However, since the Fixed Point Index around $\{y_s\}=T_s\cap N$ is given by the sign $det(D(T_s)-Id)$ this would imply the Fixed Point Index at $y_s$ is $0$ - contrary to it being $-1$ (as shown above). This implies that as we vary $L_p$ to $L$ in the generic case, provided $L$ is sufficiently $C^k-$ close to $L_p$ the orbit $T_s$ persists without changing its knot type as well. All in all, the proof of Th.\ref{persistence} for the generic case is complete.\\

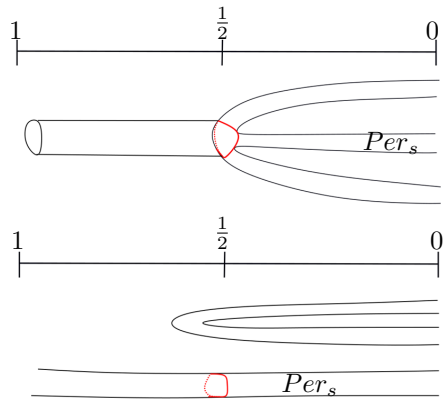
\begin{figure}[h]
\centering
\begin{overpic}[width=0.35\textwidth]{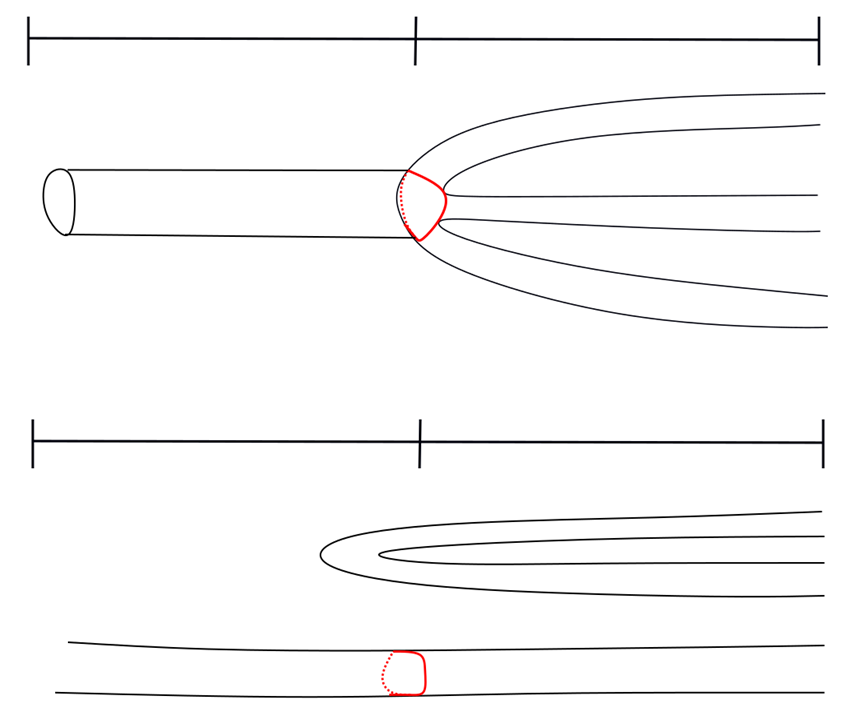}
\put(940,830){$0$}
\put(800,580){$Per_s$}
\put(450,380){ $\frac{1}{2}$}
\put(20,360){$1$}
\put(15,830){$1$}
\put(950,360){$0$}
\put(470,850){$\frac{1}{2}$}
\put(620,40){$Per_s$}
\end{overpic}
\caption{\textit{The non-generic case (above) and the generic case (below) and $per_s$ is each case (he red loop denotes $T_s$ is both scenarios). It is easy to see $T_s$ is isolated in the generic case. It is easy to see the non-generic case can be approximated with the generic one.}}
\label{lorenz1}
\end{figure}

We now prove Th.\ref{persistence} for the general case, i.e., when $L_p$ is no longer generic - or in other words, when the orbit $T_s$ is not assumed to be isolated for $L_p$. We do so using a method of approximation - and to motivate it note that regardless of the behavior of $L_p$ around $T_s$ we can choose the sub-curve $\{L_t\}_{t\in[0,\frac{1}{2}]}$, $L_0=F$, $L_\frac{1}{2}=L_p$ s.t. it is generic in the following sense: given any $t\in[0,\frac{1}{2})$, if $T'$ is a periodic orbit for $L_T$ then it is isolated. Or, in other words, the periodic orbits for the curve $\{L_t\}_{t\in[0,\frac{1}{2}]}$ stop being isolated at most in $L_\frac{1}{2}=L_p$ (see the illustration in Fig.\ref{lorenz1}). To continue, let $L$ be some $C^3$ vector field of $\mathbf{R}^3$ and similarly extend $\{L_t\}_{t\in[0,\frac{1}{2}]}$ to $\{L_t\}_{t\in[0,1]}$, a smooth curve of vector fields in $\mathbf{R}^3\times[0,1]$ s.t. $L_1=L$. This implies that whenever $T_s$ is not isolated for $L_p$ we can homotope the curve $\{L_t\}_{t\in[0,1]}$ into the generic case (i.e. making $T_s$ isolated) by some arbitrarily small $C^k-$perturbation of $\{L_t\}_{t\in[0,1]}$ - which shows we can approximate the curve $\{L_t\}_{t\in[0,1]}$ with the generic case described above.\\

With these ideas in mind, we begin by studying how $T_s$ can be destroyed by perturbation in the generic case. To do so, note that by the global continuability of $T_s$ in the generic case by Th.\ref{contith} and Def.\ref{globalconti} there are precisely three ways in which $T_s$ can terminate under perturbation: either by its period diverging to $\infty$, by a Hopf Bifurcation, or by a Blue Sky catastrophe (see Def.\ref{globalconti} and the discussion immediately following it). Now, let $\Theta_p$ denote the tube of flow lines connecting $N$ to $\psi^k_p$ - provided $v\in P$ is sufficiently close to $p$, as we deform $L_p=L_\frac{1}{2}$ to $L=L_1$ the tube $\Theta_p$ is isotoped to the tube $\Theta$ (which, similarly, is constructed from the flow lines connecting $N$ and its first-return map - see the illustration in Fig.\ref{tube}). It is easy to see that whenever $L$ is sufficiently $C^k$-close to $L_p$ the tubes $\Theta_p$ and $\Theta$ are knotted with themselves in the same way (even in the non-generic case). It is also easy to see that in the generic case, provided the $C^k$ distance between $L_p$ and $L$ is sufficiently small, by $T_s\subseteq \Theta_p$ we know $T_s$ persists as a periodic orbit trapped inside $\Theta$ as $L_p$ is varied to $L$. \\

As such, since $\Theta_p$ is bounded and lies away from the fixed points of $L_p$ (which are all stable) we immediately conclude the same is true for $\Theta$ for any $L$ sufficiently $C^k-$close to $L_p$ - in which case we know that generically, whenever $L$ is sufficiently $C^k-$close to $L_p$, $T_s$ cannot be destroyed by any Hopf bifurcation or a Blue Sky Catastrophe. Similarly, we also know that whenever $L$ is sufficiently $C^k-$close to $L_p$ the orbit $T_s$ cannot be destroyed by colliding with a fixed point (which would turn it into a homoclinic or a heteroclinic orbit). Similarly, by the genericity assumption (i.e., that $T_s$ is isolated) as the Fixed Point Index of any first-return map for an isolating cross-section is $-1$ it follows $T_s$ persists without bifurcating for all sufficiently small perturbations - i.e., that whenever $L$ is sufficiently close to $L_p$ $T_s$ cannot collapse into some bounded aperiodic motion inside $\Theta$. To summarize, provided $L$ is sufficiently $C^k-$ close to $L_p$ such that $\Theta$ remains bounded and lies away from the fixed points, the generically $T_s$ persists as a periodic orbit trapped in $\Theta$ as $L_p$ is perturbed to $L$ - and due to the Fixed Point Index argument above we know that it does so without changing its knot type.\\

\begin{figure}[h]
\centering
\begin{overpic}[width=0.5\textwidth]{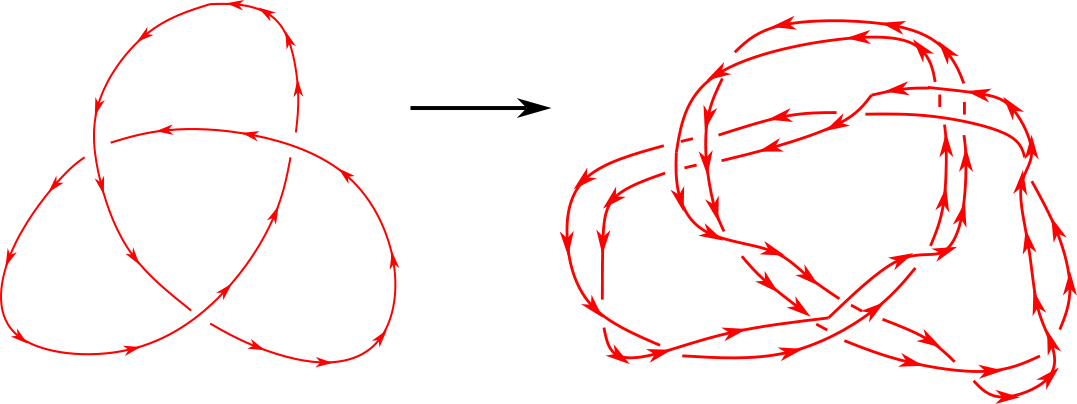}
\end{overpic}
\caption{\textit{A period-multiplying bifurcation for a periodic orbit. In this scenario we specifically see a period-doubling bifurcation.}}
\label{cable}
\end{figure}

We are now ready to use the discussion above to conclude the proof for the case when $L_p$ is non-generic. To do so, consider $L$, some vector field sufficiently close to $L_p$ s.t. as we vary $L_p$ to $L$ the set $\Theta_p$ is isotopically deformed to $\Theta$, s.t. satisfies the following (see the illustration in Fig.\ref{tube}):

\begin{itemize}
    \item $\Theta$ remains bounded away from the fixed points and $\infty$.
    \item $\Theta$ and $\Theta_p$ are knotted in the exact same way.
\end{itemize}

Now, let $\{L^n_t\}_{t\in[0,1]},n>0$ be generic approximations of $\{L_t\}_{t\in[0,1]}$ s.t. $L_0=L^n_0=F$ and for all $t$ we have $d(L^n_t,L_t)\to0$ (we $d$ is the $C^3-$distance in $\mathbf{R}^3$). Moreover, let $Per^n_s$ denote the component of periodic orbits for $\{L^n_t\}_{t\in[0,1]},n>0$ s.t. $T\times\{0\}\subseteq Per^n_s$, and recall we denote by $Per_s$ the component of periodic orbits for $\{L_t\}_{t\in[0,1]},n>0$ connecting $T\times\{0\}$ and $T_s\times\{\frac{1}{2}\}$. By the discussion above, as $\Theta$ remains bounded we know that for all sufficiently large $n$ the set $Per^n_s$ connects $\mathbf{R}^3\times\{0\}$ and $\mathbf{R}^3\times\{1\}$, and that for all $t\in[0,1]$ the intersection $Per^n_s\cap\mathbf{R}^3\times\{t\}$ is a periodic orbit which has the same knot type as $T$ and $T_s$. Moreover, it is also easy to see that as $\{L^n_t\}_{t\in[0,1]}\to \{L_t\}_{t\in[0,1]}$ we also have $\lim_n Per^n_s\to Per'_s$ (as limit sets).\\

We now claim every component of $Per'_s\cap\mathbf{R}^3\times\{t\}$, $t\in[\frac{1}{2},1]$ includes a periodic orbit (by definition, we already know this is the case for $t\in[0,\frac{1}{2})$) - to see why note 
 he first-return map from a cross-section of the tube $\Theta$ w.r.t. $L^n_t$ and $L_t$ is homotopic when the $C^k$-distances are sufficiently small. Therefore, as the Fixed Point Index for the first-return map for these first-return maps for $L^n_t$ is $-1$ and no periodic orbits appear on $\partial \Theta$ we conclude the same is true for the first-return map of $\Theta$ w.r.t. $L_t$ - i.e., $Per_s\cap\mathbf{R}^3\times\{t\}$ includes a periodic orbit for $L_t,t\in[\frac{1}{2},1]$. And since $T$ is varied smoothly along $Per^n_s$ from $F=L^n_0$ to $L^n_1$ it follows there exists a subset on $Per'_s$ which varies $T$ smoothly as $F=L_0$ is smoothly deformed to $L_1$ - and by definition, that set has to include both $T$ and $T_s$, i.e., it is included in $Per_s$. Or, in other words, we have just proven $T_s$ persists as a periodic orbit when $L_p$ is smoothly deformed to $L$.\\

We now conclude the proof by proving the knot type for periodic orbits on $Per_s$ is also constant. To do so recall $\Theta$ and $\Theta_p$ are knotted in precisely the same way. Recalling $\Theta_p$ was generated by suspending the cross-section $N$ with the flow dictated by $L_p$, it follows  $\Theta$ is generated similarly - i.e., by suspending the cross-section $N$ with the flow generated by $L$. This implies the only possibility for $T_s$ to change its knot type as $L_p$ is varied to $L$ is by undergoing some period-multiplying bifurcation which makes it cable with itself as in Fig.\ref{cable} - however since the knot type on every $Per^n_s$ is constant for all the periodic orbits on it (and independent of $n$), it follows no such cabling can occur. Or, in other words, the knot type of $T_s$ does not change as we vary $L_p$ to $L$ - and with this statement the proof of Th.\ref{persistence} is now complete.
\end{proof}
\begin{remark}
    At this point we remark similar arguments to those above can also be applied to the Rössler attractor - that is, given a Rössler system $F$ which generates a heteroclinic knot $H$ as in Th.\ref{trefoil1}, similar arguments to those used above prove each periodic orbit in $Ess(F)$ persists under sufficiently small $C^k$ perturbations of $F$, $k\geq3$.
\end{remark}
\subsection{Removable dynamics inside solid Tori}
\label{horsus}
Having studied when heteroclinic knots $H$ force $Ess(F)$ to be infinite, in this section we study the opposite case - that is, for completeness, in this section we analyze a concrete example in which a heteroclinic knot $H$ forces $Ess(F)$ to be finite. As will be clear, if the flows from Th.\ref{trefoilor} and Th.\ref{trefoil2} are analogous to Pseudo-Anosov maps, the flows discussed in this section will be analogous to periodic maps (see Th.\ref{thurston-nielsen}). In more detail, in this section we consider the scenario where $F$ is a smooth vector field of $S^3$ whose fixed points are two saddle foci, connected by a heteroclinic knot $H$ s.t. $H$ is ambient isotopic to $S^1$ - we will prove that under these assumptions the Essential Class of $F$ in $S^3\setminus H$ can sometimes include at most two periodic orbits. We will often think of $F$ as a flow on a solid torus rather than a flow on $S^3\setminus H$ - where by a "solid torus" we will always mean $\mathbf{T}=\{(x,y)|\sqrt{x^2+y^2<1}\}\times S^1$. The reason we can do so is because given any $F$ as above the flow on $S^3\setminus H$ is orbitally equivalent to a flow on a solid torus (see the illustration in Fig.\ref{s1}).\\
\begin{figure}[h]
\centering
\begin{overpic}[width=0.4\textwidth]{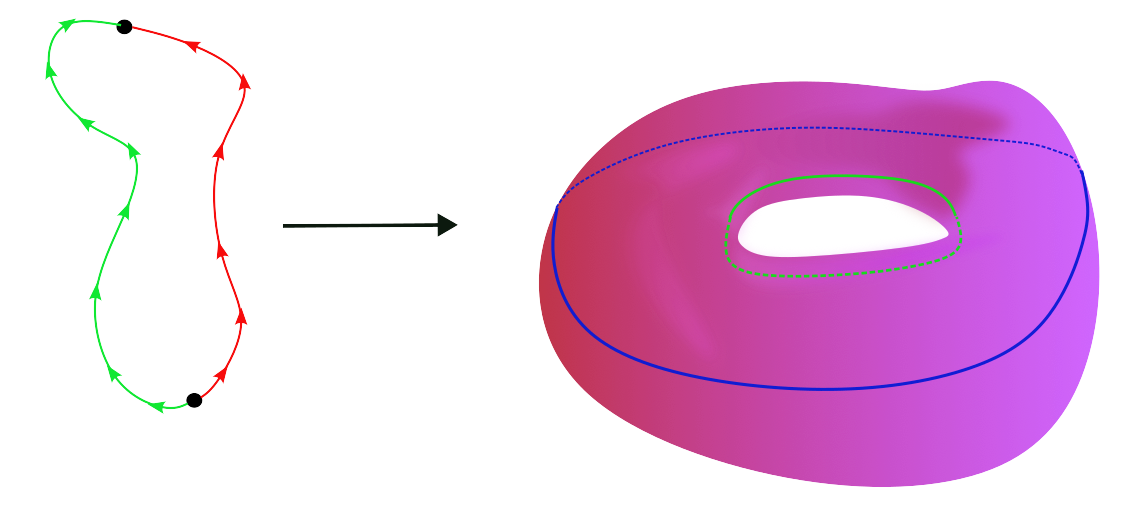}
\end{overpic}
\caption[Fig1]{\textit{A flow on the complement to a heteroclinic knot ambient isotopic to $S^1$ is orbitally equivalent to a flow in the interior of the torus $\mathbf{T}$ to the right (i.e., a solid torus).}}\label{s1}
\end{figure} 

For the sake of completeness, we remark our motivation to study such flows originally arose from the study of the Moore-Spiegel Oscillator:

\begin{equation}
\begin{cases}
\dot{x} = y \\
 \dot{y} = z\\
 \dot{z}=-z-(\tau-\rho+\rho x^2)y-\tau x
\end{cases}
\end{equation}

Where $\tau,\rho>0$. The Moore-Spiegel Oscillator, originally introduced in \cite{SM} to describe the luminosity of stars, is known to have parameter values at which the flow generates precisely one fixed point in $\mathbf{R}^3$ -  moreover, that fixed point is always a saddle-focus, $O$, which lies at the origin (see Sect.III in \cite{SM}). As proven by the author in \cite{I3} at such parameter values the flow generates an unbounded heteroclinic knot $H$ connecting $O$ and a fixed-point at $\infty$ s.t. $H$ is ambient-isotopic to $S^1$ (see Th.4.7 in \cite{I3}). Moreover, as shown in the same Theorem in \cite{I3}, by choosing an arbitrarily small smooth perturbation of the Moore-Spiegel oscillator at $\infty$ one can do two things:
\begin{itemize}
    \item Smoothly transform the Moore-Spiegel Oscillator into $G$, a smooth vector field $F$ of $S^3$ which generates a heteroclinic orbit $H$ connecting two saddle foci - the origin $O$ and $\infty$.
    \item Prove that $G$ has precisely two stable periodic orbits in $S^3\setminus H$, which together attract a.e. initial condition in $S^3\setminus H$.
\end{itemize}

As has been observed numerically, there exist parameter values for the Moore-Spiegel Oscillator in which the flow generates a seemingly chaotic attractor suspended around the saddle focus at the origin (see Fig.\ref{mooreat}). As such, even though none of our results in this section are directly applicable to the Moore-Spiegel Oscillator our results are very much inspired by it, and should be compared to numerical observations of it.\\

\begin{figure}[h]
\centering
\begin{overpic}[width=0.4\textwidth]{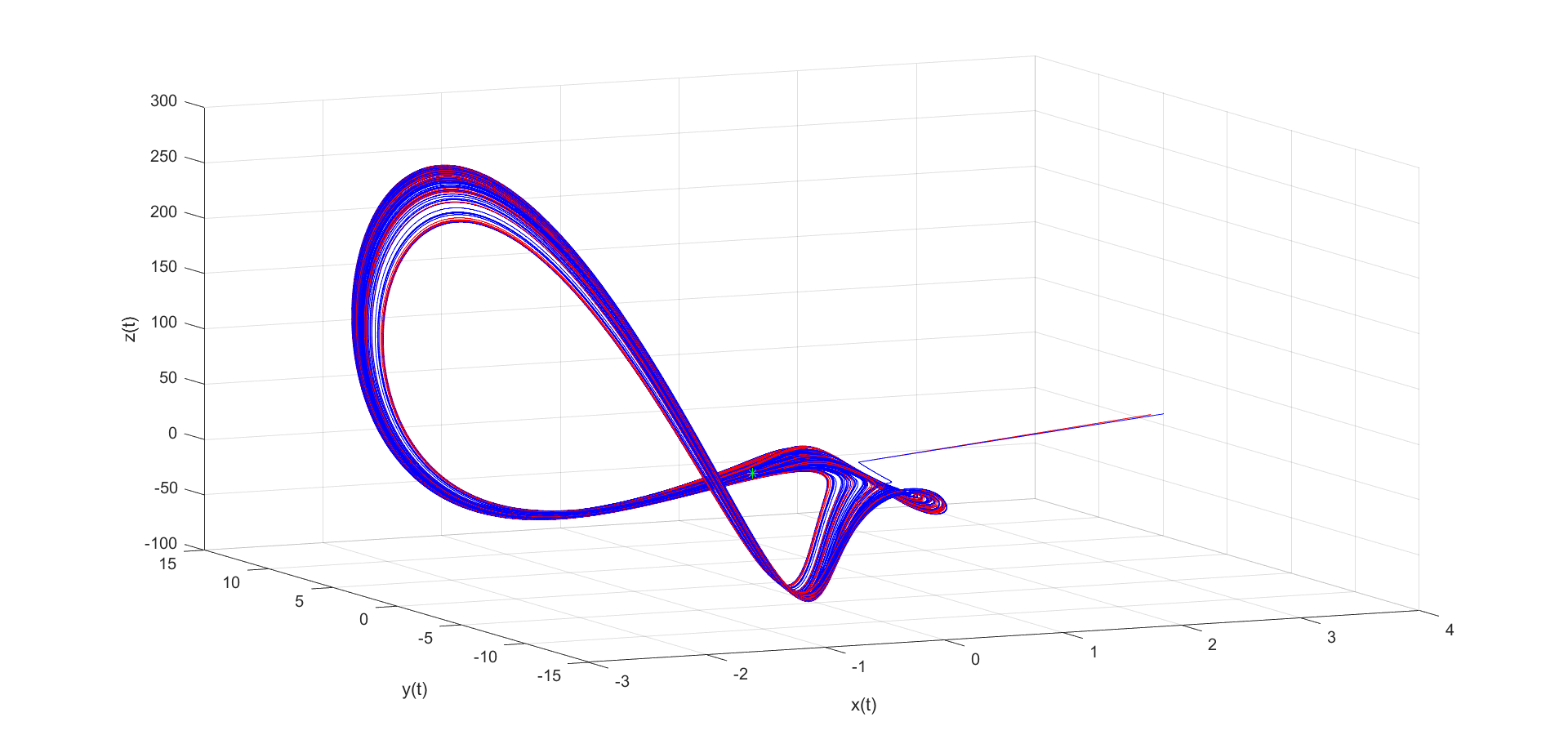}
\end{overpic}
\caption[Fig1]{\textit{The Moore-Spiegel attractor at $(\tau,\rho)=(39.25,100)$.}}
\label{mooreat}
\end{figure}

This section is organized as follows - we first prove Prop.\ref{nohyp} and Th.\ref{nohyp2} which together prove two things: first, that in certain cases hyperbolic dynamics can be easily removed with deformations $rel$ $H$, and second, that if $F$ suspends a Smale Horseshoe map inside $\mathbf{T}$ it cannot do so "too close" to $H$. Following that, we conjecture how our results can be further generalized to whenever the Essential class is finite (see Conjecture \ref{nohyp3}). Finally, we conclude this section by proving the existence of a smooth flow on $\mathbf{T}$ whose essential class is infinite - thus showing a simple topology for the heteroclinic knot $H$ does not automatically translate into simple essential class (see Th.\ref{notrivial}).\\

  \begin{figure}[h]
\centering
\begin{overpic}[width=0.25\textwidth]{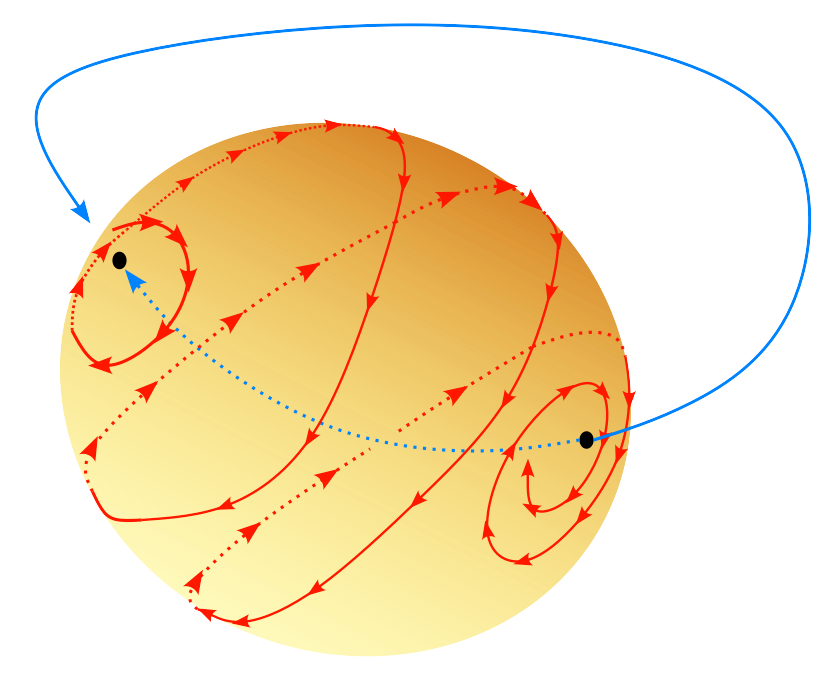}

\put(740,155){$O_2$}
\put(0,480){$O_1$}
\end{overpic}
\caption{\textit{$W$ is the yellow sphere, enclosing inside it one heteroclinic orbit. $H$ is the union of $O_1,O_2$ and the two blue heteroclinic orbit - it is easy to see $H$ is ambient isotopic to $S^1$. The red flow line is a orbit on $W$.}}
\label{coin}
\end{figure}

With these ideas in mind, we begin with the following Proposition which is an immediate consequence of Th.\ref{orbipers}:
\begin{proposition}
\label{nohyp}    There exists a smooth vector field $F$ on $S^3$ satisfying the following:

\begin{enumerate}
    \item $F$ has precisely two fixed points, $O_1$ and $O_2$, both saddle foci.
    \item $O_1$ and $O_2$ are connected by a heteroclinic knot $H$ ambient isotopic to $S^1$ (see the illustration in Fig.\ref{s1}).

    \item $F$ can be smoothly deformed $rel$ $H$ to a vector field $G$ satisfying:
    \begin{itemize}
        \item $G$ has precisely two periodic orbits, $T_1$ and $T_2$.
        \item $T_1$ and $T_2$ together attract every initial condition in $S^3\setminus H$.
    \end{itemize}
    
    \end{enumerate}

Consequently, $Ess(F)$ includes at most two periodic orbits. Moreover, we can choose $F$ s.t. its dynamics include a suspended Smale Horseshoe in $S^3\setminus H$.
\end{proposition}
\begin{proof}
The proof is a much simplified form of an argument originally used by the author in the proofs of Th.4.7 and Th.4.8 in \cite{I3} - for the sake of completeness we reproduce it below. To begin, let $G$ denote a smooth vector field on $S^3$ satisfying the following properties:

\begin{enumerate}
    \item $G$ has two fixed points, $O_1$ and $O_2$, both saddle foci of opposing indices.
    \item $O_1$ and $O_2$ are connected by a heteroclinic knot $H$, ambient isotopic to $S^1$.
    \item The two dimensional invariant manifolds of $O_1$ and $O_2$  coincide as a surface $W$ - i.e., without loss of generality the trajectory of every initial condition on $W$ tends to $O_1$ in backwards time and to $O_2$ in forward time (see the illustration in Fig.\ref{coin}).
    \item Set $\mathbf{T}=S^3\setminus H$ and let $\mathbf{T}_i$ as the components of $\mathbf{T}\setminus W$, $i=1,2$. Then, every initial condition in $\mathbf{T}_i$ is attracted to a periodic orbit $T_i$, $i=1,2$ (see the illustration in Fig.\ref{coin3}). 
\end{enumerate}

\begin{figure}[h]
\centering
\begin{overpic}[width=0.25\textwidth]{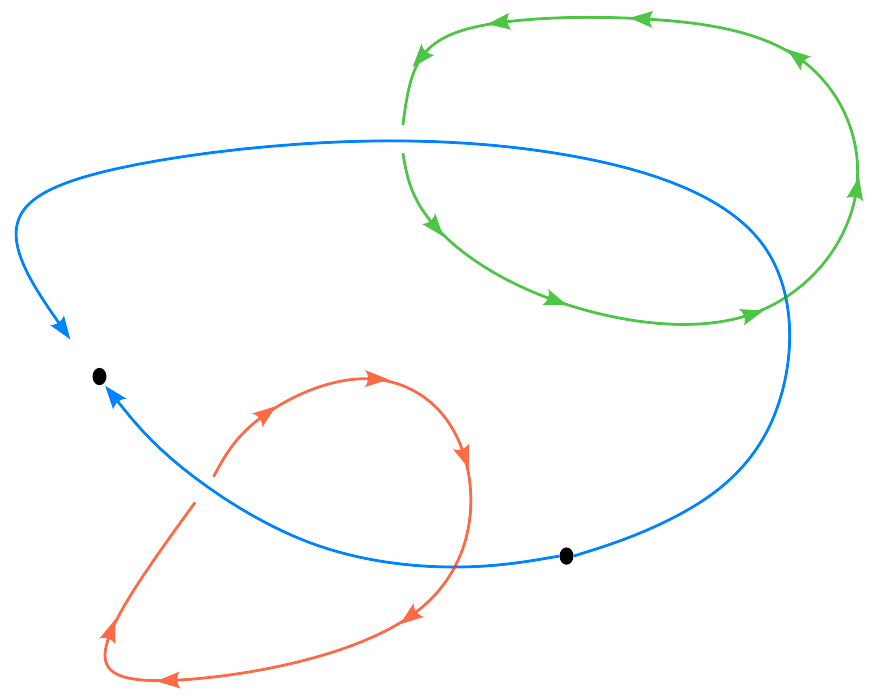}
\put(980,480){$T_1$}
\put(-40,350){$O_1$}
\put(670,90){$O_2$}
\put(410,230){$T_2$}
\end{overpic}
\caption{\textit{The periodic orbits $T_1$ and $T_2$, linked with $H$.}}
\label{coin3}
\end{figure}
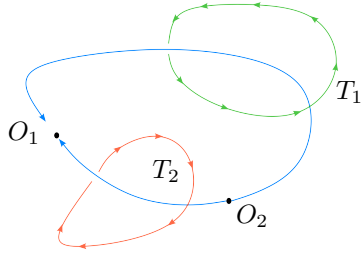

We now continue by smoothly deforming $G$ inside $T$ as follows: we first choose some cross-section $S$ transverse to either $T_1$ or $T_2$, s.t. the first-return map $g:S\to S$ is a diffeomorphism. Following that, we choose some topological rectangle $ABCD\subseteq S$, $ABCD\cap (T_1\cup T_2)=\emptyset$, and begin smoothly moving flow lines s.t. $f:ABCD\to S$ becomes a Smale Horseshoe map. This induces a smooth $rel$ $H$ deformation of $G$ to a vector field $F$ which is easily seen to satisfy the conclusion of Prop.\ref{nohyp} - and by construction, it is immediate that $Ess(F)$ includes at most two periodic orbits, $T_1$ and $T_2$. All in all, the proof of Prop.\ref{nohyp} is now complete. 
\end{proof}
\begin{remark}
    Note that by definition, $S^3\setminus H=\mathbf{T}$ is homeomorphic to an unknotted solid Torus $S^1\times S^1$. Using this observation it is easy to see we can also prove a version of Prop.\ref{nohyp} where $T_1$ and $T_2$ are both ambient isotopic to $S^1$, i.e., the unknot, and linked with $H$ as in Fig.\ref{coin3}.
\end{remark}
Having proven $Ess(F)$ is at most a finite set we now analyze the dynamics of $F$ in $\mathbf{T}=S^3\setminus H$. To illustrate why we do so, note that as shown above, the fact that $Ess(F)$ is finite it does not imply the dynamics of $F$ in $S^3\setminus H$ have to be simple. This raises the following question - given a vector field $F$ as in Prop.\ref{nohyp}, can we say anything about its topological dynamics in $S^3\setminus H$? To give a partial answer to this question we first recall the notion of Dominated Splitting, a weaker form of hyperbolicity originally introduced in \cite{Man}:

\begin{definition}
    \label{dominated}
    Let $M$ be a Riemannian manifold and let $\phi_t:M\rightarrow M, t\in\mathbf{R}$ be a smooth flow. A compact, $\phi_t$-invariant set $\Lambda$ is said to satisfy a \textbf{dominated splitting} \textbf{condition} if the following conditions are satisfied:
\begin{itemize}
    \item The tangent bundle satisfies $T\Lambda=S\oplus U$, s.t. $S=\cup_{x\in\Lambda}S(x)\times\{x\}$, $U=\cup_{x\in\Lambda}U(x)\times\{x\}$, where $T_x M= S(x)\oplus U(x)$ and $S(x)$ denotes the stable directions while $U(x)$ the unstable directions.
    \item There exists some $ c>0,0<\lambda<1$, s.t. for all $ t>0$ and every $x\in\Lambda$, $(||D\phi_t|_{U(x)}||)(||D\phi_{-t}|_{S(\phi_t(x))}||)<c\lambda^t$.
    \item   For all $x\in \Lambda$, $S(x),U(x)$ vary smoothly with $t\in\mathbf{R}$ to $S(\phi_t(x)),U(\phi_t(x))$.
\end{itemize}
\end{definition}
With these ideas in mind, we now prove:
\begin{theorem}
    \label{nohyp2} Assume $F$ is a smooth vector field of $S^3$ which generates two fixed points, $O_1$ and $O_2$, both saddle-foci connected by a heteroclinic knot $H$ as illustrated in Fig.\ref{s1}. Then there is no compact, connected $F-$invariant set $\Lambda\subseteq S^3$ that includes the fixed points on which the vector field $F$ satisfies a dominated splitting condition. 
\end{theorem}
\begin{proof}
We prove Th.\ref{nohyp2} by contradiction. To do so assume there exists an invariant set $\Lambda$ as above which includes the fixed points for $F$ - by the closure of $\Lambda$ and the continuity of the flow it follows $H\subseteq\Lambda$ Moreover, as we assume by contradiction we can decompose $T\Lambda=S\oplus U$ as in Def.\ref{singular}  it follows that at each fixed point $p\in\Lambda$ for $F$ we can write $T_pS^3=S(p)\oplus U(p)$ - where $U(p)$ corresponds to the unstable directions and $S(p)$ to the stable directions.\\

By assumption we know $F$ has precisely two fixed points in $S^3$, $p_1$ and $p_2$ - both saddle foci of opposing indices, connected by a heteroclinic trajectory $H$ which lies in the intersection of their one-dimensional manifolds. In more detail, we know $H$ forms the one-dimensional unstable manifold of, say, $p_1$, while it is also the stable manifold of, say, $p_2$. Therefore, by the discussion above above we know $U(p_1)$ is one-dimensional while $S(p_1)$ is two dimensional - and similarly, $U(p_2)$ is two dimensional while $S(p_1)$ is one-dimensional. Moreover, by the continuity of the flow we know the same is true for all initial conditions $x\in\Lambda$ sufficiently close to $p_1$ and $p_2$ - that is, of $x\in\Lambda$ is sufficiently close to $p_1$ the space $U(x)$ would be one dimensional while $S(x)$ would be two dimensional, and when $x$ is sufficiently close to $p_2$ we have the opposite.\\

Now, consider any given $x\in H$ and denote the flow by $\phi_t,t\in\mathbf{R}$. Since the trajectory of $x$ tends to $p_1$ by $x\in\Lambda$ we conclude that for all sufficiently large $t>0$ the space $U(\phi_t(x))$ is one-dimensional while $E(\phi_t(x))$ is two dimensional - and since by the dominated splitting assumption these invariant subspaces vary smoothly with $t$ it follows the same is true for $x$, i.e. $U(x)$ is also one-dimensional while $S(x)$ is two dimensional. On the other hand, since the backwards trajectory of $x$ tends to $p_2$ it follows that for all sufficiently small $t<0$ the space $U(\phi_t(x))$ would be two dimensional while $S(\phi_t(x))$ would be one dimensional - and similarly it would follow $U(x)$ is two dimensional while $S(x)$ is one dimensional. This is a contradiction, which implies there can be no such $\Lambda$ and Th.\ref{nohyp2} now follows.
\end{proof}
\begin{remark}
  It is easy to see the proof above works for any smooth vector field $F$ on $S^3$ that generates a heteroclinic knot connecting two saddle foci - regardless of whether that knot is ambient isotopic to $S^1$ or not. For example, it can also be applied to the vector field $F$ considered in Th.\ref{trefoilor}.
\end{remark}
Th.\ref{nohyp2} has the following heuristic meaning: assume $F$ is a smooth vector field satisfying the assumptions of Prop.\ref{nohyp} and Th.\ref{nohyp2} - i.e., $F$ generates a suspended Smale Horseshoe map while its essential class w.r.t. $H$ is at most finite. Then, by Th.\ref{nohyp2} this set must lie strictly away from the heteroclinic knot $H$ (see the illustration in Fig.\ref{sush}) - and moreover, no matter how we deform the flow $rel$ $H$ we cannot squeeze this horseshoe map "too close" to the heteroclinic knot. Or in other words, whatever hyperbolic, chaotic dynamics $F$ may generate in $S^3\setminus \mathbf{T}$ they must lie inside some plug which isolates them from the heteroclinic knot.\\

\begin{figure}[h]
\centering
\begin{overpic}[width=0.3\textwidth]{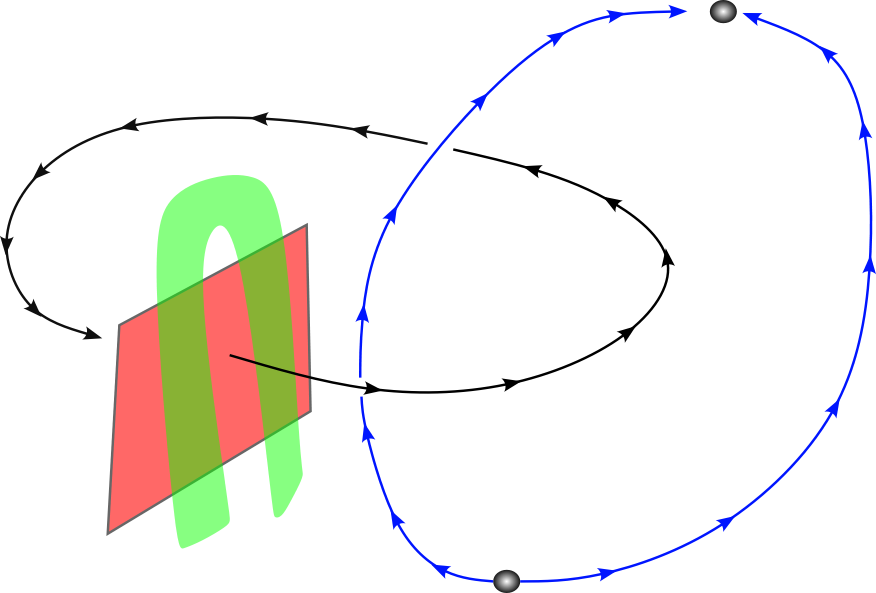}
\end{overpic}
\caption{\textit{A horseshoe suspended in $\mathbf{T}$ away from the heteroclinic knot $H$.}}
\label{sush}
\end{figure}

We conjecture this behavior is typical when $Ess(F)$ is a finite set. In order to state that conjecture precisely we need to introduce another weak form of hyperbolicity, namely the notion of singular hyperbolicity, originally introduced in \cite{MPP}:

\begin{definition}
 \label{singular}   Let $\phi_t,t\in\mathbf{R}$ be a smooth flow on $S^3$ (or some other closed $3$ manifold). We say the dynamics of $F$ on some compact invariant set $\Lambda$ are \textbf{partially hyperbolic} if  we can split the tangent bundle $T_\Lambda=E^s\oplus E^{cu}$ s.t. the following is satisfied: 
    \begin{itemize}
        \item For all $x\in\Lambda$ the space $E^{cu}(x)$ includes the direction $F(x)$, and $E^s(x)$ is one-dimensional and contracting. 
        \item There exist constants $C>0$ and $\lambda>1$ s.t. for every $t>0$ and every $x\in\Lambda$ we have: 
        
        \begin{enumerate}
            \item For $v\in E^s(x)$, $||D_{\phi_t}(x)v||<Ce^{-\lambda t}||v||$.
            \item We always have $||D_{\phi_{1,t}}(x)||||D_{\phi_{2,-t}}(x)||<Ce^{-\lambda t}$, where $D_{\phi_{1,t}}(x)$ denotes the time-$t$ differential restricted to $E^s(x)$ and $D_{\phi_{2,-t}}(x)$ denotes the restriction of the $-t$-time map to $E^{cu}(\phi_t(x))$. 
        \end{enumerate}
    \end{itemize}

 We say a partially hyperbolic set $\Lambda$ for $\phi_t$ is \textbf{volume expanding} if in addition there exists a constant $\kappa>1$ s.t. for all $x\in\Lambda$ $det(D_{\phi_{2,t}}(x))\geq\kappa$ - we say it is \textbf{volume contracting} if for all $t>0$, $x\in\Lambda$ we have $det(D_{\phi_{2,-t}}(x))\geq\kappa$. Finally, if $\phi_t$ is both partially hyperbolic on $\Lambda$, volume expanding (or contracting), and all its fixed points in $M$ are hyperbolic, we say it is \textbf{singularly hyperbolic on} $\Lambda$.
\end{definition}

Note that unlike the definition of Dominated Splitting, in the definition of Partial Hyperbolicity or Singular Hyperbolicity we do not require the invariant subspaces to vary smoothly with the flow, although we do require some domination condition on the differentials (note the arguments used to prove Th.\ref{nohyp2} in the Dominated Splitting case will not work in the singular hyperbolic case). With these ideas in mind we conjecture the following is true:

\begin{conj}
    \label{nohyp3}
    Assume $F$ is a smooth vector field on $S^3$ which generates a heteroclinic knot $H$ s.t. $Ess(F)$ w.r.t. $H$ is finite (and possibly empty). Then, there is no invariant set $I\subseteq S^3\setminus H$ on which $F$ is singular hyperbolic and $H\subseteq \overline{I}$.
\end{conj}
This conjecture is motivated by two heuristics - the first one is that per our Th.\ref{nohyp2} we know that whenever $Ess(F)$ is finite is no reason to think the heteroclinic knot can suspend complex dynamics. The second heuristic comes from certain facts in the classical theory for hyperbolic dynamics. To introduce them recall that under certain conditions hyperbolic dynamical systems are structurally stable (see \cite{Hu}, \cite{Shao} and \cite{Hay} for the details) - and that given any boundaryless, compact three-dimensional manifold $M$ we can always $C^1-$approximate smooth vector fields on $M$ by either uniformly hyperbolic, singularly hyperbolic, or homoclinic vector fields (see \cite{CY} and \cite{ARH} for the precise formulation). Due to the extreme instability of heteroclinic knots under perturbations and due to the intuitive robustness of hyperbolic and singular hyperbolic dynamics it seems intuitive one should not expect heteroclinic knots to suspend such dynamics. Therefore, when $Ess(F)$ is finite we will expect any complex, hyperbolic or singular hyperbolic dynamics to lie strictly away from the heteroclinic knot.\\
\begin{figure}[h]
\centering
\begin{overpic}[width=0.5\textwidth]{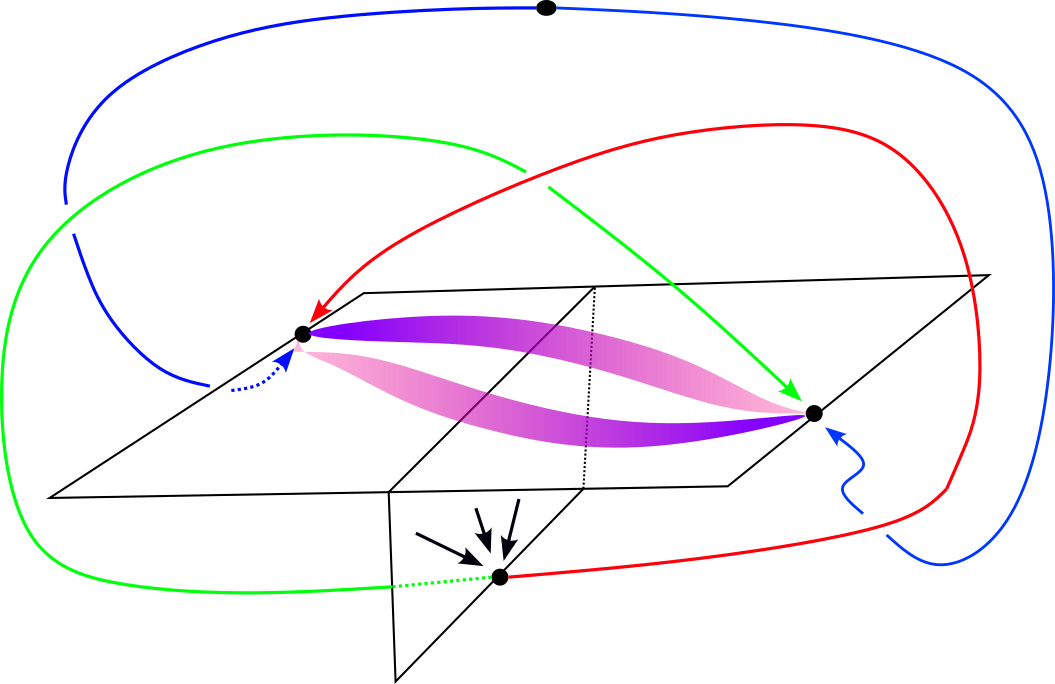}
\put(220,230){$R_{1,p}$}
\put(500,600){$\infty$}
\put(220,120){ $W^s(0)$}
\put(220,330){$p^-$}
\put(790,250){$p^+$}
\put(480,50){$0$}
\put(560,250){$R_{2,p}$}
\end{overpic}
\caption{\textit{A heteroclinic knot $H$ which is ambient isotopic to $S^1$, whose fixed points are configured precisely like those of the Lorenz system (in this scenario, $p^\pm$ are saddle foci). It is easy to see the Orbit Index of every periodic orbit is $-1$. Unlike the Lorenz system however, this vector field is not symmetric.}}
\label{simple}
\end{figure}
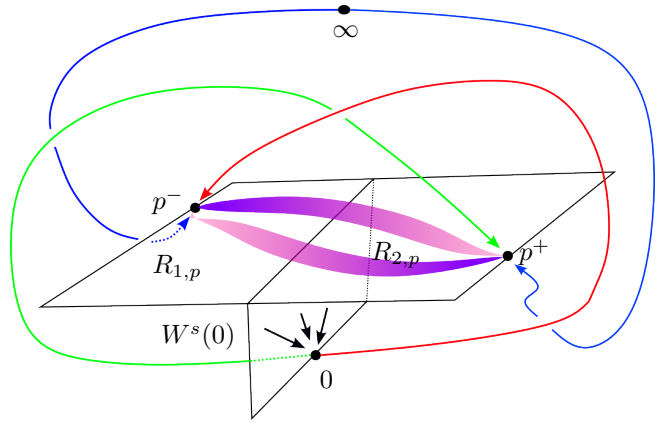

Having conjectured Conjecture \ref{nohyp3} we conclude this section by answering the following question - let $F$ be a smooth vector field of $S^3$ with a finite number of fixed points, all connected by a heteroclinic knot $H$. Then, assuming $H$ is ambient isotopic to $S^1$ (i.e., its knot type is the unknot), does it follow that $Ess(F)$ is finite (or empty)? Or, in other words, does a simple topology for a given heteroclinic knot translate into a finite (or empty) Essential Class of periodic orbits for $F$?\\

Surprisingly enough, the answer to this question is strictly negative. To illustrate, let $F$ be a smooth vector field of $S^3$ which generates heteroclinic knot $H$ configured as in Fig.\ref{simple} - i.e., $F$ generates two saddle foci $p^\pm$of Poincare Index  $-1$ and two fixed points of Poincare Index $1$: a repeller at $\infty$ and a real saddle at $0$, all connected by a heteroclinic knot $H$ as in Fig.\ref{simple}. It is easy to see the heteroclinic knot in Fig.\ref{simple} is ambient isotopic to $S^1$, the unknot - and it is also easy to see we can choose $F$ s.t. there exists a cross-section with a first-return map as sketched in Fig.\ref{simple}. That is, the knot and its configuration with the cross-section force the first return map to behave like some singular Fake Horseshoe - which implies $F$ has infinitely many periodic orbits. In addition, since this first return map is conjugate to the one we constructed in the proof of Th.\ref{trefoil2} (see Fig.\ref{deform3}) it is easy to see we can apply the same analysis to it - i.e. we can embed the first-return map inside a Pseudo-Anosov map of the punctured torus s.t. the assumptions of Th.\ref{orbipers} hold. This implies the following:

\begin{theorem}
    \label{notrivial} There exists a smooth vector field $F$ of $S^3$ which generates a heteroclinic knot $H$ ambient isotopic to $S^1$, s.t. $Ess(F)$ w.r.t. $H$ includes infinitely many periodic orbits. In other words, even when $H$ is ambient isotopic to $S^1$ the Essential Class of $F$ can still be infinite.
\end{theorem}

\section{Discussion}

Having used both Th.\ref{orbipers} and the Orbit Index Theory to study the Essential Dynamics of three-dimensional flows in this section we outline how the results of this paper can possibly be extended - and in particular, discuss how and where they possibly fit within the larger theory of low-dimensional dynamics. As stated at the introduction our results above were strongly motivated by the theory of Surface Dynamics as presented in \cite{Bo} - therefore we begin this section by comparing our results with the said theory. Following that we discuss our results in the wider context of hyperbolic dynamical systems, after which we heuristically connect them with the "Chaotic Hypothesis" known from \cite{gal}.\\

To begin, first recall that given a surface homeomorphism $h:S\to S$ we can define its Essential Class of periodic orbits in $S$ as the collection of periodic orbits which persist under isotopies (see Def.\ref{essential2}). Now, recall the Thurston-Nielsen Classification Theorem which states that if $S$ is a surface of negative Euler characteristic any homeomorphism $h:S\to S$ is isotopic to precisely one of the following (see Th.\ref{thurston-nielsen}):
\begin{enumerate}
    \item A periodic homeomorphism $f:S\to S$, i.e., there exists some $n>0$ s.t. $f^n=Id$. In that case, whatever complex dynamics $h$ may have in $S$ can be removed by an isotopy of $h$. As such, in this case the Essential Class of periodic orbits for $h$ in $S$ is at most finite.
    \item A Pseudo-Anosov map $f:S\to S$ with infinitely many periodic orbits in $S$. In this case the complex dynamics of $f$ persist as $f$ is isotoped back to $h$ (see Th.\ref{stability}) - i.e., the dynamics of $h$ are complex at the very least like those of $f$, and the Essential Class for $h$ is $S$ is infinite.
    \item A reducible map $f:S\to S$ -  i.e., $S$ can be decomposed into a finite number of surfaces $S_1,...,S_n$ of negative Euler characteristic, glued to one another along some invariant curve(s) in $S$. Moreover, for every $i=1,...,n$ $f|_{S_i}$ is either periodic or Pseudo-Anosov, and the Essential Class is as follows.
\end{enumerate}
 
As we have proven in Th.\ref{trefoilor} and Th.\ref{trefoil2} there exist vector fields on $S^3$ whose Essential Class of periodic dynamics (w.r.t. some heteroclinic knot $H$) is infinite - and in Th.\ref{nohyp} we have also proven the existence of vector fields whose Essential Class is at most a finite collection of periodic orbits. In other words, we have proven the existence of analogues to both Pseudo-Anosov and Periodic maps. We believe these analogies are not mere coincidence. In fact, we believe Th.\ref{orbipers} and Th.\ref{pers13} are both part of a larger theory which ties the topology of a given three-dimensional phase space with the possible dynamics which can exist on it - similarly to how the Arnold Conjecture connects between the dynamical properties of a given symplectomorphism and the phase space it is defined on (see Conjecture 2.1 in \cite{Arn}). This motivates us to conjecture the following:
\begin{conj}
    \label{genthunie} Let $F$ be a smooth vector field on $S^3$ with fixed points $\{x_1,...,x_n\}$, and let $H$ be a one-dimensional set invariant under the flow, s.t. each component $C$ of $H\setminus\{x_1,...,x_n\}$ satisfies one of the following:
    \begin{enumerate}
        \item $C$ is a homeomorphic to $S^1$.
        \item $C$ is homeomorphic to $(0,1)$.
        \item For every $1\leq i\leq n$ there exists some $C$ s.t. $x_i\in\overline{C}$ - i.e., all the fixed points of $F$ are on $H$.
    \end{enumerate}

Now, recall we denote the Orbit Index of a periodic orbit $T$ of $F$ by $i(T)$ (see Def.\ref{index22}). Then, under the assumptions above $F$ can be deformed $rel$ $H$ to a smooth vector field $G$ s.t. $G$ satisfies precisely one of the following:
 \begin{enumerate}
     \item $G$ is \textbf{ordered} - $G$ generates a finite number of periodic orbits, $T_1,...,T_n$, s.t. $i(T_j)=1$ for all $j=1,...,n$. Moreover, every $T_j$ is either attracting or repelling and $Ess(F)$ w.r.t. $H$ is simply $\{T_1,...,T_n\}$.
     \item $G$ is \textbf{chaotic} - $G$ generates infinitely many periodic orbits, all of which have Orbit Indices $-1$ and $0$. Moreover, the periodic orbits of $G$ all persist under deformations $rel$ $H$ without changing their knot type or collapsing to one another - in that case $Ess(F)$ w.r.t. $H$ is infinite.
     \item $G$ is \textbf{mixed} - there exists some surface $S$ invariant under the flow of $G$ $S\cap{H}\subseteq\{x_1,...,x_n\}$, s.t. on each component of $S^3\setminus(H\cup S)$ the flow is either ordered or chaotic.
 \end{enumerate}
\end{conj}
If true, this conjecture generalizes the Thurston-Nielsen Classification Theorem and Th.\ref{stability} to three-dimensional flows - as well as greatly extends Th.\ref{orbipers}. It can be interpreted as follows - it is easy to see that whenever we can deform $F$ $rel$ $H$ to some singular-hyperbolic vector field $G$ in $S^3\setminus H$ by Def.\ref{index22} the Orbit Index of every periodic orbit for $G$ would be either $-1$ or $0$. Therefore, if the conjecture above is true it would imply the dynamics of $F$ are essentially those of a deformed hyperbolic vector field - or qualitatively speaking, the dynamics of $G$ serve as a topological lower bound for the dynamical complexity of $F$, much like Pseudo-Anosov maps are dynamically minimal (see Th.\ref{stability}). Moreover, in the mixed case this Conjecture can be interpreted as a decomposition theorem for flows - i.e., any three dimensional flow in $S^3\setminus H$ which can be deformed $rel$ $H$ to a mixed vector field $G$ is essentially the gluing two dynamical systems along some invariant surface (say, some two-dimensional invariant manifold). \\

The discussion above implies that a positive answer to Conjecture \ref{genthunie} could aid the analysis of chaotic attractors whose dynamics lie in the Essential Class. However, even if true Conjecture \ref{genthunie} does not give us tools to analyze chaotic attractors whose dynamics are removable - i.e., chaotic attractors whose dynamics can be destroyed by deformation $rel$ $H$ to an ordered vector field. Inspired by Conjecture \ref{genthunie} we now conjecture how to overcome this difficulty - to this end, assume $G$ is an ordered vector field and with previous notations let let $T_1,...,T_k$ denote the periodic orbits in $Ess(G)$. By $i(T_j)=1, j=1,...,k$ we know that for a generic choice of $G$ we may assume these periodic orbits are either attractors or repellers - which implies the generic existence of some tubular neighborhoods $T'_1,...,T'_k$ for $T_1,...,T_k$ s.t. $G$ is transverse to each $\partial T'_i$, $i=1,...,k$. With these ideas in mind, inspired by Th.\ref{orbipers} and Conjecture \ref{nohyp3} we conjecture the following:

\begin{conj}
\label{rem}  As we deform $G$ $rel$ $H$ back to $F$, $F$ remains transverse to the (possibly knotted) torus $\partial T'_i$, $i=1,...,k$.
\end{conj}
If true, this conjecture implies the Conley Index of $F$ and $G$ inside each $T'_i$, $i=1,...,k$ is the same. For the sake of completeness we remark the Conley Index (see \cite{con}) is an algebraic invariant which allows one to describe the dynamical complexity of an isolated invariant set (see \cite{Chu} for a survey of the field and the precise definitions). As such, Conjecture \ref{rem} states the dynamical complexity of $F$ can be partially described by answering the question which invariant sets can exist inside $T'_i$, $i=1,...,k$ while having the same Conley Index of a periodic orbit.\\

If true, another implication of Conjecture \ref{orbipers} would be that it allows us to think of all the singular hyperbolic vector fields defined in $S^3\setminus H$ as the space of all possible idealized models for chaotic behavior. This relates it to another famous conjecture, known as the \textbf{Chaotic Hypothesis} (see \cite{gal}). To state that hypothesis, assume $F$ is some smooth vector field of $\mathbf{R}^3$ which generates a chaotic attractor $A$. Then, the Chaotic Hypothesis conjectures the following:

\begin{conj}
    The flow around $A$ is orbitally equivalent to that of a hyperbolic flow.
\end{conj}
In more detail, the interest in the said hypothesis arose from the fact that if correct, the statistical dynamics around $A$ could be studied with an SRB measure (see Sect.2 in \cite{SRB} for the precise definition) - which, to our knowledge, is unknown to exist in the absence of any hyperbolicity assumptions on $F$. Even though our methods in this paper are topological and not measurable, it is easy to see that Th.\ref{orbipers} gives a partial answer to the Chaotic Hypothesis - in the sense that whenever we have a singular-hyperbolic vector field defined on the complement of a heteroclinic knot $H$, one should expect at least some of its dynamical complexity to persist under deformations $rel$ $H$. As such, in the spirit of the Chaotic Hypothesis we conjecture the following with which we conclude this paper:

\begin{conj}
    Let $F$ be a smooth vector field of $S^3$ and let $H$ denote some one-dimensional set as in Conjecture \ref{genthunie}. Moreover, assume all the fixed points of $F$ lie on $H$, and that $F$ can be deformed $rel$ $H$ to a chaotic vector field $G$. Then, the dynamics of $F$ on $\overline{Ess(F)}$ are orbitally equivalent to those of a hyperbolic flow.
\end{conj}

\section{Appendix $A$ - approximating a DA Bifurcation from $K$}

In this appendix we illustrate how one can approximate a bifurcation scenario known as "derived from Anosov" using vector fields from $K$, the family of vector fields considered in Section \ref{orbitin} (see Prop.\ref{propk}). The setting of such bifurcations, originally described in \cite{S}, is as follows: let $T$ be a hyperbolic periodic orbit (i.e., an orbit whose Orbit Index is $-1$) for some smooth vector field $F$ and let $N$ be a tubular neighborhood of $T$ s.t. $N$ is an isolating torus for $T$ - i.e., there exists a cross-section $S\subseteq N$ transverse to both $G_t$ and $N$ s.t. the first-return map $f:S\to S$ satisfies $\{x\in S|f(x)=x\}=T\cap S$. We deform $F$ inside $N$ with a pitchfork-like bifurcation - that is, we split $T$ along its stable manifold into a repelling periodic orbit, $T_1$, and two hyperbolic saddles, $T_2$ and $T_3$ as illustrated in Fig.\ref{DA1} (for more details see the description in Lemma 2.2.11 in \cite{KNOTBOOK}).\\

\begin{figure}[h]
\centering
\begin{overpic}[width=0.4\textwidth]{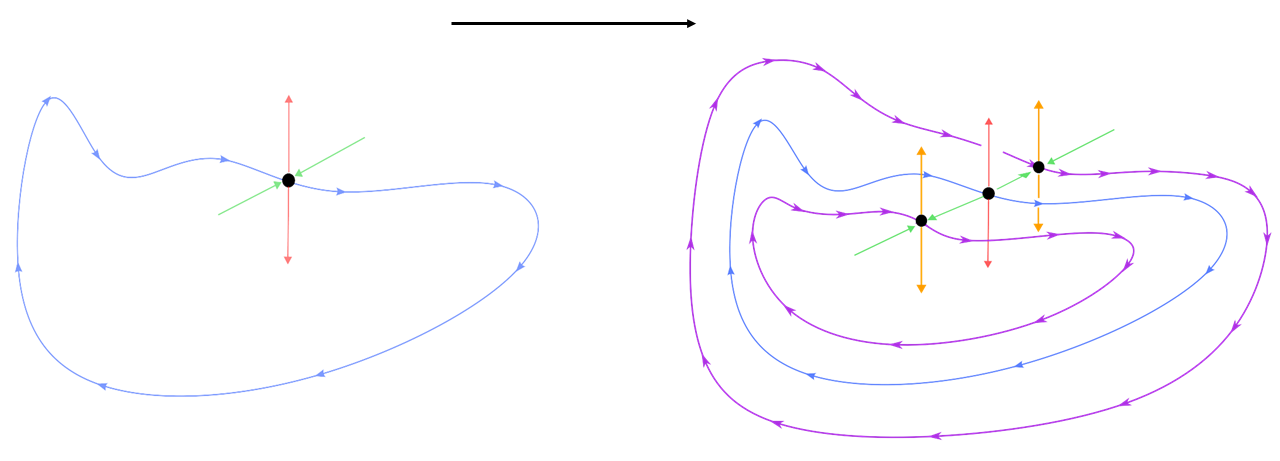}
\end{overpic}
\caption{\textit{Splitting a saddle periodic orbit to a repeller and two saddles along the stable manifold.}}
\label{DA1}
\end{figure}

Let us formalize the proccess described above in a smooth curve of vector fields $\dot{s}=F_t(s)$, $t\in[0,1]$, $s\in N$ s.t. $F_0=F$ - moreover, let us also adopt the convention that the DA bifurcation of $T$ occurs at $t=\frac{1}{2}$. Let us now consider a curve of vector fields $\dot{s}=G_t(s)$, $t\in[0,1]$, $s\in N$ satisfying the following (see Fig.\ref{DA2}):

\begin{itemize}
    \item $\{G_t\}_{t\in[0,1]}\subseteq K$ - i.e., the only bifurcations periodic orbits for $G_t$ can undergo are Hopf, Period-Doubling and Saddle Node (see Def.\ref{type}).
    \item For all $0\leq t\leq1$ the vector field $G_t$ has a periodic orbit in $N$, $T_2'$, which varies smoothly with $t$ and satisfies $i(T_2')=-1$. Moreover, for $0\leq t<\frac{1}{2}$ $N$ is an isolating neighborhood for $T'_2$ as defined above (w.r.t. the vector field $G_t$).
    \item At $t=\frac{1}{2}$ the vector field $G_\frac{1}{2}$ generates a second periodic orbit in $N$ in addition to $T'_2$, a Saddle Node (Type $I$) which we denote by $T'$. Note we can choose $G_\frac{1}{2}$ s.t. $T'$ and $T'_2$ are arbitrarily close to one another. 
    \item For $t>\frac{1}{2}$, $G_t$ has three periodic orbits in $N$ - $T'_2$ (which persists as we vary $t$) along with $T'_1$ a repeller, and $T'_3$, a hyperbolic. In particular, $T'_1$ and $T'_3$ split from $T'$ following the saddle-node bifurcation (see the illustration in Fig.\ref{DA2}). 
\end{itemize}

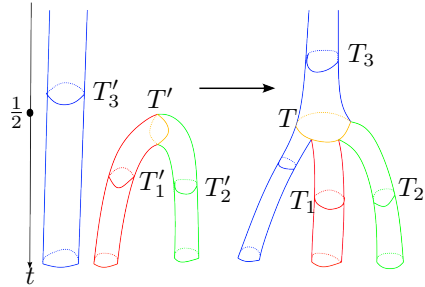
\begin{figure}[h]
\centering
\begin{overpic}[width=0.3\textwidth]{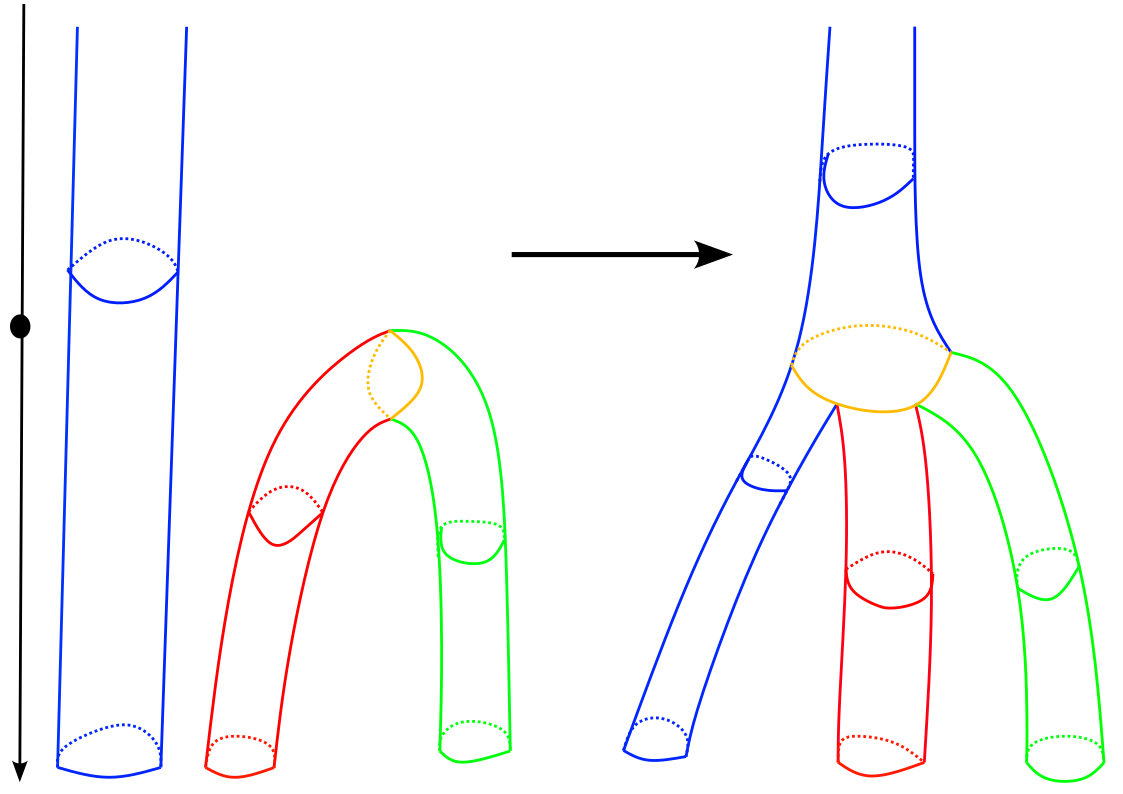}
\put(5,-30){$t$}
\put(300,215){$T'_1$}
\put(180,450){$T'_3$}
\put(320,430){$T'$}
\put(-40,400){$\frac{1}{2}$}
\put(470,200){$T'_2$}
\put(690,165){$T_1$}
\put(840,550){$T_3$}
\put(660,380){$T$}
\put(970,200){$T_2$}
\end{overpic}
\caption{\textit{Approximating a $DA$ bifurcation (on the right) by vector fields in $K$ (on the left). The yellow orbit $T$ and $T'$ correspond to bifurcation orbit, while the (generic) Orbit Indices on $T'_3$ and $T'_2$ are $-1$ while on $T'_1$ it is $1$. The limit of $T'_i$ is $T_i,i=1,2,3$.}}
\label{DA2}
\end{figure}

It is easy to see that given any $\epsilon>0$ we can choose $\{G_t\}_{t\in[0,1]}$ s.t. for all $t\in[0,1]$ we have $d_3(F_t,G_t)<\epsilon$ within $N$ (where $d_3$ denotes the $C^3-$metric on $N$). And to summarize, we have just shown we can approximate a DA Bifurcation scenario via a curve of vector fields in $K$ - for the approximation of more complex bifurcation scenario, see Section 2 in \cite{PY4}. 

\section{Appendix $B$ - the Orbit Index of suspended Horseshoes}

In this Appendix we study the Orbit Index of periodic orbits for flows generated by suspending the classical Smale Horseshoe map (see \cite{S}) - or more generally, the periodic orbits created by suspending any two-dimensional diffeomorphism whose dynamics can be conjugated to those of  the double-sided shift $\sigma:\{1,2\}^\mathbf{Z}\to\{1,2\}^\mathbf{Z}$. In particular, in this Appendix we prove Prop.\ref{dens1} from Section \ref{orbitin} and show the suspension of the classic Smale Horseshoe map always generates infinitely many periodic orbits whose Orbit Index is $-1$. As the proof would apply to a much larger class of maps, we will conclude this Appendix by briefly discussing how our can possibly be generalized.\\

    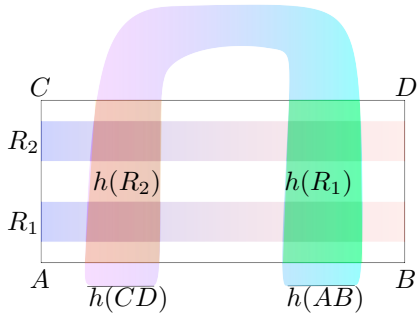
\begin{figure}[h]
\centering
\begin{overpic}[width=0.35\textwidth]{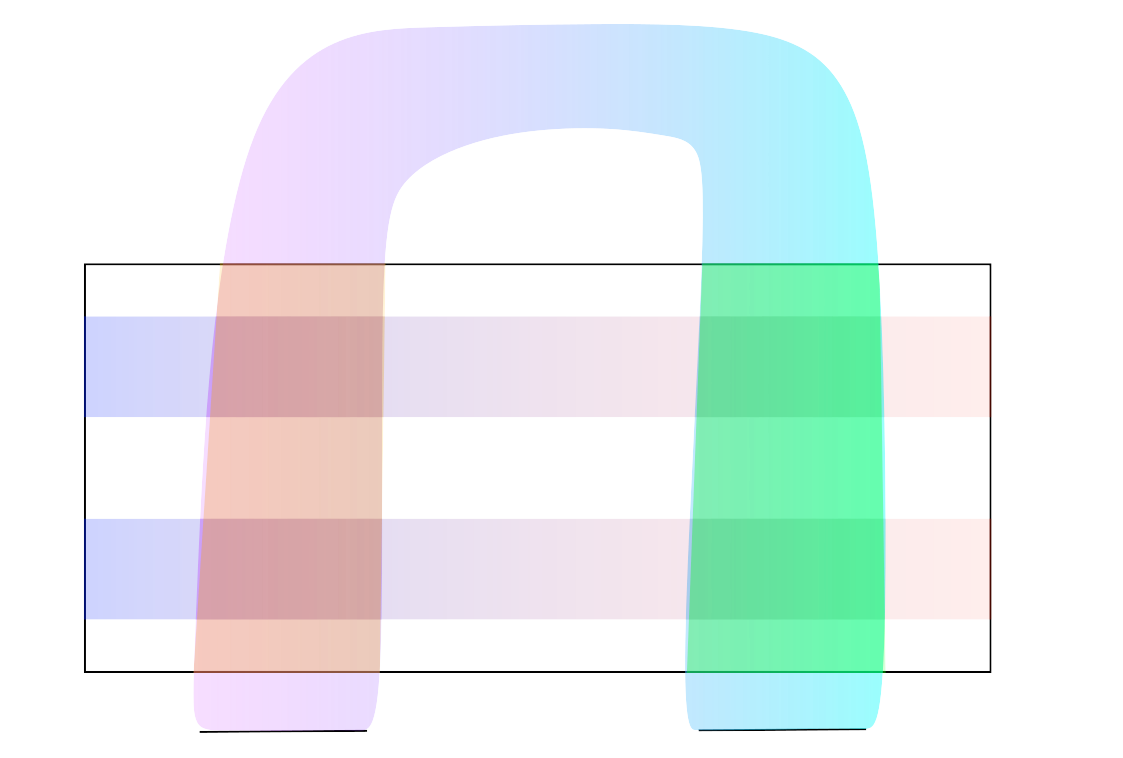}

\put(50,20){$A$}
\put(180,-20){$h(CD)$}
\put(620,-20){$h(AB)$}
\put(860,20){$B$}
\put(190,240){$h(R_2)$}
\put(615,240){$h(R_1)$}
\put(50,450){$C$}
\put(860,450){$D$}
\put(0,330){$R_2$}
\put(0,150){$R_1$}
\end{overpic}
\caption{\textit{A Smale Horseshoe map.}}
\label{hors}
\end{figure}

We first recall the following terminology, introduced in Section \ref{orbitin}. Assume $F$ is a smooth vector field on some three-manifold $M$ and let $T$ be an isolated, Type $0$ periodic orbit for $F$ - i.e., the Orbit Index of $T$ is well defined (see Def.\ref{index1} and Def.\ref{index22}), Then, we refer to $T$ as follows:

\begin{enumerate}
    \item  $T$ is said to \textbf{Hyperbolic} whenever $i(T)=-1$.
    \item $T$ is \textbf{Möbius} whenever $i(T)=0$ (see the illustration in Fig.\ref{MOB}).
    \item $T$ is \textbf{Elliptic} when $i(T)=1$ .
\end{enumerate}

With this formalism in mind, we prove:

\begin{proposition}
\label{dens}    Let $ABCD$ be a topological rectangle with vertices $ABCD$, and assume $h:ABCD\to\mathbf{R}^2$ is either a Smale Horseshoe map or a Fake Horseshoe map (see Fig.\ref{horss}). Then, the following is satisfied:

\begin{enumerate}
    \item When $h$ is a Smale Horseshoe map, every periodic orbit w.r.t. any smooth suspension of  $h$ in $S^3$ is either Hyperbolic or Möbius. When $h$ is a Fake Horseshoe map every periodic orbit w.r.t. the suspension is Hyperbolic.
    \item The periodic orbits for $h$ whose is hyperbolic are dense in the invariant set of $h$ in $ABCD$.
\end{enumerate}
\end{proposition}

\begin{proof}
We only prove the assertion when $h:ABCD\to\mathbf{R}^2$ is a Smale Horseshoe map, the argument for the Fake Horseshoe map is similar. To begin, recall that $h^{-1}(ABCD)\cap ABCD$ consists of two rectangles, $R_1$ and $R_2$ s.t. the following is satisfied (see the illustration in Fig.\ref{hors}):
    \begin{itemize}
        \item The $AB$ side lies below of $R_1$, and $h(R_1)$ is stretched upwards, connecting the $AB$ and $CD$ sides - i.e., the arc on $\partial R_1\setminus(AC\cup BD)$ which lies above the $AB$ side is mapped to the $AB$ side, while the arc on $\partial R_1\setminus(AC\cup BD)$ which is closer to the $CD$ side is mapped to the $CD$ side.
        \item The $CD$ side lies above $R_2$ while $h(R_2)$ is rotated, then stretched downwards, connecting the $CD$ and $AB$ sides - i.e., the arc on $\partial R_2\setminus(AC\cup BD)$ which lies directly below the $CD$ side is mapped to the $AB$ side, while the arc on $\partial R_2\setminus(AC\cup BD)$ closer to the $AB$ side is mapped to the $CD$ side.
    \end{itemize}
    
Further recall that at every point $x$ in $R_1\cup R_2$ the differential $D_h(x)$ has one eigenvalue in $(-1,0)\cup(0,1)$ and another in $(-\infty,-1)\cup(1,\infty)$. As $h$ is orientation preserving, by this discussion we conclude that for $x\in R_1$ the differential $D_h(x)$ has one eigenvalue in $(0,1)$ and another in $(1,\infty)$, while for $x\in R_2$ it has one eigenvalue in $(-1,0)$ and another in $(-\infty,-1)$.\\

 \begin{figure}[h]
\centering
\begin{overpic}[width=0.5\textwidth]{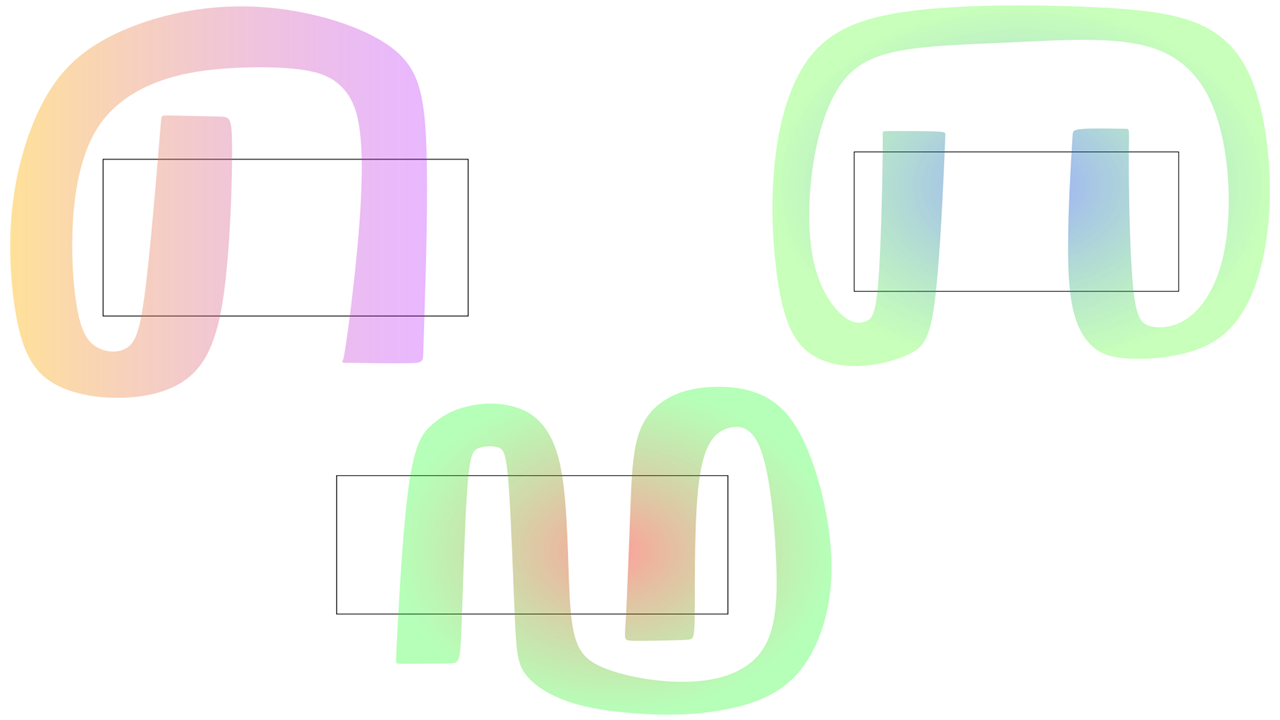}

\end{overpic}
\caption{\textit{Several variants on the Smale Horseshoe map. Using a similar argument to the one used to prove Prop.\ref{dens} it is easy to see any suspension of each of them results in a dense collection of periodic orbits of Orbit Index $-1$.}}
\label{hors2}
\end{figure}

To continue, let $P$ denote the collection of periodic orbits for $h$ in $R_1\cup R_2$ and let us recall there exists a homeomorphism between them and the collection of periodic sequences in $\{1,2\}^\mathbf{Z}$ (see \cite{S}). Now, let us suspend $h$ into a flow and let $x\in R_1\cup R_2$ be a periodic orbit for $h$ of minimal period $k$ - per definition, $x$ lies on some periodic orbit for the suspension flow. Moreover, it is easy to see that matter how we suspend $h$, as the minimal period for $x$ is $k$ the differential $D(x)$ of any local first first-return map at $x$ is given by $D(x)=\Pi_{j=0}^{k-1}D_{h}(h^j(x))$ - where $D_h$ denotes the differential of $h$.\\

Consequentially, the sign of either one of the eigenvalues of $D(x)$ is given by $(-1)^{r(x)}$ - where $r(x)$ is an integer denoting number of times the orbit of $x$ visits $R_2$ before returning to $x$. Therefore, whenever $r(x)$ is even $D(x)$ has two positive eigenvalues - in which case by Def.\ref{index1} we know the periodic orbit containing $x$is Hyperbolic (and similarly, whenever $r(x)$ is odd it is Möbius). In other words, we have proven that every periodic orbit in the suspension of $h$ has Orbit Index which is either $-1$ or $0$ - i.e., it is either Hyperbolic or Möbius.\\

It remains to prove the periodic orbits for $h$ in $ABCD$ which are suspended into Hyperbolic periodic orbits are dense in the invariant set of $h$ in $ABCD$. To do so, recall that if $x$ is suspended into a Hyperbolic periodic orbit, i.e., $r(x)$ is even, then it corresponds to some unique periodic symbol in $\{1,2\}^\mathbf{Z}$ whose minimal period includes an even number of $2'$s - and moreover, any such periodic symbol in $\{1,2\}^\mathbf{Z}$ corresponds to a unique periodic point for $h$ in $ABCD$. Since this collection of periodic symbols is dense in $\{1,2\}^\mathbf{Z}$ it follows the periodic points for $h$ corresponding to it is dense in the invariant set of $h$ in $ABCD$. All in all, the proof Prop.\ref{dens} is now complete.
\end{proof}
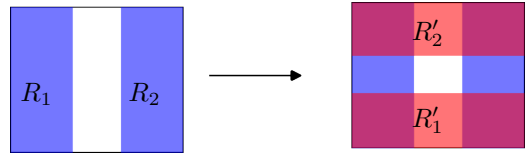
\begin{figure}[h]
\centering
\begin{overpic}[width=0.4\textwidth]{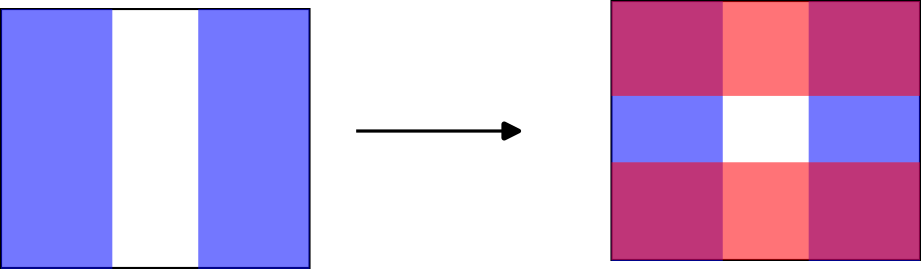}
\put(780,220){$R'_2$}
\put(780,50){$R'_1$}
\put(230,100){$R_2$}
\put(20,100){$R_1$}

\end{overpic}
\caption{\textit{The Baker's map, where $R'_i$ denotes the image of $R_i,i=1,2$. If the dynamics of some first-return map can be reduced to it, every periodic orbit generated by its suspension would be hyperbolic.}}
\label{baker}
\end{figure}

At this stage we remark the proof of Prop.\ref{dens} can be modified to extend to other variants of the Smale Horseshoe map (and many variants on the Baker map), as appear in Fig.\ref{hors2} and Fig.\ref{baker}.  This motivates us to ask the following: let $f:\mathbf{R}^2\to\mathbf{R}^2$ be an orientation-preserving diffeomorphism and assume $I\subseteq\mathbf{R}^2$ is some compact $f-$invariant subset on which $f$ is conjugate to some subshift of a finite type. Now, let $x\in I$ be periodic and let $T$ be a periodic orbit for the suspension flow of $f$ - what can we say about the Orbit Index of $T$? Or, in other words, assuming we can extend the suspension of $f$ into some smooth flow of $S^3$, what Orbit Indices should we expect to find in the suspension of $I$?\\

We will not address this question in this Appendix as it is beyond the scope of this study - however, we would like to state that heuristically we would expect $T$ to be either Hyperbolic or Möbius, at least when $f$ is conjugate to the entire double-sided shift on $I$. The reason we expect it to be so is due to the following "extreme" example in Fig.\ref{nosmal}. In the said image we see a variant on the Baker's map for which it is easy to prove the differential at each point in the invariant set in $ABCD$ has one positive and one negative eigenvalue - which, at least heuristically, could lead to the creation of infinitely many Möbius periodic orbits under any suspension. However, since $S^3$ is orientable - and hence any smooth flow on it is orientation preserving - there cannot be a flow on $S^3$ which generates such a first-return map. This discussion motivates us to conjecture the following:

\begin{conj}
\label{con1}    There is no $C^k$ vector fields $F$ on $S^3$, $k\geq3$, which satisfies both properties listed below:
\begin{enumerate}
    \item There exists a compact, connected invariant set $\Lambda$ for $F$ and a cross-section $S$ for $\Lambda$ s.t. the following holds:
    \begin{itemize}
        \item The first-return map $h:S\to S$ is continuous on $S\cap\Lambda$.
        \item $h$ has infinitely many periodic points in $S\cap\Lambda$.
        \item The dynamics of $h$ on $S\cap\Lambda$ can be factored to those of some subshift of finite type.
    \end{itemize}
    \item Every periodic orbit $T\subseteq\Lambda$ satisfies $i(T)=0$.
\end{enumerate}
\end{conj}
    \begin{figure}[h]
\centering
\begin{overpic}[width=0.35\textwidth]{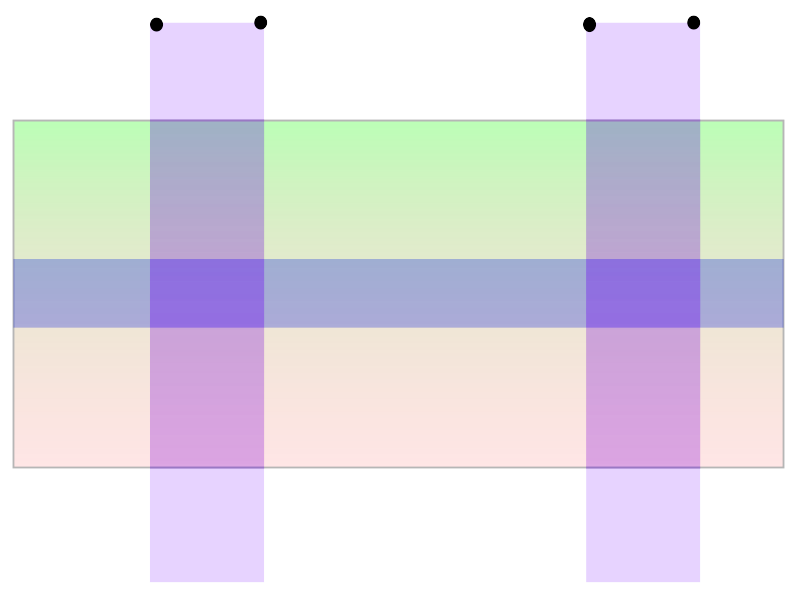}

\put(0,115){$A$}
\put(840,740){$h(B)$}
\put(640,740){$h(A)$}
\put(950,115){$B$}
\put(100,740){$h(D)$}
\put(285,740){$h(C)$}
\put(0,600){$C$}
\put(950,600){$D$}
\put(30,480){$R_2$}
\put(30,220){$R_1$}
\end{overpic}
\caption{\textit{A variant on the Baker transformation, which tears $ABCD$ along the blue rectangle in the center, and stretches $R_1$ and $R_2$ as appears above. One can verify the map is orientation reversing on both $R_1$ and $R_2$.}}
\label{nosmal}
\end{figure}
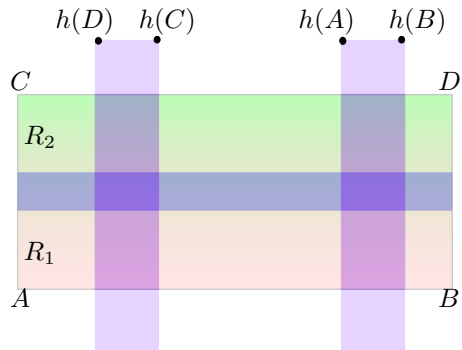

From the arguments above we know that such flows, whether they exist or not, cannot generate dynamics which include a suspension of a Smale Horseshoe. We do not know how to solve this conjecture - however, as a first step towards such a solution, we would suggest one should first tackle the following problem:

\begin{conj}
\label{con2} Let $D$ be a disc, let $\sigma:\{1,2\}^\mathbf{Z}\to\{1,2\}^\mathbf{Z}$ denote the double-sided shift, and let $f:D\to D$ be an orientation preserving diffeomorphism s.t. the following is satisfied:
\begin{enumerate}
    \item There exists some $f-$invariant subset $I\subseteq D$, a $\sigma-$invariant $\Sigma\subseteq\{1,2\}^\mathbf{Z}$ and a homeomorphism $\pi:I\to\Sigma$ s.t. $\pi\circ f=\sigma\circ\pi$.
    \item The periodic symbols are dense in $\Sigma$.
\end{enumerate}
Then, there exists a dense collection of periodic orbits in $I$ s.t. any suspension of $f$ turns them to hyperbolic periodic orbits.
\end{conj}

Assuming Conjecture \ref{con2} is proven, a possible heuristic for proving Conjecture \ref{con1} is the following: recall that given a smooth vector field $F$ of $\mathbf{R}^3$, by Th.1-4 in \cite{Kris} there exists a cross-section $S$ transverse to all the periodic orbits of $F$. Then, by sowing together all the local first-return maps of these periodic orbits (and possibly ascribing symbolic dynamics to them via \cite{Bow}) one can hope to reduce this new map to a disc diffeomorphism as in Conjecture \ref{con2} - from which the assertion would then follow. Conversely, given a counterexample to Conjecture \ref{con2} one can suspend it and derive a counterexample to Conjecture \ref{con1}.
\printbibliography
\end{document}